\documentclass[12pt]{amsart}
\usepackage{amssymb,wasysym,verbatim,graphicx,transparent}
\usepackage{graphics,epsfig,psfrag} 
\usepackage{pb-diagram,hyperref,graphics,epsfig,psfrag}
\usepackage{nameref}
\usepackage{caption}
\usepackage{subcaption}
\usepackage{color}
\usepackage{braket,tikz-cd}
\definecolor{darkred}{rgb}{0.5,0,0}
\definecolor{darkgreen}{rgb}{0,0.5,0}
\definecolor{darkblue}{rgb}{0,0,0.5}
\hypersetup{ colorlinks,
linkcolor=darkblue,
filecolor=darkgreen,
urlcolor=darkred,
citecolor=darkblue }
\usepackage[inline]{enumitem}

\newtheorem{theorem}{Theorem}[section]

\newtheorem{corollary}[theorem]{Corollary}

\newtheorem{proposition}[theorem]{Proposition}
\newtheorem{lemma}[theorem]{Lemma}
\newtheorem{lem}[theorem]{}
\theoremstyle{definition}
\newtheorem{definition}[theorem]{Definition}

\theoremstyle{remark}
\newtheorem{remark}[theorem]{Remark}
\newtheorem{example}[theorem]{Example}
\newcommand{\blem}{\begin{lem} \rm}
\newcommand{\elem}{\end{lem}}

\newcommand\M{\mathcal{M}}
\newcommand\D{\mathcal{D}}

\renewcommand\S{\mathcal{S}}

\newcommand{\T}{\mathcal{T}}
\newcommand{\J}{\mathcal{J}}

\newcommand{\U}{\mathcal{U}}

\newcommand{\N}{\mathbb{N}}
\newcommand{\R}{\mathbb{R}}
\newcommand{\HH}{\mathbb{H}}
\newcommand{\RR}{\mathcal{R}}

\newcommand{\CC}{\mathbb{C}}
\newcommand{\C}{\mathbb{C}}

\newcommand{\Z}{\mathbb{Z}}

\newcommand{\ddt}{\frac{d}{dt}}

\newcommand{\dds}{\frac{d}{ds}}

\renewcommand{\P}{\mathbb{P}}
\newcommand{\PP}{\mathcal{P}}

\newcommand\lie[1]{\mathfrak{#1}}

\newcommand{\so}{\lie{so}}

\newcommand{\on}{\operatorname}

\newcommand{\si}{\on{si}}

\renewcommand{\span}{\on{span}}

\newcommand{\univ}{\on{univ}} 
\newcommand{\wt}{\on{wt}} 

\newcommand{\dd}{\partial}

\newcommand{\ainfty}{{$A_\infty$\ }}

\newcommand{\dist}{\on{dist}}

\newcommand{\pre}{{\on{pre}}}

\newcommand{\red}{{\on{red}}}

\newcommand{\ess}{{\on{ess}}}

\newcommand{\dual}{\vee}

\newcommand{\cyl}{\on{cyl}}

\newcommand{\Edge}{\on{Edge}}
\newcommand{\graph}{\on{graph}}

\newcommand{\Lag}{\on{Lag}}

\newcommand{\loc}{{\on{loc}}}
\newcommand{\Ver}{\on{Vert}}

\newcommand\B{\mathcal{B}}
\newcommand\cU{\mathcal{U}}
\newcommand\cM{\mathcal{M}}

\newcommand{\End}{\on{End}}

\newcommand{\x}{\times}

\newcommand{\Aut}{ \on{Aut} }

\newcommand{\Hom}{ \on{Hom}}
\newcommand{\Ind}{ \on{Ind}}
\renewcommand{\ker}{ \on{ker}}
\newcommand{\coker}{ \on{coker}}
\newcommand{\im}{ \on{im}}

\newcommand{\Spin}{ \on{Spin}}

\newcommand{\Vol}{  \on{Vol}}

\newcommand{\codim}{\on{codim}}

\newcommand\dirac{/\kern-1.2ex\partial} 
\newcommand\qu{/\kern-.7ex/} 
\newcommand\lqu{\backslash \kern-.7ex \backslash} 

\newcommand\dr{r_+ \kern-.7ex - \kern-.7ex r_-}
 
\newcommand{\labell}\label

\renewcommand{\dd}{{\on{d}}}
\newcommand{\ol}{\overline}
\newcommand{\olp}{\ol{\partial}}

\newcommand\eps{\epsilon}

\newcommand{\f}{\frac}
\newcommand{\lan}{\langle}
\newcommand{\ran}{\rangle}
\newcommand{\hh}{{\f{1}{2}}}

\newcommand{\qq}{{\f{1}{4}}}

\newcommand{\ti}{\tilde}

\newcommand{\tM}{\widetilde{\M}}

\newcommand\pt{\on{pt}}
\newcommand\id{\on{id}}
\newcommand\cE{\mathcal{E}}
\newcommand\E{\mathcal{E}}
\newcommand\cL{\mathcal{L}}
\newcommand\cF{\mathcal{F}}

\newcommand\CP{\C P}
\newcommand\RP{\R P}
\newcommand\cT{\mathcal{T}}

\newcommand\cI{\mathcal{I}}

\newcommand\mE{\mathcal{E}}

\newcommand\curv{\on{curv}}

\newcommand\Map{\on{Map}}

\newcommand\ev{\on{ev}}

\newcommand\Vect{\on{Vect}}
\newcommand\ul{\underline}
\newcommand\mO{\mathcal{O}}

\renewcommand\Im{\on{Im}}
\newcommand\Ker{\on{Ker}}
\newcommand\grad{\on{grad}}
\newcommand\reg{{\on{reg}}}
\newcommand\bdefn{\begin{definition}}
\newcommand\edefn{\end{definition}}
\newcommand\bea{\begin{eqnarray*}}
\newcommand\eea{\end{eqnarray*}}
\newcommand\bcv{\left[ \begin{array}{r} }
\newcommand\ecv{\end{array} \right] }

\newcommand\bma{\left[ \begin{array}{l} }
\newcommand\ema{\end{array} \right]}
\newcommand\ben{\begin{enumerate}}
\newcommand\een{\end{enumerate}}
\newcommand\beq{\begin{equation}}
\newcommand\eeq{\end{equation}}
\newcommand\bex{\begin{example}}
\newcommand\bsj{\left\{ \begin{array}{rrr} }
\newcommand\esj{\end{array} \right\}}

\newcommand\Cone{\on{Cone}}
\newcommand\Id{{\id}}
\newcommand\XX{\mathbb{X}}

\newcommand\LL{\mathbb{L}}

\newcommand\DD{\mathbb{D}}

\newcommand\vv{\check}
\newcommand\Fuk{\on{Fuk}}
\newcommand\Bl{\on{Bl}}

\newcommand\eex{\end{example}}
\newcommand\crit{{\on{crit}}}
  
\newcommand\val{{\on{val}}}  
  
\newcommand\sx{*\kern-.5ex_X}

\newcommand\white{{\includegraphics[width=.05in]{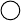}}}
\newcommand\black{{\includegraphics[width=.05in]{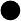}}}
\newcommand\whitet{s}
\newcommand\greyt{h}
\newcommand\blackt{g}

\newcommand{\bGamma}{\mathbb{\Gamma}}

\def\mathunderaccent#1{\let\theaccent#1\mathpalette\putaccentunder}
\def\putaccentunder#1#2{\oalign{$#1#2$\crcr\hidewidth \vbox
to.2ex{\hbox{$#1\theaccent{}$}\vss}\hidewidth}}

\makeatletter 
\@namedef{dgo@dd}{\let\dg@VECTOR=\dg@twoarrowedvector}%

\newcount\dg@XSHIFT 
\newcount\dg@YSHIFT 
\def\dg@twoarrowedvector(#1,#2)#3{%
   \begingroup 
   \dg@XTEMP=#1\relax\multiply\dg@XTEMP\m@ne\relax 
   \dg@YTEMP=#2\relax\multiply\dg@YTEMP\m@ne\relax 
   \dg@ZTEMP=#1\relax 
   \ifnum\dg@ZTEMP<\z@
     \multiply\dg@ZTEMP\m@ne\relax \fi 
   \ifnum\dg@YTEMP<\z@
     \advance\dg@ZTEMP by -\dg@YTEMP 
   \else \advance\dg@ZTEMP by \dg@YTEMP \fi 
   \dg@XSHIFT=#2\relax\multiply\dg@XSHIFT\m@ne\relax\multiply\dg@XSHIFT\twoarrowsep\relax 
     \divide\dg@XSHIFT by \dg@ZTEMP\relax 
   \dg@YSHIFT=#1\relax\multiply\dg@YSHIFT\twoarrowsep\relax\divide\dg@YSHIFT by \dg@ZTEMP\relax 
   \begin{picture}(0,0)%
      \thinlines 
      \put(-\dg@XSHIFT,-\dg@YSHIFT){\vector(#1,#2){#3}}%
      \put(\dg@XSHIFT,\dg@YSHIFT){\line(#1,#2){#3}}%
      \put(\dg@XSHIFT,\dg@YSHIFT){\vector(\dg@XTEMP,\dg@YTEMP){0}}
   \end{picture}%
   \endgroup}%
\makeatother 
\newcount\twoarrowsep 
\def\thorn{\vartheta}

\newcounter{mparcnt}

\usepackage[bbgreekl]{mathbbol}
\usepackage{stackengine,graphicx,mathtools}
\newsavebox{\foobox}
\newcommand{\slantbox}[2][.5]{\mbox{%
        \sbox{\foobox}{#2}%
        \hskip\wd\foobox
        \pdfsave
        \pdfsetmatrix{1 0 #1 1}%
        \llap{\usebox{\foobox}}%
        \pdfrestore
}}
\def\sfvarfhash{%
  \stackinset{c}{}{b}{1.5pt}{\rule{1.45ex}{\sfrlthk}\kern1pt}{%
  \stackinset{c}{}{b}{3.7pt}{\rule{1.45ex}{\sfrlthk}\kern1pt}{%
    \textsf{\itshape ff}%
  }}%
}
\def\sfvarhash{%
  \kern1pt%
  \stackinset{c}{}{b}{1.5pt}{\rule{1.45ex}{\sfrlthk}\kern1pt}{%
  \stackinset{c}{}{b}{3.7pt}{\rule{1.45ex}{\sfrlthk}\kern1pt}{%
    \slantbox[.2]{\mysfrule\kern2pt\mysfrule\kern2.3pt}%
  }}%
  \kern1pt%
}
\def\sfrlthk{.17ex}
\def\mysfrule{\rule{\sfrlthk}{1.4ex}}
\def\varfhash{%
  \kern -1.5pt%
  \stackinset{c}{}{c}{-1.7pt}{\kern1pt\rule{1.7ex}{\rlthk}}{%
  \stackinset{c}{}{c}{1.7pt}{\kern1pt\rule{1.7ex}{\rlthk}}{%
    \kern2pt\itshape ff\kern2pt%
  }}%
}
\def\varhash{%
  \kern-.5pt%
  \stackinset{c}{}{c}{-1.7pt}{\kern1pt\rule{1.7ex}{\rlthk}}{%
  \stackinset{c}{}{c}{1.7pt}{\kern1pt\rule{1.7ex}{\rlthk}}{%
    \kern2pt\kern.2pt\slantbox[.2]{\myrule\kern2.4pt\myrule}\kern.2pt\kern2pt%
  }}%
  \kern1pt%
}
\def\rlthk{.13ex}
\def\myrule{\rule[-.33ex]{\rlthk}{1.8ex}}

\usepackage[T1]{fontenc}
\usepackage{lmodern}

\begin{document}
\title[Immersed Floer cohomology and Lagrangian surgery]{Invariance of
  immersed
  Floer cohomology \\  under Lagrangian
  surgery}

\author{Joseph Palmer and Chris Woodward}

\address{ Department of Mathematics - Seeley Mudd Building, Amherst College, 31 Quadrangle Drive, Amherst, MA 01002, U.S.A.}
\email{ jpalmer@amherst.edu }

\address{Mathematics-Hill Center,
Rutgers University, 110 Frelinghuysen Road, Piscataway, NJ 08854-8019,
U.S.A.}  \email{woodwardc@gmail.com}

\thanks{This work was partially supported by NSF grant DMS 
  1711070. 
 Any opinions, findings, and conclusions or recommendations 
  expressed in this material are those of the author(s) and do not 
  necessarily reflect the views of the National Science Foundation.  }

\begin{abstract} 
  We show that Floer cohomology of an immersed Lagrangian
  brane is invariant under smoothing of a self-intersection point if
  the quantum valuation of the weakly bounding cochain vanishes and
  the Lagrangian has dimension at least two.  The chain-level map
  replaces the two orderings of the self-intersection point with
  meridional and longitudinal cells on the handle created by the
  surgery, and uses a bijection between holomorphic disks developed by
  Fukaya-Oh-Ohta-Ono \cite[Chapter 10]{fooo}.  Our result generalizes
  invariance of potentials for certain Lagrangian surfaces in
  Dimitroglou-Rizell--Ekholm--Tonkonog \cite[Theorem 1.2]{det:ref},
  and implies the invariance of Floer cohomology under mean curvature
  flow with this type of surgery, as conjectured by Joyce
  \cite{joyce:conjectures}.
\end{abstract}

\maketitle 

\tableofcontents 

\section{Introduction} 

A Lagrangian immersion in a compact symplectic manifold with
transverse self-intersection defines a homotopy-associative {\em
  Fukaya algebra} developed by Akaho-Joyce in \cite{akaho}.  The
framework of Fukaya-Oh-Ohta-Ono \cite{fooo} associates to this algebra
a space of solutions to the projective Maurer-Cartan equation. For any solution, there is a {\em Lagrangian Floer cohomology group},
independent up to isomorphism of all choices.  In Palmer-Woodward
\cite{pw:flow}, we studied the behavior of Floer cohomology under
variation of an immersion in the direction of the Maslov (relative
first Chern) class, such as a coupled mean-curvature/K\"ahler-Ricci
flow.  The main result of \cite{pw:flow} was that there exists a flow
on the space of projective Maurer-Cartan solutions with the following
property: The isomorphism class of the Lagrangian Floer cohomology is
invariant as long as the valuation of the Maurer-Cartan solution with
respect to the quantum parameter stays positive and the Lagrangian
stays immersed.  In particular, the Floer cohomology is invariant as
the immersion passes through a self-tangency.  Naturally a question
arises whether one can continue the flow through a ``wall'' created by
the vanishing valuation at a self-intersection point.

Via the mirror symmetry conjectures, this question is expected to be
related to a question on deformation theory of vector bundles on a
mirror complex manifold, or more precisely, matrix factorizations
\cite{kap:dbranes}. The mirror of the mean curvature flow is expected to
be (a deformed version of) the Yang-Mills flow \cite{jy}.  The isomorphism
class of the bundle is constant under Yang-Mills flow and, in particular, the cohomology is invariant \cite{ab:ym}.  That is, there
are no real-codimension-one ``walls'' on the mirror side, and so one
does not expect such walls in the deformation spaces for Lagrangian
branes either.  For vector bundles on projective varieties
there exist versal deformations \cite{forster} in the sense of
Kuranishi; see for example \cite{siu} for coherent sheaves.  The base
of these versal deformations are complex-analytic spaces.  The results
of this paper can be viewed as giving a theory of versal deformations
for immersed Lagrangians, in which solutions to the projective
Maurer-Cartan equation with negative $q$-exponents parametrize the actual
deformations of an immersed Lagrangian.  As in the case of
deformations of singular algebraic varieties \cite[Chapter
XI]{ar:alg2}, in order to produce the expected space of deformations
one must allow smoothings at the singularities.

A way of smoothing singularities of immersed self-transverse
Lagrangians was introduced by Lalonde-Sikorav \cite{ls:sous} and
Polterovich \cite{surger}.  Let 
\[ \phi_0: L_0 \to X \] 
be a
self-transverse Lagrangian immersion with compact domain $L_0$ with an
self-intersection point $x \in \phi_0(L_0)$. 
\begin{figure}[ht]
  \includegraphics[height=1.5in]{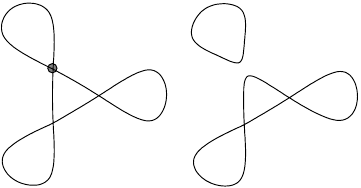}
\caption{An immersion and its surgery} 
\label{loopy} 
\end{figure} 
For a sufficiently small
{\em surgery parameter} $\eps \in \R$ denote by
\[ \phi_\eps: L_\eps \to X \] 
the surgery obtained by removing small
balls around the intersection and gluing in a cylinder. The surgery
parameter $\eps$ is closely related to \label{thethe} the difference
$A(\eps)$ from \eqref{Aeq} in the areas of  disks bounding (that is, having boundary in) \label{rep:bounding}
$\phi_0(L_0)$ and $\phi_\eps(L_\eps)$.

A long line of papers in symplectic geometry have studied the effect
of Lagrangian surgery on Floer theory.  Seidel's long exact triangle
\cite{se:lo} is perhaps the first example, since a Dehn twist is a
special case of a surgery.  More generally, holomorphic\footnote{To save space, we refer to pseudoholomorphic disks with respect to some almost complex structure as holomorphic.}  disks with
boundary in the surgery were described in Fukaya-Oh-Ohta-Ono
\cite[Chapter 10]{fooo}.  Abouzaid \cite{ab:fr}, Mak-Wu
\cite{mak:dehn}, Tanaka \cite{tanaka},
Chantraine-Dimitroglou-Rizell-Ghiggini-Golovko \cite[Chapter
8]{chantraine}, Fang \cite{fang}, and Hong-Kim-Lau \cite[Theorem
B]{hkl} proved various generalizations.  The invariance of disk potentials
was shown for certain Lagrangian surfaces by Pascaleff--Tonkonog
\cite[Theorem 1.2]{pasc:wall} and Dimitroglou-Rizell--Ekholm--Tonkonog
\cite[Theorem 1.2]{det:ref}.  In dimension two, the Lagrangians
related by the two different signs of surgery parameter are said to be
related by {\em mutation}.  Mutation-invariance of Lagrangian Floer homology 
was shown via Lagrangian cobordism techniques by Hicks \cite{hicks:wall}.
The ``wall-crossing'' formula for the
change in the local system given by the above formulas is discussed in
Auroux \cite{auroux1}, \cite{auroux2}, Kontsevich-Soibelman \cite{ks},
and Pascaleff-Tonkonog \cite{pasc:wall}.

We construct a natural identification of solutions of the projective
Maurer-Cartan equations for the surgered and unsurgered Lagrangian
branes that preserves the disk potentials and Floer cohomology.  The
version of Floer cohomology used here is the cohomology of the twisted
first composition map for a Fukaya algebra which counts {\em treed holomorphic disks} bounding the Lagrangian.    Our results show that if the quantum valuation at a self-intersection point of a family of Maurer-Cartan solutions in a mean curvature flow of Palmer-Woodward \cite{pw:flow} reaches zero then (at least under the conditions in the Theorem) the solution may be continued by Lagrangian surgery so that the Floer cohomology of the surgery is invariant.  Thus the flow may be continued after the
singular time without changing the Floer cohomology.

The assumptions necessary for invariance of Floer cohomology to hold
are encoded in the following definitions.  Let
\begin{equation} \label{di}
\Lambda = \C((q^\R)) := \Set{ \sum_{i=0}^\infty a_i q^{d_i} \ | \
  \lim_{i \to \infty} d_i = \infty, \ \forall i, \ d_i \in \R,
  \ a_i \in \C } \end{equation}
denote the Novikov field with complex coefficients,\footnote{The
  Fukaya algebras in this paper are defined with rational
  coefficients, but allowing complex coefficients gives a
  possibly-larger Maurer-Cartan space.}  equipped with $q$-valuation
\[ \val_q: \Lambda - \{ 0 \} \to \R, \quad \sum_{i=0}^\infty a_i q^{d_i} \mapsto
\min(d_i, a_i \neq 0)
.\]
Let $\Lambda_0$ denote the group of units in $\Lambda$ with
vanishing $q$-valuation
\[ \Lambda_0 = \val_q^{-1}(0) = \Set{ a_0 + \sum_{i \ge 1} a_i
  q^{d_i} \in \Lambda \ | a_0 \in \C - \{ 0 \}, \ \forall i,  \ a_i \in \C, d_i > 0 } .\]
Let $\phi_0:L_0 \to X$ be a Lagrangian immersion.  A {\em local
  system} on $\phi_0$ is a flat $\Lambda_0$-line bundle $y$ on $\phi_0(L_0)$,
or equivalently, a flat line bundle on $L_0$ together with
identifications of the fibers $y(x_-) \to y(x_+)$ at the
self-intersection points 
\[ x = (x_-,x_+), \quad \phi_0(x_-) = \phi_0(x_+) . \]   
If $\phi_0(L_0)$ is connected with fundamental group
$\pi_1(\phi_0(L_0))$ for some choice of base point then the space of
isomorphism classes of local systems is isomorphic to the space of
representations
\[ \RR(\phi_0) \cong \Hom(\pi_1(\phi_0(L_0)), \Lambda_0) \cong
\Hom(H_1(\phi_0(L_0)), \Lambda_0) .\]
For disconnected Lagrangians, $\RR(\phi_0)$ is defined by replacing
$\pi_1(\phi_0(L_0))$ with the product of the fundamental groups of the
connected components of $\phi_0(L_0)$.  Let $\phi_0: L_0 \to X$ be
equipped with a {\em brane structure} consisting of an orientation,
relative spin structure, and $\Lambda_0$-valued local system
$y \in \RR(\phi_0)$. In Sections \ref{cellsec} and \ref{sftsec} we
construct for any such datum a {\em Fukaya algebra} $CF(\phi_0)$,
which is a strictly unital \ainfty algebra.  

The Fukaya algebra has a canonical family of deformations parametrized
by odd cochains, and the cohomology is defined for solutions to the
{\em projective Maurer-Cartan equation}.  By definition, any element
$b \in CF(\phi_0)$ is given as a sum 
\[ b = \sum_{x \in \cI(\phi_0)} b(x) x \]
over generators $x \in \cI(\phi_0)$ corresponding to critical points of a Morse function, corresponding to cells in a cellular decomposition, or self-intersection points.   In this paper, we choose to index the Floer generators by cells, which we find  conceptually clearer than the indexing by Morse critical points.   In particular, if $b$ is odd then $b(x)$ vanishes for $x$ of \label{rep:evendeg}
even degree.  Let $MC(\phi_0)$ denote the space of projective Maurer-Cartan
solutions, also called ,  {\em weakly bounding cochains}, as in \eqref{mc}.  For $\delta > 0$ small let
$MC_\delta(\phi_0)$ denote the subspace satisfying
\[ \val_q(b(x)) \in (-\delta,\infty) \]
at the transverse self-intersection points $x$ of $\phi_0$.  For
sufficiently small $\delta$, associated to any
$b_0 \in MC_\delta(\phi_0)$
is a Floer cohomology group $HF(\phi_0,b_0)$, independent of all
choices up to isomorphism.  Given $x = (x_-,x_+)$, denote by
$\ol{x} = (x_+, x_-) \in L_0^2$ the self-intersection point with the
opposite ordering.  The degree of $x$ is even resp. odd if the natural
map
\[ T_{x_-} L \oplus T_{x_+} L \to T_{\phi(x_\pm)} X \]
is orientation preserving resp. reversing.  If $\dim(L)$ is even
then $x$ is odd if and only if $\ol{x}$ is odd, 
while if $\dim(L)$ is odd then $x$ is odd if and only if $\ol{x}$ even.  By
convention $b_0(\ol{x})$ vanishes if $\ol{x}$ is even.  \label{odd}

\begin{definition} \label{admissible} Let $\phi_0: L_0 \to X$ be a Lagrangian immersion 
and $b_0 \in MC_\delta(\phi_0)$ a Maurer-Cartan solution.  
An odd self-intersection point
  $x = (x_-,x_+) \in L_0^2$ is {\em admissible} for $b_0$  if and only if 
\begin{enumerate} 
\item the $q$-valuation of the
  coefficient $b_0(x)$ is close to zero in the sense that
\begin{equation} \label{closetozero}
\val_q(b_0(x)) \in (-\delta,0) \end{equation} 
and 
\item either $b_0(\ol{x}) = 0$, or $\dim(L_0) =2 $ and the $q$-valuation of $b_0(\ol{x})$ is sufficiently large in the
sense that 
\footnote{For the sake of discussing explicit examples, we also
allow $\dim(L_0) = 1$ under the following assumptions (which do not
typically hold): $b_0(\ol{x}) = 0$, every holomorphic disk
$u: S \to X$ with boundary on $\phi$ meeting $x$ has a branch change
at every $z \in \partial S$ with $u(z) = x$, and there are no
holomorphic disks $u: S \to X$ with exactly one corner at $\ol{x}$.}
\begin{equation} \label{valq}
\val_q(b_0(x) b_0(\ol{x})) > 0  .\end{equation} 
\end{enumerate}
This ends the Definition.
\end{definition}

The invariance of Floer cohomology under surgery holds after the
following change in the weakly bounding cochain.  The surgered
Lagrangian $L_\eps$ is obtained from $L_0$ by removing the
self-intersection points $x_\pm \in L_0$ and gluing in a handle
\[ H_\eps \cong S^{n-1} \times \R \] 
as in Section \ref{surgsec} below.  We
denote by
\[ \mu \cong S^{n-1} \times \{ 0 \}, \quad \lambda \cong \{ \pt \}
\times \R \]
the meridional and longitudinal cells on the handle $H_\eps$, oriented
so that the bijection of Proposition \ref{prop:bijprop} is orientation
preserving. 

\begin{definition} \label{onsurgery} Let $x = (x_-,x_+)$ with $  \phi_0(x_-) = \phi_0(x_+)$.  
For $\eps > 0$ let 
\[ CF_\delta(\phi_0,\eps) 
\subset CF(\phi_0) \]
denote the space of elements $b_0 \in CF(\phi_0)$ whose $q$-valuation at the surgery point is opposite the surgery area:
\[ \val_q(b_0(x)) + A(\eps) = 0 \] 
with notation from \eqref{Aeq} and \eqref{valq}.  Let 
\[ MC_\delta(\phi,\eps) \subset CF_\delta(\phi,\eps) \] 
denote the space of Maurer-Cartan solutions $b_0$ for which $x$ is admissible and that 
 vanish on the closure of the cells
containing $x_\pm$:
\begin{equation} \label{vcell}
b_0(\sigma) = 0 , \quad \forall \sigma \subset \ol{\tau}, \tau \ni x_\pm .\end{equation}
Define 
  \begin{multline} \label{psimap} \Psi: CF_\delta(\phi_0,\eps) \to CF(\phi_\eps), 
  \quad b_0 \mapsto  b_0 -
    b_0(x)x - b_0(\ol{x}) \ol{x}  +  \\
 \begin{cases} 
\ln(b_0(x)q^{A(\eps)}) 
\mu 
 + \ln(b_0(x)b_0(\ol{x}) + 1) 
 \lambda 
 &    \dim(L_0) = 2 \\
\ln(b_0(x)q^{A(\eps)}) \mu 
 + b_0(x) b_0(\ol{x})  \lambda 
& \dim(L_0) > 2 \end{cases} \end{multline}
where the logarithms are defined by formal power series expansion at the leading order term for any choice of branch, well-defined
by the assumption that $b_0(x)q^{A(\eps)}$ and $b_0(x)b_0(\ol{x}) + 1$
have vanishing $q$-valuation. 
This ends the Definition.
\end{definition}

The vanishing condition \eqref{vcell} can always be achieved
up to gauge equivalence by Lemma \ref{lem:gaugekill}.  In the case
$\dim(L_0) = 2$, we assume that the surgered Lagrangian $L_\eps$ is
equipped with a local system which has holonomy $-1$ around the
meridian; note that this assumption constrains the topology of the surgery.  The conditions in Definition \ref{onsurgery} are satisfied in our application
to mean curvature flow \cite{pw:flow}. 

We may now state the main result.   Let
  $MC_{\ge 0}(\phi_\eps)$ be the enlarged space of projective
  Maurer-Cartan solutions in \eqref{projmc} for $\phi_\eps$, in which
  one allows the coefficient of $\lambda$ to have vanishing
  $q$-valuation.\footnote{This enlargement is only relevant in the
    case $\dim(L) =2 $, and in this case we show that the
    Maurer-Cartan sum still converges.}

\begin{theorem} \label{thm:bij} Let $\phi_0: L_0 \to X$ be an immersed
  Lagrangian brane of dimension $\dim(L_0)$ at least three in a compact
  rational symplectic manifold $X$.  There exists a constant
  $\delta > 0$ such that for any $b_0 \in MC_\delta(\phi_0)$ and any
  admissible transverse self-intersection point
  $x \in \cI^{\on{si}}(\phi)$ and small $\eps > 0$ \label{rep:smalleps} as in Definition \ref{admissible} there
  exist perturbation systems defining the Fukaya algebras $CF(\phi_0)$
  and $CF(\phi_\eps)$ so that the following holds:
The map $\Psi$ of Definition \ref{onsurgery} satisfies 
  \begin{equation} \label{mcmap} \Psi(MC_\delta(\phi_0,\eps))\subset
    MC_{\ge 0}(\phi_\eps)  \end{equation}
preserves the disk potentials
\[
\Psi^* W_\eps = W_0, \quad W_0: MC_\delta(\phi_0) \to \Lambda, \quad W_\eps: MC_{\ge
  0}(\phi_\eps) \to \Lambda \]
and lifts to an isomorphism of Floer cohomology groups
\[ HF(\phi_0,b_0) \cong HF(\phi_\eps,b_\eps := \Psi(b_0)), \forall 
b_0 \in MC_\delta(\phi_0,\eps).\]
The same result holds for $\dim(L_0) = 2$ under the assumption described in 
\eqref{rinvd}.
  \end{theorem}

  In other words, immersed Floer cohomology is invariant under surgery
  after a suitable change in the weakly bounding cochain.   We conjecture
  that the assumption $b_0(\ol{x}) = 0$ for $\dim(L_0) > 2$ is not necessary.

\begin{remark} \label{rem:hicks}
  J. Hicks \cite[Section 2]{hicks} has given examples of Lagrangian
  spheres that have surgeries that in the Fukaya category are
  non-isomorphic depending on the sign of the surgery parameter
  $\eps$; the result above does not contradict these examples since we
  require the immersed Lagrangian $\phi_0:L_0 \to X$ itself to have a
  non-zero weakly bounding cochain $b_0$. \label{hicks}
\end{remark}

\begin{remark} Returning to the application to mean curvature flow, Theorem
    \ref{thm:bij} suggests the possibility of mean curvature flow for
    Lagrangians with {\em preventive surgery}.  Namely, similar to the
    set-up in the Thomas-Yau conjecture \cite{thomasyau} suppose one
    performs coupled mean curvature/K\"ahler-Ricci flow on a
    Lagrangian immersion $\phi_t$ with unobstructed and non-trivial
    Floer theory $HF(\phi_t)$.  The results of this paper and
    Palmer-Woodward \cite{pw:flow} imply that the non-triviality of
    the Floer homology $HF(\phi_t)$ carries along with the flow
    $\phi_t$, if a surgery {\em before} the time at which the
    geometric singularity forms is performed whenever the
    $q$-valuation $\val_q(b_t)$ of the Maurer-Cartan solution $b_t$
    crosses zero.  This type of surgery is {\em preventive} rather
    than {\em emergency} in the sense that the Lagrangian immersion
    $\phi_t$ is not about to cease to exist.  Non-triviality of the
    Floer cohomology affects the types of singularities that can occur
    as discussed by Joyce \cite{joyce:conjectures}.  One naturally
    wonders what kind of singularities can occur {\em generically}
    (meaning allowing arbitrary Hamiltonian perturbations) in the
     case of non-trivial Floer cohomology. \label{Fcase}
\end{remark} 

We may upgrade the isomorphisms of Floer cohomology to a quasi-isomorphism 
in the Fukaya category as follows.    A full construction of a Fukaya category 
containing all Lagrangians is beyond the techniques of  this paper; we consider
rather the following simplified Fukaya category with two objects.   Let $\phi_0': L_0' \to X$ be a Hamiltonian isotopy of $\phi_0: L_0 \to X$ such that $\phi_0,\phi_0'$ intersect transversally, as in Figure \ref{drawing}; the shaded region represents the disks that correspond under the bijection.  We assume that $\eps$ is sufficiently small so that the surgered
immersion $\phi_\eps: L_\eps \to X$ intersects $\phi_0'$ transversally as well. 
After a Hamiltonian perturbation we may assume that 
the Lagrangian 
\begin{equation} \label{eqn:tiphi} \ti{\phi}: \phi_0 \cup \phi_0' : L_0 \cup L_0' \to X \end{equation}
is rational, immersed, and with 
transverse self-intersection, and similarly for 
\[ \ti{\phi}_\eps: \phi_\eps \cup \phi_0' : L_\eps \cup L_0' \to X .\]
As such, we obtain a category $\Fuk_{\ti{\phi}_0}(X)$ with two objects $\phi_0,\phi_0'$ 
equipped with weakly bounding cochains, and with morphisms defined by the corresponding subspaces of $CF(\ti{\phi}_0)$.  For sufficiently small surgery parameter $\eps$,
the intersection $\phi_\eps \cap \phi_0'$ is still transverse and 
we obtain  a category $\Fuk_{\ti{\phi}_\eps}(X)$ with objects
$(\phi_\eps,b_\eps),(\phi_0',b_0')$ with weakly bounding cochains.
Invariance of Lagrangian Floer theory 
under Hamiltonian isotopy (specifically, the homotopy invariance of the \ainfty bimodule associated to a pair of Lagrangians) implies that \label{rep:thatthat} $\phi_0'$ admits a weakly bounding cochain 
$b_0'$ and so that the objects
$(\phi_0,b_0),(\phi_0',b_0')$ are quasi-isomorphic in 
$\Fuk_{\ti{\phi}_0}(X)$.  That is, there exist elements
\begin{equation} \label{alphabeta} \alpha_0 \in CF(\phi_0,\phi_0'),  \ \  \beta_0 \in CF(\phi_0',\phi_0), \ \ 
\delta_0 \in CF(\phi_0,\phi_0), \ \  \delta_0' \in CF(\phi_0',\phi_0') \end{equation}
so that 
\[ m_2^{b_0,b_0'}(\alpha_0,\beta_0) - 1_{\phi_0} = m_1^{b_0}(\delta_0), \quad 
m_2^{b_0',b_0}(\beta_0,\alpha_0) - 1_{\phi_0'} = m_1^{b_0'}(\delta'_0) . \] 
It follows from the associativity of the composition law that 
\[ HF(\phi_0,b_0) \cong HF(\phi_0',b_0') . \]
We prove a similar theorem for the surgered immersion:

\begin{theorem} \label{thm:qiso} Under the assumptions of Theorem \ref{thm:bij}, the objects $(\phi_\eps,b_\eps)$ and $(\phi_0',b_0')$
are quasi-isomorphic in $\Fuk_{\ti{\phi}_\eps}(X)$, and in particular 
their endomorphism algebras have isomorphic cohomology algebras \label{rep:havehave}
$HF(\phi_\eps,b_\eps) \cong HF(\phi_0,b_0).$
\end{theorem}
More generally, in any reasonable definition of the Fukaya category we expect
that $(\phi_0,b_0)$ and $(\phi_\eps,b_\eps)$ are quasi-isomorphic.  \label{rep:quasi} This would be the mirror statement to invariance of the isomorphism class of the bundle under (deformed) Yang-Mills heat flow.

\vskip .1in \noindent {\em Outline of proof.}  \label{rep:outlineproof} The 
main result is an application of symplectic field theory in the sense of Bourgeois-Eliashberg-Hofer-Wysocki-Zehnder \cite{bo:com}.    We make use of the neck-stretching limit for holomorphic disks along a sphere enclosing the self-intersection point.   In the limit, pseudoholomorphic maps limit to the pseudoholomorphic buildings.   The levels of the buildings for the surgered and unsurgered Lagrangians are the same except for those
in a neighborhood of the self-intersection point, and the main task is to analyze the 
holomorphic curves in the local model.   A similar argument was used in the proof of the special case by Fukaya-Oh-Ohta-Ono, Chapter 10 of  \cite{fooo} describing the effect of surgery at an intersection point of two Lagrangians in terms of a mapping cone.   The setup requires substantial groundwork.  The definition and basic properties of surgery are described in Section  \ref{surgsec}.  The construction of a Fukaya algebra associated to an immersed Lagrangian is recalled in Sections \ref{sec:fukalg}, \ref{sec:pert}, and \ref{cellsec}, using a perturbation scheme which adapts that in Cieliebak-Mohnke \cite{cm:trans}.  The details of the symplectic field theory argument are carried out in Section \ref{sftsec}.   The holomorphic disks in the local model are classified in Section \ref{handlesec}.  Details  include showing that the standard complex structure in the local model makes all holomorphic buildings regular, and the classification for the standard handle applies to the ``flattened handle'' boundary condition which is honestly cylindrical near infinity and for which the symplectic field theory results apply.  A key Lemma \ref{lem:noothers1} describes the levels in the local model that can appear with more than one strip-like end.  Having identified a bijection between moduli spaces defining the differential for the Lagrangian and its surgery, the isomorphism of  Floer cohomologies is carried out in Section \ref{fsurgery}, and in particular,  Section \ref{sec:floerequiv} by identifying the complexes up to a stabilization.   Section \ref{sec:qiso} proves various enhancements, such as the existence of a quasiisomorphism in the Fukaya category between the objects defined by the Lagrangian and its surgery, which in particular implies an isomorphism of Floer
cohomologies as algebras.  We also explain how to obtain the result of Fukaya-Oh-Ohta-Ono's Chapter 10  (equivalence of surgeries with mapping cones) of \cite{fooo}, using the same arguments.

\vskip .1in \noindent {\em We thank Denis Auroux, Soham Chanda, John Man-Shun Ma, Dmitry
  Tonkonog, Sushmita Venugopalan and Guangbo Xu for helpful
  discussions.}

\section{Lagrangian surgery} 
\label{surgsec}

Lagrangian surgery was introduced by Lalonde-Sikorav \cite{ls:sous} in
dimension two and Polterovich \cite{surger} for arbitrary dimension.
Surgery smooths a self-intersection point by removing small balls
around the preimages of the self-intersection point and gluing in a
handle.   Haug \cite{haug} introduced generalizations to handles of higher index, 
which we do not consider here.

\subsection{The local model}

The local model for the surgery is obtained by parallel transport of
the vanishing cycle of the standard Lefschetz fibration along a line
parallel to the real axis, as explained in 
Seidel \cite[Section 
2e]{se:gr}. \label{altdef}
The {\em standard
    Lefschetz fibration} is the map
\begin{equation} \label{slef} \pi: \CC^n \to \C, \quad (z_1,\ldots,
  z_n) \mapsto z_1^2 + \ldots + z_n^2 .\end{equation}
Equip $\CC^n$ with the standard symplectic form
$\omega \in \Omega^2(\CC^n)$.  The space $\CC^n - \{ 0 \} $ has a
natural connection given by a horizontal sub-bundle
\[ T^h_z \subset T_z(\CC^n - \{ 0 \}), \quad 
T^h_z = (\Ker D_z \pi)^\omega \] 
equal to the union of symplectic perpendiculars of the fibers
$\pi^{-1}(z)$.  For any path 
\[ \gamma: [0,1] \to \C - \{ 0 \} \] 
there is a symplectic parallel transport map
\[ T_\gamma: \pi^{-1}(\gamma(0)) \to \pi^{-1}(\gamma(1)) \]
by taking the endpoint of a horizontal lift of $\gamma$ with any given
initial condition.  Let $\gamma:[0,1] \to \C$ be an embedded path with endpoint
$\gamma(0) = 0$ at the critical point of the Lefschetz fibration.
Each fiber of $\pi$ over $\gamma([0,1])$ has a {\em vanishing cycle}
$C_z \subset \pi^{-1}(z)$ defined as the set of elements
$w \in \pi^{-1}(z)$ that limit to the origin $0\in \CC^n$ under
symplectic parallel transport $C_z \to C_0$.  If
$S^{n-1}\subset\R^n\subset\CC^n$ is the unit sphere, then explicitly
\[ C_z := \sqrt{z} S^{n-1} , \quad z \in \C.\]
Now let $\gamma: \R \to \C$ be an embedded path (not necessarily with $\gamma(0) = 0$) \label{rep:gamma0} so that except for $t$
in some compact interval $[-c,c]$, we have
\begin{equation} \label{gamlimits} \gamma(t) = t + i 2 \eps, \quad
  \forall t \notin [-c,c] .\end{equation}
For example, one could assume that the path is the affine linear path $\gamma(t) = t + i 2 \eps$.  The {\em handle Lagrangian}
$H_\gamma$ is the union of vanishing cycles over $\gamma$:
\[ H_\gamma := \bigcup_{t \in \R} C_{\gamma(t)} .\]
As in \cite[Discussion after (1.12)]{se:lo}, $H_\gamma$
may be equivalently defined as symplectic parallel transport of $C_z$
along $\gamma$.  \label{asy} 

  More generally, as pointed out by Seidel \cite[Section 2e]{se:gr},
  one may define surgery by allowing more general paths in the base of
  the Lefschetz fibration. By bending the path somewhat below the real
  axis one can achieve a {\em zero-area surgery} for which  holomorphic disks bounding  the surgery have the same area as the corresponding disks (as in Theorem \ref{thm:bij})   bounding the unsurgered Lagrangian.  \label{rep:zeroarea}  However, we will only use
  the straight paths for the classification of disks in Section
  \ref{handlesec}.

The handle has the following explicit description.  
Let $\CC^n \cong \R^{2n}$ be equipped with Darboux coordinates
\[ z = (z_1,\ldots, z_n), \quad z_k= q_k + i p_k, \quad k = 1,\ldots, n .\] 
For a real number $\eps$ with $|\eps|$ small define a Lagrangian
submanifold $H_\eps$ of $\CC^n$, the {\em handle} of the surgery, by
\begin{equation} \label{handledef} H_\eps = \Set{ (q_1  + ip_1
    ,\ldots, q_n+ ip_n)
    \in \CC^n \ | \ q \neq 0, \ \forall k, \ p_k = \frac{\eps
      q_k}{|q|^2 } } .\end{equation}
Identify $\CC^n = T^\dual \R^n$ in the standard way.  Denote the
standard symplectic form
\[ \omega_0 = \sum_{k=1}^n \dd q_k \wedge \dd p_k  \in \Omega^2(\CC^n). \]
Define 
\[ f_\eps: \R^n - \{ 0 \} \to \R, \quad q \mapsto \eps \ln(|q|) .\]
The Lagrangian $H_\eps$ is the graph of the closed one-form
$\dd f_\eps$:
\begin{equation} \label{Heps} H_\eps = \graph( \dd f_\eps) \subset
  \R^{2n} .\end{equation}
Also note that $H_\eps \subset \CC^n$ of \eqref{Heps} is invariant
under the anti-symplectic involution
\[ \iota: \CC^n \to \CC^n, \quad (p,q) \mapsto (q,p) .\]

\label{fhandle} For the purposes of
  symplectic field theory, it is convenient to replace the above
  Lagrangian with one that is cylindrical near infinity in the sense of Definition \ref{def:acyl} rather than only asymptotically cylindrical.   By the {\em flattened handle} we mean the 
  Lagrangian defined by parallel transport of a sphere along  a path $\gamma$ with $\gamma(t) = t$ for $t$ outside of a compact neighbourhood of $0$, and passing slightly above the critical value
$0 \in \C$.  An equivalent definition can be given explicitly as follows. 
  Define a Lagrangian submanifold $\vv{H}_\eps \subset \CC^n$ equal to
  $H_\eps$ in a compact neighborhood of $0$ and equal to
  $\R^n \cup i\R^n$ outside a larger compact neighborhood of $0$.  Following Fukaya-Oh-Ohta-Ono \cite[Chapter 10]{fooo}, let
  \[ \zeta > 0, \quad \eps \neq 0 \] 
  be constants.  The constant $\eps$ is the
  {\em surgery parameter} describing the ``size'' of the Lagrangian
  surgery, while the parameter $\zeta$ is a cutoff parameter
  describing the size of the ball on whose complement the surgery
  $\phi_\eps$ agrees with the unsurgered immersion $\phi_0$.  These
  constants will be chosen later so \label{addso} that $\zeta$ is
  large and $\zeta|\eps|^{1/2}$ is small.  Consider a function $\rho$ given by a logarithm 
  plus constant times a compactly-supported cutoff function, namely
\begin{equation} \label{zeta} \rho \in C^\infty(\R_{> 0}), \quad 
\rho(r) = 
\begin{cases} \ln(r)  - |\eps| & r \leq |\eps|^{1/2} \zeta \\
  \ln( |\eps|^{1/2} \zeta) & r \geq 2 | \eps|^{1/2}
  \zeta \end{cases} 
  \end{equation} 
satisfying 
 \[ \forall r \in \R_{ > 0}, \quad  \rho''(r) \leq 0 \leq \rho'(r). \]
Define
\begin{equation} \label{fprime} \vv{f}_\eps: \R^n \to \R, \quad \quad q \mapsto
  \eps \rho(|q|). \end{equation}
Consider the graph \label{thegraph}
\[  \graph( \dd \vv{f}_\eps) \subset T^\dual \R^n \cong \R^{2n} .\]
Let $U \subset X$ be a Darboux chart near $x$ so that the
self-intersection of $\phi$ at $x$ has the form \eqref{twobranches}.
One realization of the {\em flattened handle} is the union of the graph of the differential of $\vv{f}_\eps$ and its reflection:
\begin{equation} \label{Hepsp} \vv{H}_\eps = \left( \graph(\dd \vv{f}_\eps)
  \cap ( \CC^n - B_{|\eps|^{1/2} \zeta}(0) ) \right)
  \\
  \cup \iota \left( \graph( \dd \vv{f}_\eps) \cap ( \CC^n - B_{|\eps|^{1/2}\zeta} (0) ) \right).  \end{equation} 
The inclusion 
\[ \vv{\phi}_\eps: \vv{H}_\eps \to U . \] 
is then a Lagrangian embedding, with image equal to that of $\R^n \cup i \R^n$
outside of a compact set.

\begin{figure}[ht] 
\begin{picture}(0,0)%
\includegraphics{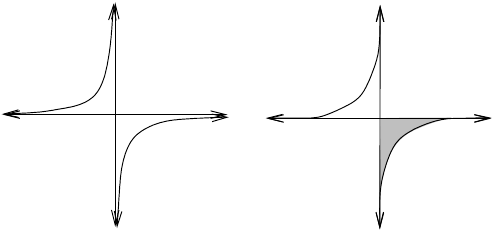}%
\end{picture}%
\setlength{\unitlength}{3947sp}%
\begingroup\makeatletter\ifx\SetFigFont\undefined%
\gdef\SetFigFont#1#2#3#4#5{%
  \reset@font\fontsize{#1}{#2pt}%
  \fontfamily{#3}\fontseries{#4}\fontshape{#5}%
  \selectfont}%
\fi\endgroup%
\begin{picture}(3945,1846)(2671,819)
\put(5720,1623){\makebox(0,0)[lb]{\smash{{\SetFigFont{12}{14.4}{\rmdefault}{\mddefault}{\updefault}{\tiny
          $A(\eps)$}%
}}}}
\put(3050,1842){\makebox(0,0)[lb]{\smash{{\SetFigFont{12}{14.4}{\rmdefault}{\mddefault}{\updefault}{\tiny$H_\eps$}%
}}}}
\put(5173,1806){\makebox(0,0)[lb]{\smash{{\SetFigFont{12}{14.4}{\rmdefault}{\mddefault}{\updefault}{\tiny$\vv{H}_\eps$}%
}}}}
\put(3400,1493){\makebox(0,0)[lb]{\smash{{\SetFigFont{12}{14.4}{\rmdefault}{\mddefault}{\updefault}{\tiny$H_0$}%
}}}}
\end{picture}%
\caption{The local models} 
\label{tri} 
\end{figure} 

This explicit definition of the handle agrees with the previous one by parallel transport. Indeed,
since $H_\eps$ projects to $\on{Im}(z) = 2\eps$, the tangent space to
$TH_\eps$ consists of a corank one sub-bundle
$T^v H_\eps := TH_\eps \cap \on{Ker}(D \pi)$ and the horizontal lift
of $T ( \{ \on{Im}(z) = 2 \eps \} )$.  Hence $H_\eps$ is
obtained by symplectic parallel transport of any fiber.  
The Lagrangian  $\hat{H}_\eps$ defined in this way is cylindrical near infinity and
the argument of Proposition \ref{prop:topology} \eqref{indepc} shows that
after a change in surgery parameter the two definitions are
Hamiltonian isotopic.

\subsection{Surgery and its properties}

The surgery of an immersed  self-transverse Lagrangian is obtained by gluing in the local model of the previous section. Let $X$ be a compact symplectic manifold. 
Let $\phi_0: L_0 \to X$
  be a self-transverse Lagrangian immersion with compact, connected
  domain $L_0$.  Let
\[ x = \phi_0(x_+) = \phi_0(x_-), \quad x_+ \neq x_- \in L_0 \]
be an intersection point.  The local model for transverse Lagrangian
self-intersections (see for example Pozniak \cite[Section
3.4]{pozniak} for the more general case of clean intersection) implies
that there exist Darboux coordinates in an open ball $U \subset X$ of
$x$
\[ q_1,\ldots, q_n, p_1,\ldots, p_n \in C^\infty(U) \] 
such that the two branches of $\phi_0$ meeting at $x$ are defined by
\begin{equation} \label{twobranches} 
 L_- = \{ p_1 = \ldots = p_n = 0 \}, \quad L_+ = \{ q_1 = \ldots = q_n
= 0 \} .\end{equation} 
Let $V \subset U$ be a subset so that $H_0$ agrees with $\vv{H}_\eps$
outside of $V$. 

\begin{definition} \label{lsurg}  
\begin{enumerate} 
\item  The {\em Lagrangian surgery} of
  $\phi_0:L_0 \to X$ is the immersion with domain $L_\eps$ defined by replacing a
  neighborhood $U \cap L_0$ of the self-intersection points
  $x_-, x_+ \in L_0$ with an open subset $U \cap \vv{H}_\eps$ of the
  cylindrical-near-infinity local model $\vv{H}_\eps$:
\[ L_\eps = \left( (L_0 - V) \cup (U \cap \vv{H}_\eps) \right) / \sim \] 
where $\sim$ is the obvious identification of $H_0$ with $\vv{H}_\eps$
on the complement of $V$.  The {\em
  surgered immersion}  
\[ \phi_\eps: L_\eps \to X , \quad \phi_\eps  = (\phi_0 |_{L_0 - V})
\cup (\vv{\phi}_\eps |_{\vv{H}_\eps \cap U} ) \] 
is defined by patching together the immersions $\vv{\phi}_\eps$ of
$\vv{H}_\eps \cap U \to X$ and $\phi_0$ on
$L_\eps - V \cong L_0 - V \to X$. 
\item The {\em area of the surgery} is the integral 
  \begin{equation} \label{Aeq} A(\eps) = \int_S v^*
    \omega \end{equation}
of the symplectic form over a small holomorphic triangle $v: S \to X$ with boundary
in $\phi_0(L_0) \cup \phi_\eps(L_\eps)$, as in Figure \ref{tri} and
Equation \eqref{Aeps} below.  Equivalently, by Stokes' theorem,
$A(\eps)$ is the difference of actions
\[ A(\eps) = \int_\R \gamma_0^* \alpha - \gamma_\eps^* \alpha \]
given by the integral of the canonical one-form $\alpha$ over paths
$\gamma_0,\gamma_\eps$ from $\infty$ in $\R^n$ to $\infty$ in $i \R^n$
along $H_0$ and $\vv{H}_\eps$ in the local model; see the proof of
Lemma \ref{lem:ncorners}. 
\end{enumerate} 
\end{definition} 
 
We collect some basic properties of the surgery, most
of which will be used later.  See \cite{surger}, \cite{se:lo}, and
\cite[Chapter 10]{fooo} for more details.

\begin{proposition} \label{prop:topology} Let $\phi_0:L_0 \to X$ be an
  immersed Lagrangian with transverse ordered self-intersection point
  $(x_-,x_+) \in L_0^2$.
\begin{enumerate} 
\item \label{orient} {\rm (Orientation)} If $L_0$ is oriented and
  $\eps > 0$ then there exists an orientation on $L_\eps$ that agrees
  with that on $L_0$ in a complement of the handle $\vv{H}_\eps$ if
  and only if the self-intersection $x \in L_0^2$ is odd.
\item \label{spin} {\rm (Relative spin structure)} Assuming $x$ is odd, 
any relative spin structure on $\phi_0: L_0 \to X$ 
  and an isomorphism  $ \Spin(TL_0)_{x_-}  \cong \Spin(TL_0)_{x_+} $  defines a relative
spin structure on the surgery $\phi_\eps:L_\eps \to X$.
\item \label{indepc} {\rm (Independence of choices)} The exact isotopy
  class of the surgery $\phi_\eps$ is independent of all choices, up
  to a change in surgery parameter $\eps$.
\end{enumerate}
\end{proposition} 

\begin{proof} \label{editing}
  Item \eqref{orient} follows from the fact that the gluing maps on
  the ends of the handle are homotopic to $(t,v) \mapsto e^t v $ resp.
  $(t,v) \mapsto i e^{-t} v $.  These maps are orientation preserving
  exactly if the intersection is odd\footnote{Recall that the
    self-intersection is odd if in the local model 
    $L_0$ in a neighborhood of $x_-$ resp. $x_+$ in $L_0$ is
    identified with $\R^n$ with the standard orientation induced by
    the volume form $\dd q_1 \wedge \ldots \wedge \dd q_n$ resp. $i\R^n$
    with the opposite orientation
    $- \dd p_1 \wedge \ldots \wedge \dd p_n$.  Reversing the sign of
    $\eps$ changes the order of the branches, and so changes the
    parity of the self-intersection if and only if $\dim(L_0)$ is odd.  Thus in
    the case that $\dim(L_0)$ is odd, there is always some choice of
    sign of $\eps$ \label{rep:signeps} for which the oriented surgery exists regardless of
    the parity of $x = (x_-,x_+)$. On the other hand, if $\dim(L_0)$
    is even then either both surgeries exist as oriented surgeries or
    neither.   }.  For item
  \eqref{spin}, suppose a relative spin structure is given as a
  relative \v{C}ech cocycle as in \cite[Section 3]{orient}.  Such a cocycle
  consists of charts $U_\alpha, \alpha \in A$ for $X$ indexed by some
  set $A$, corresponding charts
  $V_\alpha \subset \phi^{-1}(U_\alpha)$ for $L_0$, and transition
  functions defined as follows.  For $\alpha,\beta \in A$ let
  \[ V_{\alpha \beta} = V_\alpha \cap V_\beta, \quad \text{resp.}  \
  g_{\alpha \beta}: V_{\alpha \beta} \to SO(n) \]
denote the intersections of the charts for $L_0$ resp.  transition maps
for the tangent bundle $TL_0$.  A relative spin structure is a
collection of lifts $\ti{g}_{\alpha \beta}$ and signs
$o_{\alpha \beta \gamma}$ given by maps
\[ \ti{g}_{\alpha \beta}: V_{\alpha \beta} \to \Spin(n) , \quad
o_{\alpha \beta \gamma}: U_\alpha \cap U_\beta \cap U_\gamma \to \{
\pm 1 \} \]
  such that the following relative cocycle condition holds: \label{rcycle}
 \[ \ti{g}_{\alpha \beta} \ti{g}_{\alpha \gamma}^{-1} \ti{g}_{\beta \gamma} = \phi^*
  o_{\alpha \beta \gamma}, \quad \forall \alpha, \beta,\gamma \in A .\] 
  To obtain the relative spin structure on the surgery $L_\eps$ we
  take the cover on the surgery with a single additional open set on
  the handle $U_0 := H_\eps$ with no triple intersections.  The
  relative spin structure is defined by transition maps near the handle
  $ \ti{g}_{0\alpha} = \ti{g}_{0\beta} = \on{Id} .  $

\label{moreprecise} 
A more precise reformulation of the independence in item \eqref{indepc} is the following:
Let 
$U^1 , U^2 \subset X$ and $\vv{H}^1_{\eps_1}, \vv{H}^2_{\eps_2} \subset \CC^n$
be two sets of such choices and 
\[ \phi^1_{\eps_1}: L_{\eps_1}^1 \to X, \quad \phi^2_{\eps_2}: L_{\eps_2}^2 \to X \]
the corresponding families of surgeries for parameters $\eps_1,\eps_2$.   We claim that 
for any $\eps_1 \in \R$ with $|\eps_1|$ small there exists $\eps_2 \in \R$ so that
$\phi^1_{\eps_1}$ is an exact deformation of $\phi^2_{\eps_2}$.  In
particular if both $\phi^1_{\eps_1}$ and $\phi^2_{\eps_2}$ are
embeddings then $\phi^2_{\eps_2}(L_{\eps_2})$ is Hamiltonian isotopic
to $\phi^1_{\eps_1}(L_{\eps_1})$.  To prove this claim, recall that any
Lagrangian $\phi_\eps':L_\eps \to X$ nearby a given one is a graph of
a one form $\phi_\eps' = \graph(\alpha)$ for some
$\alpha \in \Omega^1(L_\eps)$ and local model
$T^\dual L_\eps \supset U \to X$.  An exact deformation is one
generated by exact one-forms, as in the terminology of Weinstein \cite{weinstein}.  Exact deformations are equivalent to deformation by Hamiltonian
diffeomorphisms in the embedded case, but not in general.  Any two
Darboux charts are isotopic after shrinking, by Moser's argument.  The
approximations $\vv{H}_\eps$ are also independent up to isotopy of the
choice of cutoff function.  Therefore, any two choices of surgery
are isotopic through Lagrangian immersions $\phi_\eps^t:L_\eps \to X$.
In particular, the infinitesimal deformation $\ddt \phi_\eps^t$ is
given by a closed one-form $\alpha_\eps^t \in \Omega^1(L_\eps)$.

The deformation may be taken to be exact by varying the surgery parameter, as follows.  If $n=2$ then the integral of
$\alpha_\eps^t$ on the additional generator corresponding to the
meridian  is, by Stokes' theorem, the evaluation of the
relative symplectic class $[\omega] \in H^2(\CC^2, \vv{H}_\eps)$ on the
generator in $H_2(\CC^2, \vv{H}_\eps)$.  Such a generator is given by a disk
$u: S \to X, S = \{ |z | \leq 1 \} $ with boundary $u(\partial S)$ on
the meridian $S^{1} \times \{ 0 \}$ of the handle.  The disk $u$ may
be deformed to a disk $u_0: S \to X$ taking values in $\R^2$, and so
has vanishing area $A(u)= A(u_0) = 0$.  
Returning to the case of arbitrary $n \ge 2$,  the action
$\int_\R \gamma^*_\eps \alpha$ along a longitude
$\gamma_\eps: \R \to \vv{H}_\eps$ takes on all positive values near $0$ as $\eps$ varies.
It follows that for any such $\phi_\eps^t, t \in [0,1]$ there exists a family
$\eps(t), \eps(0) = \eps$ such that the deformation is given by an
exact form.  Compare Sheridan-Smith \cite[Section 2.6]{ss:cob}.
\end{proof} 

\begin{remark} \label{rem:grad} {\rm (Gradings)} By Seidel
  \cite{se:gr}, absolute gradings on Floer cohomology groups are
  provided by gradings of the Lagrangians:  Given a positive integer $N$, an {\em $N$-grading} of a Lagrangian
  $L$ is a lift of the natural map 
  \[ L \to \Lag(TX), \quad x \mapsto T_x L  \] 
  from $L$ to the bundle of
  Lagrangian subspaces $\Lag(TX)$ to an $N$-fold Maslov cover
  \[ \Lag^N(TX) \to \Lag(TX) . \] 
  If $L_0$ \label{L0} is graded by a map
  $\ti{\phi}_0: L_0 \to \Lag^N(X)$ and the self-intersection point has
  degree $1$ then $\phi_\eps:L _\eps \to X$ is graded \cite[Section
  2e]{se:gr}.
\end{remark}

\begin{remark} {\rm (Brane structures)} \label{rem:branestr} A {\em brane
    structure} for $\phi_0$ consists of an orientation, relative spin
  structure, and $\Lambda_0$-valued local system
  $y \in \RR(\phi_0)$.  For \label{forfor}  any holomorphic treed disk $u: C \to X$
  the holonomy of the local system around the boundary is denoted by
\[  y(\partial u) \in \Lambda_0 , \quad 
y: \pi_1(\phi(L)) \to \Lambda_0. \]   
Any local system on $\phi_0$ induces a local system on $\phi_\eps$, 
trivial on the handle, by identifying the local system on the handle with the fiber of the local system over the self-intersection point.   Remark \ref{rem:grad} and Proposition \ref{prop:topology} imply that any brane
structure on $\phi_0$ induces a ($\Z_2$-graded) brane structure on $\phi_\eps$.
\end{remark} 

\section{Treed holomorphic disks} 
\label{sec:fukalg}

We recall the construction of a 
strictly unital \ainfty algebra from Charest-Woodward
\cite[Theorem 4.1]{flips} and Palmer-Woodward \cite[Theorem 6.2]{pw:flow}.

\subsection{Treed disks} 

As usual in symplectic topology we require terminology for stable disks and related concepts.  A {\em disk} will mean a $2$-manifold-with-boundary $S_{\white}$ equipped with a complex structure so that the surface $S_{\white}$ is biholomorphic to the
closed unit disk $ \{ z \in \C \ | \ |z| \leq 1 \}$.  A {\em sphere}
will mean a complex one-manifold $S_{\black}$ biholomorphic to the
complex projective line
$\P^1 = \{ [\zeta_0:\zeta_1] \ | \ \ \zeta_0,\zeta_1 \in \C \}$.  A
{\em nodal disk} $S$ is a union
\[ S = \left( \bigcup_{i=1}^{n_{\white}} S_{\white,i} \right) \cup 
\left(\bigcup_{i=1}^{n_{\black}} S_{\black,i} \right) \Big/ \sim \]
of a finite number of disks $S_{\white,i}, i = 1,\ldots, n_{\white}$
and spheres $S_{\black,i}, i =1,\ldots, n_{\black}$ identified at
pairs of distinct points called {\em nodes} $w_1,\ldots, w_m $.  Each
node 
\[ w_e = (w_e^-,w_e^+) \in S_{i_-(e)} \times S_{i_+(e)} \]  
is a pair of distinct points (either both interior or both boundary
points) where $S_{i_\pm(e)}$ are the (disk or sphere) components
adjacent to the node; the resulting topological space $S$ is required
to be simply-connected and the boundary $\partial S$ is required to be
connected. \label{connbound} The complex structures on the disks
$S_{\white,i}$ and spheres $S_{\black, i}$ induce a complex structure
on the tangent bundle $TS$ (which is a vector bundle except at the
nodal points) denoted $j:TS \to TS$.  A {\em boundary resp. interior
  marking} of a nodal disk $S$ is an ordered collection of non-nodal
points
\begin{equation} \label{zprime}
\ul{z} = (z_0,\ldots, z_d) \in \partial S^{d+1} \quad \text{resp.} \quad
\ul{z}' = (z_1',\ldots,z_c') \in \on{int}(S)^c \end{equation}
\label{addsp}
on the boundary resp. interior, whose ordering is compatible with the
orientation on the boundary $\partial S$.  The {\em combinatorial
  type} $\Gamma(S)$ is the graph whose vertices, edges, and head and
tail maps
\[ (\on{Vert}(\Gamma(S)), \Edge(\Gamma(S))), (h \times t):
\Edge(\Gamma(S)) \to \on{Vert}(\Gamma(S)) \cup \{ \infty \} \]
are obtained by setting $\on{Vert}(\Gamma(S))$ to be the set of disk
and sphere components and $\Edge(\Gamma(S))$ the set of nodes
(each connected to the vertices corresponding to the disks or spheres
they connect).  The graph $\Gamma(S)$ is required to be a tree, which
means that $\Gamma$ is connected with no cycles among the
combinatorially finite edges; each edge $e \in \Edge(\Gamma(S))$ is oriented
so that it points towards the outgoing leaf $e_0 \in \Edge(\Gamma(S))$
corresponding to the marking $z_0$. An edge $e$ is combinatorially finite
if neither \label{rep:ifneither} of its ends are at infinity.  The set of edges
$\Edge(\Gamma(S))$ is equipped with a partition into subsets
$\Edge_{\black}(\Gamma(S)) \cup \Edge_{\white}(\Gamma(S))$
corresponding to interior resp. boundary markings respectively.  The set
of boundary edges
$(h^{-1}(v) \cup t^{-1}(v)) \cap \Edge_{\white}(\Gamma(S))$ meeting
some vertex $v \in \Ver(\Gamma(S))$ is equipped with a cyclic ordering
giving $\Gamma(S)$ the partial structure of a ribbon graph.  Define
\begin{eqnarray*} \Edge_{\rightarrow}(\Gamma(S)) &:=&
h^{-1}(\infty) \cup t^{-1}(\infty) \\
 \Edge_{\white,\rightarrow}(\Gamma(S)) &:=& \Edge_{\white}(\Gamma(S))
\cap \Edge_{\rightarrow}(\Gamma(S))   \\
\Edge_{\black,\rightarrow}(\Gamma(S)) &:=& \Edge_{\black}(\Gamma(S)) \cap
\Edge_{\rightarrow}(\Gamma(S)) \end{eqnarray*}
The sets $\Edge_{\white,\rightarrow}(\Gamma(S))$,
$\Edge_{\black,\rightarrow}(\Gamma(S))$ of boundary and interior {\em
  semi-infinite edges} are each equipped with an ordering; these
orderings will be omitted from the notation to save space.  A marked
disk $(S,\ul{z},\ul{z}')$ is {\em stable} if it admits no non-trivial
automorphisms $\varphi:S \to S$ preserving the markings $\ul{z},\ul{z}'$.  The moduli
space of stable disks with fixed number $d \ge 3 $ of boundary
markings and no interior markings admits a natural structure of a cell complex which identifies the moduli space with Stasheff's
associahedron.

Treed disks are defined by replacing nodes with broken segments as in
the pearly trajectories of Biran-Cornea \cite{bc:ql} and Seidel
\cite{seidel:genustwo}.  A {\em
  segment} will mean a closed \label{aclosed} one-manifold with boundary
  homeomorphic to a connected closed subset of the real line. 
  Given two such subsets $T_{e,1}, T_{e,2}$, one with a non-compact end at infinity 
  and another with a non-compact end at infinity we may form 
  a new closed manifold 
  \begin{equation} \label{eq:comps}
  T_e = T_{e,1} \cup \{ \infty \} \cup T_{e,2} \end{equation}
  by gluing along the infinite ends,
  which we call a {\em singly-broken segment}; the point $\{ \infty \}$ is called a {\em breaking}; repeating the gluing process gives a (multiply) broken segment.   A {\em treed disk} is a topological
space $C$ obtained from a nodal disk $S$ by replacing each boundary
node or boundary marking corresponding to an edge $e \in \Gamma(S)$
with a (possibly broken) segment $T_e$.  Each such segment $T_e$ is naturally equipped with a {\em length}
\[ \ell(e)\in\R_{\geq 0} \cup \{\infty\} \] 
where the semi-infinite edges
$e \in \Gamma(S)$ are automatically assigned infinite length.  \label{inflength} A treed disk $C$ may be written as a union
$C = S \cup T$ where the one-dimensional part $T$ is joined to the
two-dimensional part $S$ at a finite set of points on the boundary of
$S$, called the {\em nodes} $w \in C$ of the treed disk (as they
correspond to the nodes in the underlying nodal disk.)  The
semi-infinite edges $e$ in the one-dimensional part $T$ are oriented
by requiring that the root edge $e_0$ is outgoing while the remaining
leaves $e_1,\ldots, e_d$ are incoming; the outgoing
leaf $e_0$ is referred to as the {\em root} while the
other semi-infinite edges $e_1,\ldots, e_d$ are {\em
  leaves}. 

  In particular, we have the following gluing construction which produces
  treed disks from a pair of treed disks.  Given treed disks $C_1,C_2$ and and a leaf $T_{e_2}$ of $C_2$
  one may glue together $C_1$ and $C_2$ by identifying the point at infinity along the root edge $T_{e_1}$ of $C_1$ with the point at infinity for an incoming leaf
  of $C_2$, creating a treed disk 
\begin{equation} \label{cglue} C = C_1 \cup \{ \infty \}\cup C_2 , 
\quad T_e := T_{e_1} \cup \{ \infty \} \cup T_{e_2} \end{equation} 
with a broken edge $T_e \subset C$.  We say that 
  the treed disks $C_1,C_2$ are obtained from $C$ by {\em cutting along $e$.}
    \label{root} The combinatorial type
\[ \Gamma(C)= (\on{Vert}(\Gamma(C)), \Edge(\Gamma(C))) \]  
of a treed
disk $C$ is defined similarly to that for disks with the following
addition: The set of edges $\Edge(\Gamma(C))$ is equipped with a
partition
\[ \Edge(\Gamma(C)) = \Edge_0(\Gamma(C)) \cup \Edge_{(0,\infty)}(\Gamma(C)) \cup
\Edge_\infty(\Gamma(C)) \]
indicating whether the length is zero, finite and non-zero, or
infinite.    

The space of isomorphism classes of treed disks satisfying a stability condition 
is compact and Hausdorff with a universal curve.  
A treed disk $C = S \cup T$ is {\em stable} if the
underlying nodal disk obtained by collapsing edges $T_e \subset T$ to
points is stable.  An example of a treed disk with one broken edge
(indicated by a small hash through the edge) is shown in Figure
\ref{treeddisk}.  In the Figure, the interior leaves
$e \in \Edge_{\black}(\Gamma)$ are not shown and only their attaching
points $w_e \in S \cap T$ are depicted so as not to clutter the
figure.
\begin{figure}[ht] 
\begin{center} \includegraphics[height=2.5in]{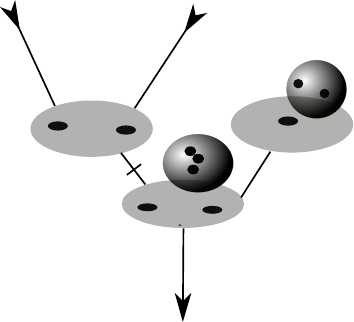} 
\end{center} 
\caption{A treed disk with $d = 2$ incoming boundary edges} 
\label{treeddisk} 
\end{figure} 
For a given combinatorial type
$\Gamma$, denote by ${\M}_\Gamma$ the moduli space of treed disks whose
domains have combinatorial type $\Gamma$ and
\[ \ol{\M}_d = \cup_\Gamma {\M}_\Gamma  \] 
the union over stable
types $\Gamma$ with $d$ leaves. The moduli space $\ol{\M}_d$ is
a compact Hausdorff space which admits locally finite  decomposition into smooth submanifolds.  It admits a universal curve $\ol{\U}_d$, given as the space of isomorphism classes of pairs $[C,z]$ where $C$ is a holomorphic treed disk and $z \in C$ is any point, either in $S$ or $T$.  The two cases correspond to a splitting
\begin{equation} \label{eq:splitting} \ol{\U}_d = \ol{\S}_d \cup
  \ol{\T}_d \end{equation}
of the universal treed disk into one-dimensional and two-dimensional
parts $\ol{\T}_d$ resp.  $\ol{\S}_d$ where the fibers of
$\ol{\T}_d \to \ol{\M}_d$ resp.  $\ol{\S}_d \to \ol{\M}_d$ are one
resp. two-dimensional.  Denote by ${\S}_\Gamma$ resp. $ {\T}_\Gamma$ the
surface resp. tree parts of the universal treed disk living over ${\M}_\Gamma$, 
and similarly their closures $\ol{\S}_\Gamma, \ol{\T}_\Gamma$
over $\ol{\M}_\Gamma$. If for some types $\Gamma',\Gamma$ the moduli space $\cM_{\Gamma'}$ is contained in $\ol{\cM}_\Gamma$ then we write $\Gamma' \preceq \Gamma$.

\subsection{Treed holomorphic disks}
\label{Jstd}

Treed holomorphic disks for immersed Lagrangians are defined as in the
embedded case, but requiring a double cover of the tree parts to
obtain the boundary lift.   Recall that a {\em Morse-Smale pair} on $L$ is a
  pair $(f,g)$ consisting of a Morse function  and
  Riemannian metric 
\[ f:L \to \R, \quad g: TL \times_L TL \to \R \]  
so that the stable and
  unstable manifolds of $(f,g)$ meet transversally.  Any Morse-Smale
  pair on $L$ gives rise (somewhat non-canonically) to a cellular
  decomposition whose cellular chain complex is equal to the Morse
  chain complex of $L$ (see for example \cite[Theorem 4.9.3]{audin:morse}).  The Morse function $f$ induces a Morse function
  $\ti{f}$ on the fiber product 
  \[ L \times_\phi L \cong L \cup \cI^{\on{si}}(\phi) , \] 
  by setting $\ti{f}$ to equal $f$ on $L$ and any function on the discrete
  set $\cI^{\on{si}}(\phi)$.    Let $J: TX \to TX$ be an almost complex structure compatible
  with the symplectic form $\omega \in \Omega^2(X)$.   We will assume that $J$ is adapted to the local
  intersections of $\phi: L \to X$ in the following sense: For any
  self-intersection there is a Darboux chart on $U \subset X$ as in
  \eqref{twobranches} so that 
\[ L \cap U \cong \R^n \cup i \R^n \]  
and $J$ is given by the standard complex structure on
$U \subset \CC^n$.  

A holomorphic treed disk consists of a map from treed disk together with a lift of the  map on the boundary to  the
Lagrangian. Since our Lagrangians are only immersed, the domain of the boundary 
map is a one-manifold with boundary as follows:
Given a treed disk $C = S \cup T$,  denote by
\begin{equation} \label{wnot} \widetilde{\partial S} = (\partial S -
  ((\partial S) \cap T)) \cup \{ w_<, w_>  \in (\partial S) \cap T
  \} \end{equation}
the compact one-manifold obtained by replacing each element $w$ of
$\partial S \cap T$ with a pair of points $w_<,w_>$ lying in the closure of the component of the boundary 
$\partial S - (\partial S) \cap T$ which lies before resp. 
after the intersection point.  Each
component of the boundary $(\partial S)_i \subset \partial S - T$ has
closure in $\widetilde{\partial S}$ that is homeomorphic to a closed
interval.  Let
\[ \iota: \widetilde{\partial S} \to S \]
denote the canonical map that is generically $1-1$ except for the
fibers over the intersection points $S \cap T$ which are $2-1$.
Let 
\[ \widetilde{\partial C} = (\widetilde{\partial S} \cup T \cup T)/ \sim  \] 
denote the union of $\widetilde{\partial S} $ with two copies of each edge, 
with the endpoints of the edges identified with the two inverse images of the intersection points.   Consider a pair of maps 
\[ u: C \to X, \quad \partial u:  \widetilde{\partial C} \to L \]
so that 
\begin{equation} \label{eq:match}    u \circ \iota = \phi \circ \partial u  .
\end{equation}
We introduce the following notation for vector fields.  Let $\Vect(X)$
denote the space of vector fields on $X$, and let
$\Vect_h(X) \subset \Vect(X)$ denote the subset of Hamiltonian vector
fields.  For any subset $U \subset X$, let
\begin{equation} \label{vecth}
 \Vect_h(X,U) \subset \Vect(X) \end{equation}
be the space of Hamiltonian vector fields vanishing on $U \subset X$.
\label{vectX}
Let 
\[ H \in \Omega^1(S,\Vect_h(X))  \]  
be a one-form with values in Hamiltonian vector fields supported in 
the interior of $S$.  Denote by
\[ \dd_H u = \dd u - H \circ u \in \Omega^1(S, u^* TX) \]
the Hamiltonian-perturbed exterior derivative.

\begin{definition} \label{jtdisk} A {\em $(J,H)$-holomorphic treed disk} with boundary in
  $\phi: L \to X$ of type $\Gamma$ consists of a treed disk $C = S \cup T$
  of type $\Gamma$ and 
  continuous maps 
  \[ u: C \to X, \quad \partial u: \widetilde{\partial C} \to L \]
satisfying the matching condition \eqref{eq:match} and 
\begin{enumerate} 
\item the map $u$ is smooth and  \label{rep:smooth} 
$(J,H)$-holomorphic on $S - (S \cap T)$, that is,
  \begin{equation} \label{conds}
    J \dd_H  (u  |_{S}) = \dd_H  (u |_{S} ) j  \end{equation}
\item the restriction of $\partial u$ to $\widetilde{T} := T \cup T \subset \widetilde{C}$ is smooth and a gradient trajectory for the given Morse function $\ti{f}$ on $L \times_\phi L$,  that is, a gradient trajectory for $f$ on the locus in $T \cup T$ giving a map to $L$, and a constant map on the self-intersection points.
\end{enumerate} 
The definition can be rephrased in terms of a matching condition at the endpoints of the edges.   For each edge $T_e$ adjacent to components $S_{v_\pm}$ with intersection points $ \{ w_+, w_-\} = T_e \cap S $, if 
 the edge $T_e$ maps to  $L \subset L \times_\phi L$ then 
 the endpoints $u(w_+), u(w_-)$ are related by the gradient flow determined by the length of $T_e$, while if the edge $T_e$ maps to a self-intersection point then 
 the endpoints $u(w_+),u(w_-) $ are equal.   We denote a treed disk by $(C,u: C \to X)$, omitting $\partial u$ to save space.  An {\em isomorphism} between treed disks $(C,u:C \to X)$ and $(C',u:C' \to X)$ is an isomorphism of treed disks $\psi:C \to C'$ so that $u' \circ \psi =  u$.
\end{definition} 

A compactified moduli space for any type is obtained after imposing a
stability condition.  

\begin{definition} 
A holomorphic treed disk
$(C, \ u: C \to X)$ is {\em stable} if it has no non-trivial
automorphisms $\psi: C \to C$, or equivalently
\begin{enumerate} 
\item each disk component $S_{v} \subset S, v \in \Ver_{\white}(\Gamma)$ on which the map
  $u$ is constant (that is, a ghost disk bubble) has at least one
  interior node $w_e \in \on{int}(S_{v})$ or has at least three
  boundary nodes $w_e \in \partial S_{v}$;
\item each sphere component $S_{v}
  \subset S, v \in \Ver_{\black}(\Gamma)$ on which the map
  $u$
  is constant (that is, a ghost sphere bubble) has at least three
  nodes $w_e \in \partial S_{v}$; and
\item on each possibly broken edge $T_e$ with components $T_{e,i}$ as in \eqref{eq:comps} if 
$T_{e,i}$ has two infinite ends then the map $u |T_{e,i}$ is non-constant. 
\end{enumerate}
\end{definition}

The combinatorial data of a treed holomorphic disk is packaged into a
labeled graph called the {\em combinatorial type}:
\begin{definition} \label{typeofmap} For a holomorphic treed disk
  $u: C \to X$ the combinatorial type $\bGamma$ of $u$ is the combinatorial
  type $\Gamma$ of the underlying treed disk $C$ together with
the labelling of vertices $v \in \on{Vert}(\Gamma)$
  corresponding to sphere and disk components
  $S_v, v \in \Ver(\Gamma)$ by the (relative) homology classes
  \[ u_*[S_v] \in H_2(X) \cup H_2(X,\phi(L)) \] 
  and the labelling 
  \[ t(e) \in \{ 1, 2\} \] 
  of edges $e \in \Edge(\Gamma)$ by their branch
  type (whether they represent a branch change of the map $\partial u:
  \widetilde{\partial S} \to \phi(L)$ or not).
The total homology class of a type
$\bGamma$ of positive area is called {\em primitive} if it is not the sum
of homology classes of types $\bGamma_1,\bGamma_2$ of smaller positive
area.   We write $\bGamma \mapsto \Gamma$ if $\Gamma$ is the domain 
type of a map type $\bGamma$.
\end{definition}
For any combinatorial type of map $\bGamma$ denote by
${\M}_{\bGamma}(\phi)$ the moduli space of stable treed
holomorphic disks bounding $\phi$ of type $\bGamma$.  \label{typefixed} Denote by 
\[ \M_\Gamma(\phi) = \bigcup_{\bGamma \mapsto \Gamma} \M_\bGamma(\phi) \] 
the union of strata of stable map type $\bGamma$ with domain type $\Gamma$
and
\[ \ol{\M}_d(\phi) = \bigcup_\Gamma \M_\Gamma(\phi) \]
the union over combinatorial types with $d$ incoming edges.  
The case that $C$ consists of a single edge so that $S = \emptyset$ and $T = \R$ is allowed; in this case $\M_\Gamma(\phi)$ is defined to be the space of non-constant unparametrized gradient trajectories on $L$. 

Each stratum is cut out by a
Fredholm map of Banach spaces as explained in Palmer-Woodward \cite{pw:flow}.  Since later we will have to generalize the discussion to buildings, we review the 
construction.  Let $u: S \to X$ be a map
of type $\bGamma$.  Let
\[ S^\circ = S - \{ w \in S \cap T_e, t(e) = 2\} \]  
(where $t(e)$ was the number of branches of the map $u$ on the edge
$e$ defined in \ref{typeofmap}) denote the complement of the points
$w \in S \cap T$ representing branch changes of the map
$\partial u: \widetilde{\partial S} \to L$.  The surface $S^\circ$ is
naturally a surface with {\em strip-like ends}: For each
$w \in S \cap T_e$ above there exists a proper embedding of manifolds
with boundary
\[ \eps_w: \R_{> 0} \times [0,1] \to S^\circ , \quad \lim_{s \to
  \infty} \eps_w(s,t) = w, \quad \forall t \in [0,1] \]
such that the complex structure on $S$ pulls back to the standard
complex structure in the coordinates $s,t$.  For a Sobolev exponent
$p \ge 2 $ and Sobolev differentiability constant $k \ge 1$ with
$kp > 2$ \label{kp}, let
\[ \Map^{k,p}(S^\circ,X) = \{ u = \exp_{u_0} (\xi), \quad \xi \in
\Omega^0(S^\circ, TX)_{k,p} \} \]
denote the space of continuous maps $u: S^\circ \to X$ of the form
$u = \exp_{u_0} (\xi)$ where $u_0$ is constant in a neighborhood of
infinity along the strip-like ends and
$\xi \in \Omega^0(S^\circ, TX)_{k,p}$ has finite $W^{k,p}$ norm.  In particular, $\xi$ has $k$ covariant derivatives
in $L^p$ using a connection on $X$ and a metric on $S^\circ$ that is of product form on the ends.
For each branched edge $e$, let 
\[ w_\pm(e) \in \partial S \cap T \]  
denote the points at the end of each finite edge $T_e \subset T$
(distinguished by requiring that with the given orientation 
of $T_e$, the segment $T_e$ points from $w_-(e)$ to $w_+(e)$)
Let $\delta_t: L \to L$ be the time $t$ gradient flow of $f$, for $  t \ge 0$.  Let $T_1$ resp. $T_2 \subset T$ be the locus on which $u$ is unbranched resp. branched and
$\Edge_1(\Gamma),\Edge_2(\Gamma)$ the corresponding subsets of edges.
Define
\begin{multline} \label{bundlebase} \B_\Gamma = \left\{ 
    \begin{array}{l} (C,u, \partial u,l) \in \left(
        \M_\Gamma(\phi) \times \Map^{k,p} (S^\circ, X)  \right.  \ \text{so} \   \\
      u \circ \iota = \phi \circ\partial u, \ 
       u(w_{e,+}) = \delta_{\ell(e)}(u(w_{e,-})), \
      \forall e \in \Edge_1(\Gamma) ,\\ u(w_{e,-}) =
      u(w_{e,+}) \ \forall e \in \Edge_2(\Gamma) 
    \end{array}  \right\}.\end{multline}
Boundary values of $W^{k,p}$ maps lie in
$W^{k-1/p,p}$ (see \cite[(0.15)]{lm}).  Maps close to any
given pair $(u,\partial u)$ are exponentials
$\exp_u(\xi), \exp_{\partial u}(\partial \xi)$ of sections
\[ \xi \in \Omega^0(S^\circ, u^* TX)_{k,p}, \quad \partial \xi \in
\Omega^0(\partial S^\circ, (\partial u)^* TL)_{k-1/p,p}, \]
where the subscript denotes Sobolev class $W^{k,p}$, satisfying
\[ \xi \circ \iota = D \phi \circ\partial \xi .\] 
Here exponentiation means geodesic exponentiation using, for example,
a metric on $X$ for which each branch of $\phi(L)$ is totally
geodesic.  The fiber of the bundle $\cE_\Gamma$ over some map $u$ is
the vector space of one-forms
\begin{equation} 
\label{bundle} 
\cE_{\Gamma,u} := 
\Omega^{0,1}(S, u_S^* TX)_{k-1,p} . \end{equation} 
Local charts are provided by almost complex parallel transport 
\begin{equation} \label{trans} \cT_u^\xi:   \Omega^{0,1}(S^\circ, \exp_u(\xi)^* TX)_{k-1,p} \to
\Omega^{0,1}(S^\circ, u^* TX)_{k-1,p} \end{equation} 
along $\exp_u(s\xi)$ for $s \in [0,1]$; note here that the connection
used for parallel transport $\cT_u^\xi$ need not be related to the
metric used for geodesic exponentiation.  In any local
trivialization of the universal curve, one can obtain Banach bundles
with arbitrarily high regularity.  Let
\begin{equation} \label{Ui} \cU^i_\Gamma \to \cM_\Gamma^i \times
  C \end{equation} 
be a collection of local trivializations of the universal curve.  Let
$\B_\Gamma^i$ denote the inverse image of $\cM_\Gamma^i$ in
$\B_\Gamma$ and $\cE_\Gamma^i$ its preimage in $\cE_\Gamma $. 
The Fredholm map cutting out the moduli space over $\cM_\Gamma^i$ is \label{noof} 
\begin{equation} \label{cutout} 
 \cF_\Gamma^i: \B_\Gamma^i \to \cE_\Gamma^i,  \quad  u \mapsto \olp_{J,H} u 
\end{equation}
   The linearization of the map \eqref{cutout} cutting out the moduli
   space is a combination of the standard linearization of the
   Cauchy-Riemann operator with additional terms arising from the
   variation of conformal structure.  With $k,p$ integers determining
   the Sobolev class as above let
\begin{equation} 
\label{du}
\begin{split}
  D_u: \Omega^0(S^\circ, u^* TX, (\partial
  u)^* TL )_{k,p} &\to     \Omega^{0,1}(S^\circ, u^* TX)_{k-1,p} \\
  \xi &\mapsto \nabla_H^{0,1} \xi - \hh (\nabla_\xi J) J \partial_{J,H} u
  \end{split}
 \end{equation}
 denote the linearization of the Cauchy-Riemann operator, cf.
 McDuff-Salamon \cite[p. 276]{ms:jh}; here
 \[ \partial_{J,H} u = \frac{1}{2} (\dd_H u - J \dd_H u j ) .\]
The complex structures on the fibers induce a family
\begin{equation} \label{localtriv} {\M}_{\Gamma}^i \to \J(S), \quad m \mapsto j(m) \end{equation}
of complex structures on the two-dimensional locus $S \subset C$, and
In particular, any tangent vector $\zeta \in T\M_\Gamma^i$ induces a
variation $Dj: TS \to TS$ of complex structure on $S$: Let
\[ \Omega^0(S^\circ, \partial S^\circ; u_S^* TX; (\partial u)^* TL
)_{k,p}  \subset \Omega^0(S^\circ; u_S^* TX)_{k,p} \] 
denote the subspace of sections $\xi$ whose boundary values lift to
$W^{k-1/p,p}$-sections $\partial \xi$ with values in $(\partial u)^*TL$.  The tangent
space to $\B_\Gamma$ is the space of
deformations  $(\zeta_S,\zeta_T,\xi)$ preserving the matching conditions given by 
\[ 
  T_{(C,u)} \B_\Gamma 
= \Set{ ( \zeta_S,\zeta_T,\xi)
    | \begin{array}{l} ( \xi(w_{e,+}) = 
      D\delta_{\ell(e)} ( \xi(w_{e,-}),  \zeta_T(e))
            \\ \forall e \in
      \Edge_1(\Gamma) \end{array} }  \] 
  where $\delta(l,t) := \delta_t(l)$ is the cellular approximation.
  The {\em linearized operator} for the map $u$ is given by the
  expression
\begin{multline} \label{linop} \ti{D}_u: T_{(C,u)} \B_\Gamma \to
  \Omega^{0,1}(S^\circ, \partial S^\circ; u^* TX )_{k-1,p} \\ \quad (
  \zeta_S,\zeta_T, \xi) \mapsto \left( D_u \xi + \frac{1}{2} J \dd u
    Dj(\zeta_S) \right) . \end{multline}
 A holomorphic treed disk $u: C \to X$ with stable domain $C$ is
\begin{enumerate} 
\vskip .1in
\item[]  {\em
   regular} if the linearized operator $\ti{D}_u$ is surjective; 

\vskip .1in 

\item[] {\em stratum-wise rigid} if $u$ is regular and $\ti{D}_u$ is
  surjective and the kernel of $\ti{D}_u$ is generated by the
  infinitesimal automorphism $\on{aut}(S)$ of $S$; and

\vskip .1in 

\item[] {\em rigid} if $u$ is stratum-wise rigid and the domain $C$
  lies in a top-dimensional stratum $\M_\Gamma$ in the moduli space of
  domains $\ol{\M}_d$.
\end{enumerate} 

\vskip .1in

The moduli space of holomorphic treed disks admits a natural version
of the Gromov topology which allows bubbling off spheres, disks, and
cellular boundaries.   For Hamiltonian-perturbed maps, the Hamiltonian-perturbed energy is
\begin{equation} \label{EH} E(u) = \frac{1}{2} \int_S \Vert \dd_H u  
  \Vert^2 \dd \Vol_S  
\end{equation} 
where $\dd \Vol_S$ is the area density on the surface $S$. 
The {\em area} of a map type $\bGamma$ is the sum of the pairings of
the homology classes of the disk and sphere components $u_*[S_v]$ with
the symplectic class $[\omega]$.   The energy
$E(u)$ is equal to the area $A(u)$ up to a curvature term explained
in \cite[Chapter 8]{ms:jh}.  Consider a sequence \label{aseq}
\[ u_\nu: C_\nu \to X, \quad \partial u_\nu: \widetilde{\partial C_\nu} \to
L, \quad \nu \in \N \]
of treed holomorphic disks with boundary in $\phi$ with bounded energy
$E(u_\nu) > 0 $.  Gromov compactness with Lagrangian boundary
conditions as in, for example, Frauenfelder-Zemisch \cite[Theorem 1.1]{totreal}
implies that there exists a subsequence with a stable limit
\[ (C,u: C \to X) := \lim_{\nu \to \infty} \ (C_\nu, u_{\nu} : C_\nu \to X) . \]   
Standard arguments using local distance functions as in McDuff-Salamon 
\cite[Section 5.5]{ms:jh} then show that for
any fixed energy bound $E > 0 $, the subset
\begin{equation} \label{Eloc}
\ol{\M}_d^{<E }(\phi) = \Set{ u \in \ol{\M}_d(\phi) \ | \ E(u) < E  
} \end{equation} 
satisfying the given energy bound $E(u) < E$ is compact and Hausdorff.  
 
The moduli space further decomposes according to the limits at
infinity and the expected dimension.    For conceptual reasons, we 
index the generators by the cells rather than critical points of the Morse function.  That is, for any $x \in \crit(f)$ let $\sigma(x) \subset L$ denote the corresponding unstable manifold in the cellular decomposition defined by $f$.  Define
\[ \cI(\phi) = \cI^c(\phi)\cup \cI^{si}(\phi) \]
where 
\[ \cI^c(\phi) := \{  \sigma(x), x \in \crit(f) \} \]
is the set of cells; and
  \[ \cI^{\si}(\phi) := (L \times_\phi L) - \Delta_L \]
is the set of ordered self-intersection points, where 
$L \times_\phi L$ is the fiber product and $\Delta_L \subset L^2$ the 
diagonal.  
For some collection of cells $\sigma_0,\sigma_1,\ldots, \sigma_d \in \cI(\phi)$ denote by
\[\ol{\M}(\phi,\ul{\sigma}) = \Set{  [C,u: C \to X] \in \ol{\M}_d(\phi)\ | \ \ \ 
   \ev_e(u) \in  \sigma_e \  \forall e \in \Edge(\Gamma) }
\]
the locus of maps such that the limit, denoted $\ev_e(u)$, of $u$ at infinity of the restriction of the map to the edge  $T_e$
is the corresponding critical point of the Morse function, so that a neighborhood of infinity maps to $\sigma_e$.
 \label{nodisks}   We consider the case of no disks in the configuration as a special case.  Let
  \[ \sigma_- \in \cI^c(\phi), \quad \sigma_+ \in \cI^{c,\dual}(\phi), \quad \deg(\sigma_+) = \deg(\sigma_-) +1  .\] 
For any integer $d$ denote by
\begin{equation} \label{Md} \M_\Gamma(\phi,\ul{\sigma})_d = \Set{ [C,u : C
  \to X] \ | \ \Ind(\ti{D}_u)  - \sum_{i=0}^d | \sigma_i | = d } \end{equation}
the locus with {\em expected dimension} $d$, where $\ti{D}_u$ is the
operator of \eqref{linop} and $| \sigma_i |$ is the codimension of the
constraint $\sigma_i$ for $i = 0,\ldots, d$.  An element of
$\M_\Gamma(\phi,\ul{\sigma})$ is {\em rigid} if it lies in the locus
$\M_\Gamma(\phi,\ul{\sigma})_0$ of expected dimension zero and
$\M_\Gamma$ is codimension zero in $\ol{\M}_d$.  
A {\em labeled type} of map is the map type $\bGamma$ with a labelling
$\ul{\sigma}$ of its edges.   A labeled map type $\bGamma$
is rigid if any (and so all) maps $(C,u: C \to X)$ of labeled type $\bGamma$ 
are rigid.

\section{Coherent perturbations} 
\label{sec:pert}

Regularization of the moduli spaces is achieved through
domain-dependent perturbations, using a Donaldson hypersurface
\cite{don:symp} to stabilize the domains as in Cieliebak-Mohnke
\cite{cm:trans}.  

\subsection{Donaldson hypersurfaces}

\begin{definition} \label{donhyp}
A {\em Donaldson hypersurface} of a compact symplectic manifold $X$ is a codimension two
  symplectic submanifold $D \subset X$ representing a multiple
  $k [\omega], k > 0$, of the symplectic class $[\omega] \in H^2(X)$.  The integer $k$ is the 
  {\em degree} of $D$.

  \vskip .1in \noindent 
  A {\em relative Donaldson hypersurface} for a Lagrangian immersion $\phi:L  \to X $ is a codimension two
  symplectic submanifold $D \subset X$ disjoint from $\phi(L)$ representing a multiple
  $k [\omega], k > 0$, of the symplectic class $[\omega] \in H^2(\phi)$.
\end{definition} 

 Donaldson's construction in \cite{don:symp} associates
  to any asymptotically holomorphic sequence of sections $s_k$ of tensor
  powers $\hat{X}^k$ of a line
  bundle $\hat{X} \to X$ with first Chern class  $c_1(\hat{X}) =
  [\omega]$ a sequence of hypersurfaces $D_k = s_k^{-1}(0)$; for $k$ sufficiently large
  the submanifold $D_k$ is a Donaldson hypersurface.  A result of Auroux \cite{auroux:asym} provides a homotopy between any
  two such choices with the same degree.  Results of Auroux-Gayet-Mohsen \cite{auroux:complement}
  show the existence of Donaldson hypersurfaces in the complement of an isotropic submanifold, 
and a result of Auroux included in 
Pascaleff--Tonkonog \cite[Theorem 3.1]{pasc:wall} extends this to the case
of cleanly-intersecting Lagrangians satisfying the Bohr-Sommerfeld condition that 
the pull-back bundle $\phi^*(\hat{X}^k \to X)$ is trivial for some $k$.  
As in  \cite[Corollary 3.4]{pasc:wall}
one may assume \label{rep:tobe} the Lagrangian to be exact  in the complement
by choosing the 
approximately holomorphic section defining the Donaldson hypersurface to
be flat on the Lagrangian.
Then Stokes' theorem implies that 
 $k$ times the area $A(u)$ of any disk $u: C \to X$ bounding $L$ is given by its intersection number $\lan [u], [D] \ran$  with $D$.  Equivalently, 
 $D$ represents $[\omega]$ in the relative cohomology group $H^2(\phi).$
 
As explained in Cieliebak-Mohnke \cite{cm:trans}, the set of
intersections of a holomorphic curve with a Donaldson
hypersurface provides an additional set of marked points that
stabilize the domain.  Let $D \subset X$ be a Donaldson hypersurface.
We say a compatible almost complex structure $J_D \in \J(X)$ is {\em stabilizing} 
if $J$ preserves $TD$, a compatible almost complex structure preserving
$D$ so that $D$ contains no non-constant holomorphic spheres as in
Cieliebak-Mohnke \cite[Section 8]{cm:trans}, and each non-constant $J_D$-holomorphic sphere
  $u: \P^1 \to X$ of energy at most $E$ intersects $D$ in
  finitely many but at least three points $u^{-1}(D)$.
\footnote{In order to
  prove independence from all choices, \cite{cm:trans} \label{rep:inorder} also consider
  tamed almost complex structures.  However, in this paper we do not prove any
  independence results so compatible almost complex structures
  suffice.}

\begin{lemma} \label{lem:sufflarge} \cite[Section 8]{cm:trans}  Suppose $D$ has sufficiently large
degree $k \gg 0$. Then any generic almost complex structure $J_D$ preserving $D$ is stabilizing, and 
for any energy bound $E > 0$ there exist an open neighborhood 
  $\J(X,J_D,E)$ of $J_D$ consisting of stabilizing almost complex structures. 
  \end{lemma} 

We use the additional markings provided by the Donaldson hypersurface to define
domain-dependent perturbations.  Choose an open neighborhood $U$ of $D$.
Recall from \eqref{vecth} that $\Vect_h(X,U)$ denotes the space of Hamiltonian vector fields
$v: X \to TX$ vanishing on $U$.

\begin{definition}  \label{def:domdep}
  For each combinatorial type of domain $\Gamma$, 
\begin{enumerate}
\item a {\em domain-dependent almost complex structure} for $\Gamma$ is a map
\[ J_\Gamma: \ol{\S}_\Gamma \to \J(X,J_D,E) \]
(notation from \eqref{eq:splitting}) smooth as a map
$\ol{\S}_\Gamma \times TX \to TX$.
\item A {\em domain-dependent Hamiltonian perturbation} for $\Gamma$ is a one-form
  \[ H_\Gamma  \in \Omega^1(\ol{\S}_\Gamma, \Vect_h(X,U))\]
  smooth as a map $T\ol{\S}_\Gamma \times X \to TX$.
\item A {\em single-valued domain-dependent matching condition} for $\Gamma$ is a map
\[ M_\Gamma : (\ol{\S}_\Gamma \cap \ol{\cT}_\Gamma) \times L \to L  \] 
such that $M_\Gamma(w_e,\cdot)$ is a diffeomorphism of $L$ for each
$w_e \in \ol{\S}_\Gamma \cap \ol{\cT}_\Gamma$.  
\item A {\em perturbation datum} is a datum
\[ P_\Gamma = (J_\Gamma,H_\Gamma, M_\Gamma)  \] 
such that $J_\Gamma$ agrees with the given almost
complex structure $J_D$ on the hypersurface $D$ and in a neighborhood
of the nodes $w_e \in S$ and boundary $\partial S$ for any fiber
$S \subset \ol{\S}_\Gamma$, and takes values in the space of almost complex structures without sphere bubbling $\J(X,J_D,
\# \Edge_\black(\Gamma)/k)$ from Lemma \ref{lem:sufflarge}.  The space of $P_\Gamma$ of
  perturbation data  is denoted 
  \[ \PP_\Gamma = \{ P_\Gamma \} .\]
\end{enumerate}
\end{definition} 

To achieve certain symmetry properties of the Fukaya algebra, multi-valued
perturbation data are required.  For example, if one \label{rep:one}  expects divisor insertions to contribute exponentials to the disk potential then one expects the factorials to appear as an averaging
factor as in Theorem \ref{thm:repthm}.
\label{punct}

\begin{definition} 
\begin{enumerate} 
\item A {\em multivalued domain-dependent matching condition} for $\Gamma$ is 
a formal sum 
\begin{equation} \label{eq:branchedM} M_\Gamma = \sum_{i=1}^k c_i 
  M_{\Gamma,i} \quad \sum_{i=1}^k c_i = 1 \quad c_i \in [0,1] \
  \forall i \end{equation}
of single-valued matching conditions $M_{\Gamma,i}$.  
\item Similarly, a {\em multi-valued domain-dependent Hamiltonian } is a
  formal sum
\begin{equation} \label{branchedH} H_\Gamma = \sum_{i=1}^l d_i
  H_{\Gamma,i} \quad \sum_{i=1}^l d_i = 1 \quad d_i \in [0,1] \
  \forall i  \end{equation}
of single-valued Hamiltonian perturbations.  
\end{enumerate}
\noindent For much of the paper, one could assume that $M_\Gamma,H_\Gamma $ are
\label{rep:much}
single-valued.  However in order to deal with repeated inputs one must
allow formal sums, that is, multivalued perturbations, as in Section
\ref{divisor}.
\end{definition} 

Given perturbations, the perturbed moduli spaces are defined as
follows. \label{given}

\begin{definition} \label{def:pmatchdef} 
For $P_\Gamma = (J_\Gamma,H_\Gamma, M_\Gamma)$, a
  $P_\Gamma$-perturbed treed holomorphic disk is a pair
  $(C,u: C \to X)$ where $C$ is of type $\Gamma$ and the equation
  \eqref{conds} is replaced with the following
  conditions:
\begin{enumerate} 
\item The map $u$ is perturbed holomorphic in the sense that 
 \begin{equation} \label{olp} \olp_{J_\Gamma,H_\Gamma} u (z) = \frac{1}{2}  \left( \begin{array}{l}                             
(\dd u(z) -
  H_\Gamma(u(z))  \\ + J_\Gamma(z,u(z))  (\dd u (z) - H_\Gamma(u(z))) j(z) ) \end{array}
\right) = 0 \end{equation}
on the surface $S$;
\item for each unbranched interior edge $e$ the perturbed matching
  condition 
  \begin{equation} \label{pmatch}  M_{\Gamma,i}(w_-(e),\delta_{\ell(e)} u(w_-(e))), M_{\Gamma,i}(w_+(e),u(w_+(e)))) \in \Delta \end{equation} 
holds for some $i$; and for each leaf $e$ labeled by a
cell $\sigma_e$ for some $i$ we have
  \[ M_{\Gamma,i}(w_e,u(w_e)) \in \sigma_e ;\]
\item  and the matching condition holds for each branched interior edge $T_e$
joining points $w_{e,\pm} \in S \cap T$:
\[ u(w_{e,-}) = u(w_{e,+}) . \]
\end{enumerate} 
The map is {\em adapted} if each connected component of $u^{-1}(D)$
contains an interior node $w_e \in S, e \in \Edge_\black(\Gamma)$ and
each such $w_e$ lies in $u^{-1}(D)$.  This ends the Definition.
\end{definition} 

\begin{remark} \label{rem:disktodiv} By Theorem \ref{thm:comeager} a generic adapted
  map $u: C \to X$ has the property that every holomorphic disk
  component $u|S_v$ meets $D$ in finitely many points $u^{-1}(D)$, and
  positively many points if the disk is non-constant. The definition
  above, however, allows constant sphere components $S_v$ mapping
  entirely to the divisor $D$, which would therefore have infinitely
  many intersections.  
\end{remark} 

The construction above naturally produces a collection of moduli
spaces satisfying an energy gap condition:

\begin{lemma} \label{lem:energyquant} Let $\phi: L \to X$ be a
  self-transverse Lagrangian immersion.  There exists an $\hbar > 0$
  such that any treed holomorphic disk $u: C \to X$ with boundary on
  $\phi$ containing at least one non-constant holomorphic component
  $u_v: C_v \to X, \dd u_v \neq 0 , v \in \Ver(\Gamma)$ has area $A(u)$
  at least $\hbar$.
\end{lemma} 

\begin{proof} By Gromov compactness, 
for $E> 0 $, the set of
  homotopy classes $[u] \in \pi_2(\phi)$ of stable holomorphic disks
  $u: S \to X, S = \{ |z| \leq 1 \}$ with energy bound $E(u) < E $
  is finite.  It follows that the set $\{ A(u), \dd u \neq 0 \}$ of
  non-zero energies of disks $u: S \to X$ bounding $\phi$ has a
  non-zero minimum $\min \{ A(u), \dd u \neq 0 \}$, which we may take to
  equal $\hbar$.
\end{proof} 

Our constructions produce Fukaya algebras over Novikov rings (with positive valuation) rather than Novikov fields by use of the following:

\begin{lemma} \label{lem:posarea} For a regular Hamiltonian
  perturbation $H_\Gamma$ that is sufficiently small in the $C^\infty$
  topology, the areas $A(u)$ of all rigid
  $(J_\Gamma,H_\Gamma)$-holomorphic treed disks $u: C \to X$ are   non-negative. 
  \end{lemma}
  
  \begin{proof} The areas of such configurations are
  topological quantities, that is, depend only on the homotopy type of
  the map.  The set of homotopy types $[u]$
  achieved by holomorphic maps $u: S \to X$ is unchanged by the
  introduction of a perturbation $H_\Gamma$, by a standard argument
  using Gromov compactness.  Any
  $(J_\Gamma,H_\Gamma)$-holomorphic map may be written as a
  $J_\Gamma'$-holomorphic map for some almost complex structure
  $J_\Gamma'$ obtained by pulling back $J_\Gamma$ under a Hamiltonian
  flow as in \cite[Chapter 8]{ms:jh}.  Suppose that
  $u_\nu: C_\nu \to X$ is a sequence of
  $(J_\Gamma, H_{\Gamma,\nu})$-holomorphic treed disks with $H_{\Gamma,\nu}$
  converging to zero in $C^\infty$.  After passing to a subsequence,
  we may assume that the domain $C_\nu$ converges to a limit $C$.
By the energy-area relation for Hamiltonian-perturbed maps, 
in particular the bound in \cite[Remark 8.1.7]{ms:jh},
 the energy of the sequence $u_\nu$ is bounded.  By
  Gromov compactness (see for example \cite[Chapter 4]{ms:jh},
  although a modification is necessary to adapt for the varying
  domain) a subsequence of $u_\nu$ Gromov converges to a limiting
  stable $J_\Gamma$-holomorphic treed disk $u: C \to X$ with the same area.  Since the Hamiltonian perturbation $H_{\Gamma,\nu}$ vanishes in the limit, the area is necessarily non-negative.
\end{proof}

The combinatorial type of an adapted map is that of the map with the
additional data of a labelling $i(e), e \in \Edge(\Gamma)$ of any
interior node by intersection multiplicity $i(e)$ with the
hypersurface $D$; let $i(e) =0$ if the map $u: S \to X$ is constant
with values in the hypersurface $D$ near $w_e$.  Denote the moduli
space of $D$-adapted treed holomorphic disks bounding $\phi$ of type
domain type $\Gamma$ with respect to the perturbation $P_\Gamma$ by \label{mphiD}
\[ \M_\bGamma(X,\phi) \subset \Set{ u: S \to X |
  \olp_{J_\Gamma,H_\Gamma} u = 0, \quad u(w_e) \in D, \quad \forall e
  \in \Edge_{\black}(\Gamma) } . \]
  Denote by 
\[ \ol{\M}(X,\phi) = \cup_\Gamma \M_\Gamma(X,\phi) \]  
the union over combinatorial types $\Gamma$.  As before, we may further 
refine to a union over map types
\[ \ol{\M}(X,\phi) = \cup_\bGamma \M_\bGamma(X,\phi). \]  

\subsection{Coherence}

In order to obtain good compactness properties of the moduli spaces of holomorphic curves,  the following coherence properties of the perturbations are required.

\begin{definition} \label{cohere}
The perturbations $\ul{P} = (P_\Gamma)$ are {\em coherent} if they satisfy the following
axioms:
\begin{itemize}
\vskip .1in
\item[] (Locality axiom) We require the
  following notation.  Given a type $\Gamma$, for each vertex
  $v \in \on{Vert}(\Gamma)$, let $\Gamma(v)$ denote the subtree of
  $\Gamma$ consisting of the vertex $v$ and all edges $e$ of $\Gamma$
  meeting $v$.  Let $\Gamma_{\white}$ denote the subgraph of $\Gamma$
  whose vertices are those of open type $v \in \Ver_{\white}(\Gamma)$
  and whose edges are $e \in \Edge_{\white}(\Gamma)$.  Let
\[ \pi = \pi_\white \times \pi_v: \U_\Gamma \to \M_{\Gamma_\white} \times \U_{\Gamma(v)} \] 
be the product of the maps where $\pi_\white$ is given by projection
followed by forgetful morphism and $\pi_v$ is the map $S \mapsto S_v$
that collapses all components other than $S_v$ onto the corresponding
special points of $S_v$.
The locality property is the following: For each vertex $v$, the
perturbation $P_\Gamma$ restricts on $S_v$ to the pull-back under
$\pi$ of some perturbation $P_{\Gamma,v}$ on
$\M_{\Gamma_\white} \times \U_{\Gamma(v)}$ to $\U_\Gamma$.
\footnote{In other words, on each component $S_v$ the perturbations  only depend on the positions of the special points on that
  component and the boundary edge lengths.  This locality principle is
  used later in Theorem \ref{thm:comeager} to rule out constant spheres with
  more than one marking, in the case of zero and one-dimensional
  moduli spaces.}
 \vskip .1in
\item[] (Cutting edges axiom) If $\Gamma$ is obtained from types
  $\Gamma_1,\Gamma_2$ by gluing along semi-infinite edges $e$ of
  $\Gamma_1$ and $e'$ of $\Gamma_2$ as in \eqref{cglue} then let 
\[ \pi_1 : \cM_{\Gamma} \to \cM_{\Gamma_1} , \quad 
\pi_2 : \cM_{\Gamma} \to \cM_{\Gamma_2}  \] 
denote the projections obtained by mapping each curve
$C = C_1 \cup_{e,e'} C_2$ to $C_1$ resp. $C_2$.  For the coherence
axioms $P_\Gamma$ is the product of the perturbations
$P_{\Gamma_1},P_{\Gamma_2}$ under the isomorphism
$\U_\Gamma \cong \pi_1^* \U_{\Gamma_1} \cup \pi_2^*
\U_{\Gamma_2}$.\footnote{That
  is, on any configuration with a broken edge, the perturbations on
  the components separated by the broken edge depend only on the
  domains on that side of the edge, rather than the domain on the
  other side.  This property is necessary for the \label{rep:addthe} boundary description in
  Theorem \ref{thm:comeager}, which in turn is used to prove the \ainfty
  axiom.}

  \vskip .1in
\item[] (Collapsing edges axiom) If $\Gamma'$ is obtained from
  $\Gamma$ by setting a length equal to zero or infinity, or collapsing
  an edge, then the restriction of $P_{\Gamma}$ to
  $\S_{\Gamma} | \M_{\Gamma'} \cong \S_{\Gamma'}$ is equal to
  $P_{\Gamma'}$.
\end{itemize} 
\end{definition}
\begin{remark} \label{fsphere} \label{israther} {\rm (Forgetting
    markings on spheres)} The locality axiom provides the following
  forgetful construction, which is a variation of the construction in
  Cieliebak-Mohnke \cite{cm:trans}: Suppose that $C$ is a curve of
  type $\Gamma$ containing a sphere component $S_v$ with more than one
  interior marking $w_e \in S_v$.  Forgetting all but one marking, say
  $w_{e_0}$ on $S_v$, and collapsing unstable components produces a
  marked curve $f(C)$ with type $f(\Gamma)$ possibly with a component
  $f(S_v)$ containing a single marking, as in Figure \ref{drawing2}.
  \begin{figure}
      \centering
\begingroup%
  \makeatletter%
  \providecommand\color[2][]{%
    \errmessage{(Inkscape) Color is used for the text in Inkscape, but the package 'color.sty' is not loaded}%
    \renewcommand\color[2][]{}%
  }%
  \providecommand\transparent[1]{%
    \errmessage{(Inkscape) Transparency is used (non-zero) for the text in Inkscape, but the package 'transparent.sty' is not loaded}%
    \renewcommand\transparent[1]{}%
  }%
  \providecommand\rotatebox[2]{#2}%
  \newcommand*\fsize{\dimexpr\f@size pt\relax}%
  \newcommand*\lineheight[1]{\fontsize{\fsize}{#1\fsize}\selectfont}%
  \ifx\svgwidth\undefined%
    \setlength{\unitlength}{393.3446737bp}%
    \ifx\svgscale\undefined%
      \relax%
    \else%
      \setlength{\unitlength}{\unitlength * \real{\svgscale}}%
    \fi%
  \else%
    \setlength{\unitlength}{\svgwidth}%
  \fi%
  \global\let\svgwidth\undefined%
  \global\let\svgscale\undefined%
  \makeatother%
  \begin{picture}(1,0.35902042)%
    \lineheight{1}%
    \setlength\tabcolsep{0pt}%
    \put(0,0){\includegraphics[width=\unitlength,page=1]{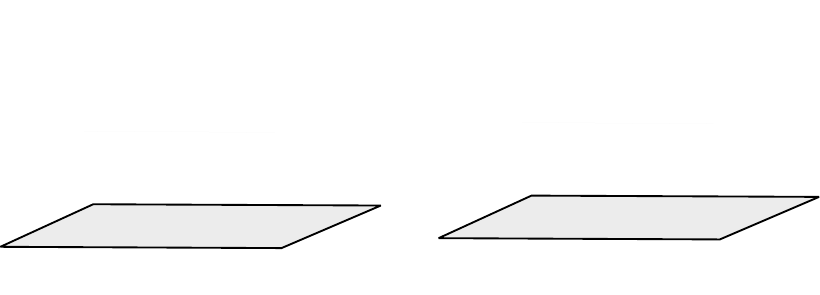}}%
    \put(0.12965862,0.00924364){\color[rgb]{0,0,0}\makebox(0,0)[lt]{\lineheight{1.25}\smash{\begin{tabular}[t]{l}Type $\Gamma$\end{tabular}}}}%
    \put(0.63773954,0.00773599){\color[rgb]{0,0,0}\makebox(0,0)[lt]{\lineheight{1.25}\smash{\begin{tabular}[t]{l}Type $f(\Gamma)$\end{tabular}}}}%
    \put(0,0){\includegraphics[width=\unitlength,page=2]{drawing.pdf}}%
  \end{picture}%
\endgroup%

      \caption{The types $\Gamma$ and $f(\Gamma)$}
      \label{drawing2}
  \end{figure}
  Define a perturbation datum
  $f(P_\Gamma)$ for $f(\Gamma)$ by taking the almost complex structure
  $J_{f(\Gamma)}$ to equal the
  base almost complex structure $J_D$ on $f(S_v)$, and the almost
  complex structures $P_\Gamma | f(C) - f(S_v)$ on the complement;
  while the Hamiltonian perturbation $H_\Gamma$ and $M_\Gamma$ remains the same.  If
  $u : S \to X$ is a $P_\Gamma$-holomorphic map constant on $S_v$,
  then one obtains an $f(P_\Gamma)$-holomorphic map on $f(S_v)$ by
  forgetting all markings except $w_{e_0}$ on $S_v$.  Since each
  interior node $w_e, e \in \Edge_{\black}(\Gamma)$ is required to map
  to the divisor $D$, the resulting type $f(\bGamma)$ is the same
  expected dimension as that of maps of type $\bGamma$.
\end{remark}

\begin{remark} \label{rem:cutrem}
The (Cutting edges) axiom implies the following relationship between
moduli spaces.  Suppose that the type $\Gamma$ is obtained by gluing together
  types $\Gamma_2$ and $\Gamma_1$ along a boundary edge.  An element of
  $\M_\Gamma(X,\phi)$ consists of a pair $C_k = S_k \cup T_k, \ u_k: S_k \to X$ of treed
  holomorphic disks of combinatorial types $\Gamma_k$ for
  $k \in \{ 1, 2\}$, glued together along some point at infinity along
  two of the semi-infinite edges.  Thus if $i$ denotes the index of the incoming edge for $\Gamma_2$ glued at the outgoing edge of $\Gamma_1$ and $j+1$ the number of incoming edges of $\Gamma_2$
the map $u \mapsto (u_1,u_2)$
defines a map
\begin{multline} \label{finitefib}
  \M_{\Gamma}(X,\phi,\sigma_0,\ldots, \sigma_d)_0 \\ \to
\cup_{\alpha} 
{\M}_{\Gamma_1}(X,\phi,\sigma_0,\sigma_1,\ldots,\sigma_{i-1},\alpha,\sigma_{i+j+1},
\dots, \sigma_d)_0 \\ \times
{\M}_{\Gamma_2}(X,\phi,\alpha,\sigma_{i},\ldots,\sigma_{i+j})_0 .\end{multline}
\end{remark} 

Obtaining strict units requires the addition of {\em
    weightings} to the combinatorial types as in Ganatra
  \cite[Section 10]{ganatra} and Charest-Woodward \cite[Section 4]{flips}. When the weighting of
an edge is infinite, we will assume that the perturbation data is
pulled back under the forgetful map forgetting that edge and
stabilizing.  For this reason, the edges where the weightings are
forced to be infinite are called {\em forgettable}.
\label{forgettable}
\label{forgettablep}

\begin{definition} \label{def:wdef}  A {\em weighting} of
  a treed disk $C= S \cup T$ of type $\Gamma$ is 
\begin{enumerate}
\item a
  partition of the boundary semi-infinite edges
\[\Edge^{\wt}(\Gamma) \sqcup \Edge^{\wt,\infty}(\Gamma) \sqcup \Edge^{\wt,0}(\Gamma) =
  \Edge_{\white,\rightarrow}(\Gamma) \]
into {\em weighted} resp. {\em forgettable} resp.  {\em unforgettable}
edges, and
\item  a map 
\[\rho: \Edge_{\white,\rightarrow}(\Gamma) \to [0,\infty]\]
satisfying the property: each of the semi-infinite $e$ edges is
assigned a {\em weight} $\rho(e)$ such that
 \[ 
	  \rho(e) \in 
\begin{cases}  \{ 0 \}  &  e \in
  \Edge^{\wt,0}(\Gamma) \\  [0,\infty] & e \in
  \Edge^{\wt}(\Gamma) \\
  \{ \infty \} & e \in \Edge^{\wt,\infty}(\Gamma) 
\end{cases}.
\] 
\end{enumerate} 
If the outgoing edge $e_0 \in \Edge_{\rightarrow}(\Gamma)$ is
unweighted (forgettable or unforgettable) then an isomorphism
$\psi: (C,\rho) \to (C',\rho')$ of weighted treed disks is an
isomorphism of treed disks $C \to C'$ that preserves the types of
semi-infinite edges
$e \in \Edge_{\rightarrow}(\Gamma) \cong \Edge_{\rightarrow}(\Gamma')$
and weightings: $\rho(e) = \rho'(e')$ for all corresponding edges
$e \in \Edge_{\white,\rightarrow}(\Gamma), \ e' \in
\Edge_{\white,\rightarrow}(\Gamma')$.  This ends the Definition.
\end{definition} 

 There is an additional notion of
equivalence in the case that the outgoing edge is weighted: If the outgoing edge $e_0$ is weighted then
an isomorphism of weighted treed disks $C \to C'$ is an isomorphism of
treed disks preserving the types of semi-infinite edges
$e \in \Edge_{\white,\rightarrow}(\Gamma)$ and the weights
$\rho(e), e \in \Edge_{\white,\rightarrow}(\Gamma)$ up to scalar
multiples:
\label{scalarl} \label{scalarlp} 
\begin{equation} \label{scalar} 
\exists \lambda\in (0,\infty), \ \forall e \in \Edge_{\white,\rightarrow}(\Gamma), e' \in
 \Edge_{\white,\rightarrow}(\Gamma'), \ \rho(e) = \lambda \rho'(e').\end{equation}
In particular, any weighted tree $T$ such that 
$\Ver(\Gamma) = \emptyset$ and a single edge
$e \in \Edge_{\white,\rightarrow}(\Gamma)$ that is weighted
$\rho(e) \in (0,\infty)$ is isomorphic to any other such configuration
$T'$ with a different weight \label{ends} \label{endsp}
$\rho(e') \in (0,\infty), e' \in \Edge(\Gamma')$.

The {\em combinatorial type} of any weighted treed disk is the tree
associated to the underlying nodal disk with additional data recording
which lengths resp. weights are zero or infinite.  Namely if
$C = S \cup T$ is a weighted treed disk then its combinatorial type is
the tree $\Gamma = \Gamma(C)$ \label{graphtotree}
\label{graphtotreep}
 obtained by gluing 
together the combinatorial types $\Gamma(S_v)$ of the disks $S_v$
along the edges corresponding to the edges of $T$; and equipped with 
the additional data of 
\begin{enumerate} 
\item the subsets 
\[\Edge^{\wt}(\Gamma) \ \text{resp.} \ \Edge^{\wt,\infty}(\Gamma) \ \text{resp.}
  \ \Edge^{\wt,0}(\Gamma) \subset \Edge_{\white,\rightarrow}(\Gamma) \]
of weighted, resp. forgettable, resp. unforgettable semi-infinite 
edges;
\item the subsets 
\[\Edge_{-}^\infty(\Gamma) \ \text{resp.} \Edge_{-}^0(\Gamma) \ \text{resp.}
  \Edge_{-}^{(0,\infty)}(\Gamma) \subset \Edge_{-}(\Gamma) \]
of combinatorially finite edges of infinite resp. zero length resp.
non-zero finite length;
\end{enumerate} 

A well-behaved moduli space of weighted treed disks is obtained after
imposing a stability condition.

\begin{definition} \label{def:wstable} A weighted treed disk
  $C = S \cup T$ of type $\Gamma$ is {\em stable} if either
\begin{enumerate}
\item  there is at least one disk component
  $S_v, v \in \Ver_{\white}(\Gamma)$, and the following conditions hold:
\begin{enumerate} 
\item each disk component $S_v, v \in \Ver_{\white}(\Gamma)$ has at least
  three edges $e \in \Edge(\Gamma)$ attached to the boundary $\partial S_v$
  or at least one edge attached to the boundary $\partial S_v$ and one
  edge to the interior $\on{int}(S_v)$;
\item each sphere component $S_v, v \in \Ver_{\black}(\Gamma)$ has at least
  three edges $e \in \Edge(\Gamma)$ attached; 
\item each combinatorially-finite edge $e \in \Edge_-(\Gamma)$ is
  broken at most once, and each semi-infinite edge
  $e \in \Edge_{\rightarrow}(\Gamma)$ is unbroken;
\item if the outgoing edge is weighted
  $e_0 \in \Edge^{\wt}(\Gamma)$ then at least one leaf
  $e_i \in \Edge_{\white,\rightarrow}(\Gamma), i > 0$ is also
  weighted, that is, $e_i \in \Edge^{\wt}(\Gamma)$. 
\end{enumerate} 
\item if there are no disks, so that
  $\Ver(\Gamma) = \emptyset$, there is a single weighted leaf
  $e_1 \in \Edge^{\wt}(\Gamma)$ and an unweighted (forgettable or
  unforgettable) root 
  $e_0 \in \Edge^{\wt,\infty}(\Gamma) \cup \Edge^{\wt,0}(\Gamma)$.
\end{enumerate} 
This ends the Definition.
\end{definition} 
 Because a configuration with no disks is allowed (namely an infinite interval)  the stability condition for
weighted treed disks is not equivalent to the absence of non-trivial
automorphisms.   The moduli space of weighted treed disks 
  of some type $\Gamma$ is denoted $\M^{\wt}_\Gamma$.  The natural map $\M^{\wt}_\Gamma \to \M_\Gamma$ forgetting the weightings is a fiber bundle with each fiber the product
  of intervals for each leaf, so that as long as there is one vertex,
\begin{equation}  \label{wtprod} \M^{\wt}_\Gamma \cong \M_\Gamma \times (0,\infty)^{\# \Edge_{\wt}(\Gamma)} \end{equation} 
where $\# \Edge_{\wt}(\Gamma)$ is the number of weighted edges with weights in $(0,1)$; 
the case of trees with no vertex is exceptional, since in this case 
both the incoming edge may be weighted but the moduli space $\M^{\wt}_\Gamma$
is still dimension zero, and there are no stable strata
$\M_\Gamma$. 

A weighted treed holomorphic disk is a holomorphic treed disk with a weighting
on the underlying treed disk and the following restriction on leaf labels.  Given a
non-constant pseudoholomorphic treed disk $u: C \to X$ with leaf $e_i$
for which the weighting $\rho(e_i) = \infty$ resp. $0 $, we view
$u$ as obtained from gluing the pseudoholomorphic treed disk
$u' : S \to X$ obtained by attaching to $e_i$ a constant configuration $u''$
with weighted incoming $e_i^-$ and forgettable resp. unforgettable outgoing edge $e_i^+$.  See Figure \ref{triv0}.
\begin{figure}[ht]
\begin{picture}(0,0)%
\includegraphics{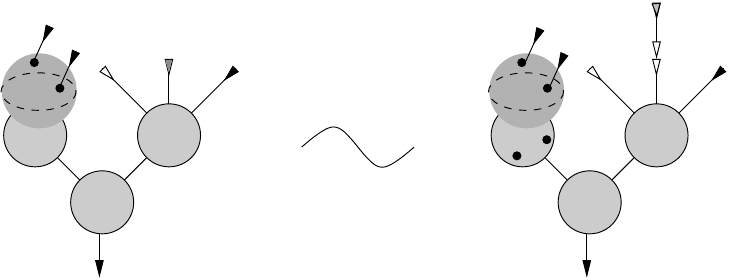}%
\end{picture}%
\setlength{\unitlength}{3947sp}%
\begingroup\makeatletter\ifx\SetFigFont\undefined%
\gdef\SetFigFont#1#2#3#4#5{%
  \reset@font\fontsize{#1}{#2pt}%
  \fontfamily{#3}\fontseries{#4}\fontshape{#5}%
  \selectfont}%
\fi\endgroup%
\begin{picture}(5818,2224)(-912,983)
\put(4393,2951){\makebox(0,0)[lb]{\smash{{{$\rho=\infty$}%
}}}}
\put(301,2789){\makebox(0,0)[lb]{\smash{{{$\rho = \infty$}%
}}}}
\end{picture}%
\caption{Equivalent weighted treed disks}
\label{triv0}
\end{figure}
Also, any two
configurations $u: S \to X, u':S' \to X$ with an outgoing weighted
edge $e_0$ with the same underlying tree $\Gamma$ are considered
equivalent.  See Figure \ref{equiv}.

\begin{figure}[ht]
\begin{picture}(0,0)%
\includegraphics{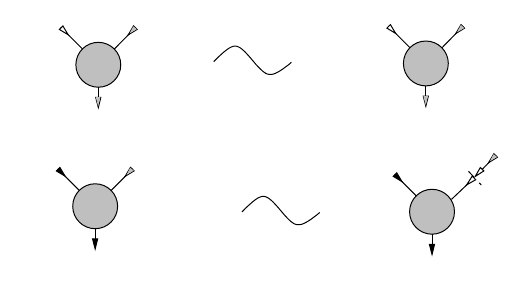}%
\end{picture}%
\setlength{\unitlength}{4144sp}%
\begingroup\makeatletter\ifx\SetFigFont\undefined%
\gdef\SetFigFont#1#2#3#4#5{%
  \reset@font\fontsize{#1}{#2pt}%
  \fontfamily{#3}\fontseries{#4}\fontshape{#5}%
  \selectfont}%
\fi\endgroup%
\begin{picture}(3846,2129)(3777,-4185)
\put(4434,-3014){\makebox(0,0)[lb]{\smash{{\SetFigFont{8}{9.6}{\rmdefault}{\mddefault}{\updefault}{\color[rgb]{0,0,0}$x^{\greyt}$}%
}}}}
\put(4709,-2193){\makebox(0,0)[lb]{\smash{{\SetFigFont{8}{9.6}{\rmdefault}{\mddefault}{\updefault}{\color[rgb]{0,0,0}$x^{\greyt}$}%
}}}}
\put(4115,-2173){\makebox(0,0)[lb]{\smash{{\SetFigFont{8}{9.6}{\rmdefault}{\mddefault}{\updefault}{\color[rgb]{0,0,0}$x^{\whitet}$}%
}}}}
\put(4790,-2415){\makebox(0,0)[lb]{\smash{{\SetFigFont{8}{9.6}{\rmdefault}{\mddefault}{\updefault}{\color[rgb]{0,0,0}$\rho_1$}%
}}}}
\put(4646,-2804){\makebox(0,0)[lb]{\smash{{\SetFigFont{8}{9.6}{\rmdefault}{\mddefault}{\updefault}{\color[rgb]{0,0,0}$\rho_1$}%
}}}}
\put(6929,-3005){\makebox(0,0)[lb]{\smash{{\SetFigFont{8}{9.6}{\rmdefault}{\mddefault}{\updefault}{\color[rgb]{0,0,0}$x^{\greyt}$}%
}}}}
\put(7205,-2183){\makebox(0,0)[lb]{\smash{{\SetFigFont{8}{9.6}{\rmdefault}{\mddefault}{\updefault}{\color[rgb]{0,0,0}$x^{\greyt}$}%
}}}}
\put(6611,-2163){\makebox(0,0)[lb]{\smash{{\SetFigFont{8}{9.6}{\rmdefault}{\mddefault}{\updefault}{\color[rgb]{0,0,0}$x^{\whitet}$}%
}}}}
\put(7608,-3140){\makebox(0,0)[lb]{\smash{{\SetFigFont{8}{9.6}{\rmdefault}{\mddefault}{\updefault}{\color[rgb]{0,0,0}$x^{\greyt}$}%
}}}}
\put(7286,-2406){\makebox(0,0)[lb]{\smash{{\SetFigFont{8}{9.6}{\rmdefault}{\mddefault}{\updefault}{\color[rgb]{0,0,0}$\rho_2$}%
}}}}
\put(7142,-2794){\makebox(0,0)[lb]{\smash{{\SetFigFont{8}{9.6}{\rmdefault}{\mddefault}{\updefault}{\color[rgb]{0,0,0}$\rho_2$}%
}}}}
\put(7509,-3459){\makebox(0,0)[lb]{\smash{{\SetFigFont{8}{9.6}{\rmdefault}{\mddefault}{\updefault}{\color[rgb]{0,0,0}$\rho=\infty$}%
}}}}
\put(6977,-4134){\makebox(0,0)[lb]{\smash{{\SetFigFont{8}{9.6}{\rmdefault}{\mddefault}{\updefault}{\color[rgb]{0,0,0}$x^{\blackt}$}%
}}}}
\put(6659,-3292){\makebox(0,0)[lb]{\smash{{\SetFigFont{8}{9.6}{\rmdefault}{\mddefault}{\updefault}{\color[rgb]{0,0,0}$x^{\blackt}$}%
}}}}
\put(4686,-3271){\makebox(0,0)[lb]{\smash{{\SetFigFont{8}{9.6}{\rmdefault}{\mddefault}{\updefault}{\color[rgb]{0,0,0}$x^{\greyt}$}%
}}}}
\put(4726,-3553){\makebox(0,0)[lb]{\smash{{\SetFigFont{8}{9.6}{\rmdefault}{\mddefault}{\updefault}{\color[rgb]{0,0,0}$\rho=\infty$}%
}}}}
\put(4410,-4092){\makebox(0,0)[lb]{\smash{{\SetFigFont{8}{9.6}{\rmdefault}{\mddefault}{\updefault}{\color[rgb]{0,0,0}$x^{\blackt}$}%
}}}}
\put(4092,-3251){\makebox(0,0)[lb]{\smash{{\SetFigFont{8}{9.6}{\rmdefault}{\mddefault}{\updefault}{\color[rgb]{0,0,0}$x^{\blackt}$}%
}}}}
\end{picture}%
\caption{Equivalent weighted treed disks, ctd.}
\label{equiv}
\end{figure}

\begin{remark} \label{rem:const} {\rm (Constant maps)} If
  $x_1 = x^{\greyt}$ and $x_0 = x^{\blackt}$ resp.
  $x_0= x^{\whitet}$ then the moduli space $\M(L,x_0,x_1)$ contains
  a configuration with no disks and single edge on which $u$ is
  constant, corresponding to a weighted leaf
  $e \in \Edge^{\wt}(\Gamma)$ and a root edge
  $e_0 \in \Edge(\Gamma)$ that is unforgettable resp.  forgettable.
  These maps are pictured in Figure \ref{triv}.
\end{remark}  

\begin{figure}[ht]
\begin{picture}(0,0)%
\includegraphics{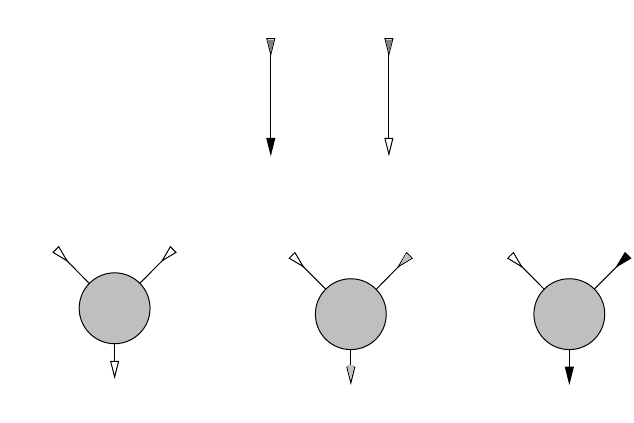}%
\end{picture}%
\setlength{\unitlength}{4144sp}%
\begingroup\makeatletter\ifx\SetFigFont\undefined%
\gdef\SetFigFont#1#2#3#4#5{%
  \reset@font\fontsize{#1}{#2pt}%
  \fontfamily{#3}\fontseries{#4}\fontshape{#5}%
  \selectfont}%
\fi\endgroup%
\begin{picture}(4819,3220)(2277,-3634)
\put(4804,-3561){\makebox(0,0)[lb]{\smash{{{$x^{\greyt}$}%
}}}}
\put(3701,-1173){\makebox(0,0)[lb]{\smash{{{$\rho = 0$}%
}}}}
\put(4111,-573){\makebox(0,0)[lb]{\smash{{{$x^{\greyt}$}%
}}}}
\put(5111,-593){\makebox(0,0)[lb]{\smash{{{$x^{\greyt}$}%
}}}}
\put(4141,-1783){\makebox(0,0)[lb]{\smash{{{$x^{\blackt}$}%
}}}}
\put(5581,-1783){\makebox(0,0)[lb]{\smash{{{$x^{\whitet}$}%
}}}}
\put(5491,-1163){\makebox(0,0)[lb]{\smash{{{$\rho = \infty$}%
}}}}
\put(2592,-2194){\makebox(0,0)[lb]{\smash{{{$x^{\whitet}$}%
}}}}
\put(3424,-2207){\makebox(0,0)[lb]{\smash{{{$x^{\whitet}$}%
}}}}
\put(2944,-3512){\makebox(0,0)[lb]{\smash{{{$x^{\whitet}$}%
}}}}
\put(4276,-2260){\makebox(0,0)[lb]{\smash{{{$x^{\whitet}$}%
}}}}
\put(5239,-2265){\makebox(0,0)[lb]{\smash{{{$x^{\greyt}$}%
}}}}
\put(6040,-2251){\makebox(0,0)[lb]{\smash{{{$x^{\whitet}$}%
}}}}
\put(6945,-2256){\makebox(0,0)[lb]{\smash{{{$x^{\blackt}$}%
}}}}
\put(6501,-3565){\makebox(0,0)[lb]{\smash{{{$x^{\blackt}$}%
}}}}
\end{picture}%
\caption{Unmarked treed disks}
\label{triv}
\end{figure}

  The set of generators of the space of Floer cochains
  $CF(\phi)$ is enlarged by adding two elements
  $1_{\phi}^{\whitet}$ (with superscript $\whitet$ denoting strict unit) 
  resp. $1_{\phi}^{\greyt}$ (with superscript $\greyt$ denoting homotopy between 
  strict and geometric unit) of degree $0$
  resp. $-1$ to $\cI(\phi)$.   Any edge $e$ labeled $1_{\phi}^{\whitet}$
  resp. $ 1_{\phi}^{\blackt}$ is required to have $\rho(e) = \infty$
  resp. $\rho(e) = 0$ while an edge with label  $1_{\phi}^{\whitet}$
  may have weighting $\rho(e) \in [0,\infty]$.     There is no constraint for the edges $e \in \Edge(\Gamma)$  labeled $1_{\phi}^{\whitet}$, while any edge with label   not equal to $1_{\phi}^{\whitet}$ or $1_{\phi}^{\greyt}$ must have 
  zero weighting.    The labels $1_{\phi}^{\whitet}$ and $1_{\phi}^{\greyt}$
  are  allowed on the outgoing leaf only if the area of the treed disk 
  is zero, the number of leaves is one or two, and the expected dimension is zero. 
  \label{rep:onlyif} Thus either there are no disks and the incoming edge is labeled
  $1_{\phi}^{\greyt}$ and the outgoing leaf is labeled $1_{\phi}^{\whitet}$
  or $1_{\phi}^{\blackt}$ or there is a single disk with no interior markings, 
  one incoming leaf labeled $1_{\phi}^{\whitet}$ and the label of the other incoming leaf and outgoing leaf are the same.  The 
  outgoing weight $\rho(e_0)$ is required to be the product of the incoming weights
  $\prod \rho(e_i)$ and in the case of zero-area configurations 
   all weightings are declared  equivalent.   By this definition, in each of these cases the moduli space $\M_{\Gamma}(X,\phi)$ is a point in each of these special configurations.
\begin{definition} 
{\rm (Forgetful axiom)} A perturbation datum $P_\Gamma$ satisfies
the forgetful axiom  if for any leaf $e \in \Edge(\Gamma)$ with
  infinite weighting $\rho(e) = \infty$, the perturbation datum
  $P_\Gamma$ is pulled back from the perturbation datum $P_{f(\Gamma)}$
  for the type $f(\Gamma)$ obtained by forgetting the leaf $e$ and stabilizing (that is, collapsing any unstable components) under the forgetful map $\U_{\Gamma} \to \U_{f(\Gamma)}$ of universal curves.
  \end{definition} 
In particular, this axiom implies that the resulting moduli spaces
admit forgetful morphisms $\M_{\Gamma}(X,\phi) \to \M_{\Gamma'}(X,\phi)$ whenever there is a leaf $e$ with weighting $\rho(e) = \infty$.  See
\cite[Section 4]{flips} for more details on the allowable weightings.  

\subsection{Transversality and compactness}

\label{ofexpected} Cieliebak-Mohnke perturbations
  \cite{cm:trans} are not sufficient for achieving transversality if
  there are multiple interior nodes on ghost bubbles. Indeed, suppose
  there exists a sphere component
  $S_v \subset S, v \in \Ver_{\black}(\Gamma)$ on which the map
  $u |_{S_v}$ is constant and maps to the divisor so that
  $u(S_v) \subset D$.  The domain $S_v$ may meet any number of
  interior leaves $T_e \subset T$.  Adding an interior leaf $T_{e'}$
  to the tree meeting $S_v$ increases the dimension of a stratum
  $\dim \M_\bGamma(X,\phi)$, but leaves the expected dimension
  $\Ind(D_u), u \in \M_\bGamma(X,\phi)$ unchanged.  It follows that
  $\M_\bGamma(X,\phi)$ is not of expected dimension for some types
  $\bGamma$ that we call {\em crowded}:

\begin{definition} \label{def:crowded} A holomorphic treed disk
  $(C,u: C \to X)$ is {\em crowded} if each such ghost component
  $S_v \subset S$ meets at least two interior leaves $T_e$, so that $\# \{ e, T_e \cap S_v \neq \emptyset \} \ge 2$, and {\em uncrowded}
  otherwise.
\end{definition} 

The 
construction of coherent perturbations for uncrowded types proceeds inductively. We summarize
the properties that we wish our perturbations to satisfy in the following definition:

\begin{definition} \label{def:good} A perturbation datum
  $\ul{P} = \{ P_\Gamma \in \PP_\Gamma \}$ has {\em good properties}
  if the following hold for each uncrowded type of map
  $\bGamma$ of expected dimension at most one:
\begin{enumerate} 
\item {\rm (Transversality)} Every element of
  $\M_{\bGamma}(X,\phi)$ is regular;
\item \label{compactness} {\rm (Compactness)} the closure
  $\ol{\M}_{\bGamma}(X,\phi)$ is a finite set, if expected
  dimension zero; or a compact one-manifold, if expected dimension
  one, with boundary contained in the adapted, uncrowded locus; and
\item \label{bdes} {\rm (Boundary description)} the boundary of
  $\ol{\M}_{\bGamma}(X,\phi)$ is a union of components
  $\M_{{\bGamma}'}(X,\phi)$ where ${\bGamma}'$ is a type with
  an edge $e$ of length $\ell(e)$ zero, an infinite length edge
  $e, \ell(e) = \infty$ connecting two disk components, or a
  leaf $e \in \Edge({\bGamma})$ with $\ev_e$ mapping to the
  boundary $\sigma_i(\partial B^{d(i)})$ of a cell;
\end{enumerate} 
\end{definition} 

Suppose that  perturbations $P_{\Gamma'}$ on the types   $\Gamma' \prec \Gamma$ have been chosen  in  Definition \ref{def:good} making the moduli space of type $\Gamma'$ regular and all moduli spaces
obtained by forgetting interior leaves in Remark \ref{fsphere} regular.   
Via a gluing construction the perturbations $P_{\Gamma'}$ induce
regular perturbations in some neighborhood $\S^+_{\Gamma'}$ of $\S_{\Gamma'}$  in $\S_{\Gamma}$.
Namely any curve $C$ of type $\Gamma$ near $\M_{\Gamma'}$ is obtained from a curve $C'$
of type $\Gamma'$ by some combination of removing small balls  from the nodes and identifying the complements by gluing maps given in local coordinates $z \to \delta/z$; and varying the edge lengths.  Since the perturbations by assumption vanish near the nodes, one obtains
perturbations on $C$ from those on $C'$.
Denote by $\ol{\PP}_\Gamma \subset \PP_\Gamma$ the subset of perturbations that agree with $P_{\Gamma'}$ on $\S^+_{\Gamma'}$ on the types $\Gamma' \prec \Gamma$.
The following is proved in a standard way, using Sard-Smale; see Palmer-Woodward \cite[Section 4]{pw:flow}.

\begin{theorem} \label{thm:comeager} There exists a comeager subset
  $\PP^{\reg}_\Gamma$ of the subspace $\ol{\PP}_\Gamma$ making the moduli space of type $\Gamma$ regular and all moduli spaces obtained by forgetting interior leaves from domains of type $\Gamma$ in Remark \ref{fsphere} regular.   Furthermore, 
  the perturbations chosen inductively in this way have the good properties
  in Definition \ref{def:good}.
\end{theorem}

\subsection{Orientations} 
\label{orientsec}

Orientations on the moduli spaces may be constructed following
Fukaya-Oh-Ohta-Ono \cite[Orientation chapter]{fooo}, \cite{orient},
given a relative spin structure.  For this purpose, we may ignore the
constraints at the interior nodes $w_1,\ldots, w_m $ in
$ \on{int}(S)$.  The tangent spaces to these nodes and the
linearized constraints $\dd u(w_i) \in T_{u(w_i)}D$ are even
dimensional and oriented by the given complex structures.  Suppose 
the type $\bGamma$ has at least one vertex $v \in \Ver(\bGamma)$.  Consider a 
regular element 
\[ (C,u:C \to X) \in \M_\bGamma(X,\phi,\ul{\sigma})\]
of type $\bGamma$.  The tangent space is the kernel of the linearized operator:
\[ T_u \M_\bGamma(X,\phi) \cong \ker(\ti{D}_u) \]
where (abusing notation) $\ti{D}_u$ is the restriction of the operator in \eqref{linop} to
the space of sections $(\zeta,\xi: S \to u^* TX)$ satisfying
constraints
\[ \xi(w_e) \in T\sigma_e , e \in \Edge_{\white,\rightarrow}(\bGamma), 
\quad \xi(w_e) \in TD, e \in \Edge_{\black,\rightarrow}(\bGamma) . \]   
The operator $\ti{D}_u$ admits a homotopy
\[ \ti{D}_u^t, t \in [0,1], \quad \ti{D}_u^1 = \ti{D}_u, \quad 
\ti{D}_u^0 = 0 \oplus D_u \]
so that $\ti{D}_u^0$ is a direct sum of the zero operator 
and the linearized Cauchy-Riemann operator  $D_u$.  For any vector spaces $V,W$,
the determinant line of the direct sum admits an isomorphism
$\det(V \oplus W) \cong \det(V) \otimes \det(W)$.  The deformation
$\ti{D}_u^t, t \in [0,1]$ of operators induces a family of determinant
lines $\det(\ti{D}_u^t)$ over the interval $[0,1]$, necessarily
trivial.  One obtains by parallel transport of this family an
identification of determinant lines
\begin{equation} \label{split1} \det(T_{u} \M_\bGamma(X,\phi)) \to
  \det(T_C \M_\Gamma )
 \otimes \det(D_{{u}}) \end{equation}
well-defined up to isomorphism.  
\begin{figure}
\includegraphics[width=3in]{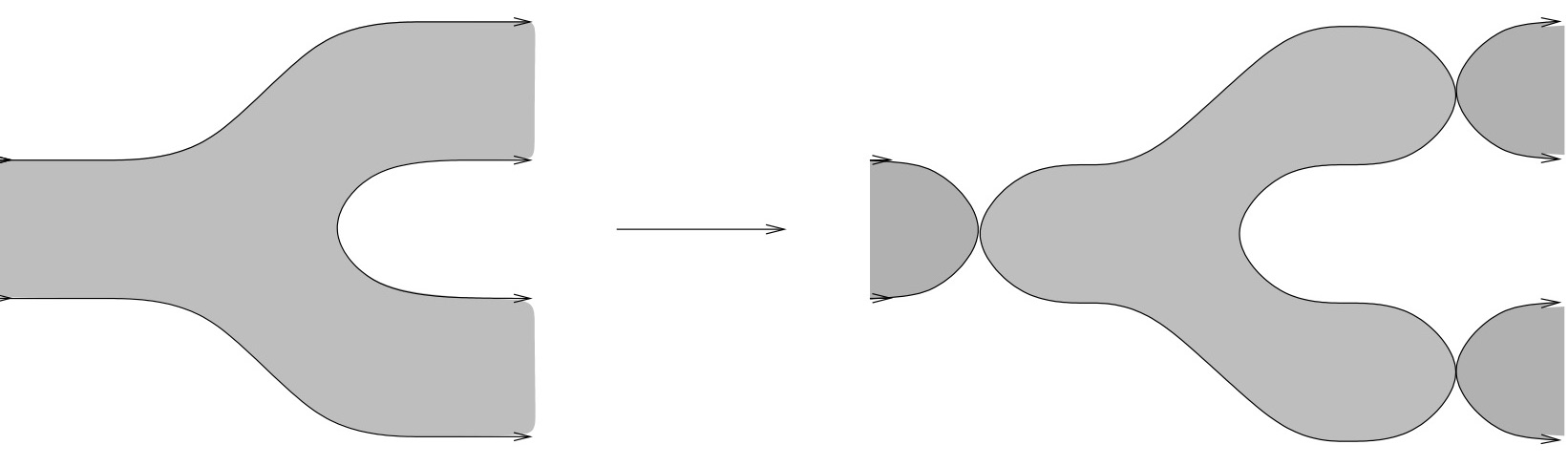}
\caption{Bubbling off the strip-like ends}
\label{closed}
\end{figure}
In the case of nodes of $S$ mapping to self-intersection points
$x \in \cI^{\on{si}}(\phi)$, the determinant line $\det(D_{u})$ is
oriented by ``bubbling off one-pointed disks'', as in \cite[Theorem
44.1]{fooo} or \cite[Equation (36)]{orient}.  For each
self-intersection point
\[ (x_- \neq x_+) \in L^2, \quad \phi(x_-) = \phi(x_+), \]  
choose a path of Lagrangian subspaces
\begin{equation} \label{gammax}
\gamma_x:[0,1] \to \Lag(T_{\phi(x_-) = \phi(x_+)} X) \end{equation} 
\[  \gamma_x(0) = D_{x_-} \phi (T_{x_-} L )
\quad \gamma_x(1) = D_{x_+} \phi (T_{x_+} L) .\] 
Let $S$ be the unit disk with a single boundary marking
$1 \in \partial S$.  The path $\gamma_x$ defines a totally real
boundary condition on $S$ on the trivial bundle with fiber $T_x X$.
Let $\det(D_x^+)$ denote the determinant line for the Cauchy-Riemann
operator $D_x^+$ with boundary conditions $\gamma_x$ as in
\cite{orient}.   \label{agrees} Let $D_x^-$ be the operator as in the
previous discussion but with the direction of the path $\gamma_x$
reversed and
\[ \DD^+_{x} = \det(D_{x}^+), \quad \DD^-_{x} = \det(D_{x}^-) \otimes
\det(T_x L) \]
The once-marked disks with boundary conditions $\gamma_{x}$ and
$\gamma_{\ol{x}}$ glue together along the strip like end to a disk
with no-strip like end whose boundary condition is the concatenation
of $\gamma_x$ and $\gamma_{\ol{x}}$.  This boundary condition is
isotopic to the constant boundary condition, and the determinant line
extends over the isotopy giving a canonical isomorphism
\begin{equation} \label{caniso} 
\DD^-_{x} \otimes \DD^+_{x} \to \R
  .\end{equation}
A choice of orientations $o_{x} \in \DD^\pm_{x}$ for the
self-intersection points $x$ are {\em coherent} if the isomorphisms
\eqref{caniso} are orientation preserving with respect to the standard
orientation on $\R$.  For each critical point choose an orientation on the 
corresponding cell and let $\DD^-_{\sigma_j}$ denote the determinant line of
the tangent space to the cell at any point.   Choose
orientations $\DD^+_{\sigma_k^\dual}$ for the stable manifolds so that the combined orientations for stable and unstable manifolds at the critical points give the orientation on $X$.  Given a relative spin structure for $\phi:L \to X$, the
orientation at $u$ is determined by an
isomorphism \label{familytransversality2}
\begin{equation} \label{split2} \det(D_{{u}}) \cong \DD^+_{\sigma_0}
  \otimes \DD^-_{\sigma_1} \otimes \ldots \otimes \DD^-_{\sigma_d}.
 \end{equation}
 The isomorphism \eqref{split2} is determined by degenerating the
 surface with strip-like ends to a nodal surface as in Figure
 \ref{closed}.  Thus each end $\eps_e, e \in \mE(S_v)$ of a component
 $S_v$ with a node $w$ mapping to a self-intersection point is
 replaced by a disk $S_{v^\pm(e)}$ with one end attached to the rest
 of the surface by a node $w_e^\pm$.  After combining the orientations
 $o_e$ on the determinant lines on $S_{v^\pm(k)}$ with orientations
 $o_\sigma$ on the tangent spaces to cells $\sigma$ in the case of
 broken edges or semi-infinite edges
 $e \in \Edge(\Gamma), \ell(e) = \infty$, one obtains an orientation
 $o_u$ on the determinant line of the parameterized linear operator
 $\det(\ti{D}_u)$.  The orientations on the determinant lines
 give orientations on the regularized moduli spaces
 $\M_\Gamma(X,\phi,\ul{\sigma})$. 
 
 There is a similar discussion for weighted
 moduli spaces.  The moduli space
 of weighted trees $\M_\Gamma^{\wt}$ is oriented
 via the product description \eqref{wtprod}.   in the case of labels $1_\phi^{\whitet}$
 or $1_\phi^{\blackt}$ the orientations of the moduli spaces 
 are defined by considering the normal bundles of the inclusions $\{\infty \} \to [0,\infty]$
 resp. $\{ 0 \} \to [0,\infty]$ to be positively resp. negatively weighted.  In particular, 
 the components $\M_\Gamma(X,\phi,\ul{\sigma})_{\leq 1}$ of expected dimension
  at most one are equipped with orientations satisfying the standard
  gluing signs for inclusions of boundary components described in
  \cite[Theorem 4.10]{ainfty}.

In particular, for labeled map types $\bGamma$ of expected
dimension zero the  strata
 $\M_{\bGamma}(X,\phi)$ inherit orientation maps
 \begin{equation} \label{orienteq} o: \M_\bGamma(X,\phi) \to \{ +1,
   -1 \} \end{equation}
comparing the constructed orientation to the canonical orientation of
a point.

\section{Fukaya algebras} 
\label{cellsec}

Fukaya algebras of immersed Lagrangians are introduced in Palmer-Woodward \cite{pw:flow}. 
The generators of the Floer complex consist of critical points, or equivalently, cells in the associated cell decomposition; self-intersection points, and additional generators for  homotopy units.  We denote 
\begin{equation} \label{hu} \cI(\phi) = \cI^{c}(\phi) \cup \cI^{\si}(\phi) \cup 
\cI^{\on{hu}}(\phi) \end{equation} 
where 
\[ \cI^c(\phi) := \{ \sigma(x), x \in \crit(f) \} \]
is the set of cells,
  \[ \cI^{\si}(\phi) := (L \times_\phi L) - \Delta_L \]
is the set of ordered self-intersection points, where 
$L \times_\phi L$ is the fiber product and $\Delta_L \subset L^2$ the 
diagonal; and 
\[ \cI^{\on{hu}}(\phi) := \{ 1_\phi^{\whitet}, 1_\phi^{\greyt} \} \] 
are two additional generators added as part of the homotopy unit
construction.  The sum
\[ 1_\phi^{\blackt} := \sum_{\codim(\sigma_i) = 0 }
\sigma_i \]
is the {\em geometric unit}.  \label{geomunit} Thus $\cI(\phi)$ consists of the cells
in $L$ together with two copies of each self-intersection point, plus
two extra generators. 

In order to obtain graded Floer cohomology groups, a grading on the set
of generators is defined as follows.  Given an orientation, there is a
natural $\Z_2$-valued map
\[ \cI(\phi) \to \Z_2, \quad x \mapsto |x| \]
obtained by assigning to any cell $\sigma \in \cI^c(\phi)$ the
codimension mod $2$ and to any self-intersection point
$(x_-,x_+) \in \cI^{\si}(\phi)$ the element $|x| = 0$ resp. $|x| = 1$
if the self-intersection is even resp. odd.  The grading degrees of
the cells are determined by the codimensions
$\codim(\sigma_i) = \dim(L) - d(i)$ for the cells $\sigma_i$.  
For the extra generators
$1_\phi^{\whitet}, 1_\phi^{\greyt}$ the gradings are given by.\footnote{Here we work only with  $\Z_2$ gradings, so the extra generators are simply even and odd
  respectively; see Remark \ref{rem:grad}.}  
\[ |1^{\whitet}_\phi| = 0, \quad |1^{\greyt}_\phi| = -1.
\]
Denote by $\cI^k(\phi)$ the
subset of $\sigma \in \cI(\phi)$ with $|\sigma| = k$ mod $2$.
  \label{rep:setofcells}

The space of Floer cochains is freely generated by the above
generators over the Novikov field.  
The space of Floer cochains is the $\Z_2$-graded vector space
\[ CF(\phi) = \bigoplus_{k \in \Z_2} CF^k(\phi), \quad CF^k(\phi) =
\bigoplus_{x \in \cI^k(\phi)} \Lambda x.\] 
The $q$-valuation on $\Lambda$ extends naturally to $CF(\phi)$:
\[ \val_q:CF(\phi) - \{ 0 \} \to \R, \quad \sum_x c(x) x \mapsto
\min_x(\val_q(c(x)), c(x) \neq 0) .\]

\subsection{Composition maps} 

The composition maps in the  Fukaya algebra are counts of
rigid holomorphic treed disks weighted by areas and holonomies.  For
 perturbations from the last section, define {\em higher composition maps}
\[ m_d: CF(\phi)^{\otimes d} \to CF(\phi)[2-d], \quad d \ge 0 \]
on generators as follows.  Let
$\sigma_1,\ldots, \sigma_d \in \cI(\phi)$ and let
\[ \M_\Gamma(X,\phi,\ul{\sigma})_0 \subset \M_\Gamma(X,\phi) \] 
denote the subset of rigid maps with constraints given by generators
$\ul{\sigma} = (\sigma_0,\ldots, \sigma_d)$ as defined in \eqref{Md}. 

\begin{definition} {\rm (Composition maps)}  
On generators $\sigma_1,\ldots, \sigma_d$ define
\begin{equation} \label{eq:highercomp} m_d(\sigma_1,\ldots,\sigma_d) =
  \sum_{ \bGamma, \substack{\sigma_0 \in \cI(\phi) 
      \\ u \in\ol{\M}_{\bGamma}(X,\phi,\sigma_0,\ldots,\sigma_d)}}
      \wt(u,\gamma) \sigma_0
      \end{equation}
where the weight $\wt(u,\gamma)$ is defined by 
\[ \wt(u,\gamma) =   \frac{(-1)^{\heartsuit} }{\theta(u)!} y(\partial u) q^{A(u)} o(u) \] 
with the notation 
\begin{itemize} 
\item $\theta(u) \in \Z_{> 0}$ is the number of interior leaves 
  $e \in \Edge_\black(\Gamma)$, corresponding to intersections 
  $u(w_e) \in D$ with the Donaldson hypersurface $D$;
\item $y( \partial u) \in \Lambda_0$ is the holonomy of the local 
  system $y$ around the boundary $u(\partial S) \subset \phi(L)$ as in 
  Remark \ref{rem:branestr};
\item $A(u) \in \R_{\ge 0}$ is the sum of the areas $A(u_v)$ of the 
  disks and spheres \label{areas} $u_v: S_v \to X$ for 
  $v \in \Ver(\Gamma)$;
\item $o(u) \in \{ \pm 1 \}$ is an orientation sign defined in 
  \eqref{orienteq} using the relative spin structure for 
  $\phi: L \to X$;
\item the exponent $\heartsuit \in \Z$ is given by 
\begin{equation} \label{heartsuit}  \heartsuit = {\sum_{i=1}^d i|\sigma_i|} ;\end{equation}
\item and the sum is over all types of rigid maps $\bGamma$.
\end{itemize} 
we have written tensor products as commas to save space.  If a
matching condition $M_\Gamma$ is a formal sum (rather than a single
diffeomorphism) the contributions are weighted by the coefficients
$c_i,d_j$ of the perturbations $M_{\Gamma,i},H_{\Gamma,j}$ in
\eqref{eq:branchedM}, \eqref{branchedH}.   This ends the Definition.  
\end{definition}

The composition maps involving one input of type
$1_\phi^{\greyt}, 1_\phi^{\whitet}$ are also defined geometrically by
the above sum, as in Lemma \ref{lem:highq} below computing
$m_1(1_\phi^{\greyt})$.  In particular
\begin{equation} \label{only} m_2( 1_\phi^{\whitet}, 1_\phi^{\whitet})
  = 1_\phi^{\whitet}, \quad - m_2( 1_\phi^{\greyt}, 1_\phi^{\whitet}) =
  m_2( 1_\phi^{\whitet}, 1_\phi^{\greyt}) =
  1_\phi^{\greyt} \end{equation}
since the corresponding moduli spaces are points.   Recall
that $1_\phi^{\whitet}$ is a strict unit if and only if 
\begin{equation} \label{strictunit} m_2(1_\phi^{\whitet},a) = a =
  (-1)^{|a|} m_2(a,1_\phi^{\whitet}), \quad m_d(\ldots, 1_\phi^{\whitet},
  \ldots) = 0, \forall d \neq 2 .\end{equation}

\begin{theorem} \cite[Theorem 4.1]{pw:flow} \label{thm:morsemodel} For a perturbation system $\ul{P} = (P_\Gamma)$ with good properties
as in Theorem \ref{thm:comeager} the maps $(m_d)_{d \ge 0}$ satisfy the axioms of a (possibly curved)
  \ainfty algebra $CF(\phi)$ with strict unit
  $1_\phi^{\whitet} \in CF(\phi)$.
\end{theorem}

\begin{remark} 
  The second \ainfty relation gives a condition for the existence of a
  coboundary operator.  The element
\[  m_0(1) \in CF(\phi) \]
is the {\em curvature} of the Fukaya algebra and has positive
$q$-valuation $\val_q(m_0(1)) > 0 $ by Lemma \ref{lem:qvalpos}.  The
Fukaya algebra $CF(\phi)$ is {\em flat} if $m_0(1)$ vanishes, and {\em
  projectively flat} if $m_0(1)$ is a multiple of the identity
$1^{\whitet}_\phi$.  The first two \ainfty relations are the analogs
of the Bianchi identity and definition of curvature respectively in
differential geometry:
\[ m_1(m_0(1)) =0, \quad m_1^2(\sigma) = m_2( m_0(1) , \sigma) -
(-1)^{|\sigma|} m_2( \sigma, m_0(1)) ,  \ \ \forall \sigma \in \cI(\phi). \]
Thus if $CF(\phi)$ is projectively flat then $m_1^2 = 0$ and the {\em
  undeformed Floer cohomology} $HF(\phi) = \ker(m_1)/\on{im}(m_1)$ is
defined.  
\end{remark} 

\begin{lemma}  \label{lem:highq}  For the composition maps $m_d$ defined
using \eqref{eq:highercomp}, $ m_1(1_\phi^{\greyt}) $ is equal to
  $1_\phi^{\whitet} - 1_\phi^{\blackt} $ 
plus terms that are higher
  order in $q$.
\end{lemma} 

\begin{proof} By definition, $m_1(1_\phi^{\greyt})$ counts
  configurations with a single input and output edge. 
  By definition, constant configurations from a single edge $T_e$ with input 
  $1_\phi^{\greyt}$ and output 
  $1_\phi^{\whitet}$ or $  1_\phi^{\blackt}$ are stable.
The moduli space of such configurations with weight $\rho(e) = \infty$
occur as the positive end of the moduli space 
$\M_\Gamma^{\wt}$ and by definition is positively oriented, 
while the locus with $\rho(e) = 0$ is negatively oriented.
Configurations
  with no disks contribute $1_\phi^{\whitet} - 1_\phi^{\blackt} $,
  while configurations $(C,u: C \to X)$ with at least one disk
  $u_v: S_v \to X$ contribute terms with positive area $A(u) > 0$,
  since at least one of the disks $u_v$ must be non-constant by the
  stability condition.
\end{proof}

\begin{lemma} \label{lem:qvalpos} \label{gapcond} 
 For composition maps $m_d$ defined
using \eqref{eq:highercomp}, the curvature $m_0(1)$ satisfies the gap condition
  $\val_q(m_0(1)) \ge \hbar$, where $\hbar >0 $ is the energy
  quantization constant of Lemma \ref{lem:energyquant}.
\end{lemma} 

\begin{proof} Any configuration $(C,u : C \to X)$ with no 
  leaves $T_e$ must have at least one non-constant 
  holomorphic disk $u_v|S_v: S_v \to X$, by the stability 
  condition.  Thus the area of any configuration $(C, u: C \to X)$
  contributing to $m_0(1)$ must be at least $A(u_v) \ge \hbar$ by Lemma 
  \ref{lem:energyquant}. 
\end{proof}

More generally, the Fukaya algebra may admit projectively flat
deformations even if it itself is not projectively flat.  Consider the
sub-space of $CF(\phi)$ consisting of elements with positive
$q$-valuation
\[CF(\phi)_+ = \bigoplus_{\sigma \in \cI(\phi)} \Lambda_{>0} \sigma .\]
where $\Lambda_{> 0} = \{ 0 \} \cup
\val_q^{-1}(0,\infty)$.\footnote{In fact one can only require positive
  valuations of the coefficients of the degree-one generators, and the
  self-intersection points.  The requirement of positivity at the
  self-intersection points can be slightly weakened, see \eqref{projmc}
  below. \label{onlyfoot}}
  Define
the {\em Maurer-Cartan map}
\[ m: CF(\phi)_+ \to CF(\phi), \quad b \mapsto m_0(1) + m_1(b) +
m_2(b,b) + \ldots .\]
Here $m_0(1)$ is the image of $1 \in \Lambda$ under 
\[ m_0: CF(\phi)^{\otimes 0} \cong \Lambda \to CF(\phi) . \]   
Let $MC(\phi)$ denote the space of (weakly) {\em bounding cochains}:
\begin{equation} \label{mc}
  MC(\phi) = \Set{ \begin{array}{l} b \in CF^{\on{odd}}_+(\phi) \\
                     \val_q(b) > 0 \end{array} \ | \ \begin{array}{l}
                                                       m(b) \in
                                                       \span(1^{\whitet}_\phi) \end{array}
  } .\end{equation}
The value
$W(b)$ of $m(b)$ for $b \in MC(\phi)$ defines the {\em disk potential}\label{dp}
\[ W: MC(\phi) \to \Lambda, \quad m(b) = W(b) 1_\phi^{\whitet} .\]
For any $b \in MC(\phi)$ define a projectively flat {\em deformed
  Fukaya algebra} $CF(\phi,b)$ with the same underlying vector space
but composition maps $m_d^b$ defined by
\begin{multline}  \label{mb} m_d^b(a_1,\ldots,a_d) = \sum_{i_1,\ldots,i_{d+1}} m_{d + i_1 +
  \ldots + i_{d+1}}(\underbrace{b,\ldots, b}_{i_1}, a_1,
\underbrace{b,\ldots, b}_{i_2}, a_2,b, \\ \ldots, b, a_d,
\underbrace{b,\ldots, b}_{i_{d+1}}) ;\end{multline} 
note that these maps only satisfy the \ainfty axiom if $b$ has odd
degree because of additional signs that appear in the case $b$ even.
Occasionally we wish to emphasize the dependence of $MC(\phi)$ on the
local system $y \in \RR(\phi)$ and we write $MC(\phi,y)$ for
$MC(\phi)$. \label{repet} For $b \in MC(\phi)$, the maps $m^b_d, d \ge 1$ form a
projectively flat \ainfty algebra.  The resulting cohomology is
denoted
\[ HF(\phi,b) = \ker(m_1^b)/\on{im}(m_1^b) \]
The deformed second composition map $m_2^b$ induces on 
$HF(\phi,b)$  with $b \in MC(\phi)$ the structure of an associative algebra.  
The union of $HF(\phi,b)$ for $b \in MC(\phi)$ mod gauge equivalence,
see the following section, is a homotopy invariant of $CF(\phi)$ and
independent of all choices up to isomorphism of 
algebras
and change of
base point $b$.

In the case of self-intersection points, the condition that the
Maurer-Cartan solutions have positive $q$-valuation may be relaxed
using the following lemma, which is a sort of energy quantization for
corners at self-intersections.  The following is an analog of
\cite[Lemma 2.6]{det:ref}.

\begin{lemma} \label{lem:corners}
  \label{vqv} Let $\dim(L_0 ) > 2$ and $k > 2$.  There exists a
  constant $\delta > 0$ such that the following holds: Suppose that
  $(C,u: C \to X)$ is a rigid treed holomorphic disk with $k+1$
  leaves.  If $s$ is the number of corners $w_e \in S$ mapping
  to transverse self-intersection points
  $\sigma_e \in \cI^{\si}(\phi)$, then $A(u) \ge s \delta$.
\end{lemma} 

\begin{proof}  We convert the area into action via an application of Stokes' theorem.  Let $x \in \cI^{\si}(\phi)$ be a
  self-intersection point.  We may assume without loss of generality
  that the Darboux chart $X \supset U \to \CC^n$ has image that
  contains the radius $r$ ball $B_r(0) \subset \CC^n$ for $r \in (0,\infty)$
  small.  Recall from Section \ref{Jstd} \label{Jstdref} that the
  complex structure $J_\Gamma \in \J(X)$ near the self-intersection
  point is standard so that $J_\Gamma |U = J_0$, where $J_0 z = iz$ for any
  tangent vector $z \in T_x U \cong \R^{2n} \cong \CC^n$.
  The symplectic form $\omega_0$ on $\CC^n$ is exact with 
\[ \omega_0 = \dd \alpha_0, \quad \alpha_0 := \sum_{j=1}^n \hh ( q_j \dd p_j - p_j \dd q_j ) \in 
\Omega^1(\CC^n) . \]
 By
Stokes' theorem,
\begin{equation} \label{eq:alphint} \int_{u^{-1}(U)} u^* \omega_0 =
   \int_{u^{-1}(\partial U)} u^* \alpha_0 .\end{equation}
Here we have used that the restriction of $\alpha$ to the Lagrangian branches
$\R^n, i\R^n$ vanishes. 

We now take the limit as the path approaches the intersection point. 
  On the locus $U^* = \{ z \in S, u(z) \neq 0 \} $ where $u$ is
  non-zero in the local chart the map $u$ descends to a map  
\[ [u]: U^* \to \C
  P^{n-1}, \quad z \mapsto \on{span}(u(z)) .\] 
  \label{Ustar}  Consider the corresponding section
  $z \mapsto ([u(z)], u(z))$ of the pull-back $[u]^*T$ of the
  tautological bundle
  \[ T = \Set{ (\ell, z) \in \C P^{n-1} \times \CC^n \ | z \in \ell } \to 
  \C P^{n-1} .\]
The restriction of $\alpha_0$ to the boundary of the ball $B_r(x)$,
viewed as the unitary frame bundle of the tautological bundle $T$, is
$-r^2$ times the standard connection one-form
$\alpha_{T} \in \Omega^1(T_1)$ on the unit circle bundle $T_1$ in
the tautological bundle $T \cong S^{2n-1}$  over $\C P^{n-1}$
with projection $\pi: T \to \C P^{n-1}$.  Let
\[ \curv(T) \in \Omega^2(\C P^{n-1}), \quad (\pi |_{T_1})^* \curv(T) =
\dd \alpha_{T}
\]  
denote the curvature two-form of $\alpha_T$.  One checks easily from,
for example, a Taylor series expansion that removal of singularities
holds in this case and the map $[u]$ extends to a holomorphic map
$ u^{-1}(U) \to \C P^{n-1}$.  Since $[u]$ is also holomorphic, the
pull-back of minus the curvature
$- [u]^* \curv(T) \in \Omega^2(u^{-1}(U))$ is a positive two-form.  On the other hand, on the locus $u \neq 0$
the map $u$ determines a section of $U$ whose normalization
$v = u/ \Vert u \Vert$ trivializes $u^* T$. The integral \eqref{eq:alphint}
is up to a scalar the parallel transport in the frame defined by the
section $u$: Let $B_\eps(u^{-1}(0))$ denote a union of $\eps$-balls
around the finite set $u^{-1}(0)$, and denote the fractional winding
number
\[ d(u,z) := (2 \pi)^{-1} \int_{\partial B_\eps(z) \cap S} v^*
\alpha_T \]
of the phase of the section $u$ along the path
$\partial B_\eps(z) \cap S$; note that this integral is well-defined even if
$z$ is a boundary point.  By Stokes' theorem
\begin{eqnarray}
  \int_{u^{-1}( \partial U)} u^* \alpha_0 &=& - r^2 \int_{u^{-1}( \partial 
                                              U)} v^* \alpha_T \\
                                          &=& \lim_{\eps \to 0}  - r^2
                                              \left( \int_{[u | U -
                                              B_\eps(u^{-1}(0))]}
                                              [v]^* \curv(T)
                                              \right. \\   \quad  &  & \left.
                                              - 
                                              \int_{\partial B_\eps(u^{-1}(0)) } v^* \alpha_T  \right) \\ 
                                          &=& - r^2 \int_{[u | U]}
                                              [v]^* \curv(T) + 2\pi r^2  
                                              \sum_{z \in u^{-1}(0)} d(u,z) \label{rhs}.
 \end{eqnarray}
 \label{cornersproof} The tautological bundle $T$ has curvature
 $- \curv(T)$ that is a positive two-form, see for example Demailly
 \cite[Section 15.B]{de:ca}.  It follows that the first term on the
 right-hand side of \eqref{rhs} is non-negative.  
 Let $\delta$  be the minimum of constants $r^2 \pi/2$, as $x$
 varies over transverse self-intersection points.
 The angle change at any self-intersection point is a multiple of $\pi/2$, which  proves the claim.

Finally, we rule out constant disks mapping to
self-intersections.  Constant disks $u: S \to X$ with image
$\phi(x), x \in \cI^{\si}(\phi)$ must have corners with alternating
labels
\[ \sigma_1 = x, \sigma_2 = \ol{x},\sigma_3 = x,\sigma_4 = \ol{x},
\ldots, \sigma_{2k} = \ol{x} . \] 
The sum of the degrees of these constraints is $k \dim(L)$, while the
moduli space of $2k+1$-marked disks has dimension $2k -2$.  The
expected dimension of the moduli space of holomorphic treed disks is
therefore 
\begin{eqnarray*}
\dim \M_{\bGamma}(X,\phi,\ul{\sigma})  &=& 
(2k -2)  - k (|x| + (\dim(L) - |x|) \\
&=& (k-1) (2 - \dim(L)) .\end{eqnarray*} 
Thus the labeled type $\bGamma$ is rigid only if $\dim(L_0) = 2$
or $k = 1$.   \end{proof}

\begin{corollary}\label{cor:mcbar}   
  Let $\phi_0: L_0 \to X$ be a self-transverse immersed Lagrangian brane 
  of dimension $\dim(L_0) \ge 2$.  The projective Maurer-Cartan equation 
 \begin{equation} \label{projmc} 
 \sum_{d \ge 0} m_d(b,\ldots, b) \in \on{span}
 1^{\whitet}_\phi \end{equation} 
is well-defined for $b$ of the form $b = b^{{\si}} + b^c $ satisfying
the condition in Definition \ref{admissible} for the $\delta$
described in Lemma \ref{lem:corners}.  Any such solution $b$ has
square-zero $m_1^b$ and so a Floer cohomology group
\[ HF(\phi,b) = \frac{\on{ker}(m_1^b)}{\on{im}(m_1^b)} .\]
\end{corollary}

\begin{proof} By Lemma \ref{lem:corners}, the infinite sum in the
  Maurer-Cartan equation \eqref{projmc} has $q$-valuations approaching
  infinity and is well-defined in $CF(\phi)$.  A similar argument
  shows that the deformed Fukaya maps $m_d^b$ from \eqref{mb} are
  well-defined.
\end{proof}

Denote the set of solutions in Corollary \ref{cor:mcbar} by
\[ MC_\delta(\phi) = \{ b \in CF(\phi) | \eqref{projmc} \} .\]

\begin{remark} 
  In the case $\phi = \phi_\eps$ is a surgery, we allow the
  coefficients $b_\eps(\mu), b_\eps(\lambda)$ of the meridian and
  longitude to have vanishing $q$-valuation.  Theorem \ref{thm:repthm}
  implies that for the \label{thep} perturbation systems we use, the
  potential $W(b_\eps)$ and Floer cohomology $HF(\phi_\eps,b_\eps)$
  are still well-defined for such elements.  
\end{remark}

\begin{remark} We briefly describe the invariance properties of
  Fukaya algebras.  The
  argument using quilted disks with diagonal seam condition, see
  Charest-Woodward \cite[Section 3]{flips} and Palmer-Woodward
  \cite[Remark 6.3]{pw:flow} extends to the cellular setting to define
  \ainfty morphisms between \ainfty algebras defined using different
  choices.  Given two sets of choices $J_k, D_k , \ul{P}_k$ this
  argument gives an \ainfty morphism
\[ CF(\phi,J_0,D_0,\ul{P}_0) \to CF(\phi,J_1,D_1,\ul{P}_1) \]
inducing in particular a morphism of Maurer-Cartan spaces
\[ MC(\phi,J_0,D_0,\ul{P}_0) \to MC(\phi,J_1,D_1,\ul{P}_1) \]
preserving the Floer cohomologies.  We expect that the homotopy type
of the immersed Fukaya algebra $CF(\phi)$ is independent of the choices of almost complex
structure, divisor, and perturbations.  However, we prove no such invariance result here.  \end{remark}

\begin{example} \label{threelobes} The following example of an
  immersion of a circle in the plane shown in Figure \ref{loopy} is an
  easily visualizable example of the invariance of the disk
  potential.  In this case,
  the correspondence between holomorphic curves in $X$ bounding
  $\phi_0$ and $\phi_\eps$ is an application of the Riemann mapping theorem.  The
  Floer cohomology $HF(\phi_0)$ is trivial since the circle is
  displaceable by a compactly-supported Hamiltonian flow.  The disk
  potential $W(\phi_0)$ is non-trivial and will be computed below.
  Let $\phi_0: S^1 \to \R^2$ be the immersion with three
  self-intersection points\
  \[ x,x',x'' \subset \phi_0(S^1) .\]
  The complement of the image $\phi_0(S^1) \subset X = \R^2$ has five
  connected components as in Figure \ref{loopy}.

  We identify a particular weakly bounding cochain.  Suppose that the
  area of the central region in $X - \phi_0(L)$ is $A_0 >0 $ while the
  area of each of the lobes is $A_1 > 0$.  For simplicity, choose a
  Morse function on $L_0 \cong S^1$ so that there is a single $0$-cell $\sigma_0$  on the lobe containing $x$, and a single $1$-cell $\sigma_1$;  the actual cell structure used for the proof is somewhat more
  complicated but the difference in cell structures is irrelevant for
  the example.   Consider the
cochain
\[ b_0 = i q^{ (- A_0 + 3A_1)/2} 1_{\phi_0}^{\greyt} + i q^{(A_1 -
  A_0)/2} (x + x' + x'') \in CF(\phi_0) \]
with coefficient $i q^{(A_1 - A_0)/2}$ on the self-intersection points
$x,x',x''$ and a multiple of the degree $-1$ element
$ i q^{ (-A_0 + 3A_1)/2} 1_{\phi_0}^{\greyt}$. 

We compute the twisted curvature $m_0^{b_0}(1)$ as follows.  The three
outer lobes with no inputs contribute
$q^{A_1} (\ol{x} + \ol{x'} + \ol{x''})$ to $m_0(1)$, and also to
$m_0^{b_0}(1)$.  The disk $u: S \to X$ whose interior $\on{int}(S)$
maps to the central region of $X - \phi_0(L_0)$ contribute to
$m_0^{b_0}(1)$ with outputs on $x,x',x''$.  Since for each such output
there are two inputs labeled $b_0$, the contribution of this region
is
\[ q^{A_0} (i q^{(A_1 - A_0)/2})^2 (\ol{x} + \ol{x'} + \ol{x''}) \in
CF(\phi_0) .\]
The holomorphic strip connecting $x$ to the zero-dimensional cell
contributes to $m_0^{b_0}(1)$ as well, with a single $g$ input and
so a contribution of $i q^{(A_1 - A_0)/2} q^{A_1} \sigma_1 $.
Finally, the constant disk with input
$i q^{ (- A_0 + 3A_1)/2} 1_{\phi_0}^{\greyt}$ contributes
\[ i q^{ (- A_0 + 3 A_1)/2} ( 1_{\phi_0}^{\whitet} -
1_{\phi_0}^{\blackt})  \in CF(\phi_0) \]
to $m_1(b_0)$, hence $m_0^{b_0}(1)$.  Thus
\begin{eqnarray*}
  m_0^{b_0}(1) &=& q^{A_1} (\ol{x} + \ol{x'} + \ol{x''}) 
+ q^{A_0} (i 
                     q^{(A_1 - A_0)/2})^2 (\ol{x} + \ol{x'} + \ol{x''}) \\
  &+& 
(i q^{(A_1 -
                    A_0)/2}) q^{A_1} \sigma_1 + 
i q^{ (- A_0 + 3 A_1)/2} ( 1_{\phi_0}^{\whitet} -
      1_{\phi_0}^{\blackt})  \\ 
                 &=& i q^{(- A_0 + 3 A_1)/2} 
                     1_{\phi_0}^{\whitet} \end{eqnarray*} 
                   is a multiple of the unit $1_{\phi_0}^{\whitet}$.
                   Therefore, the element $b_0 \in MC(\phi_0)$ is a solution
                   to the projective Maurer-Cartan equation.

                   The self-intersection points of $\phi_0$ are
                   admissible in the sense of Definition \label{admex}
                   \ref{admissible},  which implies that the
                   Floer cohomology is well-defined.  Any disk
                   $u: S \to X$ with boundary on $\phi_0$ and meeting
                   one of the self intersection points
                   $x = (x_-,x_+) \in S^1$ without a branch change
                   must contain in its image $u(S)$ the exterior
                   non-compact region in $X$ outside the curve
                   $\phi_0(S^1)$.  This is impossible since the image
                   of a compact set must be compact.

                   A small Lagrangian surgery produces a Lagrangian
                   immersion of a disjoint union of circles.  Choose
                   $\eps > 0$ sufficiently small so that the surgery
                   is defined and 
\[ (A_1 - A_0)/2 = - A(\eps) \] 
where $A(\eps) > 0$ is the area from Definition \ref{lsurg}.  Let
$\sigma_1',\sigma_1''$ denote the top-dimensional cells on the two
components near the self-intersection point $x$, as in Figure
\ref{loopy}. \label{clarex}
As explained below in \eqref{Mshift}, the shift from $b_0$ to $b_\eps$ is equivalent to 
a shift in the local system. Define a local system 
$y_\eps$ on $\phi_\eps$  by 
    \[ y_\eps([\sigma_1']) = y_\eps([\sigma_1'']) = i q^{(A_1 -
      A_0)/2} q^{A(\eps)} = i .\]
Define $b_\eps$ by removing the $x$-term so that
    \[ b_\eps = i q^{ (- A_0 + 3 A_1)/2} 1_{\phi_0}^{\greyt} 
+ i q^{(A_1 -
      A_0)/2} ( x' + x'') .\]
We have 
\begin{eqnarray*} 
  m_0^{b_\eps}(1) &=&   i  q^{A_1 - A(\eps)} \sigma_1' + ( i q^{(A_1- A_0)/2})^2  q^{A_0 - A(\eps)} \sigma_1''
  \\ &&+ q^{A_1} (\ol{x}' + \ol{x}'') + 
        q^{A_0 - A(\eps)} i (i q^{(A_1 - A_0)/2}) (\ol{x}' + \ol{x}'') \\
                    && + i q^{ (- A_0 + 3A_1)/2} (1_{\phi_0}^{\whitet} - 1_{\phi_0}^{\blackt}) 
  \\ &=& i q^{( - A_0 + 3A_1)/2} 1_{\phi_0}^{\whitet}.  \end{eqnarray*}
It follows that $m_0^{b_\eps}(1)$ is a multiple of the strict unit
$1_{\phi_0}^{\whitet}$ on the right-hand-side with the same value of
the potentials
\[ W_0(b_0,y_0) = i q^{(3A_1 - A_0)/2} = W_\eps(b_\eps,y_\eps) \]
as the unsurgered immersion $\phi_0$.  This ends the Example.
\end{example} 

\subsection{Gauge equivalence} 

A notion of gauge equivalence relates solutions to the weak
Maurer-Cartan equation so that cohomology is invariant under gauge
equivalence.  Let $b_0,\ldots,b_d \in CF(\phi)$ have odd degree and let
$a_1,\ldots, a_d \in CF(\phi)$.  Define
%
\begin{multline} \label{lotsbs}
 m_d^{b_0,b_1,\ldots,b_d}(a_1, \ldots, a_d) =
 \\ \sum_{i_0,\ldots,i_{d}} m_{d + i_0 + \ldots +
   i_{d}}(\underbrace{b_0,\ldots, b_0}_{i_0}, a_1,
 \underbrace{b_1,\ldots, b_1}_{i_2}, a_2,b_2, \\ \ldots, b_2, \ldots,
 a_d, \underbrace{b_d,\ldots, b_d}_{i_{d}}). \end{multline}
Two odd elements $b_0,b_1 \in CF(\phi)_+$ are {\em gauge equivalent} if
and only if \label{geq}
\[ \exists h \in CF(\phi)_+, \ b_1 - b_0 = m_1^{b_0,b_1}(h), \quad
\deg(h) \ \text{even} .\]
We then write $b_0 \sim_h b_1$.  The discussion on \cite[p. 75]{flips}
shows that $\sim_h$ is an equivalence relation. The linearization of
the above equation is $m_1(h) = b_1 - b_0$, in which case we say that
$b_0$ and $b_1$ are {\em infinitesimally gauge equivalent}.

For notational convenience, we define a ``shifted valuation'' 
\[ \begin{array}{ll} 
\val_q^\delta(b^{\si}) = \val_q(b^{\si}) + \delta & b^{\si} \in
  \on{span}(\cI^{\si}(\phi)) - \{ 0 \}  \\
\val_q^\delta(b^c) =   \val_q(b^c) & b^c \in \span(\cI^c(\phi))  - \{
                                     0 \} \\
\val_q^\delta(b^c + b^{\si}) = \min(\val_q^\delta(b^c), 
\val_q^\delta(b^{\si})), & b^c, b^{\si} \neq 0 .\end{array} 
\]
Then $MC_\delta(\phi)$ is the space of solutions to the projective
Maurer-Cartan equation with non-negative $\val_q^\delta$.

\begin{lemma} \label{lem:gaugekill} Let $\phi: L \to X$ be a
  self-transverse immersed Lagrangian brane and
  $b_0,b_1 \in CF(\phi)$.
\begin{enumerate} 
\item \label{firsti} {\rm (Preservation of the Maurer-Cartan space
    under gauge equivalence)} If $b_0 \sim_h b_1$ for some
  $h \in CF(\phi)_+$ and $b_0 \in MC_\delta(\phi)$ then
  $b_1 \in MC_\delta(\phi)$ as well.
\item \label{secondi} {\rm (Integration of infinitesimal gauge
    equivalences into gauge equivalences)} Suppose that
  $h,b_0,b_1 \in CF(\phi)$ and $\zeta > 0$ are such that
 \begin{equation} 
\label{infgauge}
m_1^{b_0,b_1}(h) = b_1 - b_0, \quad  \text{mod} \ ( \val_q^\delta)^{-1}(
( \zeta,\infty))  , \quad \val_q^\delta(h) > \zeta  .\end{equation}
Then there exists an element
$b_\infty \in CF(\phi), \val_q^\delta(b_\infty) >0$ with
  \[ m_1^{b_0,b_\infty}(h) = b_\infty - b_0, \quad \val_q^\delta(b_\infty - b_1) >
  \val_q^\delta(b_1 - b_0) + \zeta .\]
\end{enumerate}
\end{lemma} 

\begin{proof} 
  For item \eqref{firsti}, define $W(b_1) \in \Lambda$ so that
\begin{equation} \label{ceq}
 m_0^{b_1}(1) = W(b_1) 1_\phi^{\whitet} + c, \quad 
c: = 
(m_0^{b_1}(1) -  W(b_1) 1_\phi^{\whitet}) .\end{equation} 
The element $c \in CF(\phi)$ has coefficient of the strict unit
$1_\phi^{\whitet}$ equal to zero.  We have
\begin{eqnarray*}  m^{b_1}_0(1)  - m^{b_0}_0(1)
 &=& \sum_{d,i \leq
d-1}
m_d(\underbrace{b_0,\ldots,b_0}_i
,b_1 - b_0,
b_1,\ldots, b_1)
\\ &=&  m_1^{b_0,b_1}( m_1^{b_0,b_1}(h)) \\
&=& 
- m_2^{b_0,b_0,b_1}( m_0^{b_0}(1), h) 
+ m_2^{b_0,b_1,b_1}(h, m_0^{b_1}(1)) 
\\ 
&=& 
W(b_1) h - W(b_0) h + m_2^{b_0,b_1,b_1}(h, m_0^{b_1}(1) - W(b_1)
    1_\phi^{\whitet})
\\ 
&=& 
W(b_1) h - W(b_0) h + m_2^{b_0,b_1,b_1}(h,c)
\end{eqnarray*}
where the last inequality uses the definition of $c$ in \eqref{ceq}
and the strict unit identities \eqref{strictunit}.  Rearranging terms we have
\begin{eqnarray} \nonumber
( W(b_1) - W(b_0))  ( 1_\phi^{\whitet} - h)  &=& 
( ( m^{b_1}_0(1) -  c)  - W(b_1) h ) 
 - (m^{b_0}_0(1)  - W(b_0) h) \\
&=& \label{so} 
   m_2^{b_0,b_1,b_1}(h, c)  -  c
    .\end{eqnarray} 
  Since the two terms on the right have no coefficient of
  $1_\phi^{\whitet}$ by \eqref{only}, we must have $W(b_0) = W(b_1)$.
  
  We now apply an induction to show that the correction $c$ vanishes.   Suppose that there exists $\zeta > 0$
  and $k \ge 1$ such that $c$ is divisible
  by $q^{k\zeta}$ and $\val_q^\delta(h) > \zeta$; note that this holds
  for $k=1$ and some $\zeta > 0$ sufficiently small by the previous
  paragraph.  The equation \eqref{so} implies that
  \begin{eqnarray*}  m_0^{b_1}(1) &=& m_2^{b_0,b_1,b_1}(h, c)   + W(b_1) 1^{\whitet} \\
&\in&  W(b_0) 1^{\whitet}_\phi + (
  \val_q^\delta)^{-1}( ((k+1) \zeta,\infty)) .\end{eqnarray*}
Since this holds for every $k$, the claim \eqref{firsti} follows.

 The second item \eqref{secondi} follows from a 
filtration argument.  Suppose
\[ b_k = m_1^{b_0,b_k}(h) + b_0  \quad 
 \text{mod} \ ( \val_q^\delta)^{-1}(
(k \zeta,\infty)) .\] 
Define a solution $b_{k+1}$ to order $(k+1)\zeta$ by defining
\[ b_{k+1} = m_1^{b_0,b_k}(h) + b_0 .\]
The desired element is the
limit of the elements $b_k$.
\end{proof}

The following gives a way of ``gauging away'' the weakly bounding
cochain in a neighborhood of the self-intersection. 

\begin{proposition} \label{prop:smallneg} Let $\phi: L \to X$ be a
  Lagrangian immersion and $U\subset L$  a contractible 
  open set in $L$.   Then any
  $b_0 \in MC_\delta(\phi)$ is gauge equivalent to some
  $b_\infty \in MC_\delta(\phi)$  that vanishes on critical points  
  contained in $U$.
\end{proposition}

\begin{proof}   We solve for the adjusted weakly bounding cochain order by order, using the fact that the leading order term in the Floer differential is the Morse differential.  Suppose that
  $b_k \in MC_\delta(\phi)$ vanishes on critical points contained in $U$ modulo
  terms of order $k\zeta$ for some $k \in \Z_+$.  Let 
  \[ b_k(U) = \sum_{x \in \cI(\phi) \cap U} b(x) x  \] 
  denote the 
part of $b_k$ lying in $U$.   We may assume that $\zeta$ has been chosen sufficiently small so that any non-constant holomorphic disk bounding $\phi$
has area at least $\zeta$.  In the following, we work up to 
terms so the Morse differential of $b_k$ vanishes  terms with valuation 
greater than $k\eta$ everywhere. 
The fact that $m_0^{b_k}(1) \in \on{span}(1_\phi) $ 
mod $q^{k \zeta}$ implies that $m_1(b_k) =0 $, 
hence $\delta(b_k) = 0$ where $\delta$ is the Morse differential. 
Since $U$ is contractible, the cohomology of $(X,X - U)$ is acyclic except in top degree. Hence $b_k = \delta(c_k)$ for some even cochain $c_k$ mod critical points supported outside of $U$. Let $b_{k+1}^{\pre} = b_k  - \delta(c_k)$, so that 
$b_{k+1}(U)$ vanishes up to order $(k+1) \zeta$ on $U$.
By Lemma \ref{lem:gaugekill} there exists $b_{k+1} \in MC_\delta(\phi)$ gauge
  equivalent to $b_k$ such that
  \[ b_{k+1} - b_k = m_1^{b_k,b_{k+1}}(c_k), \quad \val_q^\delta(b_{k+1}
  (\sigma)) > (k+1) \zeta\]
  for any critical point $\sigma$ contained in $U$.  The gauge transformation does not affect
  the lower order terms in $b_k$ and so the limit 
  \[ b_\infty := \lim_{k \to \infty} b_k \] 
  exists, lies in $ MC_\delta(\phi)$, is gauge
  equivalent to $b_0$.  Furthermore, $b_\infty$ vanishes on all 
  critical points contained in $U$, as desired.
\end{proof}

\subsection{Divisor insertions} 
\label{divisor}
\label{repeated}

The divisor equation for Lagrangian Floer cohomology is a hoped-for
relation, \eqref{diveq} below, for the insertion of a degree one cocycle into the
composition maps.  In this section, we prove a related result for the
contribution of any configuration with a codimension one cell as input
up to repetition of the input.\footnote{
  The results of this section are not necessary if $\dim(L_0) \ge 3$
  and one uses the shift in local system \eqref{Lshift} and
\[  b_\eps = b_0 - b_0(x) x - b_0(\ol{x}) \ol{x} + b_0(x) b_0(\ol{x})  \lambda \] 
instead of shifting the weakly bounding cochain in Definition
\ref{onsurgery}, or in dimension $\dim(L_0) = 2$ with the local system
formulas \eqref{Lshift}, \eqref{Mshift}.}
The divisor equation for Fukaya algebras is similar to the familiar
divisor equation in Gromov-Witten theory in Kontsevich-Manin 
\cite[2.2.4]{km:cohft}.  For $k \ge 0$ write
\[ m_k = \sum_{\beta \in H_2(\phi)} m_{k,\beta} : CF(\phi)^{\otimes k}
\to CF(\phi) \]
 where $m_{k,\beta}$ is the contribution
to $m_k$ arising from holomorphic disks of class
$\beta \in H_2(\phi)$.  The divisor equation for a codimension one
cycle $c$ reads
\begin{equation} \label{diveq} \sum_{i=1}^{k+1} m_{k+1,\beta}(
  x_1,\ldots, x_{i-1}, c , x_{i}, \ldots, x_k) = \lan [c], [\partial
  \beta] \ran m_{k,\beta}(x_1,\ldots, x_k) \end{equation}
see \cite[Proposition 6.3]{cho}.  In particular, the divisor equation
implies that for $x$ a degree one cycle in $\phi(L)$
\begin{equation} \label{exp} \sum_{k \ge 0} m_k(x,\ldots, x) = \sum_{k
    \ge 0} \sum_{\beta \in H_2(\phi)} \frac{ \lan x, [\partial \beta]
    \ran^k}{ k! }  m_{0,\beta}(1) .
\end{equation}  
The right hand side of \eqref{exp} is the contribution of $m_0(1)$
with local system $y$ shifted by
\[ \exp(x) \in \Hom(H_1(\phi(L),\Z), \Lambda_0) \cong \RR(\phi)
. \] 
In this sense, variations of the weakly bounding cochain
$b \in MC(\phi)$ should be equivalent to variations of the local
system $y \in \RR(\phi)$.  In general the truth of the divisor
equation typically depends on the existence of regularized moduli spaces
 of holomorphic disks equipped with forgetful maps.  The existence
 of such maps  is rather difficult in the Morse setting. \label{truthholds} 

We prove an identity for contributions to the composition maps with
repeated inputs related to the divisor equation
\eqref{diveq}; the terminology will be explained in the following discussion.

\begin{theorem} \label{thm:repthm} 
Suppose $\ul{P}^{\red}$ is a reduced-regular perturbation system.   There exists
a system of  perturbations 
$P_\Gamma$ satisfying the conditions in Definition \ref{def:good} such that for each 
codimension one cell $\sigma_i$, 
rigid treed disk $u: C \to X$ whose boundary meets $\sigma_i$ transversally, and collection
of positive integers
$\ul{d} = (d(z) \ge 0)$, there exists a unique configuration 
$u_{\ul{d}}$ obtained by repeating $d(z)$ inputs 
at each $z \in u^{-1}(\sigma_i)$ 
in the neighborhood $U^{\red}_u$  with weight given by the inverse factorial
\[ \wt(u_{\ul{d}},\gamma) = \prod_{z,\sigma_i} 
(d(z)!)^{-1}  \wt(u,\gamma) . \] 
\end{theorem} 

The terminology is explained by the following definitions. 

\begin{definition} \label{repdef}
{\rm (Inserting repeated inputs)}
Let $u: S \to X$ be a holomorphic disk bounding $\phi$ and let 
$z \in \partial S$ be a point 
where $u(z)$ intersects a codimension one cell $\sigma$
transversally.   Given 
such a disk and an integer $d \ge 1$, 
let  
\[ u_d^{\red}: C^{\red}_d \to X \] 
be the configuration whose domain 
\[ C^{\red}_d := C \cup T_{e_0} \cup S_v \cup T_{e_1} \cup \ldots \cup T_{e_d} \]
consists of an
additional 
disk $S_{v}$ on which $u_d^{\red}$ is constant, attached
at $z$ via an edge $T_{e_0}$ of some length $\ell(e_0)$ and $d$ edges
$T_{e_1},\ldots, T_{e_d}$ 
attached to $S_v$ mapping to $\sigma$, as in the middle drawing in Figure \ref{cluster}.
\end{definition}

\begin{figure}[ht]
    \centering
    \scalebox{.7}{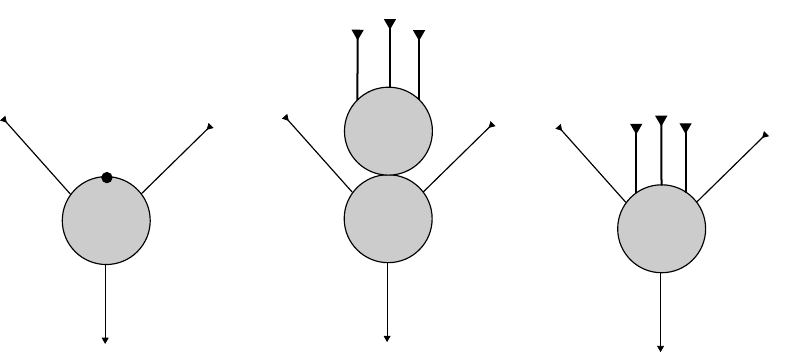}
    \caption{Configurations corresponding to an intersection with a codimension one cell: Left, a configuration meeting the cell transversally; Middle, adding a constant disk with multiple edges to the left configuration; Right, the center configuration after perturbation}
    \label{cluster}
\end{figure}

\begin{definition}  A perturbation system $\ul{P}^{\red}$ is {\em reduced regular} for a cell $\sigma$ if
\begin{enumerate} 
\item the matching conditions $\ul{M} = (M_\Gamma)$ are equal to the identity map in an open neighborhood of $\sigma$ and 
\item all rigid configurations $u$ not of the form $(u^{\red})_d$ for some $u$ and $d$ are regular.
\end{enumerate} 
In other words, rigid maps $u:C \to X$ not obtained by repeating inputs are all regular. 
\end{definition}  

Suppose a reduced regular perturbation system $\ul{P}^{\red}$ has been chosen.  Consider a new perturbation system $\ul{P}$ obtained from 
$\ul{P}^{\red}$ by perturbing the matching conditions for the semi-infinite edges $T_e, e \in \Edge_{\rightarrow}(\Gamma)$.  

\begin{lemma} \label{lem:remove} For a generic 
perturbation system $\ul{P}$ obtained from $\ul{P}^{\red}$
as above,  each rigid treed disk $u$ is regular and  obtained from some 
$u_d^{\red}$ by removing the ghost component and changing
the positions of the edges $T_{e_1},\ldots, T_{e_d}$ so that
the corresponding nodes  map to the perturbed cell $M_\Gamma(w_{e_i}, \cdot)^{-1}(\sigma)$.
\end{lemma} 

\begin{proof}  Transversality is a standard consequence of Sard-Smale. 
To see that any rigid treed disk is of the form $u_d^{\red}$
for some $u$ for a \label{rep:adda} sufficiently small perturbation, suppose that 
$u_\nu$ is a family of such maps for perturbations
$M_{\Gamma,\nu}(w_{e_i}, \cdot)$ that  converges to the identity in 
the $C^\infty$ topology, but $u_\nu$ is not obtained as above. By Gromov compactness, there exists a subsequence of $u_\nu$ that converges to some $u_d^{\red}$.
But by transversality, there is a unique $u_\nu$ close to $u_d^{\red}$
satisfying the perturbed matching conditions.  This is a contradiction.
\end{proof}

That is, perturbed configurations are clustered around the configurations obtained by repeating inputs.  

\begin{example} \label{ex:twosphere} 
We explain how repeating inputs appear in the product structure on the Floer cohomology of a Lagrangian given by the circle in the two-sphere. 
Let $X = S^2$ and $L = S^1$ a circle dividing 
$X$ into two regions of areas $A_+$ and $A_-$.  Equip $L$ with a
Morse function so that it has standard cell decomposition
into two cells consisting of a $0$-cell $\sigma_0$ and
a $1$-cell $\sigma_1$.  Equip $X$ with its standard complex structure. 
\begin{figure}[ht]
    \centering 
     \scalebox{.5}{
\begingroup%
  \makeatletter%
  \providecommand\color[2][]{%
    \errmessage{(Inkscape) Color is used for the text in Inkscape, but the package 'color.sty' is not loaded}%
    \renewcommand\color[2][]{}%
  }%
  \providecommand\transparent[1]{%
    \errmessage{(Inkscape) Transparency is used (non-zero) for the text in Inkscape, but the package 'transparent.sty' is not loaded}%
    \renewcommand\transparent[1]{}%
  }%
  \providecommand\rotatebox[2]{#2}%
  \newcommand*\fsize{\dimexpr\f@size pt\relax}%
  \newcommand*\lineheight[1]{\fontsize{\fsize}{#1\fsize}\selectfont}%
  \ifx\svgwidth\undefined%
    \setlength{\unitlength}{175.47016174bp}%
    \ifx\svgscale\undefined%
      \relax%
    \else%
      \setlength{\unitlength}{\unitlength * \real{\svgscale}}%
    \fi%
  \else%
    \setlength{\unitlength}{\svgwidth}%
  \fi%
  \global\let\svgwidth\undefined%
  \global\let\svgscale\undefined%
  \makeatother%
  \begin{picture}(1,1.00000136)%
    \lineheight{1}%
    \setlength\tabcolsep{0pt}%
    \put(0,0){\includegraphics[width=\unitlength,page=1]{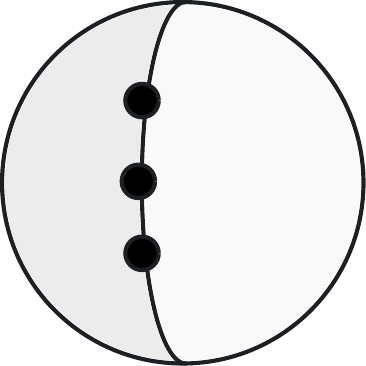}}%
    \put(0.53064483,0.80905015){\color[rgb]{0,0,0}\makebox(0,0)[lt]{\lineheight{1.25}\smash{\begin{tabular}[t]{l}$A_+$\end{tabular}}}}%
    \put(0.10181677,0.44026079){\color[rgb]{0,0,0}\makebox(0,0)[lt]{\lineheight{1.25}\smash{\begin{tabular}[t]{l}$A_-$\end{tabular}}}}%
    \put(0.45170666,0.66045314){\color[rgb]{0,0,0}\makebox(0,0)[lt]{\lineheight{1.25}\smash{\begin{tabular}[t]{l}$\sigma_{1,1}$\end{tabular}}}}%
    \put(0.46589503,0.464363){\color[rgb]{0,0,0}\makebox(0,0)[lt]{\lineheight{1.25}\smash{\begin{tabular}[t]{l}$\sigma_{1,2}$\end{tabular}}}}%
    \put(0.45687747,0.27087071){\color[rgb]{0,0,0}\makebox(0,0)[lt]{\lineheight{1.25}\smash{\begin{tabular}[t]{l}$\sigma_{1,3}$\end{tabular}}}}%
  \end{picture}%
\endgroup%
}
    \caption{Disks with point constraints on the sphere}
    \label{sphere}
\end{figure}
Since there are two Maslov-index-two disks corresponding to the two hemispheres, for the trivial bounding cochain
we obtain 
\[ m_0(1) = (q^{A_+} + q^{A_-}) \sigma_0 \] 
and the relation 
\[ m_1(\sigma_1) = (q^{A_+} - q^{A_-}) \sigma_0 . \] 
Consider the 
family of maps $u_d^{\red}$ obtained from one of the two hemispheres
as above with $d$ by repeating incoming 
edges labeled by $\sigma_1$.   The unperturbed complex structure and matching condition is reduced-regular for $\sigma_1$, since every configuration 
without constant disks labeled by multiple codimension-one constraints is regular.  In particular, the map $u_1$
obtained from $u$ by adding an edge $T_{e_1}$ mapping to 
$\sigma_0$ is regular.  A choice of perturbed
matching conditions $M_\Gamma$ at the incoming edges $T_{e_1}, \ldots, T_{e_d}$ amounts to a collection of perturbations
$\sigma_{1,1},\ldots, \sigma_{1,d}$, that is, points near
$\sigma_1$.    Suppose the perturbations $\sigma_{1,i}$ are in the same order around the boundary $\partial S_v$
as the order given by the indices.  Then there
is a $P_\Gamma$-perturbed map $u_d$ obtained from 
$u$ by repeating inputs.   Otherwise, if the order is different, then no such map exists.
In particular, for $d = 2$ if the perturbations
$\sigma_{0,1}, \sigma_{0,2}$ follow the cyclic order
around the disk with area $A_+$ then we have 
\begin{equation} \label{pert1}
m_2(\sigma_1,\sigma_1) = q^{A_+} \sigma_0 \end{equation}
while if the perturbations are in the opposite order
then 
\begin{equation} \label{pert2}
m_2(\sigma_1,\sigma_1) = q^{A_-} \sigma_0 .\end{equation}
If $A_+ = A_1$, then the Floer cohomology is non-trivial 
and we obtain the relation on cohomology  
\[ m_2([ \sigma_1], [\sigma_1]) = q^{A_+} [\sigma_0]
= q^{A_-} [ \sigma_0] \] 
so the product on cohomology is independent of the choice of perturbation.    
\label{twosphere2} \label{twosphere3}
The standard complex structure $J$ on $X = S^2$ combined with the unperturbed matching condition $M_\Gamma ( w,l) = l, \forall l \in L$ is reduced
regular since the only maps of expected dimension zero
are those containing a single Maslov index two disk constrained by degree one cells.
Any such disk $u_v: S_v \to X$ is necessarily one of those above with area $A_+$ or $A_-$ and is regular 
as long as there is a single incoming leaf.  The 
only other configurations of expected dimension
are those with multiple leaves labeled $\sigma_1$, and these are not required to be regular. The contributions to $m_d(\sigma_1,\ldots, \sigma_1)$
are clustered around the two maps arising from the two disks
with areas $A_+,A_-$, and obtained by slightly perturbing
the points $\sigma_1,\ldots, \sigma_1$ to points in general position $\sigma_{1,1},\ldots, \sigma_{1,d}$.
\end{example} 

We wish to choose perturbations so that the count of perturbed configurations with repeated inputs is controlled by the weight of the original configuration.

\begin{definition} \label{def:pinv} The matching condition $M_\Gamma$ for a type $\Gamma$ is
  {\em permutation-invariant}
  on an open subset $U \subset L$ for $\sigma$ if 
  for any two edges $e_1,e_2 \in \Edge(\Gamma)$
    $M_\Gamma( w_{e_1},l) = M_\Gamma(w_{e_2},l)$
    as multivalued perturbations for all $l \in U$.
\end{definition} 

\begin{example} We continue the two-sphere example in 
\ref{twosphere3} and compute the second structure map for a permutation-invariant matching condition.    A permutation-invariant  
multi-valued perturbation is given by matching conditions assigning
the nodes to map to perturbations of the $0$-cells
$\sigma_{1,1}, \sigma_{1,2}$ in clockwise order
around the boundary of the disk with coefficient $1/2$,
and $\sigma_{1,1}, \sigma_{1,2}$ in the counterclockwise order also with coefficient $1/2$.  The resulting structure map is 
\[ m_2(\sigma_1,\sigma_1) = \hh (q^{A_+} + q^{A_-})
\sigma_0 . \]
In the case $A_+  = A_-$, \label{rep:aplus} this agrees with the 
formulas \eqref{pert1}, \eqref{pert2} obtained without averaging, and so induces the same product
on Floer cohomology.
\end{example}

\begin{proof}[Proof of Theorem \ref{thm:repthm}]
The statement of the Theorem follows by choosing \label{rep:choosing}
the perturbations so that they are permutation-invariant 
near the reduced configurations.  Let  $\ul{P}^{\red} = (P_\Gamma^{\red})$ be a reduced-regular
perturbation system for $\sigma_1,\ldots,\sigma_k$.
Inductively construct, as in the proof of Theorem \ref{thm:comeager}, a nearby  regular perturbation system $\ul{P}$  so that the perturbed matching conditions $\ul{M} = (M_\Gamma)$
are permutation-invariant in a neighborhood of each cell $\sigma_i$.
Suppose that the reduced configurations of type $\Gamma^{\red}$ meet the codimension one 
cells $\sigma_i$ in a finite set $Z_\Gamma$.
It follows from the implicit function theorem that for 
sufficiently small perturbations $P_\Gamma$ of $P_\Gamma^{\red}$ there is a bijection between 
$P_\Gamma^{\red}$-perturbed configurations $u_{\ul{d}}^{\red}$ with $\ell(e_0) =0 $
and $P_\Gamma$-perturbed configurations.  
After a generic perturbation $M_\Gamma^\circ$ of $M_\Gamma$, the points $z_1,\ldots, z_d$ are distinct. Since the number of
  permutations is finite, the set of single-valued conditions
  $g^* M_\Gamma^\circ$ that are regular for all permutations
  $g$ of the edges $T_{e_1},\ldots, T_{e_d}$ is comeager.  The average $M_\Gamma$ is then 
  regular.    Assuming the perturbations $M_\Gamma$ are
  invariant under permutations of the points on the constant disks, of
  the $d!$ possible orderings of the perturbations
  $M_{\Gamma,i}(\sigma)$ of $\sigma$ induced by the matching
  conditions $M_\Gamma$ exactly one ordering is achieved by a sequence
  of points $z_1,\ldots, z_d$ in cyclic order around the boundary of
  $S$.  It follows that the weight of any point in the fiber is
  $(d!)^{-1}$ times the weight of the image configuration.
    In the recursive construction of the perturbation system
  $\ul{P} = (P_\Gamma)$, at each stage we are
  given $P_\Gamma$ on the boundary of $\cU_\Gamma$
  and wish to extend it over the interior.    Because the space of perturbations is contractible, the inductive procedure may be carried out as before.
  \end{proof}

\begin{remark} \label{rem:antisym} 
{\rm (Divisor edges attached to constant disks)} In the moduli space of rigid treed disks
there may also be configurations with boundary edges labeled by $\sigma_i$
so that the adjacent disk  $S_v$ is constant.   Suppose that the perturbations vanish, so that all treed disks are regular without perturbation.   If the configuration two incoming edges $T_{e_1}, T_{e_2}$ then there is another configuration $u'$ obtained by interchanging the order of the inputs  $T_{e_1}, T_{e_2}$ around the boundary of $S_v$.  These two configurations contribute with opposite signs and cancel. 
\end{remark}

We will need a similar ``repeating input'' type formula for disks with
repeating alternating inputs at the self-intersection points in the
case of Lagrangians of dimension two.   \label{reptwo}
Suppose that $\dim(L_0) = 2$ and $x = (x_-,x_+), \ol{x} = (x_+, x_-)$
are ordered self-intersection points.   In this case, there are additional constant disks $u |S_v: S_v \to \{ \phi(x) \} \subset X$ of expected
  dimension zero with corners alternating
\[ \sigma_1= x, \sigma_2 = \ol{x}, \sigma_3 = x, \sigma_4 =
\ol{x},\ldots \in \cI^{\si}(\phi) \] 
where $\sigma_{\pm} \in \cI^c(\phi)$ is the top-dimensional cell containing $x_+$ resp. $x_-$. Let $\Gamma$ denote the corresponding combinatorial
type of domain, with $2d+1$ boundary leaves and no interior leaves.   The unperturbed relevant moduli space $\M_\bGamma(\phi)$
is not of expected dimension.  Indeed, any $2d+1$-marked treed disk
$(C,u: C \to X)$ with disk component mapping entirely to the self-intersection point $\phi(x)$ is holomorphic.  The moduli space of such maps $\M_{\bGamma}(\phi)$
is dimension $2d - 2$, not the expected dimension zero.  

\begin{definition} \label{rinvd} Let $\phi:L \to X$ be an immersed
  Lagrangian brane of dimension $\dim(L) = 2$ and
  $x \in \cI^{\si}(\phi)$ a self-intersection point flowing to a
   critical point $\sigma_{\pm,0}$ in the branch $\pm$.  The immersed
  Fukaya algebra $CF(\phi)$ is {\em rotation-invariant} at $x$ if and
  only if for any $d \ge 1$ we have
 \begin{eqnarray*}
 m_{2d}(x,\ol{x},x,\ldots, \ol{x}) & \in &  \frac{(-1)^{d-1}\sigma_{+,0}}{d}
+
  \val_q^{-1}(0,\infty)  \\  m_{2d}(\ol{x},x,\ldots, \ol{x},x) &\in &
  - \frac{(-1)^{d-1} \sigma_{-,0}}{d}+
  \val_q^{-1}(0,\infty) .\end{eqnarray*}

\end{definition}

\begin{remark} \label{rotexist} It seems that rotation-invariant perturbations
  exist, essentially because perturbations for zero-energy disks are
  the first step in the inductive procedure for constructing
  perturbations.  We take \ref{rinvd} as an assumption in the case $\dim(L) =2 $ and both $b_0(x)$ and $b_0(\ol{x})$ are non-vanishing.  
\end{remark}  

\section{Holomorphic disks and neck-stretching}
\label{sftsec}
 
In symplectic field theory, one studies the behavior of holomorphic
curves as the almost complex structure on the target changes in a
family corresponding to {\em neck-stretching}.  
Following Bourgeois-Eliashberg-Hofer-Wysocki-Zehnder
\cite{bo:com} and Venugopalan-Woodward \cite{tfuk} we describe 
the limit of the Fukaya algebra of a Lagrangian under neck-stretching.
The limit of a sequence of holomorphic disks with respect to 
such a neck-stretching is a {\em holomorphic building}
in the language of Bourgeois-Eliashberg-Hofer-Wysocki-Zehnder
\cite{bo:com}.   

The situation in this paper differs \label{rep:differs} from the situation in the above papers in several ways. First of all,  the Lagrangian is allowed to pass through the neck region, as it is in Fukaya-Oh-Ohta-Ono \cite[Chapter 10]{fooo}.  Second, since our Fukaya algebras are defined using treed disks, it may be that the tree part, rather than surface part, breaks in the neck-stretching limit.  Thus the levels of the buildings in this case
have not only strip-like and cylindrical ends going to infinity, 
but also additional breakings of the segments.  
Finally, 
we wish to degenerate our moduli spaces to products of treed disks in the pieces, rather than fiber products.  For this we take \label{rep:anan}
an additional limit which degenerates
the matching conditions at the separating hypersurface.

\subsection{Broken holomorphic disks}
\label{brokendisks}

Broken disks arise by the following neck-stretching limit studied by
Bourgeois-Eliashberg-Hofer-Wysocki-Zehnder \cite{bo:com} in the
context of symplectic field theory.  Recall that if $Z \subset X$ is a
coisotropic submanifold, then the {\em null foliation} of $Z$ is the
distribution defined by 
\[ \ker(\omega |_{TZ}) = \bigcup_{z \in Z} \{ \xi \in T_zZ | \omega(\xi, \zeta)  = 0 , \ \ \forall \zeta \in T_zZ
\} \subset TZ . \]
The null foliation $\ker(\omega |_{TZ})$ is called {\em fibrating} if
there exists a fiber bundle $p: Z \to Y$ so that the null foliation
is $\ker( \omega |_{TZ} ) = \ker(D p)$.  In this case, the bundle
$p:Z \to Y$ is unique up to isomorphism and called the {\em null
  fibration}. \label{nullfib}

\begin{definition} \label{def:nstretch} {\rm (Neck-stretching for almost complex
structures on symplectic manifolds)}  
Let  $Z \subset X$ be a
codimension one coisotropic submanifold admitting the structure of an
$S^1$-null-fibration $p : Z \to Y$ over a symplectic manifold $Y$. Thus 
\[ \ker(Dp) =  \ker(\omega |_{TZ}) \subset TZ \] 
is the vertical subspace.  Let
\[ \omega_Z = p^* \omega_Y \in \Omega^2(Z) \]  
denote the pullback of the symplectic form $\omega_Y$ to $Z$. 

The neck-stretched manifold is obtained by cutting along the hypersurface and inserting a cylinder.
Let $X^\circ$ denote the manifold with boundary obtained by cutting
open $X$ along $Z$.  Let $Z', Z''$ denote the resulting copies of $Z$.
For any $\tau > 0$ let
\begin{equation} \label{eq:Xtau} X^\tau = X^\circ \bigcup_{ Z'' = \{ -
    \tau \} \times Z, \{ \tau \} \times Z = Z' } \left( [-\tau,\tau] \times
  Z  \right) \end{equation}
be the manifold  obtained by gluing together the ends $Z', Z''$ of $X^\circ$ using a
neck $ [-\tau,\tau] \times Z$ of length $2\tau$.  

Define almost complex structures on the neck-stretched manifold as follows. 
  The $\R$ action by translation on $\R$ and $U(1)$ action on $Z$
  combine to a smooth $\CC^\times \cong \R \times U(1)$ action on
  $\R \times Z$ making $\R \times Z$ into a $\CC^\times$-bundle. 
  Consider the projections
\[p_\R: \R \times Z \to \R, \quad  p_Z: \R \times Z
\to Z, \quad p_Y: \R \times Z \to Y\]
onto factors $\R$ and $Z$ resp. onto $Y$. An
  almost complex structure $J$ on $\R \times Z$ is called {\em
    cylindrical}\footnote{These conditions are stronger than the 
    definition in 
Bourgeois-Eliashberg-Hofer-Wysocki-Zehnder  \cite{bo:com}, which deals with a more general situation.}  if $J$ is $\CC^\times$-invariant, preserves the \label{cylindrical}
  tangent spaces to the fibers of $p_Y: \R \times Z \to Y$ and $J$
  is equal to the standard almost complex structure on any fiber
\[ p^{-1}_Y(y) = \R \times Z_y \cong \R \times U(1) \cong \CC^\times .\]  
In particular, each orbit of $\CC^\times$ is holomorphic.  Any cylindrical almost complex structure $J$ on $\R \times Z$ induces
an almost complex structure $J_Y$ on $Y$ by projection by the formula
\[  D p_Y (J w) = J_Y D p_Y w , \quad w \in T(\R \times Z) .\] 
We assume that $J_Y$ is compatible with the symplectic form $\omega_Y$ on $Y$. \footnote{For many purposes, it suffices to assume that $J_Y$
  tames $\omega_Y$, see for example \cite{cm:com}. } 
There are complementary {\em vertical} resp. {\em horizontal} rank 
resp. corank two sub-bundles 
\bea\label{V} V &=& \ker (D p) \oplus \ker (D p_Z) 
  \subset T(\R \times Z) \\  H &= & TZ \cap J(TZ) \subset p_Z^* TZ \subset T(\R 
  \times Z) .
  \eea 
with the first bundle $V$ being a a trivial bundle over $\R \times Z$.
We have a splitting into complex vector bundles
\begin{equation} \label{splitT} T(\R \times Z) \cong H \oplus V  .\end{equation}
Since $Z$ is assumed to admit the structure of a principal $S^1$-bundle, there is a unique connection  \label{rep:unique} one-form compatible with the splitting
\[ \alpha \in \Omega^1(Z)^{S^1}, \quad  \ker(\alpha) = H \]   
(and in particular $Z$ is a stable hypersurface in the terminology of 
symplectic field theory.)  Conversely, given such a one-form, there is a unique almost complex structure $J$ given by $J_Y$ on $H$ and \label{rep:andon} the standard almost complex
structure on $V$.

The neck-stretched submanifolds of \eqref{eq:Xtau} are all diffeomorphic,
and the construction provides a family of almost complex structures on the original manifold.  The neck-stretched manifold $X^\tau$ is
diffeomorphic to $X$ by a family of diffeomorphisms given on the neck
region by a map
\begin{equation} \label{eq:stretch} (-\tau,\tau) \times Z \to 
  (-{\tau_0},{\tau_0}) \times Z \end{equation} 
equal to the identity on $Z$ and a translation in a neighborhood of
$ \{ \pm \tau \} \times Z$.  Given an almost complex structure $J$ on
$X$ that is of cylindrical form on $(-{\tau_0},{\tau_0}) \times Z$, we
obtain an almost complex structure $J^\tau$ on $X^\tau$ by using the
same cylindrical almost complex structure on the neck region.  Via the
diffeomorphism $X^\tau \to X$ described in \eqref{eq:stretch}, we obtain
an almost complex structure on $X$ also denoted $J^\tau$.
This ends the Definition.  \end{definition} 

Compactness results in symplectic field theory \cite{bo:com} describe the limit
of holomorphic curves as the length of the neck approaches infinity.  The complement of $Z$
in $X$ divides $X$ into regions $X_\subset$ and $X_\supset$, which we
consider as symplectic manifolds with cylindrical ends.  Similarly,
suppose that $\phi: L \to X$ is a possibly immersed Lagrangian
submanifold intersecting $Z$ transversally in a submanifold
$L_Z = Z \cap L$, so that in a neighborhood of $\{ 0 \} \times Z$ in
the tubular neighborhood $(-\eps , \eps) \times Z \to X$ $L$ is the
image of $(-\eps,\eps) \times L_Z \to X$.  We denote by
\[ L_\subset = \phi^{-1}(X_\subset), \quad L_\supset = \phi^{-1}(X_\supset) \]
the pieces of $L$ in $X_{\subset}$ and $X_{\supset}$.   In the neck-stretching limit, the symplectic field theory compactness results 
produce a configuration of holomorphic maps with Lagrangian
boundary conditions called a {\em building}.  

\begin{definition} 
  The {\em broken symplectic manifold} arising from the triple
  $ (X_\subset, X_\supset, Y) $ above is the topological space
\[\XX = X_\subset \cup Y \cup X_\supset \]
obtained by compactifying $X_\subset, X_\supset$ by adding a copy of
$Y$ and identifying the copies.  Thus $\XX$ is the singular space
obtained
by gluing the smooth manifolds 
\[ \ol{X}_\subset = X_\subset \cup Y, 
\quad \ol{X}_\supset = X_\supset \cup Y \] 
along $Y$.  The space $\XX$ inherits a natural topology by viewing
$\XX$ as the quotient of $X$ by the equivalence relation on $Z$ given
by the $S^1$-fibration.  Thus $\XX$ is a stratified space and the
link \label{fixlink} of a point in $Y$ in $\XX$ is a disjoint union of
two circles.  The space $\XX$ comes equipped with an isomorphism of
normal bundles
 \begin{equation} \label{niso} N =  \frac{(TX_\subset)_Y}{TY} \cong 
 \left(
   \frac{(TX_\supset)_Y}{TY} \right)^{-1} .\end{equation}
 The {\em infinite neck} is the product $\R \times Z$ and may be
 compactified by adding copies of $Y$ at $\pm \infty$.  For an integer
 $k \ge 1$, define the {\em $k$-broken symplectic manifold}
 \begin{equation} \label{XXk} \XX[k] = X_\subset \sqcup (\R \times Z) \sqcup \ldots \sqcup (\R
 \times Z) \sqcup X_\supset. \end{equation}
 The $k-1$ copies of $\R \times Z$ are called the {\em neck pieces}.
 Define
\begin{equation} \label{eq:pieces} \XX[k]_0 = X_\subset, \ \ \XX[k]_1 =
  \R \times Z,\ldots, \XX[k]_{k} = \R \times Z, \ \ 
\XX[k]_k = X_\supset .\end{equation}
For each piece we denote by $\ol{\XX[k]_i}$ the compactified space
obtained by adding one or two copies of $Y$ at infinity.  The complex
torus $(\CC^\times)^{k-1}$ acts $\XX[k]$ via the action of $\C^\times$
on each neck piece, 
\[ \CC^\times \times \P(N_\pm \oplus \ul{\C}) \to \P(N_\pm \oplus \ul{\C}), \quad
(z,[n,w]) \mapsto z [n,w] := [zn,w] .\]
Similarly, define the {\em broken Lagrangian}
\begin{equation} \label{eq:brokL}
\LL = L_\subset \cup L_Y \cup L_\supset \end{equation} 
where $L_Y = p(L_Z)$.  Let
\[ \LL[k] = L_\subset \sqcup (\R \times L_Z) \sqcup \ldots \sqcup (\R
\times L_Z) \sqcup L_\supset .\]
The group $(\R^\times)^{k-1}$ acts by real translations on the neck pieces. 
\end{definition}

\begin{definition} \label{def:building}  A {\em holomorphic
building}  with $k+1$ levels consists of 
\begin{enumerate}
\item a collection of surfaces $S_i, i = 0,\ldots, k$ with strip-like and cylindrical ends and
\item holomorphic maps 
\[ u_i: S_i \to \XX[k]_i \] 
called the {\em levels} of the building;
satisfying the given boundary conditions  $u_i(S_i) \subset \LL[k]_i$
\item a pairing of the outgoing ends of $S_i$ with the incoming ends
of $S_{i+1}$ so that the limits along the ends satisfying matching conditions:
For each such pair of ends,  there exists a {\em multiplicity} $\mu \in \R$ 
(possibly non-integer in the case of strip-like ends) and coordinates $(s,t)$ on the ends so that 
so that 
\begin{equation} \label{matchderiv}
\lim_{s \to \infty} \exp( -2\pi \mu(s + it)) u_{i} (s,t) =
\lim_{s \to - \infty} \exp( 2\pi \mu(s + it) ) u_{i+1} (s,t) .\end{equation}
In particular, the completions $\ol{u}_i$ to $\ol{u}_{i+1}$ have matching
value at $Y = \XX[k]_i \cap \XX[k]_{i+1}$ and the same multiplicity of intersection with $Y$. 
\end{enumerate}
Automorphisms \label{rep:automorphisms} of buildings $u: S \to \XX[k]$ are pairs
\[ \phi \in \Aut(S), \quad \psi \in \Aut(\XX[k]) \] 
consisting of translations
on the neck pieces so that 
\[ u \circ \phi = \psi \circ u .\]   
A building 
$u: S\to \XX$ is {\em stable} if it has finitely many automorphisms. 
\end{definition}

The limits of the levels in a building along the ends 
is a collection of Reeb chords and orbits.  

\begin{definition} \label{reebchord} {\rm (Reeb orbits and chords)} 
Loops in a fiber of constant speed 
  \[ \thorn: S^1 \to Z_y, \quad \alpha\left( \ddt \thorn(t)
  \right)  \ \text{constant} \]
  are called  {\em Reeb orbits}.  Paths  of constant speed beginning 
  and ending at the Lagrangian 
  \[ \thorn:[0, 1] \to Z_y, \quad  \alpha \left( \ddt \thorn(t) \right) \ \text{constant}
 , \quad \thorn(k) \in L_Z, \ k \in
\{ 0, 1\} \]
are called {\em Reeb chords}. This ends the
definition. \end{definition}

\begin{remark}  Given a Reeb orbit or chord $\thorn$ with 
$\alpha \left( \ddt \thorn(t) \right) = \mu $
the {\em trivial strip} or {\em trivial cylinder} corresponding to $\thorn$ is the map with domain $S = \R \times [0,1]$ resp. $S = \R \times S^1$
\[ S \mapsto \R \times Z, \quad (s,t) \mapsto  (s\mu, \thorn(t)). \] 
Equivalently, a building $u$ is stable if and only if each component 
\[ u_v: S_v \to \XX[k]_0 \cup \XX[k]_k \] 
in the components $\XX[k]_0,\XX[k]_k $  without translation automorphisms is stable  and each level 
\[ u_i:S_i \to \XX[k]_i, i = 1,\ldots, k-1 \]  
in the neck region is a union of stable maps and trivial cylinders or strips, and has at least one component that is not a trivial cylinder or strip; the latter condition 
prevents the existence of automorphisms arising from the translation action on
the neck pieces.  
\end{remark} 

We now turn to treed holomorphic buildings.   We assume that on the neck 
region of the Lagrangian, the Morse function is given by the radial coordinate:
For some constant $a> 0,$
\[ f(s,z) =  a + s .\] 
Thus the gradient vector field of $f$ is $\grad(f)(s,z) = \partial_s$, independent of $s,z$.

\begin{definition} \label{def:tbuilding}
A {\em treed building} is a treed disk $C$ equipped with a decomposition
structure $C = C_0 \cup \ldots \cup C_k$ so that any edge
$T_e$ connecting different levels has a single breaking separating it into components $T_{e,i} \subset C_i$ and $T_{e,i+1} \subset C_{i+1}$ for some $i$.
Each component $C_i$ will be called a {\em treed level}; see Figure 
\ref{puncttreeddisk}.    

\vskip .1in \noindent 
A {\em holomorphic treed building} 
is a treed building $C = S \cup T$ equipped with a holomorphic map 
$u: C \to \XX[k]$ for some $k$ so that the intersections $T \cap S$
occur away from the nodes joining  levels and 
the collection $u_i := u|S_i$ satisfies the conditions \eqref{matchderiv}
as well as the matching conditions for the edges $T_e$ connecting $T_{e,i}$
and $T_{e,i+1}$ at a point $w$, the limits match in the sense
that if $s_i$ resp. $s_{i+1}$ is a coordinate on $T_{e,i}$ resp. $ T_{e,i+1}$ then
\begin{equation} \label{trajlim} 
 \lim_{s_i \to \infty} \pi_Z u_i(s_i) =  \lim_{s_{i+1} \to -\infty} \pi_Z u_i(s_{i+1}) .\end{equation}

\vskip .1in \noindent 
{\em Stability} for treed buildings is defined in the same way as for buildings, with the following special case:  If a building 
$C$ consists of a single edge $T_e$ with no disks with an incoming label  $1^{\greyt}_\phi$ and an outgoing label $1^{\whitet}_\phi$ or $1^{\blackt}_\phi$ then we declare the building stable; similarly  a building consisting of a single edge with input labeled $\sigma_{e_1}$ and output labeled $\sigma_{e_0}$, where $\sigma_{e_0}$ appears in the boundary
of $\sigma_{e_1}$ is stable. 
\end{definition}

\begin{figure}[ht]
    \centering
    \includegraphics[height=2in]{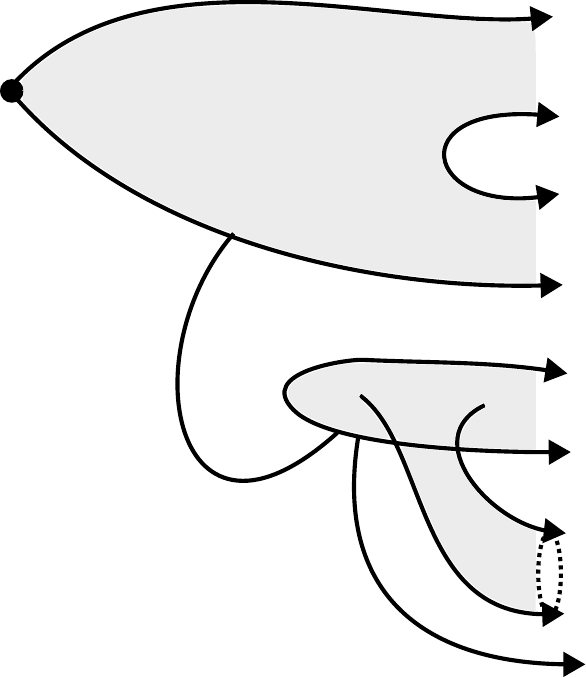}
    \caption{A treed level}
    \label{puncttreeddisk}
\end{figure}

We introduce the following notation for moduli spaces and evaluation maps. 
For any type of treed building $\bGamma$ denote by 
$\M_\bGamma(\XX,\phi)$ the moduli space of buildings of type $\bGamma$.
Similarly, if $\bGamma_i$ is the type of level of $\bGamma$ then 
$\M_{\bGamma_i}(\XX,\phi)$ denotes the moduli space of levels of type $\bGamma_i$.
For each treed level type $\bGamma_i$ evaluation at the semi-infinite edges defines a map 
\begin{equation} \label{evinfty} \ev: \M_{\bGamma_i}(\XX,\phi) \to 
  L_Z^{e_i(\white)}
  \times Y^{e_i(\black)}  \times L^{d_i(\white)}   \end{equation}
assigning to each map $u_i: S \to \XX$ the beginning points
$\thorn_e(0) \in Z$ of the limiting $e_i(\white)$ Reeb chords or $e_i(\black)$ orbits $\thorn_e$
at infinity along each strip-like or cylindrical end of $S$,
as well as the evaluations at the ends of the  $d_i(\white)$ semi-infinite edges.  For any subset 
\begin{equation} \label{sigma} \Sigma \subset
L_Z^{e_i(\white)} \times Y^{e_i(\black)} \times L^{d_i(\white)} 
  \end{equation} 
  denote by 
\begin{equation} \label{constrainedM}
\M_{\bGamma_i}(\phi,\Sigma) = \ev^{-1}(\Sigma) \end{equation}
the moduli space of maps with the given constraints. 
  
We will regularize these moduli spaces by passing to maps adapted
to a Donaldson hypersurface.   A {\em
    broken divisor} $\DD = (D_{\subset}, D_\supset)$ is a pair of
  divisors $D_{\subset} \subset X_{\subset}$ and
  $D_{\supset} \subset X_{\supset}$ with
\[ D_{\subset} \cap Y = D_Y = D_\supset \cap Y \] 
such that 
\[ \phi: L_Y \to Y, \quad \phi_\subset: L_\subset \to X_\subset, \quad
\phi_\supset: L_\supset \to X_\supset \]
are exact in the complement of $D_Y$ resp. $D_\subset$
resp. $D_\supset$.  Any broken divisor
$\DD = (D_{\subset}, D_\supset)$ gives rise to a family of divisors
$D$ such that $\phi: L \to X$ is exact in the complement of $D$,
as in \cite[Section 8.3]{flips}.   As in \cite{flips}, one may first choose a Donaldson
hypersurface $D_Y $ \label{DY} for $L_Y$ disjoint from the Lagrangian
$L_Y \subset Y$.  One may then extend to Donaldson hypersurfaces
$D_\subset \subset \ol{X}_\subset$ and $ D_\supset \subset \ol{X}_\supset$, by
choosing extensions of the asymptotically holomorphic sequence of
sections. The definition of adapted buildings is then similar
to that of adapted maps: Each component of $u^{-1}(\DD)$ is required
to contain an interior edge, and each such edge is required to map to $\DD$.

\subsection{Fredholm theory and exponential decay}

In this section, we collect some technical results 
on holomorphic maps asymptotic to Reeb orbits or chords.  In order
to carry out the necessary classification of levels in the local model, we allow Lagrangian boundary conditions that are asymptotically cylindrical rather than cylindrical in a neighborhood of infinity.  

\begin{definition} \label{def:acyl} 
\begin{enumerate}
\item An almost complex manifold $X$ has a {\em cylindrical
  end modelled on $Z$} if there exists an embedding
\[ \kappa_X: \R_{> 0} \times Z \to X \]
such that the image of $\kappa_X$ has compact complement. 
A {\em cylindrical end almost complex structure} is an almost complex
structure $J: TX \to TX$ for which the pull-back
$J |_{\R_{> 0} \times Z}$ to $\R_{> 0} \times Z$ is of cylindrical
form in the sense of  Definition \ref{def:nstretch},  that is, the restriction of a cylindrical almost complex structure $J_{\R \times Z}$ on $\R \times Z$. 
\item 
Let $\phi: L \to X$ be a Lagrangian immersion.  
     Call $\phi$ {\em cylindrical near infinity}
    if there exists a smooth manifold $L_Z$ of $Z$ and 
    $s \in \R$ so that $L_Z$ is cylindrical in $(s,\infty) \times Z$
    on the end: That is,
    \[ (\kappa_X)^{-1}(\phi(L)) \cap ((s,\infty) \times Z) = (s,\infty) \times L_Z .\]
\end{enumerate}
\end{definition}

Since we are considering only the circle-fibered case, our cylindrical 
end manifolds have natural compactifications at infinity.  Given a manifold $X$ with cylindrical almost complex
  structure $J$ as above, the {\em compactification} of $X$ is the
  almost complex manifold $\ol{X} = X \cup Y$ obtained by gluing in a
  copy of $Y$ at infinity.  In terms of charts, we have
\begin{equation} \label{olx}
 \ol{X} = X \cup_{\R_{> 0} \times Z} (  Z \times_{\CC^\times} \C 
 ) \end{equation} 
where $Z \times_{\CC^\times} \C$ is the line bundle associated to $Z$.
The inclusion of $\R_{ > 0} \times Z$ in 
$Z \times_{\CC^\times} \C $ is given by the isomorphism 
\[ \R_{ > 0} \times Z \cong Z \times_{S^1} \CC^\times .\]

\begin{proposition} \label{prop:cleanly} Suppose that $\phi: L \to X$ is cylindrical-near-infinity.  Then the closure $\ol{\phi(L)} \subset \ol{X}$ is contained in the image of a Lagrangian immersion $\ti{\phi}: \ti{L} \to \ol{X}$ with (without boundary, but possibly non-compact) clean self-intersection.   
\end{proposition} 

\begin{proof} 
The subset $\R_{> 0} \times Z$ glues into
the chart $Z \times_{\CC^\times} \C$ near infinity by the map
$(s,z) \mapsto [z, e^{-s}]$.  Let $\ol{L}$ denote the
union 
\[ \ol{L} = \phi(L) \cup (\R_{> 0} \times L_Z) \cup L_Y \] 
in $\ol{X}$.   In a neighborhood of $Y$ the closure $  \ol{L}$ is contained in the cleanly-self-intersecting submanifold $\ti{L}$ given as the image of 
$(\R_{> 0} \times (-L_Z \cup L_Z)) \cup L_Y $.
\end{proof} 

\begin{definition} \label{def:asym} 
Let $\ol{X} = X \cup Y $ be as above equipped with a 
symplectic structure.   A Lagrangian submanifold $L \subset X$ is {\em asymptotically cylindrical} to a cylindrical-near-infinity Lagrangian $L_0$ if the closure $\ol{L}$
is an immersed submanifold-with-boundary in $\ol{X}$ tangent to the closure
of $\ol{L}_0$ at $L_Y$.
\end{definition} 

\begin{example}   The unflattened handle $H_\gamma \subset \C^n$ is asymptotically cylindrical, by Lemma \ref{lem:unflat} below. 
\end{example}

\begin{proposition} \label{prop:cleanly2} Let $L \subset X$ be an asymptotically cylindrical 
Lagrangian manifold asymptotic to a cylindrical-near-infinity Lagrangian
submanifold $L_0$. The  closure $\ol{L}$ is contained in a (possibly non-compact)
cleanly-self-intersecting Lagrangian submanifold of $\ol{X}$.
\end{proposition}

\begin{proof} The closure of $L$ is $L_Y$ and $\ol{L}$ is tangent to $\ol{L}_0$, 
which is contained in a cleanly-self-intersecting Lagrangian by Proposition \ref{prop:cleanly}.
We may write $\ol{L}$ near $L_Y$ as the graph of an exact one-form 
$\dd f$ on $\ol{L}_0$ where $f: \ol{L}_0 \to \R $ is smooth, using Weinstein neighborhoods
of each branch of $L'$ in $\ol{X}$ to write nearby Lagrangians as graphs of 
one-forms.  By, for example, the Seeley extension theorem 
\cite{seeley} $f$ extends to a function $f'$ on $L'$. 
After possibly shrinking $L'$, there are no self-intersection points of $\graph(\dd f')$
other than $L_Y$, that is, the extensions of the branches do not intersect,
and $\graph(\dd f')$ provides the desired extension.  
\end{proof}

Let $S$ be a
holomorphic curve with cylindrical and
strip-like ends
\bea \kappa_{e,\black}: \R \times S^1 \to S & e = 1,\ldots,
e(\black) \\
\kappa_{e,\white}: \R \times [0,1] \to S & e = 1,\ldots, e(\white).
\eea

\begin{definition} \label{def:holasym} {\rm (Holomorphic maps asymptotic 
    to Reeb chords)}
Given a cylindrical or asymptotically cylindrical Lagrangian $\phi: L \to X$, a
{\em map from a surface $S$ with strip and cylindrical ends  to $X$ with boundary in $\phi$} is a map 
\[ u: S \to X, \quad  u(\partial S) \subset \phi(L) .\]  
A map $u: S \to X$ is {\em asymptotic to a Reeb chord} $\thorn$ on an
end of $S$ if there exist $s_0 \in \R$ and a {\em multiplicity}
$\mu \in \R_+$ such that in cylindrical coordinates $(s,t)$ on each end
the distance in the cylindrical metric $\dd_{\cyl}$ on $\R \times Z $, 
with coordinates given by $ \kappa_{e,\black}$ or $ \kappa_{e,\white}$
\begin{equation} \label{expfast}
 \dd_{\cyl} ( u(s,t) , 
(s_0 + \mu s,
  \thorn(t))) < C e^{- \theta s} \end{equation}
for some constant $\theta > 0$ and $s_0 \in \R$.  The definition of an
end asymptotic to a Reeb orbit is similar.  This ends the
Definition. \end{definition}

 The exponential decay above is closely related to a finite-energy condition. 
 Our case is a special case of a more general definition
for stable Hamiltonian structures in \cite{bo:com}.  For simplicity
consider holomorphic maps to $U = \R \times Z$, where $Z$ is equipped
with closed two-form $\omega_Z \in \Omega^2(Z)$ with fibrating
null-foliation $\ker(\omega_Z) \subset TZ$ and connection form
$\alpha \in \Omega^1(Z)$.  Let $J: TU \to TU$ be a cylindrical almost
complex structure.  The {\em horizontal energy} of a holomorphic map
\[ u = (\psi,v): (S,j) \to (\R \times Z,J) \]  
is (\cite[5.3]{bo:com}) with $S^\circ \subset S$ denoting the complement of the corners
(points where $u | \partial S$ has a branch change) 
\[ E^h(u) = \int_{S^\circ} v^* \omega_Z .\]
The {\em vertical energy} is (\cite[5.3]{bo:com})
\begin{equation} \label{alphaen} E^v(u) = \sup_{\zeta} \int_{S^\circ} (\zeta
  \circ \psi) \dd \psi \wedge v^* \alpha \end{equation}
where the supremum is taken over the set of all non-negative
$C^\infty$ functions
\[\zeta: \R \to \R, \quad  \int_\R \zeta(s) \dd s = 1 \]
with compact support.  The {\em Hofer energy} (\cite[5.3]{bo:com})
is the sum
\[ E(u) = E^h(u) + E^v(u) .\]
Let $X^\circ$ be a symplectic manifold with cylindrical end modelled
on $\R_{> 0} \times Z$. The vertical energy $E^v(u)$ on the end is
defined as before in \eqref{alphaen}.  The Hofer energy $E(u)$ of a
map $u: S^\circ \to X^\circ$ from a surface $S^\circ$ with cylindrical
ends to $X^\circ$ is defined by dividing $X^\circ$ into a compact
piece $X^{\on{com}}$ and a cylindrical end $\R_{> 0} \times Z$, and
defining
\[ E(u) = E(u |_{X^{\on{com}}} ) + E(u |_{\R_{\ge 0} \times Z}) \]
where $E(u)$ is the Hamiltonian-perturbed energy from \eqref{EH}.

\begin{lemma} \label{lem:limlem} 
Let $\phi: L \to X$ be an cylindrical-near-infinity Lagrangian immersion.  Any $J$-holomorphic map $u: S \to X$ with boundary on $\phi(L)$  and finite Hofer energy extends to a $\ol{J}$-holomorphic map  $\ol{u}:\ol{S} \to \ol{X}$, and the extension defines a bijection between maps to $\ol{X}$ and maps to $X$.
\end{lemma}

\begin{proof} Exponential convergence on strips with finite Hofer
  energy is proved in Cieliebak-Ekholm-Latschev \cite[Proposition
  3.2]{ciel:switch}.  The exponential convergence implies that $u$ 
  is finite area.  Removal of singularities for holomorphic maps with
  boundary on immersed Lagrangians with clean self-intersection
  Schm\"ashke \cite{clean} implies that the map $u$ extends to a map
  $\ol{u}: \ol{S} \to \ol{X}$ with boundary on $\ol{L}$.  Conversely, any   map $\ol{u}: \ol{S} \to \ol{X}$ restricts to a map from $S$ to $X$   by removing the points mapping to $\ol{X} - X$.  The finite 
  Hofer energy condition $E(u) < \infty$ follows from the fact that
  by the constant rank embedding theorem, 
  the symplectic form in a neighborhood of $Y$ in $\ol{X} - X$
  is \label{rep:isdiffeo} diffeomorphic to a neighborhood of the zero section in the normal 
  bundle $N_Y$ may be written  $\dd (\zeta v^* \alpha)  +
  \pi_Y^* \omega_Y  $ where $\zeta$ is the norm-square
  function.  This form is cohomologous to that one for which $\zeta$
  has compact support, and the Stokes' formula computation in 
  \cite[Section 5.7]{bo:com} implies that bounded Hofer
  energy is equivalent to bounded area.  See \cite[Remark 5.9]{bo:com}.
 \end{proof} 

The condition that a holomorphic map has finite Hofer energy
implies asymptotic convergence to Reeb chords at infinity for an
exponential decay constant that is related to the minimum angle of
intersection between the Lagrangians.

\begin{lemma} {\rm (Removal of singularities for cylindrical
    maps)} \label{lem:expfastlem} Let $\phi:L \to X$ be an asymptotically
  cylindrical Lagrangian immersion.  For any finite energy
  $J$-holomorphic map $u: S \to X$ either
\begin{enumerate} 
\item there exist $x \in X$ such that $u(s,t)$ converges to $x$ as
  $s \to \infty$, uniformly in $t$ for cylindrical coordinates $(s,t)$
  along the end $e$ (so that $u$ has a removable singularity) or 
\item there exists a Reeb chord resp. orbit $\thorn_e$ such that
  $u(s,t)$ converges exponentially fast to $\thorn_e(s)$ as
  $s \to \infty$, for $s \to \infty$  with constant $\theta$ in the
  sense of \eqref{expfast} depending only on $\thorn_e$.
\end{enumerate}
\end{lemma} 

\begin{proof} The
  desired convergence for strip-like ends $\kappa_{e,\circ}$ is a consequence of
  Schm\"ashke \cite[Theorem 3.2]{clean}:   There
  exist positive constants $\theta', c_0,c_1,c_2,\ldots $ and an
  eigenfunction
  \[ v: [0,1] \to T_x \ol{X}, \quad \partial_t v = \theta v , \quad
  v(0) \in T_x \ol{L}_{i_-}, \ v(1) \in T_x \ol{L}_{i_+} \]
  with eigenvalue $\theta$ so that for every integer $k \ge 0$
 \begin{equation} \label{leading}
 u(s,t) = 
\exp_{\ol{x}} \left( \frac{-1}{\theta} e^{-\theta s}
    v(t) + w(s,t) \right) 
 , \quad \Vert w \Vert_{C^k([s,\infty] \times 
    [0,1])} \leq c_k e^{- (\theta + \theta' ) s } .\end{equation} 
  The eigenfunctions $v$ of $\partial_t$ on the vertical parts of
  $T_x \ol{L}_{k_-}, T_x \ol{L}_{k_+}$ correspond to Reeb chords (cf.
  Robbin-Salamon \cite[Appendix E]{rs:asym}) $\thorn$, and in
  cylindrical coordinates on $X$ the exponential of
  $e^{-\theta s} v(t)$ is equal to $ (\theta s, \thorn(t)) $.  The
 second estimate in \eqref{leading} implies the desired exponential  convergence.
\end{proof}

We develop Fredholm theory for holomorphic treed maps to cylindrical end manifolds.  
Given a holomorphic map $u: S \to X$ with finite Hofer energy,
denote by $\Gamma$ the type of the domain $S$ and  $\M_\Gamma(X,\phi)$ the space of maps $u: S \to X$ with domain type $\Gamma$.

\begin{proposition} \label{prop:loccut} For any domain type $\Gamma$, the
  space $\M_\Gamma(X,\phi)$ of finite-energy holomorphic maps
  $(C,u: C \to X)$ with domain type $\Gamma$ is locally cut out by a
  Fredholm map of Banach spaces.  \end{proposition}

Before beginning the proof, we remark that there are two possible approaches to the Fredholm theory. 
  By Lemma \ref{lem:limlem}, the moduli space of finite-energy maps
  $\M_\Gamma(X,\phi)$ is in bijection with the space of maps
  $\ol{u} : \ol{S} \to \ol{X}$ to the compactification bounding
  $\ol{\phi}(\ol{L})$. The statement of the Proposition follows from the Fredholm theory  for holomorphic maps with boundary on a clean intersection  Lagrangian \cite{clean}.    The second version of Fredholm theory treats   the target as a cylindrical end manifold, and is required to prepare
  for the needed gluing result later in Section \ref{gluingsec}.  
  We carry out the second approach. \label{report:second}
  
  \begin{proof} We suppose
  for simplicity that the limits along the strip-like or cylindrical ends are Reeb chords or orbits, so that there are no self-intersection points.  Since
  the intersection $L_k \cap Z_y$ with each branch $L_k$ of $L$ at infinity 
  with each fiber $Z_y, y \in Y$ is by assumption finite, we may
  assume that the boundary of $u$ on the strip-like ends maps to
  branches $L_{k_-},L_{k_+}$ of the Lagrangian on the boundary at
  infinity.  The two branches differ by
\begin{equation} \label{thetaangle} L_{k_+} \cap Z_y \cong 
e^{i \theta} (L_{k_-} \cap Z_y) \end{equation}
for some angle $\theta \in [0,2\pi)$.  Choose a {\em Sobolev decay constant}
$\lambda \in (0,2\pi)$ smaller than the angles $\theta$, if
$\theta \neq 0$.  Let $\beta$ be a {\em good cutoff function} 
\begin{equation} \label{beta} 
\beta \in C^\infty(\R, [0,1]), \quad 
\begin{cases} 
  \beta(s) = 0 & s \leq 0 \\ \beta(s) = 1 & s \ge
  1 \end{cases}. \end{equation}
Define a {\em Sobolev weight function}
\begin{equation} \label{weightfunction} \aleph_\lambda: {S^\circ} \to
  [0,\infty), \quad (s,t) \mapsto \beta(s) p \lambda s \end{equation} 
where $ \beta(s) p \lambda $ is by definition zero on the complement
of the cylindrical ends.  Let 
\[ \Omega^0({S^\circ}, u^*TX)_{k,p,\lambda}' = \Set{ \xi \in 
\Omega^0({S^\circ}, u^*TX)_{W^{k,p}_{\loc}}  |  \Vert \xi \Vert^p_{k,p,\lambda} < \infty } \]  
denote the {\em weighted Sobolev space} of exponent $p$, differentiability class $k$, and decay constant
$\lambda$.  By definition this space consists of sections with finite norm \label{rep:finitenorm}
$\xi: {S^\circ} \to u^* TX$ with limits
\[ \lim_{s \to \infty} \xi \circ \kappa_{e,\white}(s,\cdot) =: \xi(e) \in \Omega^0([0,1], \R \times \thorn_e^* TZ) \] 
at infinity defined by 
\begin{multline} \label{1pnorm} 
  \Vert \xi \Vert^p_{k,p,\lambda} := \sum_e \Vert (\xi(e)) \Vert^p +
 \int_{{S^\circ}} \left( \sum_{k  > 0} \Vert \nabla^k \xi \Vert^p
 \right. \\ \left. +  \Vert \xi -
  \sum_e \beta(|s| - | \ln(\delta)|/2) \cT^u ( \xi(e))  \Vert^p \right) \exp(
  \aleph_\lambda) \dd \Vol_{{S^\circ}}
\end{multline}
where $\cT^u$ is parallel transport from $\thorn_e(t)$ to $u(s,t)$
along $u(s',t)$.  By definition, these Sobolev spaces have evaluation-at-infinity maps 
\begin{equation} \label{evinftyt}  
\ev_\infty: \Omega^0({S^\circ},
  u^*TX)'_{k,p,\lambda} \to \bigoplus_{e \in \cE({S^\circ})} T(\R \times Z),
  \quad \xi \mapsto (\xi(e) )_{e \in \cE({S^\circ})} .\end{equation}
By a {\em model map} we mean a map $u_0: {S^\circ} \to X$ which has the form 
$(s,t) \mapsto \exp( \mu(s + it)) z $ for some $z \in Z$ and constant $\mu \in \R$;
that is, mapping to a single fiber of $\R \times Z$ where it is identified with a one-parameter subgroup.   Let
\[ \Map({S^\circ},X)_{k,p,\lambda} = \{ \exp_{u_0}(\xi), \quad \xi \in
\Omega^0({S^\circ}, u_0^*TX)_{k,p,\lambda} \} \]
denote the space of maps $u: {S^\circ} \to X$ equal to $\exp_{u_0}(\xi)$ for
some model map $u_0: {S^\circ} \to X$ by
an element of the weighted Sobolev space
$\xi \in \Omega^0({S^\circ}, u_0^*TX)_{k,p,\lambda}$.   Let
\[ \Omega^{0,1}\left({S^\circ},u^* TX\right)_{k-1,p,\lambda} = \{ \eta \in 
\Omega^{0,1}\left({S^\circ},u^* TX\right)_{0,p,\lambda}, \quad 
\Vert \eta \Vert_{k-1,p,\lambda} < \infty \} \] 
 denote the space of
$\left(0,1\right)$-forms with finite $\left(k-1,p,\lambda\right)$ norm, given in the case $k-1 = 0$ by 
\[ \Vert \eta \Vert_{0,p,\lambda} = \left( \int_{S^\circ} \Vert \eta
  \Vert^p \exp\left( \aleph_\lambda \right) \dd \Vol_{S^\circ}
\right)^{1/p} .\]
Using the local trivializations
\eqref{localtriv} define a Banach manifold resp. Banach vector bundle 
\begin{eqnarray*} 
\B_\Gamma &=& \M_\Gamma^i \times \Map\left({S^\circ},X,L\right)_{k,p,\lambda} \\ 
\cE_\Gamma &=& \cup_{u \in \B_\Gamma} \E^i_{\Gamma,u}, \ \cE^i_{\Gamma,u} = \Omega^{0,1}\left({S^\circ},
u^* TX\right)_{k-1,p,\lambda} .\end{eqnarray*}
Note that 
\begin{equation}  T_{C,u} \B_\Gamma = 
T\M_\Gamma \oplus \Omega^0\left( {S^\circ}, u^* TX, \left(\partial
  u\right)^* TL\right)'_{k,p,\lambda}.\end{equation} 
  As usual we obtain a Cauchy-Riemann operator
\begin{equation}  \label{tcutout}
 \cF^i_\Gamma: \B^i_\Gamma \to \E^i_\Gamma, \quad u
\mapsto \olp_{J,H} u \end{equation} 
whose zeros cut out the space of holomorphic maps from ${S^\circ}$ to
$X$ locally.   For cylindrical ends, the linearized operator $\ti{D}_u$ is Fredholm by standard results
on elliptic operators on cylindrical end manifolds in Lockart-McOwen
\cite{lm}, and in the case with Lagrangian boundary condition, results
described in Schm\"ashke \cite[Section 5]{clean}; 
note that these
results require that the almost complex structure is
compatible.  
\end{proof}

We compare the linearized operators for the map to the cylindrical-end
manifold and its compactification as follows.  Given a holomorphic,
finite energy  map $u: S \to X$, let $\ol{u}: \ol{S} \to \ol{X}$ denote
its extension to the compactification described in Lemma \ref{lem:limlem}
above.  The pull-back-bundle $\ol{u}^* T \ol{X}$ gives
an extension of $u^* TX$ over the cylindrical and strip-like ends.  However, we wish 
to choose an extension so that there is an isomorphism of kernels and cokernels of the linearized operator and its compactification. 
For this, we introduce a twisting of the bundle, similar to the notion of twisting by a divisor in algebraic geometry.   On the cylindrical end, we have a splitting
\begin{equation} \label{vhsplit}
\kappa_X^* TX  \cong 
T^v X \oplus T^h X \end{equation}
where 
\[ T^v X := \ker\left(D p_Y \right) , 
\quad T^h X :=  p_Y^* TY   \] 
into the vertical and horizontal parts. We assume that the almost complex structure respects the splitting, as in, for example, the standard complex structure on $\C^n - \{ 0 \}$.
The restriction of $u^* TX$ to $u^{-1}\left( \kappa_X\left(\R_{> 0} \times Z\right)\right)$
splits into vertical and horizontal parts as well:
\[  u^* TX | \left( \kappa_X\left(\R_{> 0} \times Z\right)\right) = 
(u^* TX)^v \oplus (u^* TX)^h . \] 

\begin{definition}
\begin{enumerate}
\item  For the horizontal bundle  define an extension
\[ (u^* TX)^v_c := (u_Y^* TY)_c := \ol{u}_Y^* TY  \to \ol{S}. \] 
The Cauchy-Riemann operator $D_{u_Y}$
extends to $\ol{u}_Y^* TY$ as the operator
$D_{\ol{u}_Y}$.  
\item For the vertical bundle define an extension by triviality:
The bundle $  (u^* TX)^v $ 
is spanned in each fiber by the vector fields $\partial_s, \partial_t$, where $s+ it$
is the coordinate on each fiber of $\R \times Z_y \cong \R \times S^1$. So 
\[ (u^* TX)^v \cong 
u^{-1}\left( \kappa_X\left(\R_{> 0} \times Z\right)\right) \times \R^2  \] 
is the trivial bundle.  
The Cauchy-Riemann operator by assumption preserves the vertical part and is trivial
the frame given by $\partial_s, \partial_t$,
and so extends to the trivial operator over the compactification 
\[ (u^* TX)^v_c := \left( u^{-1}\left( \kappa_X\left(\R_{> 0} \times Z\right)\right) \cup \{ z_e, e \in \mE(S) \} \right)  \times \R^2  \to \ol{S} .\]
\item 
Gluing the extensions of the horizontal and vertical  bundles together with the 
bundle $u^* TX$ (via the canonical identification
over $u^{-1} \kappa_X(\R_{> 0} \times Z) $)
defines an extension  
\[ (u^* TX)_c = \left( (u^* TX) \cup ( u^* TX)^v_c \oplus (u^ TX)^h_c) \right) \to \ol{S} . \] 
\item Define an extension of the boundary condition $(\partial u)^* TL := (u | \partial S)^* TL$
as follows:  On the horizontal component
the condition $TL_Y$ over $\partial S \cap \kappa_X( \R_{> 0} \times Z)$ extends as
$\ol{u}_Y |_{\partial \ol{S}} TY$, and the vertical part $T\R \cong \R \times \R$
extends as the trivial bundle.  
\item Denote by the Cauchy-Riemann operator
on the compactified bundle $(u^* TX)_c$ with boundary $( (\partial u)^* TL)_c$ 
\[ (D_u)_c: \Omega^0(\ol{S}, (u^* TX)_c; (\partial u^* TL)_c) \to 
\Omega^{0,1}(\ol{S}, (u^* TX)_c; (\partial u^* TL)_c)  \] 
\end{enumerate}
\end{definition}

\begin{remark}
The extended bundle is {\em not} generally isomorphic to $ \ol{u}^* T \ol{X}$;  For example, in the case
$X$ is the cylinder $\R \times S^1$ and $u$
the identity map, $\ol{u}^* T \ol{X}$ has Chern number $2$ while $ (u^* TX)_c$ is the trivial bundle.  
\end{remark}

\begin{proposition} \label{prop:linremov}
Let $u: S \to X$ be a finite energy map bounding $L$, and $\ol{u}: S \to \ol{X}$ 
its extension bounding
  $\ol{\phi(L)}$.  Suppose further that the almost complex structure preserves the splitting into horizontal and vertical parts on the cylindrical ends.
  Then the Cauchy-Riemann operator $D_u$ extends to an operator 
  $D_{\ol{u}}$ on $ (u^* TX)_c$ and restriction defines an isomorphism of kernels and
  cokernels
  \[ \ker(\ti{D}_{\ol{u}}) \cong \ker(\ti{D}_u), \quad
  \coker(\ti{D}_{\ol{u}})  \cong \coker(\ti{D}_u). \]
In particular these operators have the same index.
\end{proposition} 

\begin{proof}  
Restriction defines an isomorphism of kernels. 
Any section of $(u^* TX)_c$ bounding $(u^* TL)_c$ and in the kernel of $D_{\ol{u}}$ restricts to a section of $(u^* TX)$.
By Taylor's theorem, smooth sections $\xi$ of $ (u^* TX)_c$ have a power series
expansion in local coordinate on each end
and so in cylindrical or strip-like coordinates on each end for each $k \ge 0$
there exists   a real constant $c_k$ so that 
\begin{equation} \label{asym} \xi (s,t) =  v_0 + e^{- \pi s} v(t)
  + w(s,t), \quad \Vert w \Vert_{C^k([s,\infty] \times
  [0,1])} \leq c_k e^{- \pi (1 + \theta) s } \end{equation} 
where $v_0, v(t)$ are the leading 
eigenvector the tangential part of $D_u$. 
Thus $\xi$ has exponential convergence to a constant and lies in the kernel of $D_u.$  On the other hand, 
cf. Schm\"ashke \cite[Appendix
B]{clean},  elements of the kernel of $D_u$ have the same exponential convergence.  Thus any $\xi \in \ker(D_u)$
extends to a continuous section $\ol{\xi}$ which is a weak solution 
to $ (D_u)_c \ol{\xi} = 0$.   The section $\ol{\xi}$ is then a strong solution by elliptic regularity, as in Harvey-Polking
\cite{hp:remov}.  The reader may compare with similar results in Abouzaid \cite[(4.19)]{ab:ex}, or in the cylindrical end case in Ekholm \cite[Lemma 6.4]{ekholm:morse}.  

The identification of 
cokernels follows from a similar statement for the kernels of the adjoints:
First note that the cokernel of $D_u$ is identified with  the subset of elements of the cokernel of the operator  
$D_u$ acting on the space $\Omega^0( u^* TX, (\partial u)^* TL)_{k,p,\lambda}$
that are also orthogonal to the sections were thrown in by hand, so to speak,
as in \eqref{1pnorm}.  As in Lockart-McOwen \cite{lm}, one may assume, after treat $D_u$ as a perturbation of a  translation-invariant operator $(D_u)'$ on the ends, by viewing  a solution to $(D_u)^* \eta = 0$ as a solution to $(D_u') \eta = (D_u)'\eta - (D_u) \eta$ with exponentially-decaying inhomogeneous term. 
In the translationally-invariant case, a one-form $\eta$  in the kernel of the adjoint has 
an eigenfunction expansion for the self-adjoint operator $\partial_t $ acting on functions on $[0,1]$ with real boundary condition on the ends:
\[ \eta(s,t) = \sum_{\nu \ge 0}  c_\nu e^{- \nu \pi (s + it)} \left( \dd s - i \dd t \right) \]
for some real constants $c_\nu \in \R^n$, after taking a bundle trivialization.   Orthogonality to the image of functions constant at infinity implies that the coefficient $c_0$ vanishes: If $\rho(s)$ is a real-valued function equal to $0$ for $s$ small and $1$ for $s$ large then 
\[    \int_{\R_{< 0} \times Z} ((\olp \rho(s))  , \eta(s,t)) \dd s \dd t  = 
  \int_{\R_{< 0} \times Z}  \left( \dds \rho(s) c_0 \right) \dd s  = c_0  \] 
  which must vanish.  It follows that  $|\eta(s,t)|$ decays at least as fast as $e^{-\pi s}$.  On the other hand, the kernel of $(D_u)_c^*$ consists of one-forms $\eta$ which by Taylor's theorem satisfy the decay estimate $| \eta(s,t) | < c e^{- \pi s}$. Indeed, if $z$ is a coordinate near the point $z_e \in \ol{S}$, the form $\dd \ol{z} $ is equal to 
\[ \dd \ol{z} = \ol{z} (\dd s -  i \dd t) = \pi e^{ \pi(-s + it)} (\dd s - i \dd t). \] 
Thus, any such section $\eta$ extends to a weak solution  of the adjoint equation $(D_u)_c^* \eta = 0$,  by removal of singularities again.  Thus restriction defines an isomorphism of cokernels as well. 
\end{proof}

\subsection{Compactness for buildings}

The relative form of the compactness theorem in 
Bourgeois-Eliashberg-Hofer-Wysocki-Zehnder \cite[Section
11.3]{bo:com} and Abbas \cite{abbas:com} in symplectic field theory describes
the limits of subsequence of holomorphic maps with Lagrangian boundary
conditions and Morse-Bott non-degeneracy conditions as in \cite[Remark 5.9]{bo:com}.  Compactness in symplectic field theory is also treated
  in related situations by Cieliebak-Mohnke \cite{cm:com} without
  Lagrangian boundary and Venugopalan-Woodward \cite{tfuk} in the case
  that the Lagrangian is disjoint from the stretching hypersurface.
Chanda \cite{chanda} gives further details in the Lagrangian case.  

We will need an extension of these results to the case of treed holomorphic curves. 
The first result, Theorem \ref{thm:stretchcompact} below, describes a compactness theorem in 
a neck-stretching limit.  The second, Theorem \ref{thm:varyL}, describes compactness
for buildings under variation of the Lagrangian boundary condition. 

\begin{theorem} \label{thm:stretchcompact} Given a sequence of adapted stable holomorphic treed disks
  $(C_\nu, u_\nu: C_\nu \to X^{\tau_\nu}), \tau_\nu \to \infty$ with Lagrangian
  boundary conditions in $\phi$ and bounded energy, there exists a
  subsequence of $u_\nu$ converging to an adapted stable holomorphic treed building
  $(C,u : C \to \XX)$ with boundary  mapping to the
  broken Lagrangian $\LL$.  Furthermore, the limit of any Gromov convergent sequence is unique.
  \end{theorem}

 \begin{proof}[Sketch of proof] 
 Venugopalan-Woodward \cite{tfuk} prove sft compactness
 for the case of Lagrangians not meeting the neck, and Chanda \cite{chanda} extends this to the case
  of Lagrangians passing through the neck.  For completeness, we sketch the modifications
  of the argument in \cite{tfuk} necessary to handle the case in hand. 
  We assume that we have chosen a broken divisor $\DD = (D_\subset, D_\supset)$, and  a
  family $D_{\tau_\nu}$ of Donaldson hypersurfaces in $X^{\tau_\nu}$
  limiting to $\DD$ in the sense that $D_{\tau_\nu}$ is the pull-back of $D_Y$
  on the neck region, as in \cite[Lemma 5.15]{tfuk}.   We suppose furthermore that we have a collection of 
  almost complex structures $\ul{J}(\XX) = (J_\Gamma(\XX))$ for $\XX$ for which every 
  adapted holomorphic building in $\XX$ of expected dimension at most one 
  is regular and for which there are no non-constant holomorphic spheres contained
  in $\DD$. By \cite[Lemma 5.29]{tfuk}, there exists a collection 
    $\ul{J}(X^{\tau_\nu}) = (J_\Gamma(X^{\tau_\nu}))$  of 
  almost complex structures for $X^{\tau_\nu}$ converging to $\ul{J}(\XX)$ so 
  that every adapted stable map to $X^{\tau_\nu} $
  of expected dimension at most one 
  is regular and for which there are no non-constant holomorphic spheres contained
  in $D_{\tau_\nu}$.  Let  
  \[ (C_\nu, u_\nu: C_\nu \to X^{\tau_\nu}), \tau_\nu \to \infty \] 
  be a sequence as in the statement of the Theorem.  Since $C_\nu$ is stable, 
after passing to a subsequence the sequence $C_\nu$ Gromov-converges to a limiting
treed disk $C$.  The argument in \cite[Step 2, Proof of Theorem 8.4]{tfuk}
shows that the derivatives of $u_\nu$ are bounded on the neck regions in $C_\nu$, 
as otherwise one would obtain by rescaling a component of $C$ mapping non-trivially 
into $\XX$ but with no intersections with $\DD$.  This is impossible since $\DD$
is a Donaldson hypersurface in each component.  On the neck region $[-\tau_\nu,\tau_\nu] \times Z \subset X^{\tau_\nu}$, the restriction of $u$ to each $T_e^\nu$ is a gradient trajectory.  Suppose that $u_\nu | T_e^\nu$ intersections $[-\tau_\nu,\tau_\nu] \times Z$.
After passing to a subsequence, the restriction $(u_\nu | T_e^\nu)^{-1}([-\tau_\nu,\tau_\nu] \times Z)$  converges to a gradient trajectory in the cylinder $\R \times Z$ for the height function.    This condition implies \eqref{trajlim} in the limit $\nu \to \infty$.
\end{proof}

We will also need various compactness results for moduli spaces of buildings under variation of the Lagrangian boundary condition.   Let $X_{\subset}$ be a manifold with a cylindrical end as in Definition \ref{def:acyl}.
Let $\phi_\subset: L_\subset \to X_{\subset}$ denote an asymptotically cylindrical 
Lagrangian embedding.   By Proposition \ref{prop:cleanly2}, the closure $ \ol{L}_\subset$ is contained in a cleanly-self-intersecting Lagrangian submanifold of $\ol{X}_\subset$.

\begin{definition} Denote by $\XX_\subset[k]$ the union of $X_{\subset}$
with $k-1$ neck pieces $\P(N_\pm \oplus \C)$.   A {\em treed holomorphic 
building} in $X_\subset$ is a collection of levels
\[ (C_i,u_i: S_i \to \XX_\subset[k]_i, i = 1,\ldots, k ) \] 
as in Definition \ref{def:building}, satisfying matching 
conditions for any collection of inter-level edges 
between $u_i$ and $u_{i+1}$.
\end{definition}

Any treed building in a broken manifold may be viewed as a pair of treed buildings in the corresponding cylindrical end manifolds, although not in a canonical way.  Let $X = (X_\subset,X_\supset)$.  A building in $\XX[k]$
of type $\Gamma$
consists of a building $u_\subset$ in $\XX_{\subset}[k_\subset]$ and a
building $u_\supset$ in $\XX_\supset[k_\supset]$ for some $k_\subset, k_\supset$ with $k = k_\subset + k_\supset$ satisfying matching conditions at the leaves
$e_\subset \in \Edge(\Gamma_\subset), e_\supset \in \Edge(\Gamma_\supset)$ corresponding to Reeb chords and orbits
that are glued to form $\Gamma$.

The key point in the following Theorem, whose proof will occupy the rest of the section, 
is that the Lagrangians are not required to be \label{rep:cylin} cylindrical-near-infinity, but only 
asymptotically cylindrical.  As such, the Theorem is not a consequence of known results
about compactness of buildings with Lagrangian boundary conditions.

\begin{theorem} \label{thm:varyL} Given a sequence of asymptotically 
cylindrical Lagrangian
  boundary conditions $\phi_{\subset,\nu}$ converging to some limiting boundary condition $\phi_\subset$ (at least in the $C^2$ topology on submanifolds) and a sequence  
  of stable holomorphic treed buildings
  $(C_\nu, u_\nu: C_\nu \to \XX_\subset[k_\nu])$ bounding $\phi_{\subset,\nu}$, there exists a
  subsequence of $u_\nu$ converging to a stable holomorphic treed building
  $u : C \to \XX_\subset[k]$ with boundary $(\partial u)(S)$ mapping to the
  broken Lagrangian $\LL_\subset[k]$.  Furthermore, the limit of any Gromov convergent sequence is unique.
  \end{theorem}

 We will need the following 
generalization of Gromov compactness for Lagrangian boundary conditions with clean
self-intersection in Schm\"aschke \cite[Section 4]{clean}.

\begin{theorem} \label{thm:grom} Let $X$ be a compact symplectic manifold and 
$\phi: L \to X$ a possibly non-compact Lagrangian immersion with clean self-intersection. Let $L_0 \subset L$ be a compact subset of $L$ that is a submanifold with boundary.
Suppose that $J_\nu$ is a sequence of tamed almost complex structures  on $X$ converging
in $C^2$ to a limiting tamed almost complex structure
$J$.  Suppose that $u_\nu: S_\nu \to X$ is a sequence of $J_\nu$-holomorphic maps with bounded area
$A(u_\nu)$ bounding $L_0$.   Then a subsequence of $u_\nu$ Gromov-converges to a $J$-holomorphic  stable map $u: S \to X$ bounding $L$.
\end{theorem} 

\begin{proof}[Sketch of proof] With $L$ compact, the statement is the standard Gromov compactness for clean intersection, as explained in \cite[Section 4]{clean}.  The extension is a kind of target-local 
Gromov-compactness theorem.   One constructs the components $u_v: S_v \to X$ of the 
limit $u: S \to X$ by composing $u$ with a sequence of embeddings $\phi_{v,\nu}: S_{v,\nu} \to S_\nu$, where $S_{v,\nu}$ is obtained from $S_v$ by removing a sequence of small balls $B_{v,\nu}$ around the nodes $Z_v \subset S_v$.   Consider a sequence $\phi_\nu: S_{v,\nu} \to S_\nu$ so that the maps $u_\nu \circ \phi_\nu$ have bounded first derivative on compact sets of $S_v - Z_v$.  The compositions $u_\nu \circ \phi_\nu$ have 
 boundary in $L_0$ and so converge, after passing to a subsequence, to a collection of components $u_v: S_v - B_{v,\nu} \to X$ bounding $L_0$, uniformly on compact sets.  The exponential decay results on cylinders and strips with small energy (used to show that bubbles connect) follow by considering $u_\nu$ as maps bounding $L$ and do not require compactness of $L$.
\end{proof}

\begin{proof}[Sketch of proof of Theorem \ref{thm:varyL}]
We indicate the modifications necessary for sft compactness as presented in,  for example, Venugopalan-Woodward \cite{tfuk} to go through.   Consider a sequence of treed disks 
 $(C_\nu, u_\nu: C_\nu \to X_\subset)$  with bounded energy 
with boundary values in $L_\subset$.
Because $\ol{L}_\subset$ has clean self-intersection, Theorem \ref{thm:grom} implies the existence of a subsequence converging to a limit $(C, u_\infty: C \to \ol{X}_\subset)$ where $C$ is a treed disk \label{rep:tdisk}  with surface component $S = \bigcup_v S_v$ mapping into the compactification $\ol{X}_\subset$.  

By adding marked points, we may assume that the limiting stable map has stable
domain. For example, choose a Donaldson hypersurface $\ol{D} \subset \ol{X}_\subset$ transverse to the limit $u_\infty$ and add leaves to $C_\nu$
according to the intersections of $u_\nu$ with $\ol{D}$. 
We denote by $S_v^\circ$ the complement of the nodes in $S_v$.
For each edge $T_e$ meeting $S_v$ choose $\eps_e $ small and denote by $B_{v}(\eps) = \cup_e B_{T_e \cap S_v}(\eps_{e})$ the complement  of the $\eps_e$-balls around the intersection $T_e \cap S_v$.  The surface $S_\nu$ is obtained by gluing together the surfaces $S_v - B_v(\eps_\nu)$ for suitable choices of $\eps_{e,\nu}$ converging to $0$ as $\nu \to \infty.$
We denote by $u_{\nu,v}$ the restriction of $u_\nu$ to $S_v - B_v(\eps_\nu)$.

Construct the levels of the limiting building by rescaling the target locally as follows. 
By assumption, an open neighborhood of $Y$ in $\ol{X}_\subset$ is 
isomorphic to the normal bundle $N_-$ of $Y$.  Identify 
the complement $N_-^\times$ of the zero section with $\R \times Z$ as above, and consider the action $e^s: N_-^\times \to N_-^\times$ of scalar multiplication of $e^s$ for a real number $s \in \R$, equivalent to translation in the $\R$-factor by $s$.   Suppose $u_v$ has image in $Y$.  Fix a point $z \in S_v^\circ$ and choose a sequence $s_{\nu,v} \in \R$ so that the translations $e^{s_{\nu,v}} u_{\nu,v}(z)$ converge
to a point in $N_-^\times$.  The argument in \cite[Section 10.4]{tfuk}
shows that the derivatives of $u_{\nu,v}$ are bounded with respect
to the cylindrical-end metric on $S_v^\circ$, so that 
after passing to a subsequence we may assume that $u_{\nu,v}$
converges to a level in $\P(N_\pm \oplus \C)$.  

It remains to show that the matching conditions between levels are
satisfied.  Suppose that $u_{v_1}$ and $u_{v_2}$ are adjacent
components of the limit connected by a boundary node; the case of an
interior node is similar.  Denote by $u_{\nu,e}$ the restriction of
$u_\nu$ to the neck region $[-\zeta_e,\zeta_e] \times S^1$ resp. strip
$[-\zeta_e,\zeta_e] \times S^1$ connecting the two components of the
limit. For $z$ lying in some such strip, choose a sequence $s_\nu$ so
that the maps $e^{s_\nu} u_{\nu,e}(z)$ converge. Since the derivative
of $u_{\nu,e}$ is bounded and the Lagrangians
$e^{s_\nu} L_{\subset,\nu}$ converge in $C^\infty$ to a boundary
condition $\R \times L_Z$, after passing to a subsequence the maps
$e^{s_\nu} u_{\nu,e}(z)$ converge in $C^\infty$ on compact sets to a
limit $u_e$ which is contained in a fiber of $\P(N_- \oplus \C)$,
necessarily with a single intersection with the divisors at zero and
infinity corresponding to the two ends of $S_{e,\nu}$.  The boundary
conditions on $e^{s_\nu} u_{e,\nu}$ converge to the cylindrical
boundary condition $\R \times L_Z$ as $\nu \to \infty$.  It follows
that $u_{e,\nu}$ converges to a trivial strip of the form
$u_{e}(s,t) = (\mu s, \thorn_\mu(t))$ in some fiber
$N_{-,y}^\times \cong \R \times S^1$ for some $\mu \in \R$ and Reeb
chord or orbit $\thorn_\mu$ with total angle change $\mu$.

The Reeb orbit appearing in the limit on the thin parts of the surface is independent of the choice of rescaling.  Indeed, 
suppose by way of contradiction that there exist
two rescaling sequences $e^{s_\nu} u_{e,\nu}$
and $e^{s_\nu'} u_{e,\nu}$ converging to different 
to trivial cylinders corresponding to different 
Reeb chords or orbits $\thorn_{\mu}, \gamma_{\mu'}$ with 
different angle changes $\mu,\mu'$.
Since the angle change of $u_{e,\nu}( s, \cdot)$
is a continuous function of $s$ and the set of angle changes 
\begin{equation} \label{eq:discrete}  \int \gamma^* \alpha , 
\quad \gamma:[0,1] \to Z \ \text{resp.} \ S^1 \to Z 
\end{equation} 
of Reeb chords and orbits is discrete, the intermediate value theorem 
implies the existence of a rescaling sequence $s_\nu''$
for which the limit of $e^{s_\nu'} u_{e,\nu}(0, \cdot)$ 
has angle change $\mu'' \in (\mu,\mu')$
which is not the angle change of any Reeb chord or orbit.
This is a contradiction.   

The limiting building is constructed as follows. 
Assign each component $u_v$ to a level $S_i$
by comparing the translation sequences $s_{v,\nu}$ necessary 
to construct the limit.  By the discussion from the previous paragraph, 
if two components $u_{v_-}, u_{v_+}$ are in different levels
and are joined by an edge then for one of the components,
say $u_{v_-}$, the image of $S_{v_-} \cap S_{v_+}$ maps to 
the divisor at infinity in $\P(N_- \oplus \C)$ and the other maps 
to the divisor at zero.  After possibly adding trivial strip or cylinders, 
we obtain a building $(C,u)$ with the matching conditions that the Reeb chords at either side of the node match, the projections to $Y$ match on either side of the node, and the matching occurs at the same copy of $Y$ in $\XX_{\subset}[k]$.  Since the limit in $\ol{X}_\subset$ was unique
and the rescaling sequences $s_{\nu,v}$ are unique up to addition of constants, the limiting building in $\XX_{\subset}[k]$ is unique
up to translation in the neck pieces.  The statement of the Theorem 
follows. 
\end{proof}

  \subsection{Transversality for buildings}
\label{regrem} 

Regularization of the moduli
  spaces of buildings may be carried out using Donaldson hypersurfaces following
  Charest-Woodward \cite{flips} and Venugopalan-Woodward \cite{tfuk}. We modify the construction to allow boundary in asymptotically-cylindrical broken Lagrangians 
\[ \LL = (L_\subset,L_\supset) .\] 
The broken analog of Theorem \ref{thm:comeager} gives an
inductive construction of regular perturbation data.  For any type  $\bGamma$ we denote by $\M_{\bGamma}(\XX,\phi)$ the
regularized moduli space of buildings with type $\bGamma$.  As in
Proposition \ref{prop:loccut}, the moduli space of buildings
$\M_{\bGamma}(\XX,\phi)$ is locally cut out by a Fredholm map.  
Let $C_i \subset C$ denote the subset of the
domain mapping to $\XX[l]_{i(v_\pm)}$, and $u_i$ the restriction of
$u$ to $C_i$.   In the construction of the linearized operator, we will 
focus on the surface parts $S_i \subset C_i$.  Using the Sobolev norms from \eqref{1pnorm},
let
\[ \B_\Gamma \subset \M_\Gamma \times \Pi_{i =0}^l \Map
(S_i^\circ, \XX[l]_i )_{k,p,\lambda} 
\] 
be the space of maps of class $k,p$ from each $S^\circ_i$ to the
spaces $\XX[l]_i$, lifting to a map to $\LL[l]_i$ on the boundary, 
and satisfying the following conditions:  The
matching conditions along cylindrical and strip-like ends via the
evaluation maps including \eqref{evinftyt} and the deformed
matching conditions of \eqref{pmatch}.  The linearized operator for
such buildings is defined as in the discussion after
\eqref{tcutout}. 

\begin{definition} \label{def:linop} The {\em linearized operator
for a holomorphic building} $(S,u)$ is 
\begin{eqnarray} \label{linop2} \ti{D}_u: T_{(C,u)} \B_\Gamma & \to
 &  \Omega^{0,1}(S^\circ, u^* T \XX )_{k-1,p,\lambda} 
 \\ (\zeta,\xi) & \mapsto & D_u \xi + \frac{1}{2} J \dd u
  Dj(\zeta_S) \end{eqnarray}
where
$T_{(C,u)} \B_\Gamma = \{ (\zeta,\xi) \}$
restricts to deformations $\xi$ satisfying in addition to the
conditions in \eqref{linop} the matching conditions at infinity
\[ \ev_e(\xi_{S_i}) = \ev_e (\xi_{S_{i+1}} )\] 
for all edges $e$ of
$\Gamma$ connecting different levels $S_i, S_{i+1}$.  A holomorphic building
$u: S \to \XX$ is called {\em regular} if the operator $\ti{D}_u$ is
surjective.    \end{definition}

\begin{theorem}  \label{thm:pertind}
For any broken
type $\Gamma$, given perturbations $P_{\Gamma'}$ for 
strata $\M_{\Gamma'}(\XX,\phi)$ with $\Gamma' \prec \Gamma$, there exists a perturbation 
$P_\Gamma$ so that for each uncrowded map type $\bGamma$ with underlying domain type $\Gamma$ and expected dimension at most one, the closure of $\M_{\bGamma}(\XX,\phi)$ is contained in the uncrowded locus and is a finite set or a compact one-manifold with boundary.
\end{theorem} 

\begin{proof}[Sketch of proof]
The proof of this theorem is similar to that  in Charest-Woodward \cite[Theorem 4.20]{flips}.
At any point $(C,u: C \to \XX, P_\Gamma)$ in the universal moduli space 
${\M}^{\univ,i}_{\Gamma}(\XX,\phi)$ one must show that 
any element in the cokernel of $\ti{D}_u$ vanishes.  The elements of 
the cokernel $\eta = (\eta_v)$ have vanishing restriction $\eta_v = 0 $ to any component $S_v$ on which $u$ has non-trivial horizontal derivative.
The restriction $\eta_v$ of $\eta$ must satisfy $D_{u_v}^* \eta_v = 0$ and be perpendicular to domain-dependent variations of the cylindrical almost-complex structure $J_\Gamma$.  
The last condition in particular implies
that $\eta_v$ vanishes in a neighborhood of any point at which $\dd (p \circ u)$
is non-zero.  The claim follows by unique continuation. 

Multiple covers of trivial cylinders, meaning maps 
whose image is contained in a fiber of the projection $\R \times Z \to Y$,
are transversally cut out.  Indeed, if $u$ maps to $\R \times Z_y$ for
some $y \in Y$ then the Cauchy-Riemann operator $D_u$  splits
\[ D_{u} \cong \olp_{T_y Y} \oplus  D_{u}^{v} \] 
as the standard Cauchy-Riemann operator $  \olp_{T_y Y} $ on maps to 
$T_y Y$ with boundary $T_y L_Y$
plus the linearized operator $D_u^v$ for a map of a
genus zero surface $S^\circ_v$ into the fiber $\C - \{ 0 \} $ with boundary conditions
$(\R \cup i \R) - \{ 0 \}$.   Such operators are  surjective
by any number of arguments; for example, by Proposition \ref{prop:linremov}, $D_u^v$ compactifies to a rank one Cauchy-Riemann operator $D_{\ol{u}}^v$
with a non-trivial kernel 
\[ \ker(D_{\ol{u}}^v) \cong \R \] 
given by dilation.  
In rank one, any Cauchy-Riemann operator 
cannot have both non-trivial kernel and cokernel by Oh \cite{oh:rh}.  It follows that the cokernel 
of $D_{u}^v$ must vanish.   Similarly, $\olp_{T_y Y}$ has non-trivial kernel and 
vanishing cokernel as well, so $D_u$ is surjective.  In particular, the usual problem in 
symplectic field theory of multiple covers of trivial cylinders or strips lacking regularity does not occur.    Thus, the operator $\ti{D}_{u_v}$ on any component $u_v$ that covers a trivial cylinder is surjective.  In particular, $\eta_v = \ti{D}_{u_v} \xi_v$ for some 
$\xi_v$ possibly with non-trivial evaluations on the ends of $S_v$. 

An induction shows that the restriction of $\ti{D}_u$ to any 
union of components $S_v$ that are covers of trivial components, or on which the map is constant, is also surjective:  The kernel of $D_u$ on any disk with strip like ends 
consists of constant sections and is identified with
$\R \oplus T_y Y$ via the splitting of the symplectization.  As such,
the matching conditions are cut out transversally as in the proof of 
Theorem \ref{thm:comeager}.   Thus there are no non-trivial elements of the cokernel. 
Compare also Ekholm \cite[Lemma 6.4]{ekholm:morse} and especially Venugopalan-Woodward \cite[Corollary 6.35]{tfuk} where a similar proof is given 
for the context of multi-directional symplectic field theory.
\end{proof}

 \subsection{Gluing with Lagrangian boundary conditions}
\label{gluingsec} 

 The gluing argument produces from any holomorphic building a limiting
 family of holomorphic maps.  The proof is probably standard and similar to that  in Charest-Woodward \cite{flips}.    First recall the gluing construction on domains and targets.  
 
 \begin{definition}  \label{def:gluing}
 Given
 gluing parameters $\delta_1,\ldots, \delta_k > 0 $, the glued domain
 $S^{\delta_1,\ldots,\delta_k}$ is obtained from $S$ by gluing necks
 $ [-|\ln(\delta_i)|/2, |\ln(\delta_i)|/2] \times S^1$ of length
 $|\ln(\delta_i)|$ at each node of $S$ separating two levels.  There is a similar construction of the glued target
 $X^{\delta}$ obtained by gluing in a neck of length 
 $|\ln(\delta)|$ in
$\XX$.
\end{definition} 

For simplicity we state the gluing result for the case of 
two levels only:

 \begin{theorem} \label{thm:gluing2} Let
   $\XX = \ol{X}_\subset \cup_Y \ol{X}_\supset$ and
   $\LL = L_{\subset} \cup_{L_Y} L_\supset $ be a broken rational
   symplectic manifold and rational self-transverse immersed cylindrical-near-infinity broken Lagrangian in the sense of \eqref{eq:brokL}.  Suppose that 
   $ (C,u: C \to \XX) $
   is a regular treed building
   with limiting eigenvalues $\mu_1,\ldots,\mu_k$ of Reeb
   chords or orbits at the separating hypersurface $Y \subset \XX$
   with boundary in $\phi_\eps$ for some $\eps < 0 $.  Then there
   exists $\delta_0 > 0$ such that for each gluing parameter
   $\delta \in (0,\delta_0) $ there exists a treed building 
   \[ (C^{\delta/\mu_1,\ldots, \delta/\mu_k} , u_\delta: S^{\delta/\mu_1,\ldots, \delta/\mu_k} \to
   X^{\delta} ) \]
   with the property that $u_\delta$ depends smoothly on $\delta$.  Furthermore 
      the Gromov limit recovers the original map:
\[ \lim_{\delta \to 0} u_\delta = u .\]
\end{theorem}

 We construct from any holomorphic building a holomorphic map to the manifold
 with long neck,  using Floer's version of the Picard Lemma.  Afterwards we show that
  any such map for sufficiently long neck length is obtained by
  such a construction.  Recall Floer's version of the Picard Lemma,
  \cite[Proposition 24]{floer:monopoles}).

 \begin{lemma} \label{lem:piclem} Let $f : V_1 \to V_2$ be a smooth map
   between Banach spaces that admits a Taylor expansion
\[ f(v) = f(0) + df(0)v + N(v) \] 
satisfying the following condition:   There exists a constant $C > 0$ such that 
$df(0): V_1 \to V_2 $ has a right inverse $G:V_2 \to V_1$
satisfying the uniform bound
\[ \Vert GN(u) - GN(v) \Vert \leq C( \Vert u\Vert + \Vert v
\Vert)\Vert u - v \Vert, \quad \forall u,v \in V_1 .\]
Let $B_\eps(0)$ denote the open
$\eps$-ball centered at $0 \in V_1$ and assume that
\[\Vert Gf(0) \Vert \leq 1/(8C) .\]
For $\eps< 1/(4C)$, the zero-set of $f^{-1}(0) \cap B_\eps(0)$ is a
tranversally-cut-out (hence smooth) submanifold of dimension $\dim(\Ker(df(0)))$ diffeomorphic to
the $\eps$-ball in $\Ker(df(0))$. \end{lemma}

\begin{proof}[Proof of Theorem] \label{proofoftheorem}
To simplify notation, we consider 
only the case that the building consists of a pair of maps joined by
strip-like ends; the general case is left to the reader. 
  To construct the approximate solution, we begin by recalling the
  construction of the deformation of a complex curve at a node.  Let
  $S$ be a broken curve with two sublevels $S_+,S_-$.  Let
  $\delta > 0 $ be a small gluing parameter.  
\label{misscom}
Variations of the domain may be represented as variations of the
conformal structure on a fixed curve together with variations of the
edge lengths.  Let
  \bea u_-: S_- &\to&  X_- := X_\subset  \\
  u_+: S_+ &\to& X_+ := X_\supset  \eea
  be maps from components $S_\mp$ containing points $w_\pm \in S_\pm$
  corresponding to the ends satisfying \eqref{matchderiv} so
  that $u = (u_-,u_+)$ form a building in $\XX$.  Let $\Gamma_\pm$
  denote the combinatorial types of the domains of $u_\pm$ and let
\begin{equation} \label{localtriv2} {{\S}}_{\Gamma_\pm}^i \to {{\M}}_{\Gamma_\pm}^i \times S_\pm, , i =
1,\ldots, l \end{equation} 
be local trivializations of the universal treed disk.  These local
trivializations identify each nearby fiber with
$(S_\pm,\ul{z},\ul{w})$ such that each point in the universal
treed disk is contained in one of the local trivializations
\eqref{localtriv2}.  We may assume that $\M_{\Gamma_\pm}^i$ is
identified with an open ball in Euclidean space so that nodal fiber 
containing $S_- \cup S_+$ lies over $0$.  Similarly, we assume we have a local trivialization of the universal bundle near the glued curve as a
smooth fiber bundle.  The local trivialization gives rise to a family
of complex structures
\begin{equation} \label{localtriv3} {\M}_{\Gamma}^i \to \J(S^\delta) \end{equation}
that are constant on the neck region.
We consider metrics on the punctured curves $S_\pm^\circ$ that are
cylindrical on the neck region.  That is, on the images of the maps
\[ \kappa_{C,\pm}: \pm [0,\infty) \times [0,1] \to S_\pm \]
the metrics are the product of the standard metrics on the two factors. 
By assumption we have cylindrical ends so that the images of 
\[ \kappa_{X,\pm}: \ \mp [0,\infty) \times Z \to X_\pm \]
are isometric.  Both the glued target $X^{\delta^\mu}$ and glued domain
$S^\delta$ are defined by removing the part of the end with
$|s| > |\ln(\delta)| $ and identifying
\begin{eqnarray*} (s,t) \sim (s- |\ln(\delta)|,t) &  (s,t) \in (0,
                                                    |\ln(\delta)|) \times S^1 \\ 
(s,t) \sim (s- |\mu\ln(\delta)|,t) & (s,t) \in  (0,
                                     |\ln(\delta)|) \times Z. \end{eqnarray*}

The prerequisite for Floer's version of the Picard lemma is an
approximate solution to the Cauchy-Riemann equation on the glued
curve.  Choose a cutoff function
\begin{equation} \label{beta2} 
\beta \in C^\infty(\R, [0,1]), \quad
\begin{cases} 
\beta(s) = 0 & s \leq 0 \\ \beta(s) = 1 &  s \ge
1 \end{cases}. \end{equation}
\label{itsastrip} 
We denote by $ \exp_x: T_x X^\delta \to X^\delta$ geodesic
exponentiation, using the given cylindrical metric on the neck region.
We write using geodesic exponentiation in cylindrical coordinates
\[ u_\pm(s,t) = \exp_{(\mp \mu s,t^\mu z)} (\zeta_\pm(s,t))
.\] %
Define $u^{\pre}_\delta$ to be equal to $u_\pm$ away from the neck
region, while on the neck region of $S^\delta$ with coordinates $s,t$
define
\begin{multline} \label{preglued} u^{\pre}_\delta(s,t) = \exp_{(\mu 
    s,t^\mu z)} ( \zeta^\delta(s,t)), \\ \quad 
\zeta^\delta(s,t) = \beta(-s) \zeta_- \left(- s + \frac{|\ln(\delta)|}{2}, t \right) +
  \beta(s ) \zeta_+ \left( s - \frac{|\ln(\delta)|}{2},t \right)  .
\end{multline}
In other words, one translates $u_+,u_-$ by some amount $|\ln(\delta)|$,
and then patches them together using the cutoff function and geodesic
exponentiation. 

To obtain the estimates necessary for the application of the Picard
lemma, we work in Sobolev spaces with weighting functions close to
those needed for the Fredholm property on cylindrical and strip-like
ends in \eqref{weightfunction}.  The surface part $S^\delta$ satisfies
a uniform cone condition and the metrics on $X^{\delta^\mu}$ are
uniformly bounded.  These uniform estimates imply uniform Sobolev
embedding estimates and multiplication estimates.  Denote by
\[ (s,t) \in \left[-\frac{|\ln(\delta)|}{2}, \frac{|\ln(\delta)|}{2} \right] \times S^1  \] 
the coordinates on the neck region in $S^\delta$ created by the
gluing.  For $\lambda > 0 $ small, define a {\em Sobolev weight
  function}
\[ \aleph^\delta_\lambda: S^\delta \to [0,\infty), \quad (s,t) \mapsto
\beta \left(\frac{|\ln(\delta)|}{2} - |s| \right) p \lambda \left ( \frac{|\ln(\delta)|}{2} - |s| \right) .\]
By definition $\aleph^\delta_\lambda$ is zero on the complement of the neck region.
We will also use similar weight functions on the punctured curves 
\[ \aleph_\lambda^{\pm}: S_\pm^\circ \to [0,\infty), \quad (s,t)
\mapsto \beta(|s|) p \lambda |s|. \]
Holomorphic maps near the pre-glued solution are cut out locally
by a smooth map of Banach spaces. 
\label{deltafix} Given an element $m \in \M_\Gamma^i$ and a section
$\xi: S^\delta \to u^* TX^{\delta}$ define as in Abouzaid
\cite[5.38]{ab:ex} a norm based on the decomposition of the section
into a part constant on the neck and the difference:
\begin{multline} \label{1pl2} \Vert (m,\xi) \Vert_{1,p,\lambda}^p :=
   \Vert m \Vert^p + \Vert \xi \Vert^p_{1,p,\lambda} \\ 
\Vert \xi \Vert^p_{1,p,\lambda} := 
\Vert (\xi(0,0)) \Vert^p +
    \int_{S^\delta} ( \Vert \nabla \xi \Vert^p  \\  + \Vert \xi - \beta(|
    \ln(\delta)|/2 - |s|) \cT^u ( \xi(0,0) ) \Vert^p ) 
    \exp( \aleph_\lambda^\delta) \dd \Vol_{S^\delta} 
\end{multline}
where $\cT^u$ is parallel transport from $u^{\pre}(0,t)$ to
$u^{\pre}(s,t)$ along $u^{\pre}(s',t)$.  \label{nolonger} Pointwise geodesic
exponentiation defines a map (using Sobolev multiplication estimates)
\begin{equation}
  \exp_{u_\delta^{\pre}}: \Omega^0(S^\delta, (u_\delta^{\pre})^*
  TX^{\delta^\mu})_{1,p,\lambda} \to \Map^{1,p}(S^\delta,X^{\delta^\mu}) \end{equation}
and $\Map^{1,p}(S^\delta,X^{\delta^\mu})$ denotes maps of class
$W_{1,p}^{\loc}$ from $S^\delta$ to $X^{\delta^\mu}$.  In the case of
Lagrangian boundary conditions, we have a similar map assuming that the
exponential map sends tangent vectors to the Lagrangian to points in
the Lagrangian boundary condition; we omit the Lagrangian boundary
condition from the notation.  Similarly, for the punctured surfaces we
have Sobolev norms
\begin{multline} \label{eq:1pl3} \Vert (m,\xi) \Vert_{1,p,\lambda} :=
 \left( \Vert m \Vert^p + \Vert \xi \Vert_{1,p,\lambda}^p \right)^{1/p}, \\ 
\Vert \xi \Vert_{1,p,\lambda} :=     \left( \begin{array}{l} 
\Vert \xi(\pm \infty,0) \Vert^p +    \int_{S^\circ_\pm} ( \Vert \nabla \xi \Vert^p +  \\
   \Vert \xi - \beta(|s|) \cT^u \xi(\pm \infty,0) \Vert^p ) \exp(
    \aleph_\lambda^\pm) \dd \Vol_{S^\circ_\pm}  \end{array} \right)^{1/p}
\end{multline}
where the limits $\xi(\pm \infty,t)$ are assumed to exist. Geodesic exponentiation defines maps 
\begin{equation}
  \exp_{u_\delta^{\pre}}: \Omega^0(S^\circ_\pm, (u_\delta^{\pre})^*
  TX)'_{1,p,\lambda} \to \Map^{1,p,\lambda}(S^\circ_\pm,X_\pm^\circ) \end{equation}
where, by definition, $\Map^{1,p,\lambda}(S^\circ_\pm,X_\pm^\circ)$ is
the space of $W_{1,p}^{\loc}$ maps from $S^\circ_\pm$ to $X_\pm$ that
differ from a Reeb chord \label{notorbit} at infinity by an element of
$\Omega^0(S^\circ_\pm,(u_\delta^{\pre})^* TX_\pm^\circ)'_{1,p,\lambda}$
(which may vary at infinity because of the inclusion of constant maps
on the end in the Banach space).  
In the case of the cylindrical end manifolds \label{nosuper},
 the assumption $\lambda$ small on the
Sobolev decay constant implies that the linearized operators
\[ D_{u_\pm} : 
\Omega^0(S_\pm^\circ,
u_\pm^* TX_\pm)'_{1,p,\lambda} 
\to \Omega^{0,1}(S_\pm^\circ,
u_\pm^* TX_\pm)_{0,p,\lambda} \]
are Fredholm.  The kernel contains any infinitesimal variation of the
map by Lemma \ref{lem:expfastlem}.  By the regularity assumption, the fiber
products
\begin{equation} \label{eq:trancut} \ker(\ti{D}_{u_-})
  \times_{\ev_{\infty,-},\ev_{\infty,+}}
  \ker(\ti{D}_{u_+}) \end{equation}
are transversally cut out, where $\ev_{\infty,\pm}$ are the maps of
\eqref{evinftyt}.  

The space of holomorphic maps near the pre-glued solution is cut
out locally by a smooth map of Banach spaces.  For a $0,1$-form
$\eta \in \Omega^{0,1}( S^\delta, u^* TX)$ define
\[ \Vert \eta \Vert_{0,p,\lambda} = \left( \int_{S^\delta} \Vert \eta
\Vert^p \exp( \aleph_\lambda^\delta) \dd \Vol_{S^\delta} \right)^{1/p} .\]
Parallel transport using an almost-complex connection defines a map
\[ \cT_{u_\delta^{\pre}}(\xi) :\ \Omega^{0,1}(S^\delta,
(u_\delta^{\pre})^*TX)_{0,p,\lambda} \to \Omega^{0,1}(S^\delta,
(\exp_{u_\delta^{\pre}}(\xi))^*TX)_{0,p,\lambda} .\]
Because we are working in the adapted setting, our curves $S^\delta$
are attached to a collection of interior leaves $T_{e_1},\ldots, T_{e_n}$.  We require
\begin{equation} \label{eq:constrain} (\exp_{u_\delta^{\pre}} (\xi) )(T_{e_i}) \in D, \quad i = 1,\ldots, n
.\end{equation} 
By choosing local coordinates near the attaching points
$w_e = T_e \cap S$, the constraints \eqref{eq:constrain} may be
incorporated into the map $\cF_\delta$ to produce a map
\begin{multline} 
  \cF_\delta: \M_\Gamma^i \times \Omega^0(S^\delta,
  (u_\delta^{\pre})^* TX, (\partial u_\delta^{\pre})^* TL)'_{1,p,\lambda}  \to
  \Omega^{0,1}(S^\delta, (u_\delta^{\pre})^* TX)_{0,p,\lambda} \times
  V \end{multline}
where the space $V$ is the direct sum of additional factors enforcing the matching and divisor conditions; namely those in \eqref{matchderiv} together with the sum
$ \bigoplus_{e=1}^n T_{u(w_e)} X/T_{u(w_e)} D$ enforcing the
conditions that the interior markings map to the Donaldson
hypersurface.  The first component of this map is
\[ 
  \cF_\delta (m,\xi) = \left( \cT_{u_\delta^{\pre}}
    (\xi)^{-1} \olp_{J_\Gamma,H_\Gamma, j(m)} \exp_{u_\delta^{\pre}}
    (\xi), \ldots \right) .\]
Zeroes of $\cF_\delta$ correspond to {\em adapted} holomorphic
maps near the preglued map $u_\delta^{\pre}$.  The expression
$ \cF_\delta(0) $ has contributions created by the cutoff function
and difference in the maps:
\begin{eqnarray*} \Vert \cF_\delta(0) \Vert_{0,p,\lambda} &=& \left\Vert
  \olp_{J_\Gamma,H_\Gamma} \exp_{(\mu s,t^\mu z)} \left( \beta(-s)
  \zeta_-\left(- s + \frac{|\ln(\delta)|}{2}, t \right) \right.  \right. \\ 
  && + \beta(s ) \zeta_+ \left. \left. \left(s -
  \frac{|\ln(\delta)|}{2},t \right) \right) 
                                          \right\Vert_{0,p,\lambda}\\
&=& \left\Vert \left( D \exp_{(\mu s,t^\mu z) } \left( \dd \beta(-s) \zeta_-\left(- s +
 \frac{|\ln(\delta)|}{2}, t \right) \right. \right. \right. \\ 
 && +  \left. \dd \beta(s ) \zeta_+ \left(s -
  \frac{|\ln(\delta)|}{2},t \right)  \right) + 
   \\ 
&&
   \left. \left.
 \left( \beta(-s) \dd \zeta_-\left(- s +
  \frac{|\ln(\delta)|}{2}, t \right) \right. \right. \right. \\ && + \left. \left. \left. \beta(s ) \dd \zeta_+\left(s - \frac{|\ln(\delta)|}{2},t\right) \right)
  \right)^{0,1} \right\Vert_{0,p,\lambda}.
\end{eqnarray*}
Holomorphicity of $u_\pm$ implies an
estimate
\begin{multline} 
 \left\Vert \left( \left( \beta(-s) \dd \zeta_-\left( - s + \frac{|\ln(\delta)|}{2}, t \right) + \beta(s
) \dd \zeta_+\left(s - \frac{|\ln(\delta)|}{2},t \right) \right) \right)^{0,1} \right\Vert_{0,p,\lambda} 
\\ \leq
c e^{ - |\ln(\delta)| (1- \lambda)} = c \delta^{1 - \lambda }
,\end{multline} 
cf. Abouzaid \cite[5.10]{ab:ex}.  Similarly, from the terms involving
the derivatives of the cutoff function and exponential convergence of
$\zeta_\pm $ to $0$ we obtain an estimate
\begin{equation} \label{zeroth} \Vert \cF_\delta(0) \Vert_{0,p,\lambda} < c \exp( - |
\ln(\delta)| ( 1- \lambda)) = c \delta^{1 - \lambda } \end{equation}
with $c$ independent of $\delta$.  

To perform the iteration, we apply a uniformly bounded right inverse to
the failure of the approximate solution to solve the Cauchy-Riemann
equation.  Given
\[ \eta \in \Omega^{0,1}( S^\delta, (u^{\pre})^* TX)_{0,p} \] 
one obtains elements
\[ \ul{\eta} = (\eta_-, \eta_+) \in \Omega^{0,1}(S_\pm, u_\pm^*
TX_\pm) \]
by multiplication with the cutoff function $\beta$ and parallel
transport $\cT^{u_\pm}$ to $u_\pm$ along the path
\[ \exp_{(\mu s, t^\mu z)} ( \rho (\zeta^\delta(s,t) + (1- \rho)
\zeta_\pm(s,t))), \quad \rho \in [0,1] .\]
Define
\[ \eta_+ = \cT^{u_+}  \beta(s - 1/2) \eta, \quad \eta_- = \cT^{u_-}
\beta(1/2 - s)
\eta .\]
Since the fiber product \eqref{eq:trancut} is transversally cut out, there
exists
\[ (\xi_+,\xi_-) \in \Omega^0(S_\pm,u^*
TX_\pm)_{1,p,\lambda}, \quad D_{u_\pm} \xi_\pm = \eta_\pm,
\quad \ev_\infty(\xi_+) = \ev_\infty(\xi_-) \]
where $\ev_\infty$ are the evaluation-at-infinity maps of
\eqref{evinftyt}.  \label{eiadd} Denote 
\[ \xi_\infty = \ev_\infty(\xi_\pm) \in \R \times T_{\ev_\infty(u_\pm)} Z  .\]
Define $Q^\delta \eta$ equal to $( \xi_-, \xi_+)$ away from
$[- \frac{|\ln(\delta)|}{2}, \frac{|\ln(\delta)|}{2}] \times Z$ and on the neck region
by patching the solutions $(\xi_-,\xi_+)$ together using a cutoff
function that vanishes three-quarters of the way along the neck:
\begin{multline} Q^\delta \eta := \beta \left( - s + \qq |\ln(\delta)|
  \right) ( (\cT^{u_-})^{-1} \xi_- - \cT^u \xi_\infty) 
\\ + \beta \left(
    s + \qq | \ln(\delta)| \right) ( ( \cT^{u_+})^{-1} \xi_+ - \cT^u
  \xi_\infty) \\ + \cT^u \xi_\infty \in \Omega^{0,1}(S^\delta,
  (u^{\pre}_{\delta})^* TX)_{1,p,\lambda}
    \end{multline}
    where $\cT^{u_\pm}$ denotes parallel transport to $u_\pm$ from
    $u^\pre_\delta$
 along the path 
\[ \exp_{(\mu s, t^\mu z)} ( \rho (\zeta^\delta(s,t) + (1- \rho) 
    \zeta_\pm(s,t))), \rho \in [0,1] .\] 
   Since 
\[  \eta = (\cT^{u_-})^{-1}  \eta_- + (\cT^{u_+})^{-1} \eta_+ \] 
we have
\begin{eqnarray} \label{Dest} \nonumber
 \Vert D_{u_{\pre}^\delta} Q^\delta \eta - \eta \Vert_{1,p,\lambda}
 &=& \Vert D_{u^{\pre}_\delta} Q^\delta \eta -  (\cT^{u_-})^{-1} D_{u_-^\delta}
 \xi_-  - (\cT^{u_+})^{-1} D_{u_+^\delta} \xi_+
 \Vert_{1,p,\lambda} \\  \nonumber &\leq& c \exp( (1 - \lambda) | \ln(\delta)/4|) 
 \Vert
 \eta \Vert_{0,p,\lambda}  \\ \nonumber && + c \Vert \dd \beta( s - |\ln(\delta)|/4)
 Q^\delta_- \ul{\eta} \Vert_{0,p,\lambda} \\ \nonumber && + c \Vert \dd \beta( -s + |\ln(\delta)|/4) Q^\delta_+
 \ul{\eta}\Vert_{0,p,\lambda} \end{eqnarray}
where the first term arises from the difference between
$ D_{u^\pre_\delta} $ and $(\cT^{u_\pm} )^{-1} D_{u_\pm}
\cT^{u_\pm}$ and the second from the derivative $\dd
\beta$ of the cutoff function
$\beta$.  The difference in the exponential factors
\[ \aleph_\lambda^\pm = \aleph_\lambda^\delta \exp( \pm 2s \lambda), \quad 
\mp s \ge \frac{|\ln(\delta)|}{2} \] 
in the definition of the Sobolev weight functions implies that
possibly after changing the constant $c$, we have since
$|\ln(\delta)| = - \ln(\delta)$ \label{nosign}
\[ \Vert \dd \beta( s - |\ln(\delta)|/4) Q^\delta_\pm \eta
\Vert_{1,p,\lambda} < c e^{-  \lambda \frac{|\ln(\delta)|}{2}} =
c \delta^{\lambda/2} .\]
Hence one obtains an estimate as in Fukaya-Oh-Ohta-Ono
\cite[7.1.32]{fooo}, Abouzaid \cite[Lemma 5.13]{ab:ex}: for some
constant $c > 0$, for any $\delta > 0$, 
\begin{equation} \label{first} \Vert D_{ u^{\pre}_\delta } Q^\delta -
  \on{Id} \Vert < c \min( \delta^{\lambda/2} , \delta^{(1 - \lambda)/4})
  .\end{equation}
It follows that for $\delta$ sufficiently large an actual inverse may
be obtained from the Taylor series formula \label{tseries}
\[ D_{u^{\pre}_\delta}^{-1} = Q^\delta   ( D_{u^{\pre}_\delta} Q^\delta)^{-1}
=\sum_{k \ge 0} Q^\delta (I - Q^\delta D_{u^{\pre}_\delta})^k 
 .\]

 The variation in the linearized operators can be estimated as
 follows. After redefining $c > 0$ we have for all $\xi_1,\xi$
 sufficiently small
\begin{equation} \label{second} \Vert D_{\xi} \cF_\delta (0,\xi_1) -
  D_{u^{\pre}_\delta} \xi_1 \Vert \leq C \Vert \xi_1
  \Vert_{1,p,\lambda} \Vert \xi \Vert_{1,p,\lambda}. \end{equation}
To prove this we require some estimates on parallel transport.  Let
\[ \cT_{z}^{{\delta},x}(m,\xi): \Lambda^{0,1} T_z^* S_\delta \otimes
T_x X \to \Lambda^{0,1}_{{j}^{{\delta}}(m)} T_z^* S_\delta \otimes
T_{\exp_x(\xi)} X \]
denote pointwise parallel transport.  Consider its derivative
\[ D\cT_{z}^{{\delta},x}(m,\xi,m_1,\xi_1;\eta) = \nabla_t |_{t = 0}
\cT_{u_\delta^{\pre}} ( m + t m_1, \xi + t\xi_1) \eta .\]
For a map $u: S \to X$ we denote by $D\cT_{u}$ the corresponding map
on sections.  By Sobolev multiplication (for which the constants are
uniform because of the uniform cone condition on the metric on
$S^\delta$ and uniform bounds on the metric on $X^{\delta^\mu}$) there
exists a constant $c$ such that
\begin{equation} \label{Psiest}
 \Vert D\cT_{{u}}^{\delta,x}(m,\xi,m_1,\xi_1; \eta )
 \Vert_{0,p,\lambda} \leq c \Vert (m,\xi) \Vert_{1,p,\lambda} \Vert
 (m_1, \xi_1) \Vert_{1,p,\lambda} \Vert \eta \Vert_{0,p,\lambda}.
\end{equation}
Differentiate the equation
\[ \cT_{{u}}^{{\delta},x} (m,\xi) \cF_{\delta}(m,\xi) =
\olp_{J_\Gamma,H_\Gamma,{j}^\delta(m)}(\exp_{{u}^\delta_{\pre}}(\xi)) \]
with respect to $(m_1,\xi_1)$ to obtain
\begin{multline}
 D\cT_{{u}_\pre^\delta}(m,\xi,m_1,\xi_1, \cF_{\delta}(m,\xi) ) +
 \cT_{{u}}^{\delta}(m,\xi)( D \cF_{\delta}(m,\xi,m_1,\xi_1)) = \\ (D
 \olp )_{j^\delta(m),\exp_{u^\pre_\delta}(\xi)} (Dj^\delta
 (m,m_1),D\exp_{{u}^\delta} (\xi,\xi_1)) .\end{multline}
Using the pointwise inequality
\[ | \cF_\delta(m,\xi) | < c | \dd {\exp_{{u}^{\pre}_{\delta}(z)}(\xi)} | < c ( | \dd {u}^{\pre}_{\delta}
| + | \nabla \xi | )
\]
for $m,\xi$ sufficiently small, the estimate \eqref{Psiest} yields a
pointwise estimate
\[ | \cT_{u_\delta^{\pre}}(\xi)^{-1}
D\cT_{{u}_\pre^{\delta}}(m,\xi,m_1,\xi_1,\cF_{\delta}(m,\xi)) |
  \leq c (| \dd {u}^\delta_{\pre} | + | \nabla \xi |) \, | ( m,\xi ) | \, |
  (\xi_1,m_1) | .\]
Hence
\begin{multline} \Vert  \cT_{u_\delta^{\pre}}(\xi)^{-1}
D\cT_{{u}_\pre^{\delta}}(m,\xi,m_1,\xi_1,\cF_\delta(m,\xi))
\Vert_{0,p,\lambda} \\ \leq c ( 1+ \Vert \dd {u}^\delta
\Vert_{0,p,\lambda} + \Vert \nabla \xi \Vert_{0,p,\lambda} ) \Vert
(m,\xi) \Vert_{L^\infty} \Vert (\xi_1,m_1) \Vert_{L^\infty}
.\end{multline}
It follows that 
\begin{equation} \label{firstclaim}
  \Vert \cT_{u_\delta^{\pre}}(\xi)^{-1}
  D\cT_{{u}_{\pre}^{\delta}}(m,\xi,m_1,\xi_1,\cF_{\delta}(m,\xi))
  \Vert_{0,p,\lambda} \leq c \Vert (m,\xi) \Vert_{1,p,\lambda} \Vert
  (m_1,\xi_1) \Vert_{1,p,\lambda}
\end{equation}
since the $W^{1,p}$ norm controls the $L^\infty$ norm by the uniform
Sobolev estimates.  As in McDuff-Salamon \cite[Chapter
10]{ms:jh}, Abouzaid \cite{ab:ex} there exists a constant $c > 0$ such
that for all $\delta$ sufficiently small,
\begin{multline} \label{secondclaim} \Vert
  \cT_{{u}^\delta_{\pre}}(\xi)^{-1}
  D_{\exp_{u_\delta^{\pre}}(\xi)}(D_m{j}^{{\delta}}(m_1),D_{\exp_{{{u}^\delta_{\pre}}}
    (\xi)} \xi_1)) - D_{u_\delta^{\pre}} (m_1,\xi_1)
  \Vert_{0,p,\lambda} \\ \leq c \Vert m,\xi \Vert_{1,p,\lambda} \Vert
  m_1,\xi_1 \Vert_{1,p,\lambda} .\end{multline}
Combining the estimates \eqref{firstclaim} and \eqref{secondclaim} and
integrating completes the proof of claim \eqref{second}.
Applying the estimates
 \eqref{zeroth}, \eqref{first}, \eqref{second} produces a unique
 solution $m(\delta),\xi(\delta)$ to the equation
\[ \cF_\delta(m(\delta),\xi(\delta)) = 0 \] 
for each $\delta$, such
 that the maps 
\[ u_{\delta} := \exp_{u_\delta^{\pre}}(\xi(\delta)) \] 
 depend smoothly on $\delta$.  Note that the implicit function theorem
 by itself does not give that the maps $u_\delta$ are distinct, since
 each $u_\delta$ is the result of applying the contraction mapping
 principle in a different Sobolev space.  
\end{proof}

We now state the main
result on the behavior of the moduli spaces under the neck-stretching
limit.  Let $\M^{< E}(X,\phi_\gamma,D)$
denotes the locus in $\M(X,\phi_\gamma,D)$ with area less than $E$.
Similarly, let $\M^{<E }(\XX,\phi_\gamma,\DD)$ denote the locus with area
less than $E$ in $\M(\XX,\phi_\gamma,\DD)$.

\begin{theorem} \label{thm:bthm} Let
  $\XX = \ol{X}_\subset \cup_Y \ol{X}_\supset$ be a broken rational
  symplectic manifold and $\phi: \LL \to \XX$ a broken self-transverse
  Lagrangian immersion in the sense of \eqref{eq:brokL}.  Suppose perturbations
  $\ul{P} = (P_\Gamma)$ have been chosen so
  that every rigid labeled map in $\M^{< E}(\XX,\phi_\gamma,\D,\ul{\sigma})$ is regular. 
  There exists $\delta_0$ such that for
  $ \delta < \delta_0$, the assignment $[u] \mapsto [u_\delta]$
  from Theorem \ref{thm:gluing2} defines a bijection between the rigid moduli spaces
  $\M^{< E}(\XX,\phi_\gamma,\DD,\ul{\sigma})_0$ and
  $\M^{< E}(X^{\delta},\phi_\gamma,D,\ul{\sigma})_0$.
\end{theorem} 

\begin{proof}  To prove injectivity of gluing, suppose that $u_{\delta_\nu} = v_{\delta_{\nu}}$ for some pseudoholomorphic buildings $u \neq v$ and gluing parameters $\delta_\nu \to 0$.   Then the sequence $u_{\delta_\nu}$ has two stable Gromov limits, which is a contradiction to Theorem \ref{thm:stretchcompact}.  

To prove surjectivity of gluing, it suffices to
 prove the following:  Given a converging family  $u'_\delta$ with parameter $\delta$ converging to $u$, the map $u'_\delta$
 is close to $u_{\delta}$ in the norms used in the gluing formula.
 Indeed, this closeness implies that $u'_\delta = u_\delta$ by the uniqueness  part of the implicit function theorem. 

 To check the necessary exponential decay, note the map on the neck region may be decomposed into horizontal and
 vertical component.  First consider the horizontal part of the map
 $p_Y \circ u_\delta': S_\delta \to Y$. Denote by $R(l)$ the rectangle
 \[ R(l) = [- l/2, l/2 ]  \times [0,1] .\]
 Since there is no area loss in the limit $\delta \to 0$, for any
 $C > 0$ there exists $\delta' > \delta_C$ such that the restriction
 of $p_Y \circ u'_\delta$ to the annulus $R(|\ln(\delta')|/2)$
 satisfies the energy estimate of \cite[Lemma 3.1]{totreal}.  Thus
\begin{multline} \label{longcyl} p_Y u'_\delta(s,t) = \exp_{p_Y
    u^{\pre}_\delta(s,t)} \xi^h(s,t), \quad \Vert \xi^h(s,t) \Vert \leq C
  ( e^{ \pi( -  |\ln(\delta')|/2   - s ) } + e^{ \pi( - |\ln(\delta')|/2 + s)}) \\
  \quad s \in [-|\ln(\delta')|/2, |\ln(\delta')|/2] .\end{multline}
A similar estimate holds for the higher derivatives $D^k \xi^h(s,t)$
by elliptic regularity, for any $k \ge 0$.  
The necessary 
exponential decay result on the neck region is proved
in the case of cylinders in Venugopalan-Woodward \cite[(8.13)]{tfuk}; the case of strips is similar.
\end{proof} 

\begin{corollary} If $\M^{< E}(\XX,\phi_\gamma,\DD)_0$ is regular, then
  there exists $\delta_0 $ such that for $\delta > \delta_0$,
  $\M^{< E}(X^\delta,\phi_\gamma,D)_0$ is regular.
\end{corollary}

\begin{proof} We may assume that every $(C^{\delta/\mu_1,\ldots, \delta/\mu_k} , u_\delta) $ in $  \M^{< E}(X^\delta,\phi_\gamma,D)$
for $\delta$ sufficiently small is obtained by an application of Lemma \ref{lem:piclem} 
to the approximate solution in the proof of Theorem \ref{thm:gluing2}.  By the transversality statement in the Picard Lemma \ref{lem:piclem}, the nearby solution $u_\delta$ produced from $u_0$ by the implicit function theorem also has surjective linearized operator $\ti{D}_{u_\delta}$.  (In fact, the operator
is surjective even restricting to variations of conformal structure on 
$C^{\delta/\mu_1,\ldots, \delta/\mu_k} $ arising from variations on $S.)$
\end{proof} 

We will need a similar bijection for the case of buildings
in $X_\subset$ consisting of a map $u_{\subset}: S_\subset \to X_\subset$
and a neck piece $u_0: S_0 \to \P(N_- \oplus \C)$, where 
the Lagrangian $L_\subset$ in $X_\subset$ is only asymptotically cylindrical.

\begin{theorem} \label{thm:bthm2}
Let $\bGamma$ be a type of building in $\XX_\subset$ with two components as above.  Suppose perturbations $\ul{P} = (P_\Gamma)$ have been chosen so
  that every rigid labeled map in $\M_{\bGamma}(\XX,\phi_\gamma,\DD,\ul{\sigma})$ is regular, and let $\bGamma'$
  be the type with a single level in $X_\subset$ obtained by gluing. 
Then each $u \in \M_{\bGamma}(\XX,\phi_\gamma,\DD,\ul{\sigma})$
is in the closure of a unique component of $\M_{\bGamma'}(\XX,\phi_\gamma,\DD,\ul{\sigma})$.
\end{theorem} 

\begin{proof}   
The statement of the Theorem is a type of result known as ``surjectivity of gluing'' in the literature (as in for example \cite[Section 7.6]{clean}) \label{rep:surjglue} in which one must show that the sequence on the long cylinders is close to the approximate solution in the chosen Sobolev norm.  Let $u$ be as in the statement of the Theorem. 
By the gluing construction,  there exists a one-parameter family $u_\delta$ 
of buildings of type $\Gamma'$ Gromov-converging to $u$ in the limit
$\delta \to 0.$  To see that $u_\delta$ is the unique such limit, 
it suffice to check that if $u'_\nu$ converges to $u = (u_\subset,u_0)$ then  $u'_\nu$ is close to the approximate
solution $u_\delta^{\pre}$ of \eqref{preglued}. 

To prove this, we examine the vertical and horizontal parts of the map. 
Denote by $\ol{u}$ the map to $\ol{X}_\subset$ obtained by 
projecting $u_0$ to the base $Y$ of the neck piece $\P(N_- \oplus \C).$
Similarly, let $\ol{u}_\nu$ denote the map to $\ol{X}_\subset$
induced by $u_\nu$.   Then $\ol{u}_\nu$ Gromov converges
to $\ol{u}$, and in particular the domain $C_\nu$
of $\ol{u}_\nu$ converges to the domain $C$ of $\ol{u}$.
Hence $C_\nu$ is obtained from $C$ by removing small balls
around the node and gluing in cylinders of length $|\ln(\delta_\nu)|$ for some sequence of gluing parameters $\delta_\nu$.  Let $y \in Y$ be the image of the node in $\ol{u}$.
We trivialize the bundle $\R \times Z \to Y$ in a neighborhood $B_\eps(y)$ of $y$. Denote by 
\[ e^{-\tau_\nu}: \R \times Z \to \R \times Z, \quad (\sigma,z) 
\mapsto  (\sigma - \tau_\nu, z) \] 
translation by $-\tau_\nu$.   We may pass to a subsequence so that  
\[ \lim_{\nu \to \infty} e^{-\tau_\nu} u_\nu(0,0) = (0,z) \] 
for some point $(0,z)$ over $y$.
By the annulus lemma for maps to $\ol{X}_\subset$, we have on the 
long strips of length $|\ln(\delta_\nu)|$ connecting the components
\begin{multline} \label{longcyl2} \ol{u}_\nu(s,t) = \exp_y \xi(s,t), \quad \Vert \xi(s,t) \Vert \leq C
  ( e^{ \pi(s - |\ln(\delta_\nu)|/2) \mu_0 } + e^{\pi(-|\ln(\delta_\nu)|/2 - s)\mu_0} ) \\
  \quad s \in [-|\ln(\delta_\nu)|/2, |\ln(\delta_\nu)|/2] \end{multline}
where the exponential decay constant $\mu_0$ is determined by the 
angles at the clean intersection; see Abouzaid \cite[10.12]{ab:ex}.

To control the Sobolev norms of the vertical part of the map on the 
neck pieces, we compare the given almost complex structure and boundary 
conditions to a model problem in which the almost complex structure
and boundary conditions are constant.   Let 
\[ y(s,t) = p(e^{-\tau_\nu} u_\nu(s,t)) \in Y ,  \quad s  \in [-|\ln(\delta_\nu)|/2, |\ln(\delta_\nu)|/2] \]
be the projection of the given map to $Y$.
Choose local coordinates on $B_\eps(y)$ so that $L_Y$
is linear.   Denote by $L_\subset^\pm$ the branches of $L_\subset$
containing the images of $u_\nu(s,t)$ for $t = 0$ resp. $t = 1$
and $s$ sufficiently large and $\theta_\pm \in S^1$ the angles
of the branches at $y$.  Define affine-linear model boundary conditions
in $\R \times S^1 \times B_\eps(y)$ by 
\[ L_\nu^{\on{model},\pm} = \R \times \{ \theta_\pm \} \times L_Y .\]
Let $\sigma_\nu$ be the $\R$-coordinate of the evaluations
$u_\nu(0,0)$.  Define $T^\pm_\nu$ by 
\[ u( [-|\ln(\delta_\nu)|/2, |\ln(\delta_\nu)|/2] \times \{ 0 \}) 
=  [T^-_\nu,T^+_\nu] .\]
Since $L_\subset$ has smooth cleanly-self-intersecting compactification $\ol{L}_\subset$ in $\ol{X}_\subset$, the translations 
\[ e^{-\tau_\nu} ( L_\subset^\pm \cap  [T_\nu^- + \tau_\nu,T_\nu^++ \tau_\nu] \times Z)   \subset \R \times S^1 \times B_\eps(y)  \] 
differ from  $L_\nu^{\on{model},\pm}$ near $u(s,t)$ by a map 
\[ \beta_\pm:    L_\subset^\pm \cap  \left( [T_\nu^- + \tau_\nu,T_\nu^+ + \tau_\nu] \times Z  \right)  \to  N_{L_\nu^{\on{model},\pm}}  \]
satisfying an exponential decay estimate 
\begin{equation} \label{longcyl3} 
 \Vert e^{-\tau\nu} \beta_\pm(\sigma,z) \Vert  \
\leq C (e^{-\tau_\nu - \sigma}  + \dist(z,z'))
  \quad \sigma \in [T_\nu^-, T_\nu^+], z \in L_\nu^{\on{model}} \cap Z .\end{equation}
Choose a diffeomorphism identifying the boundary condition
with its model 
\[ \psi_\nu \in \on{Diff}(\R \times S^1 \times B_\eps)  , 
\quad \psi_\nu(e^{-\tau_\nu} L_\subset^\pm ) = L^{\on{model},\pm}_\nu .\]
Because of \eqref{longcyl3},
the diffeomorphism $D\psi_\nu$ may be taken to satisfy an estimate similar to that for the boundary conditions:
\begin{equation} \label{longcyl4} 
 \Vert D \psi_\nu (\sigma,z')  - \on{Id} \Vert  \
\leq C ( e^{-\tau_\nu - \sigma}  + \dist(z,z'))
  \quad \sigma \in [T_\nu^-, \infty)  .\end{equation}
The composition $\psi_\nu e^{-\tau_\nu} u_\nu$ satisfies the model boundary conditions 
and the Cauchy-Riemann equation up to an error term arising from 
the failure of the map $\psi_\nu$ to be $J$-holomorphic:
Let $J_0(t)$ denotes the  almost complex 
structure obtained by evaluating at $\psi_\nu e^{-\tau_\nu} u_\nu (0,t).$
By assumption,  the $\R$-derivative of $u_\nu(s,t)$ on the
neck so that if $\sigma_\nu(s,t)$ denotes the $\R$ coordinate
of $e^{-\tau\nu} u_\nu(s,t)$ then  for any $\eps > 0$
we have  for $\nu$ sufficiently large 
\[  \partial_s \sigma_\nu(s,t) \in  \left( \frac{\mu}{2}, \frac{3\mu}{2} \right) . \]
This implies 
\[  |\sigma_\nu(s,t)  - \sigma_\nu(0,t)| \ge \frac{\mu}{2} |s| .\]
Let 
\[ u^{\on{model}}(s,t) = (\mu s, t^\mu z) \]
as in \eqref{preglued}.   As in the discussion around \eqref{eq:discrete}, 
we have 
\[   p_Z( \psi_\nu e^{-\tau_\nu}u_\nu(s_\nu,t))  \to t^\mu z \quad \text{in} \quad C^\infty([0,1]) \]
for any sequence $s_\nu$ with  $|s_\nu|/|\ln(\delta_\nu)| \to 0$. 
Define
\[ \zeta_\nu(s,t) = (\psi_\nu  e^{-\tau_\nu} u_\nu ) (s,t)  - 
 u^{\on{model}}(s,t) . \]
Then $\zeta_\nu$ satisfies linear totally-real boundary conditions and 
is approximately holomorphic:  
Since the difference between $J$ and $J_0$ depends only 
on the projection to $Y$, after re-defining $\mu_0$ we have
\begin{eqnarray*} \label{longcyl5}
\Vert \olp_{J_0} \zeta_\nu (s,t)  \Vert  
&\leq&  
\Vert \olp_J  \zeta_\nu (s,t)  \Vert  
+ 
\Vert (\olp_J -  \olp_{J_0}) \zeta_\nu (s,t)  \Vert   \\
&\leq& C
  ( e^{  - \tau_v - \mu(s - |\ln(\delta_\nu)|/2)/2 } + 
    e^{ (s - |\ln(\delta_\nu)|/2) \mu_0 } + e^{(|\ln(\delta_\nu)|/2 - s)\mu_0 } ) \\
   && s \in [-|\ln(\delta_\nu)|/2, |\ln(\delta_\nu)|/2]
  \end{eqnarray*}
where the first term arises from exponential convergence of the 
vertical part of the boundary conditions, and the second from 
the annulus lemma for the horizontal map.
Write
\[ \eta_\nu := \olp_{J_0} \zeta_\nu \in 
\Omega^{0,1}([-|\ln(\delta_\nu)|/2, |\ln(\delta_\nu)|/2] \times [0,1],
\C^n)  .\]
Denote by 
\[ f_i \in C^\infty([0,1],\R^{2n}) , i \in I \] 
the eigenfunctions of $J_0 \partial_t$ with boundary conditions
the linear subspaces $TL_\nu^{\on{model},\pm}$ and the eigenvalues 
$\lambda_i \in \R$, the operator $J_0 \partial_t$ being self-adjoint.  Consider the  decomposition of the maps $\zeta_\nu, \eta_\nu$ into eigenfunctions of $J\partial_t$
with boundary conditions $TL_\nu^{\on{model},\pm}$ with coefficients
$c_{\nu,i}, d_{\nu,i} \in \R$:
\[   \zeta_\nu(s,t)  = \sum_{i \in I} c_{\nu,i}(s) f_i(t) ,
\quad \eta_\nu(s,t)  = \sum_{i \in I} d_{\nu,i}(s) f_i(t)  .\] 
We obtain a solution for the coefficients by integration
\begin{eqnarray} \label{integrate} \ \ \ \ \ \quad c_{\nu,i}(s)   &=& 
c_{\nu,i}( |\ln(\delta_\nu)|/2 ) 
\exp \left(
\lambda_i (s - |\ln(\delta_\nu)|/2) +  \int_{|\ln(\delta_\nu)|/2}^{s} d_{\nu,i}(s') \dd s'    \right) \\
  &=& 
c_{\nu,i}( -|\ln(\delta_\nu)|/2 ) 
\exp \left(  
\lambda_i (s + |\ln(\delta_\nu)|/2)  + \int_{-|\ln(\delta_\nu)|/2}^{s} d_{\nu,i}(s') \dd s'    \right) .\end{eqnarray}
We may assume, by redefining $\mu_0$, that the exponential decay constant $\mu_0$ is smaller than the minimum of the
non-zero eigenvalues $|\lambda_i|$.
For any $i \in I$, \eqref{integrate} gives 
\begin{multline} \label{longcyl6}
\Vert c_{\nu,i}(s)  \Vert  
\leq C \left( 
  \Vert c_{\nu,i}( |\ln(\delta_\nu)|/2) \Vert
   \ e^{ (s - |\ln(\delta_\nu)|/2)\mu_0 } + 
 \Vert c_{\nu,i}( - |\ln(\delta_\nu)|/2) \Vert \ 
  e^{(-|\ln(\delta_\nu)|/2 - s)\mu_0} \right)  \\
  \quad s \in [-|\ln(\delta_\nu)|/2, |\ln(\delta_\nu)|/2] . \end{multline}
Using elliptic estimates, one obtains similar estimates for the derivatives
of $\zeta_\nu$.   We obtain by
\eqref{longcyl6}  and elliptic regularity that for any 
constant $C$, for $\nu$ sufficiently large  the estimate 
\begin{equation} \label{diffest} \Vert \psi_\nu e^{-\tau_\nu} u_\nu(s,t)  - u^{\on{model}}(s,t) \Vert_{k,2,\lambda} 
< C (1 + \delta_\nu^{\mu_0 - \lambda}/(\mu_0 - \lambda) )  \end{equation}
holds for any $k \ge 0$.    Note that in the norm defined
in \eqref{eq:1pl3}, the zero modes $c_i(s,t), \lambda_i = 0$ have norm determined
by evaluation, and are not required to have small $k,p$ norm on the neck.
The Sobolev embedding theorem implies the same estimate for the $k,p,\lambda$ norm for any $kp$ with $kp > 1$.  
 Write 
\[ \psi_\nu  e^{-\tau_\nu} u_\nu(s,t) = 
\exp_{e^{-\tau_\nu} u_{\delta_\nu^{\pre}} (s,t)} (\xi_\nu(s,t)) \] 
with notation as around \eqref{eq:1pl3}.    
The difference between geodesic exponentiation and addition
vanishes uniformly in the limit.  Hence, the estimate \eqref{diffest} holds
for $\xi_\nu$ by comparability of geodesic exponentiation with addition in the local model.   Away from the neck $u_\nu$ converges to $u_{\delta_\nu^{\pre}}$ uniformly in all derivatives.  Thus the $k,p,\lambda$ norm of $\xi_\nu$ tends to zero 
on $S^{\delta_\nu}$ as $\nu \to \infty$.  It follows that the map $u_\nu'$ is the unique solution appearing in the implicit function theorem. 
\end{proof}

\subsection{Deformation to split form} 

\label{splitform}

As in Charest-Woodward \cite{flips} we consider a deformation of the  matching conditions between levels to split form.  In this limit, 
the moduli spaces of treed buildings become products, rather than fiber products, of moduli spaces of their treed levels.   This deformation is similar to the theory introduced by Bourgeois \cite{bo:com}.   The resulting Fukaya algebra is homotopy equivalent to the original.   The deformation replaces the matching conditions at the Reeb
orbits and chords with deformed matching conditions using deformations of the diagonal.  
Fix a Morse-Smale pair on $Y = Z/S^1$ and $L_Z := L \cap Z$ with Morse functions
\[ h_Y: Y \to \R, \quad h_Z : L_Z \to \R  .\]
In our special case considered in this paper, will have $L_Z  = S^{n-1}$ and the Morse function can be taken to be projection on an axis, so that in particular $h$ has two critical points.  

\begin{definition} \label{def:tlbuild} Let $\wp \in [0,\infty]$. A {\em treed $\wp$-building $C$}
is a  treed building $C^{\pre}$ with levels $C_0,\ldots C_k$ so that 
$C_i$ is joined to $C_{i-1}$ by segments $T_i$ of length $\wp_i$ so that the sum 
\[ \sum_{i=1}^{k}  \wp_i = \wp \] 
is equal to $\wp$. 

\vskip .1in \noindent 
A {\em holomorphic treed $\wp$-building} with $k$ levels is a pair $(C,u: S \to \XX)$
where $C$ is a treed $\wp$-building and $u: C \to \XX$ is a collection of
levels $u|C_i: C_i \to \XX[2k]_{2i}$ and gradient segments $u|T_i :T_i \to L_Z \subset \XX[2k]_{2i+1}$ of $h_Y$ resp. $ h_Z$ connected Reeb orbits resp. chords satisfying the obvious matching conditions at the intersection points  $T_i \cap C_i, T_i \cap C_{i+1}$.
\end{definition}

\begin{figure}[ht]
  \includegraphics[height=1.5in]{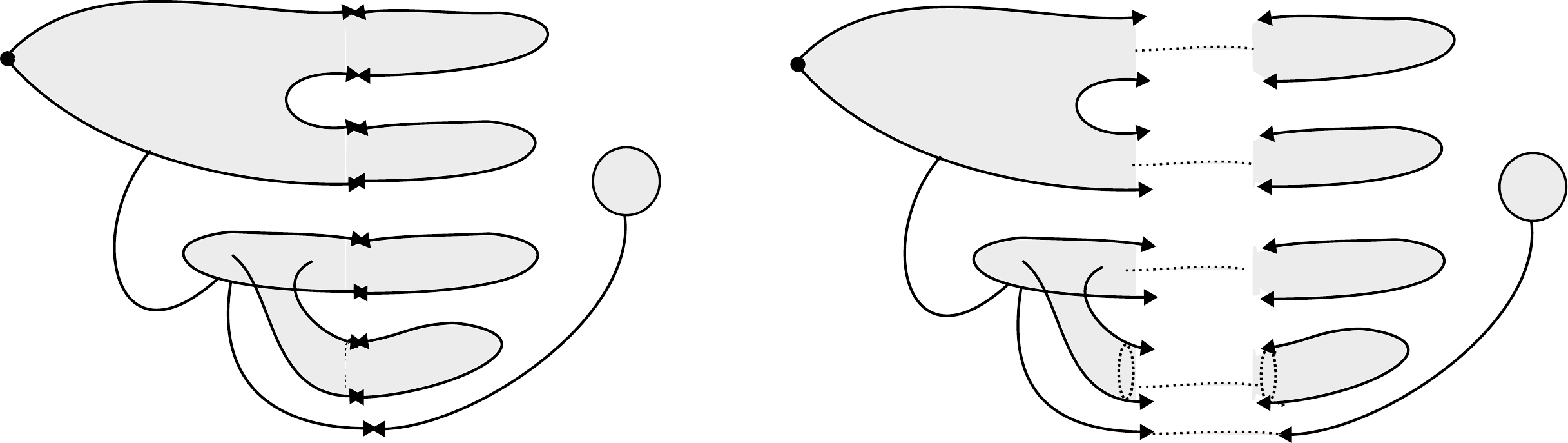}
\caption{Replacing nodes with segments: A $\wp$-building is on the right} 
\label{interleaves} 
\end{figure}

That is, each node connecting levels is replaced by a gradient trajectory as in Figure
\ref{interleaves}.  See Charest-Woodward \cite[Chapter 8]{flips}.  We now repeat the construction of the broken Fukaya category using  $\wp$-buildings rather than treed buildings.   The moduli space of holomorphic $\wp$-treed buildings is denoted 
$\M^\wp(\XX,\phi)$.  The same regularization procedure as in the previous section leads to regularized moduli spaces $\M^\wp(\XX,\phi,\DD)$ with good compactness properties  for the components of expected dimension at most one.     Using moduli spaces of buildings, we may define a broken analog of the Fukaya algebra.  The underlying vector space $CF(\XX,\phi)$ is defined in the same way as $CF(X,\phi)$, but in equation
\eqref{eq:highercomp} the count of elements of $\M_{\bGamma}(X,\phi,D)$
is replaced by a count of elements of $\M_{\bGamma}(\XX,\phi,\DD)$.
Counts of $\wp$-treed buildings  lead to a family of broken Fukaya algebras  $CF^\wp(\XX,\phi)$.

\begin{remark} 
In the case of infinite breaking parameter, each moduli space of buildings is a product
of the moduli spaces of  levels in the following sense.
Each holomorphic building $(C,u: S \to \XX[k])$ breaks up into a collection
of pairs 
\[ (C_i,u_i: S_i \to \XX[k]_{l_-(i),l_+(i)} ) \] 
where $\XX[k]_{l_-(i),l_+(i)} \subset \XX[k]$ is the union of components in the decomposition 
\eqref{XXk} between levels $l_-(i),l_+(i)$.
We denote in particular $\XX_{\subset}[k] \subset \XX[k]$
the union of the components except the last $X_\supset$, 
and similarly for $\XX_{\supset}[k]$.
Thus, in particular an $\infty$-level 
may be a pair $(C,u:S \to \XX_{\subset}[k])$ which consists of a map to $X_\subset$ and some collection of maps 
to $\R \times Z$ with matching conditions. 
For a type of building $\bGamma$ and collection of constraints $\Sigma$ the moduli space 
of treed buildings $\M_{\bGamma}^\infty(\XX,\phi,\DD)$ of labeled
type $\bGamma$ is  a product 
of the moduli space of its labeled treed levels 
\[ \M_{\bGamma}^\infty(\XX,\phi,\DD) = \bigcup_{(\Sigma_i)} \prod_{i=0}^k \M_{\bGamma_i}(\XX,\phi,\DD,\Sigma_i) .\] 
Here the union is over possible labellings $\Sigma_i$, representing the collection of constraints given by the (un)stable manifolds of the Morse functions chosen on $\Pi$ and $Y$ for the inter-level edges and the cellular constraints for each boundary leaf.    We adjust our 
terminology and call the elements $u_i$ of $\M_{\bGamma_i}(\XX,\phi,\DD,\Sigma_i)$
{\em levels}.  Each $u_i$ further decomposes into {\em sublevels} $u_{i,j}$ with domain $S_{i,j} \subset S_i$ mapping into some $\XX[k]_{l(i,j)}$. 
\end{remark}

\subsection{Homotopy equivalences}

We consider various kinds of homotopy equivalences of \ainfty 
algebras involving broken Fukaya algebras in this section.  The first Theorem \ref{thm:htpy} is an \ainfty equivalence which arises in the limit of infinite neck length under the neck stretching limit; this is essentially the same as considered
in Charest-Woodward \cite{flips}.  
 For this theorem, the Lagrangian
boundary condition is required to be cylindrical on the neck. 
Suppose that $\phi: L \to X$ is a Lagrangian boundary condition
that is cylindrical in a neighborhood of a hypersurface
$Z \subset X$.  We denote by $m_d^\tau$ the composition maps on
$CF(X,\phi)$ associated to the neck-stretched almost complex
structure $J^\tau \in \J(X)$.

\begin{theorem} \label{thm:htpy}  
The maps $m_d^\tau$ have a limit $m_d^\infty$ as 
  $\tau \to \infty$ equal to the composition map $m_d$ for the algebra 
  $CF(\XX,\phi)$.  The broken Fukaya algebra 
  $CF(\XX,\phi)$ is homotopy-equivalent to $CF(X,\phi)$.  Similarly, 
  for any breaking parameter $\wp$, the broken Fukaya algebra 
  $CF(\XX,\phi)$ is homotopy equivalent to $CF^\wp(\XX,\phi)$.
  \end{theorem}

The proof was given in \cite[Chapter 8]{flips} for the case that the Lagrangian 
does not pass through the neck; the proof is the same in the case here. 
We summarize the argument for completeness, which uses counts of quilted disks.
A quilted treed disk is a treed disk  a collection of disk components
$S' \subset S$ equipped with {\em quiltings}, meaning circles
in $S'$ intersecting a boundary component exactly once.  These 
components are called {\em quilted components} and, in the treed context, the lengths $\ell(e)$ of edges connecting these components
satisfy a system of equalities, if the number of quilted components
is greater than one.  The composition of these homotopy equivalences converges to a homotopy equivalence with the broken Fukaya algebra. 
For any energy bound $E$, the terms in 
the homotopy-equivalence $\zeta_\tau$ relating the neck-stretched alost complex structures   with coefficient $q^{A(u)}, A(u) < E$ vanish for sufficiently large
 $\tau$ except for constant disks.  It follows that there exist limits of the successive  compositions of the homotopy equivalences.  For $N,\tau \in \Z_{> 0}$
 consider the composition
\[ \zeta_{N,\tau} := 
\zeta_{N +   \tau - 1} 
\circ
 \zeta_{N+\tau - 2}
 \circ \ldots \circ 
\zeta_N 
:
 CF(X^{N},\phi) \to CF( X^{N+\tau},\phi).
\]
Because of the bijection in Theorem \ref{thm:bthm}, the limit
\[
\zeta_N = \lim_{\tau \to \infty} \zeta_{N,\tau}: CF(X^{N},\phi) \to \lim_{\tau \to
  \infty} CF(X^{N+\tau},\phi) \]
exists.  Similarly, the limit
\[ \psi_N = \lim_{\tau \to \infty} \psi_{N,\tau}, \quad \psi_{N,\tau}
:= 
\psi_{N+\tau} \circ \psi_{N+\tau-1} \circ \ldots \circ \psi_{N} \]
exists.  The composition of strictly unital morphisms is strictly
unital, so the composition $\psi$ is strictly unital mod terms
divisible by $q^E$ for any $E$.

The limiting morphisms are also homotopy-equivalences.  Let $h_\tau,g_\tau$
denote the homotopies satisfying
\[\zeta_{\tau} \circ \psi_{\tau} - \Id = m_1(h_\tau), 
\quad \psi_\tau \circ \zeta_\tau - \Id = m_1(g_\tau), \]
from the homotopies relating $\zeta_\tau \circ \psi_\tau$ and
$\psi_\tau \circ \zeta_\tau$ to the identities in \cite[Section
1e]{se:bo}.  In particular, $h_{\tau+1},g_{\tau+1}$ differ from
$h_\tau,g_\tau$ by expressions counting {\em twice-quilted} disks.
For any $E > 0$ and $\tau$ sufficiently large, all terms in
$h_{\tau+1} - h_\tau$ are divisible by $q^E$.  It follows that the
infinite composition
\[ h_N = \lim_{\tau \to \infty}  h_{N,\tau}, \quad g_N =  \lim_{\tau
  \to \infty} g_{N,\tau} \]
exists and gives a homotopy-equivalence between
$ \zeta_N \circ \psi_N$ resp. $\psi_N \circ \zeta_N$ and the
identities on  $CF(\XX,\phi)$ and $CF(X,\phi)$.  The proof of homotopy equivalence
with $CF^\wp(\XX,\phi)$ is similar.

\section{Holomorphic disks bounding the handle}
\label{handlesec} 

In this section, we review some results of Fukaya-Oh-Ohta-Ono
\cite[Chapter 10]{fooo} on the moduli spaces of holomorphic disks with
boundary in the local model.   The main result is Proposition 
\ref{prop:bijprop}, which gives a correspondence, up to repetition of codimension one inputs, between rigid maps in the local model with surgered and unsurgered boundary condition (after adding a longitudinal constraint, in the
case of wrong-way corners.)

\subsection{Classifying disks with a single end}
\label{sec:mintype} 

We first classify disks with a single end.  Let $\gamma(t) = t + i 2 \eps$ be the standard
path and $\M_{\bGamma}(\phi_\gamma)$ denote the
space of holomorphic maps $u: S \to X= \CC^n$ with boundary condition
in ${\phi}_\gamma: {H}_\gamma \to X$ of some type of map
${\bGamma}$.  The target $X = \CC^n$ is naturally a cylindrical-end
manifold with end modelled on a cylinder $\R \times Z$ on
the unit sphere $Z = S^{2n-1}$ defined using coordinates
$q_j + i p_j, j = 1,\ldots, n$ on $\CC^n$ by
\[ Z = \{ q_1^2 + p_1^2 + \ldots + q_n^2 + p_n^2 = 1 \} .\]
The Reeb flow on $Z$ is periodic with period $2\pi$ and the quotient
$Z/S^1$ is a complex projective space
\[ Y = Z/S^1 \cong \C P^{n-1} .\] 
The handle Lagrangian ${H}_\gamma$ defines a Lagrangian in the
projective space $\C P^n$, whose intersection with the divisor at
infinity is $\R P^{n-1}$.   The Reeb chords from $\R^n$ to $i \R^n$ (or vice-versa)  through
$0 \neq (a_1,\ldots,a_n) \in \R^n$ are classified by half-integers
$m \in \Z/2$ and are of the form 
\[ \thorn_{m,\ul{a}}(t) = e^{m \pi it/2 }(a_1,\ldots, a_n) . \]
Consider the case that $\Gamma$ is a type of configuration consisting of a 
disk $S$ attached to
single leaf $T$ at a node $w \in S$.  The following classification of curves with {\em right-way} and {\em wrong-way} corners is a modification of Fukaya-Oh-Ohta-Ono
\cite[Theorem 60.26]{fooo}.

\begin{definition}  \label{def:mintype} 
Let $\Gamma$ be a type of domain
  $S$ with a single strip-like end $e \in \mE(S)$.   Let $\bGamma_+$
  resp. $\bGamma_-$ be a type of finite-energy map 
  given by  sections of the Lefschetz
  fibration $\pi: \CC^n \to \C$ over a half space bounding 
  $\phi_\gamma$ asymptotic to a Reeb chord of angle change $\pi/2$ from $\R^n$ to $i \R^n$ resp. $i \R^n$  to $\R^n$.  
  We say that the map types $\bGamma_\pm$ are {\em minimal types} and 
 all other map types are {\em non-minimal}.  Let  $\hat{\bGamma}_\pm$   be the type obtained from the minimal types $\bGamma_\pm$ by adding a boundary 
  leaf $e \in \Edge(\hat{\bGamma}_\pm)$.  Let 
  \[ \ev : \M_{\hat{\bGamma}_-}({\phi}_\gamma) \to
  {H}_\gamma \times  S^{n-1}
\]
denote  the combined evaluation map for the leaf and end.  
\end{definition}

\begin{proposition} 
  \label{localspaces} \label{prop:bijprop}
  For $\gamma$ be the standard path $t \mapsto t + i 2 \eps, \eps > 0$
  and  $J_\Gamma = J_0$
 the standard complex structure, the maps of type $\bGamma_\pm$
  and $\hat{\bGamma}_\pm$ are regular and the following hold: 
\begin{enumerate} 
\item \label{rightway} {\rm (Right-way corners)} Evaluation at
  infinity \eqref{evinfty} defines a diffeomorphism 
\[ \M_{\bGamma_+}(\phi_\gamma) \to S^{n-1}, \quad u \mapsto \thorn_e(0). \]
\item \label{wrongway} {\rm (Wrong-way corners)} Evaluation at
  infinity \eqref{evinfty} defines a map
  \begin{equation} \label{fiberbundle} \M_{\bGamma_-}({\phi}_\gamma)
    \to S^{n-1}, \quad u \mapsto \thorn_e(0)  \end{equation}
  giving $\M_{\bGamma_-}({\phi}_\gamma)$ the structure of an $S^{n-2}$
  bundle over $S^{n-1}$ diffeomorphic to the unit sphere bundle
  $T_1 S^{n-1}$ in $T S^{n-1}$. For generic 
$a,c \in S^{n-1}$, the inverse image $\ev^{-1}(\R \times \{c \} \times \{a \})$
is a single transverse point.
\end{enumerate}
Furthermore, the homology classes of maps of type $\bGamma_\pm$ are primitive.  In the dimension two case $\dim(H_0) = 2$, the orientations of the two points in any fiber of \eqref{fiberbundle} agree for the trivial relative spin structure. 
\end{proposition}

\begin{proof}  We adopt a proof similar to Seidel's
  computation in \cite{se:lo}, which studied a boundary value problem
  for sections of a Lefschetz fibration with Lagrangian boundary
  condition obtained by parallel transport of the vanishing cycle
  around a circle, rather than a line considered here.

  We compare the indices of the map with its projection to the base of
  the standard Lefschetz fibration.  Let $u: S \to X = \CC^n$ be a map
  with boundary in $H_\gamma$.  The composition $\pi \circ u$ of $u$
  with the Lefschetz fibration $\pi:\CC^n \to \C$ of \eqref{slef}
  produces a map $\pi \circ u$ from $\HH$ to $\C$ with boundary
  condition $(\pi \circ u)(\partial S) \subset \R +i 2\eps$.  The map
  $\pi \circ u$ is an isomorphism from $\HH$ to $\HH + i 2 \eps$ by
  assumption.  After composing on the right with the shift
  $z \mapsto z + i 2 \eps$ and an automorphism of $\HH$, the map $u$
  becomes a section of the Lefschetz fibration:
  \[ \pi \circ u(z) = z , \quad \forall z \in \HH + i 2 \eps. \]
  Thus the components $u_j, j = 1,\ldots, n$ of the map
\[ u:(\HH,\partial \HH) \to (\CC^n, H_\gamma) \] 
  satisfy equations 
  \[ u_j(z) \in (z + i 2 \eps)^{1/2} \R, \quad z \in \R.\]

  The rank one problems in the previous paragraph are easily solvable.
  A change in sign of $\eps$ is equivalent to an interchange
  of the types $\bGamma_\pm$, so it suffices to consider
  the case of maps whose image is half-plane above the line $\on{Im}(z) = \eps$
  and consider the cases $\eps > 0$ and $\eps < 0$ respectively.  The components $u_j$ are solutions to a rank one boundary value  problem of index zero resp. one in the case $ \eps > 0$ resp.  $ \eps < 0$.  
    Each component $u_j$ of $u$ must be of the form for
  $z \in \HH $
  \begin{equation} \label{ujs} u_j(z + i 2 \eps) = \begin{cases} a_j (z + i
      2 \eps)^{1/2} & \eps > 0
      \\
      (a_j z + b_j) ( z -  i 2 \eps )^{-1/2} & \eps <
      0 \end{cases} \end{equation}
  for some $ a_j \in \R_{ > 0}$ resp. $a_j \in \R_{> 0}, b_j \in \R$.
  One can check explicitly that each such $u$ is a solution to
  the given boundary value problem: In the first case $\eps > 0$ the
  map has the required boundary values by inspection while in the second
  case we have for $ x \in \R$,
\begin{eqnarray*} 
 u_j(x + i 2 \eps)(x + i 2 \eps )^{-1/2}  &=& ( a_j x + b_j ) (x + i 2 \eps )^{-1/2} 
( x -   i 2 \eps )^{-1/2}
\\ &=& (a_j x + b_j ) (x^2 + 4 \eps^2)^{-1/2} \in \R .\end{eqnarray*} 
%

The constants are fixed by requiring that the given map is a section
of the Lefschetz fibration over its projection to the base.  Solving
for the condition $\pi u(z) = z$, that is, $u(z)$ is a section of the
Lefschetz fibration, we obtain
\begin{equation} \label{cases} \begin{cases}  a^2 = 1 & \eps > 0 \\ 
 a^2 = 1, \quad a \cdot b = 0, \quad b^2 = \eps^2 & \eps < 0 .\end{cases} \end{equation}
Indeed, if $\eps < 0$ then
\begin{eqnarray}  \pi u(z) = z 
&\iff&   (a ( z - i 2 \eps ) + b)^2  = (z - i 4  \eps )  z  \\
\label{these} &\iff&  \left( \begin{array}{rcl} a^2 &=& 1 \\
 2 a \cdot b - 4 a^2 \eps i &=& - 2 \eps i \\
 - 4 \eps^2 a^2 -
4 \eps i a \cdot b + b^2
&=& 0 \end{array} \right) .  \end{eqnarray}
The equations \eqref{these} are equivalent to the equations
\[ a^2 = 1, \quad a \cdot b= 0, \quad b^2 = 4 \eps^2 .\]

A similar computation computes the kernel of the linearization.  In
the first case $\eps > 0$, the kernel at $a$ is the set of solutions
$a'$ to $a' a = 0$, and so has dimension $n-1$. In the second case
$\eps < 0$, the kernel of the linearization at $(a,b)$ is the set of
solutions $(a',b')$ to
\[ a' \cdot a =0 , \quad a' \cdot b + a \cdot b' = 0, \quad b \cdot b' = 0 \]  
and so dimension $2n - 3$.

An index computation implies that the cokernel is trivial.  Indeed, 
for $\eps > 0$ the bundle $u^* (\CC^n \to \C)$ is a trivial symplectic
fibration.  It follows the vertical part of the index is equal to the dimension of the boundary condition, that is, $n-1$.
On the other hand, if $\eps < 0$ then the index problem is related to
that obtained by a connect sum with the index problem over the disk,
which has index $2n-3$ by Seidel \cite[Proof of Lemma 2.16]{se:lo}.
Since the horizontal index is the same as the dimension of the space
of automorphisms of the domain, the total index is $2n-3$ as well.
Triviality of the cokernel implies that the moduli spaces are
transversally cut out by the equations \eqref{cases}.  The equations give a
sphere of dimension $n-1$ in the first case, and fibration in the
second case with spherical fibers of dimension $n-2$.

For \eqref{wrongway}, it remains to prove the claim on the intersection with
a generic line on the handle.  Given 
\[ c,a \in S^{n-1}, \quad c \neq a, -a \] 
there exist unique $x \in \R $ and $b \in S^{n-1}$ with $ a \cdot b = 0$ such that
\[ \frac{ u(x+ i 2\eps) }{  |u(x + i 2 \eps)|} =  \frac{ a x + b }{
 ( (ax)^2 + b^2)^{1/2} }  = c . \] 
Indeed, the set of points  
\[ \left\{  ( a x + b ) / \Vert ax + b \Vert, \quad x \in \R ,b \in \on{span}(a)^\perp \right\}  \] 
is the complement of the two poles $a,-a$ in $S^{n-1}$.  The
claim follows.

To prove the claim about primitivity, we classify the possible
homology classes.  The second relative homology group
$  H_2(\C P^n , \ol{H}_\gamma) \cong H_2(\C P^n - \{ 0 \}, \ol{H}_\gamma) $
can be computed by Mayer-Vietoris.  Write 
$\CP^n = \CC^n \cup \CP^{n-1}$ and consider the cover
of $ \C P^n$ by the open sets $  (\CC^n - \{ 0 \}) $ and $( \C P^n - B_R(0) ) $ 
\label{rep:closedparen}
where $B_R(0)$ is a ball around $0 \in \CC^n$ of radius $R$.  The classes
corresponding to the first homology of the intersection are generated
by disks in the line $\C$ with boundary in $H_\gamma \cap \C$ that
have area $\pi/2 \pm A(\eps)$.  The remaining classes arise from
the classes of disks and spheres in $\C P^{n-1}$ with boundary in
$\R P^{n-1}$, which have areas equal to multiples of $\pi$.  It follows that
there are no decompositions of the classes with areas
$\pi/2 \pm A(\eps)$ into classes with positive smaller areas, so that
the homology classes of the maps in the Proposition are primitive.

To prove the claim on orientations in the dimension two case, we must
compare the contributions from the two points $u,u'$ in each fiber of
the fibration of \eqref{fiberbundle}.  The orientations
$o(u),o(u')$ may be compared by deforming the Lefschetz fibration
by bubbling off a disk containing the critical value of the Lefschetz fibration as in Seidel \cite{se:lo}.  By the computation in \cite[Proof of Corollary 4.31]{wo:ex} the orientations of the two different elements $u,u'$ in a single fiber agree.
\end{proof}

The areas of the disks on the handle and on the self-transverse
Lagrangian are related by the area correction in Definition
\ref{lsurg} and indicated (conceptually; the graph does not exactly
match the definition) in Figure \ref{tri}.  As in the Introduction
we denote by $\phi_\eps$ resp. $\phi_0$ the surgered resp. unsurgered
immersion.  

\begin{lemma}\label{lem:ncorners} Suppose that $u_0, u_\eps: S \to X$ are maps with
  boundary in $\phi_0$ resp. ${\phi}_\eps$ that are equal except in a
  neighborhood of a self-intersection point
  $x \in \cI^{\on{si}}(\phi_0)$ as in Figure \ref{tri}; and suppose that in a neighborhood of the surgery, the map $u_\eps$ is obtained
  by replacing a right-way resp. wrong-way corner in $u_0$ by its smoothing
  above.  Then the areas of $u_\eps$ and $u_0$ are related by 
\[ A(u_\eps) = A(u_0) - (\kappa - \ol{\kappa}) A(\eps) \]  
where $\kappa \in \Z_{\ge 0} $ resp. $\ol{\kappa} \in \Z_{\ge 0}$ is
the number of times $u_0$ passes through $x$ resp. $\ol{x}$.
\end{lemma} 

\begin{proof} We compute the difference in areas  using Stokes'
  theorem.  The symplectic form $\omega_0$ on $\CC^n$ is exact with
  primitive \label{rep:primitive}
\[ \alpha_0 = \sum_{j=1}^n \hh ( q_j \dd p_j - p_j \dd q_j ), \quad \dd \alpha_0
= \omega_0. \] 
The restriction of $\alpha_0$ to the Lagrangian branches $\R^n, i\R^n$
vanishes.  The maps $u_0,u_\eps$ agree away from the corner, 
and the difference of $u_0,u_\eps$ in the 
relative homology with respect to $\phi_0 \cup \phi_\eps$
is the class of a map 
\[ v: \R \times [0,1] \to \CC^n \] 
bounding $H_0,{H}_\eps$ and constant outside a compact set $[-T,T] \times [0,1]$ with boundary
\[ \gamma_0:[-T,T] \to \phi_0(H_0), \quad \gamma_\eps: [-T,T] \to
{\phi}_\eps({H}_\eps). \]
The first path $\gamma_0$ travels from the negative to positive
branches of $\phi_0(H_0)$, while the second $\gamma_\eps$ travels
along ${\phi}_\eps({H}_\eps)$ in the same direction, as in
Figure \ref{tri}.  By Stokes' theorem, the area of $v$ is
\begin{equation} \label{Aeps} A(\eps) = \int_{\R \times [0,1]} v^* \omega=
  \int_{[-T,T]} \gamma_\eps^* \alpha_0 - \gamma_0^* \alpha_0.
\end{equation}
independent of the homotopy class of the map $v$.  \end{proof}

\subsection{Ruling out disks with large angle in the unsurgered handle}
\label{rulingout2} 

 In this section we show that the only rigid curves for the unsurgered handle are those appearing in
Proposition \ref{localspaces}, that is, those curves with a single strip-like end
asymptotic to a Reeb chord of minimal length.   Given a map type of punctured
surface $\bGamma$, denote by $\M_\bGamma(\phi_0)$ the moduli space of 
holomorphic treed disks bounding $H_0= \R^n \cup i \R^n$ of type $\bGamma$.  Denote by $e(\black)$ resp. $e(\white)$ the number of Reeb orbits resp. chords at infinity,
and $d(\white)$ the number of boundary leaves in total, so that $f(\white) = d(\white) - e(\white)$
represents the number of boundary leaves not corresponding to Reeb chords.  We have a natural evaluation map 
\[  \ev:  \M_{\bGamma}(\phi_0) \to  H_0^{f(\white)}  
\times(S^{n-1})^{e(\white)} \times \C P^{e(\black)} \]
which assigns to any configuration $(C,u: C \to X)$ the projection of the limiting Reeb orbits and chords, and evaluates the map at the intersection of the remaining leaves $T_e $ with $S$. As in \eqref{sigma} let 
\[ \Sigma \subset
H_0^{f(\white)}
\times(S^{n-1})^{e(\white)} \times \C P^{e(\black)} \] 
be a submanifold (later, $\Sigma$ will be the image under the evaluation map of the "outside pieces" in the symplectic field theory decomposition) intersecting the evaluation map for $ \M_{\bGamma}(\phi_0)$ transversally. 
As in \eqref{constrainedM} define
\[ \M_{\bGamma}(\phi_0,\Sigma) = \ev^{-1}(\Sigma) \]
denote the moduli space of maps with the given constraints. 
We wish to know in what conditions $\M_{\bGamma}(\phi_0,\Sigma)$ may be rigid, that is, of expected dimension zero. 

 We recall the  classification of holomorphic maps of disks
 to the complex projective line. Let $L \subset S^2$ be an embedded circle.    For the case $\gamma(t) = t+ i 2 \eps$, the Blaschke classification 
 in \cite{chooh:toric} implies that after a change of coordinates any holomorphic disk $u$ bounding $L$ is of the form
  \begin{equation} \label{blaschke} u (z) = \left[ c_-
      \prod_{i=1}^{d_-} \frac{z - a_{i,-}}{1 - z \ol{a}_{i,-} } , c_+
      \prod_{i=1}^{d_+} \frac{z - a_{i,+}}{1 - z \ol{a}_{i,+} }
    \right] \end{equation}
  for some integers $d_\pm$ and constants $c_\pm$ with 
\[ |c_\pm| = 1, \quad  a_{i,\pm} \in \C, \quad d_\pm \in 
\Z_{\ge 0} .\]
  Returning to the case of arbitrary paths, we introduce the following notation
  for topological type. Identify $H_2( \C P^1, \R P^1) \cong \Z^2$.  Let
  \[ d = (d_-,d_+) \in \Z^2 \] 
denote the {\em degree} of the composition
  $\pi \circ \ol{u}$ given by the image $u_* ([S,\partial S])$
  in the relative homology group $H_2( \C P^1, \R P^1)\cong \Z^2$.     For each point $z \in \ol{S}$
mapping to infinity $\infty$, let $m_\white(z),m_\black(z)$ denote
the ``multiplicity'' of the corresponding Reeb chord or orbit,
indicating how many Taylor coefficients vanish at $z$.  Thus
$m_\black(z)$ are the intersection numbers with the divisor at
infinity $Y$ at $z$, while $m_\white(z) = k$ if the Reeb chord
represents an angle change of $k\pi$ for some $k \in \Z_{\ge 1}$.

\begin{lemma} \label{lem:indlem}
For a holomorphic map $u: S \to \C$ bounding $\R$ asymptotic to Reeb 
chords and orbits, the Fredholm index of the linearized operator $D_u$ is 
  \begin{equation} \label{indexdimone} \Ind(D_u) = 1 + 2(d_- + d_+)
 - \sum_{z \in \ol{u}^{-1}(\infty) \cap \partial S}(  m_\white(z) - 1) 
  - \sum_{z \in \ol{u}^{-1}(\infty) \cap \on{int}(S)}  2 (m_\black(z) - 1) 
     \end{equation} 
  \end{lemma}

  \begin{proof} The Maslov index  of the map $\ol{u}$ 
  is $I(\ol{u}) = 2(d_- + d_+)$.  The Fredholm index
  is $\Ind(D_{\ol{u}}) = 1 + 2(d_- + d_+)$ by Riemann-Roch.  The tangency requirements lower the index by $  m_\white(z) - 1  $ resp. $2 (m_\black(z) - 1) $ 
  at each   Reeb chord or orbit, as in \cite[Lemma 6.7]{cm:trans}.
  \end{proof}

\begin{lemma} \label{lem:quadrant} Let $n \ge 2$.  Assume that the standard complex structure makes every punctured holomorphic disk or sphere $u: S \to \C^n$  bounding $\R^n \cup i \R^n$ regular. Then any rigid such map  
is an isomorphism from a disk onto a quadrant $(\R_{\ge 0} 
+ i \R_{\ge 0} )v$ in $\span(z) \cong \C \subset \C^n$
for some $v \in \R^n \subset \C^n$.
\end{lemma}

\begin{proof}  We use the fact that the boundary condition is invariant under dilation.  The action of $\lambda \in \R^\times$ on $\C^n$ 
induces a one-parameter family of punctured curves $\lambda u$ with the same evaluation at infinity.    Suppose $u$ is rigid.  The dilation $\lambda u$ must equal the composition $u \circ \phi_\lambda$ for some automorphism  $\phi_\lambda: S \to S$.
  Thus, in particular $S$ has a non-trivial automorphism, at most two strip-like or cylindrical ends, and image contained in a line in $\C^n$, and we are in the 
  one-dimensional case considered in \eqref{indexdimone}.   
  For holomorphic spheres, the degree equals the signed count of points in a fiber $u^{-1}(\zeta)$  over a point $\zeta$ near $\infty$,
  so each puncture with multiplicity $m(z)$ contributes $m(z)$ to the degree.   Hence,
  for spheres  $u$ we have 
  \[ \Ind(D_u) \ge 2+ (d_- + d_+) . \]  
  Since the automorphism group is at most dimension four,
  rigidity forces $(d_-,d_+) = (1,1)$.
  In this case, the map 
  $u(z) = v_1 z $ is linear with first derivative given by some $v_1 \in \C^n$. 
  The map $u$ admits a deformation $u_t(z) = v_1 z+ t v_0$ with $v_0$
  not in the span of $v_1$, and so is not rigid.

Thus the domain is a disk, and we claim that the map is linear.
   Since the image of the boundary is $\R$-invariant, $u(\partial S)$ must intersect infinity and so $u$ has at least one strip-like end.  If $u$ has only one strip-like end (mapping either to $0$ or to $\infty$)  then $u$ is a disk in $\C P^1 \cong \on{span}(c) \cup \{ \infty \}$ bounding 
  $\R P^1$ or $i \R P^1$. By \eqref{indexdimone}, $u$ cannot be rigid.
 Thus $S$ has two strip-like ends, and by homogeneity one must map to
the corner $\{ 0 \} = \R^n \cap i \R^n$.  The map $u \cdot u$ has real boundary conditions, so $u$ is of the form 
\[ u(z) = \sqrt{ a(z) } v, \quad \on{Im}(z) \ge 0 \] 
for some $v \in \R^n$ and polynomial $a(z)$.  Such maps are rigid only if $a(z)$ has degree one, since otherwise the variation in the sub-leading coefficients in $a(z)$ preserving  the condition $a(0) = 0$ produces a non-trivial variation of the map.
\end{proof}

\begin{lemma} \label{lem:singlequadrant}   Assuming the standard complex
structure on $\C^n$ is regular, for any rigid treed level $(C,u: C \to \C^n[k])$ bounding $\R^n \cup i \R^n$ with at least one component in $\C^n$, the domain $S$ is connected with image $u(S)$ equal to a quadrant in a one-dimensional subspace of $\C^n$.
\end{lemma} 

\begin{proof}  The general linear group acts on the set of disks with the given 
boundary condition.  Let $u: S \to \C^n$ be a rigid treed disk for the given constraint $\Sigma$ with components
$u_v : S_v \to \C^n, v \in \Ver(\Gamma)$.
Any deformation $(\xi_v)$ of the maps $u_v$ by, for example, 
a dilation  by $\lambda \in \R - \{ 0 \}$ extends to a dilation of $u$, 
by translating the remaining components $u_{v'}, v' \neq v$ by elements of 
$GL(n,\R)$ so that the matching conditions hold. 
It follows that each component $u_v$ must be rigid separately.
Each
$u_v$ is therefore an isomorphism onto a quadrant in some line $\on{span}(c) \subset \C^n$.  The attaching points $w_e$ of the edges $T_e$ to the surface part $S$ are necessarily invariant under the action of $\R^\times$. Each such $w_e$ must be a point
on the boundary $\partial S$ with $u(w_e) = 0 \in \C^n$.  That is, the non-constant components of $u$ must be quadrants, with adjacent components joined by nodes mapping to $0 \in \C^n$.  (For example, the domain $S$ of $u$ could consist of two components $S_{v_1}, S_{v_2}$, 
with the first component $S_{v_1}$ mapping
onto the first quadrant, and
the second $S_{v_2}$ mapping onto onto the second quadrant.) The edge length $\ell(e)$ of the 
edge $T_e$ connecting two such quadrants is free to vary, 
so the configuration $u$  cannot be rigid.
\end{proof} 

It remains to justify the regularity assumption, to which end we prove:

\begin{lemma} \label{lem:standard1} The standard complex structure on $\C^n$ makes all treed levels $(C,u: C \to \C^n[k])$ bounding the unsurgered handle $\R \cup i \R^n$ regular.  
\end{lemma}

 \begin{proof} We use the fact that the boundary value problem for the unsurgered problem splits into one-dimensional problems.  
 Let $u: S \to \C^n$ be a holomorphic map from a connected surface $S$  to $ \C^n$ asymptotic to some collection $\ul{\thorn}$ of Reeb chords and orbits.   The pair $(\C^n, \R^n \cup i\R^n)$ splits into one-dimensional problems $(\C,\R\cup i\R)$, each invariant under the action of dilation.   It follows that each summand has a non-trivial element
 of the kernel and vanishing cokernel.  
 Thus any such holomorphic map $u$ is regular.  

The case of disconnected domain requires an induction. 
     Let $u: S \to \C^n[k]$ be a treed level with surface
     components $S_v, v \in \Ver(\Gamma)$ and line segments
     $T_e, e \in \Edge(\Gamma)$; recall that the notation $\C^n[k]$
     means that $u$ has various components some of which map to neck pieces $\C^n - \{ 0 \}$.  The moduli space of configurations of type $\bGamma$ with any given set of finite edge lengths $\ell(e)$
     is transversally cut out.   Let $\M'_\bGamma(\phi_0)$ denote the moduli space of configurations with no matching conditions at the nodes.       It suffices to show  that the evaluation map $\M'_\bGamma(\phi_0) \to (S^{n-1})^{\# (T \cap S)}$ is transverse    to the matching condition.   This follows from an inductive
     argument:  If $e$ is the only finite edge
     adjacent to a vertex $v$ and $\Gamma(v) \subset \Gamma$ is the subgraph of edges containing $v$ then one sees from the previous paragraph that the map 
     $\M'_{\bGamma(v)}(\phi_0) \to L$ is a submersion at $T_e \cap S_v$.  Removing $v$ and $e$  one obtains a tree $\Gamma_1$ with fewer vertices and edges, and the claim follows from the inductive assumption of transversality for such trees. 
     \end{proof} 

\begin{remark} \label{lesrem}
Alternatively, one may prove regularity using a long exact sequence. 
Let $u: S \to \C^n$ be a punctured holomorphic disk or sphere
avoiding $0 \in \C^n$.   Let 
\[ p: \C^n - \{ 0\} \to \C P^{n-1} \] 
denote the projection.    Consider the short exact sequence of vector bundles
\[ u^* T^v \C^n  \to u^* T \C^n 
\to u^* p^*  T\C P^{n-1}  \] 
where the vertical sub-bundle $T^v \C^n$ is the kernel of $Dp$.  Denote by
$D_u^v, D_u^h$ the vertical and horizontal parts of the linearized
operator $D_u$.  The short exact sequence of complexes induces a long
exact sequence of kernels and cokernels denoted 
\begin{equation} \label{les}
\begin{diagram}
 \node{\coker(D_u^v)} \arrow{e}
    \node{ \coker(D_{u})} 
\arrow{e} \node{ 
\coker(D_u^h ) }
\\
    \node{\ker(D_u^v)} \arrow{e}
    \node{\ker(D_u)}
\arrow{e} \node{ 
\ker(D_u^h ) }
    \arrow{wnw}
 \end{diagram} . \end{equation} 
  As in Proposition \ref{prop:linremov}, the kernel and cokernel
  of $D_{p \circ u}$ may be identified with those of its
  extension $D_{\ol{p \circ u}}$ obtained by adding in the points at infinity
  along the cylindrical and strip-like ends.  It is a standard consequence
  of homogeneity of $\C P^{n-1}$ that the cokernel of such operators vanish:
  Since $ \ol{p \circ u}^* T \C P^{n-1}$ is generated by the image of $\lie{o}(n)$ in $ \ker(D_{\ol{p \circ u}})$ at any point, 
  $ \ol{p \circ u}^* T \C P^{n-1}$ must be a sum of a line bundles
  with boundary conditions with positive Maslov index as in Oh \cite{oh:rh}. 
  Since the cokernel of a Cauchy-Riemann operator on any one-dimensional problem 
  with positive Maslov index   vanishes \cite{oh:rh}, 
  the cokernel of $D_{p \circ u}$ vanishes.      Similarly, the vertical 
      bundle $u^* T^v (\C^n - \{ 0 \})$ has a section given by the action of $\R^\times$
      on $\C^n - \{ 0 \} $ which preserves the boundary conditions.  
           It follows
      that $\coker(D^v_u)$ also vanishes.  By the long exact sequence \eqref{les},
      $\coker(D_u)$ vanishes as well. 
\end{remark}

 We require similar results for ruling out holomorphic disks mapping to 
 $\C^n - \{ 0 \}$ with boundary in $(\R^n \cup i \R^n ) - \{ 0 \}$.
 Consider a map 
 $u$ from $S$ to $X_0 := \C^n - \{ 0 \} $ with boundary on $\phi_0: (\R^n \cup i \R^n) - \{ 0 \} \to \C^n - \{ 0 \} $ with multiple ends or non-minimal Reeb orbits.  

\begin{lemma} \label{lem:singlequadrant2} Suppose $(C, u: C \to \C^n - \{ 0 \} )$ is a treed map bounding $(\R^n \cup i \R^n) - \{ 0 \}$ and $\Sigma$ is a collection of constraints so that all cellular
constraints in $\phi_0$ are equal to $\sigma_1$.  Then $(C,u)$
consists of a single component with image in a fiber of $\C^n - \{ 0 \} \to \CP^{n-1}$.
\end{lemma} 

\begin{proof}  The proof is similar to that of Lemma \ref{lem:quadrant} and uses the
action of the group  $\R_{> 0}$ on $\C^n - \{ 0 \} $ by dilations preserving the boundary condition
$(\R^n \cup i \R^n ) - \{ 0 \}$ and \label{rep:preservelimits} preserving the limits at infinity.  Suppose the surface part $S$ of the domain has 
components $S_v$ connected by edges $T_e$.   Consider a component $u_v$  of $u$.   By dilation by a constant $c \in \R_{> 0}$ one obtains a new map $c u_v$ with the same boundary condition not isomorphic to the original map.  Given two strips $S_v, S_{v'}$
joined by an edge $T_e$ joined at points $z \in S_v, z' \in S_{v'}$,
and given $c$ the point $u( c z), u(c' z')$ are still connected to  Morse trajectory of the distance-to-zero function.  It follows that $S$ has a single component that is a  trivial strip. 
\end{proof}

\subsection{Ruling out disks with large angle in the surgered handle}
\label{rulingout} 

We now wish to show a similar statement for curves bounding the surgered handle, namely that rigidity forces  strip-like ends of minimal length.    We first introduce some notation.  Consider the natural projections
\[ \pi: \C^n  \to \C, \quad (z_1,\ldots, z_n) \mapsto z_1^2 +
  \ldots z_n^2 \]
and
\[ p: \C^n - \{ 0 \} \to \CP^{n-1}, 
\quad z \mapsto \on{span}(z) .\]
The null-cone in $\CP^{n-1}$ is 
\[ N := p( \pi^{-1}(0)) = \{ [z_1,\dots, z_n], z_1^2 + \ldots + z_n^2 =  0 \} .\]
Given a holomorphic map bounding the handle, we obtain maps by composition with either projection.  Namely let 
  $u: S \to \CC^n$ be a smooth map with boundary on the handle $H_\gamma$
  and avoiding $0$.  The composition $\pi \circ u$  maps to $\C$
  with boundary in $\gamma$, while the composition $p \circ u$ maps to $\CP^{n-1}$ with boundary in $\RP^{n-1}$.  We wish to compare the indices of the linearized
  operator of the map $u$ with that of its projections.  Because the map $u$ may limit to the null-cone along the strip-like ends, the composition $\pi \circ u$
  may have finite limits along any particular end.    To compare indices we introduce a long exact sequence.
We have a short exact sequence of bundles over $(S,\partial S)$
\begin{equation} \label{ses} 0 \to (T^v \C^n, T^v H_\gamma) \to
  ( T \CC^n, TH_\gamma) \to  ( T^h \C^n, T^h H_\gamma) \to 0 .\end{equation}
The short exact sequence of bundles \eqref{ses} induces a short exact sequence of complexes of $0$ and $0,1$-forms.    We denote by
$D_u^v, D_u^h$ the vertical and horizontal parts of the linearized
operator $D_u$, and  $\ti{D}_u^h$ the parametrized linear operator for the horizontal part. 
Consider the map 
\[ D \pi:  u^* TX \to (\pi \circ u)^* T\C. \]
The short exact sequence of complexes induces a long
exact sequence of kernels and cokernels given by 
\begin{equation} \label{eq:les3}
\begin{diagram}
    \node{\coker(D_u^v)} 
 \arrow{e} \node{\ldots} \node{}  \\
    \node{\ker(D_u^v)} \arrow{e}
    \node{\ker(D_u)} 
\arrow{e} \node{ 
 \ker(D_u^h ) }
    \arrow{wnw}
 \end{diagram} . \end{equation} 
 The long exact sequence equally holds with the unparametrized linear operator $\ti{D}_u,D_u^h$
 replaced with the parametrized linear operators $\ti{D}_u, \ti{D}_u^h$ whose domain does not allow variations of the conformal structure, by the same argument.

  We introduce convenient terminology corresponding to the ``multiplicity'' of the 
  map at infinity.  
  
  \begin{definition} 
  \begin{enumerate} \item For each point $z \in S$
mapping to infinity in $\ol{X}$, let $m_\white(z),m_\black(z)$ denote
the ``multiplicity'' of the corresponding Reeb chord or orbit,
indicating which eigenvector appears in \eqref{leading}.  Thus
$m_\black(z)$ are the intersection numbers with the divisor at
infinity $Y$ at $z$, while $m_\white(z) = k$ if the Reeb chord
represents an angle change of $k\pi$ for some $k \in \Z_{\ge 1}/2$.
\item  Define
\[ m^{\on{ex}}(z) = \begin{cases} 2 m_\white(z)  - 1 & \text{if $z$ represents a
    strip-like end, or } \\ 
 4m_\black(z) - 2 & \text{if $z$ represents a cylindrical end}. \end{cases}
\] 
\end{enumerate}
\end{definition}

\begin{proposition}  \label{prop:index} \label{posm} Let $u: S^\circ \to \CC^n$ be a finite
  energy holomorphic map bounding $H_\gamma$ whose evaluations $\ev_e(u)$ at cylindrical ends avoid the null-cone $N$.  The Fredholm  index of the horizontal part $D_u^h$ of the linearized operator 
  $D_u$ is
  by \eqref{indexdimone}
\[ \Ind(D_u^h) = 
1 + 2 (d_+
+ d_-) - \sum_{z \in \ol{u}^{-1}(Y)} m^{\on{ex}}(z)
  \]
while the index of the vertical part $D_u^v$ is 
\[ \Ind(D_u^v) = (n-1) + d_-( n-2) . \]
\end{proposition} 

 \begin{proof}
Denote the extensions of Proposition \ref{prop:linremov} by
$(( \pi \circ u)^* T \C, (\pi \circ u)^* T \gamma(\R))_c $
over $(\C P^1, \ol{\gamma(\R)})$.  The Maslov index of the horizontal part is 
\[ I ( ( \pi \circ u)^* T \C, (\pi \circ u)^* T \gamma(\R))_c = 
2 (d_+
+ d_-) - \sum_{z \in \ol{u}^{-1}(Y)} m^{\on{ex}}(z) \]
by \eqref{indexdimone}.
To compute the vertical part of the index, note that the map $u$ may
be viewed as a section of the pull-back of the Lefschetz fibration
$\pi: \CC^n \to \C$ under $(\pi \circ u)$.  The ``bubbling off
singularities'' computation in Seidel \cite[p. 253]{se:bo} implies
that each bubble containing a singularity contributes the Maslov index
of the corresponding holomorphic map to the disk, computed in
Seidel \cite[Proof of Lemma 2.16]{se:lo} to equal $n-2$.  
\end{proof}

\begin{lemma} \label{dimonereg}\label{horizlem} Let 
 $\gamma: \R \to \C$ be an asymptotically cylindrical path and 
$u: S \to \C$ a level  bounding $\phi_\gamma$.   Then the parametrized linearized operator $\ti{D}_{u}$
is surjective.  For the standard path $\gamma$, the unparametrized 
linear operator $D_u$ is surjective. 
\end{lemma} 

\begin{proof}  We first prove the last claim for the standard path.  The kernel and cokernel are 
  identified via Proposition \ref{prop:linremov} with a one-dimensional real boundary value problem.  Note that $(\CP^1, \RP^1)$ has a family of automorphisms $\phi_s$ preserving $\infty$.
  The derivative $\dds |_{s = 0} \phi_s \circ u$   gives a non-trivial element of the kernel $\ker(D_u)$. Hence  $\coker(D_u)$ vanishes, since in rank one either the kernel or cokernel of any Cauchy-Riemann operator vanishes by Oh \cite{oh:rh}.  

For arbitrary paths that are asymptotically cylindrical, the statement of the Lemma regarding the parametrized linear operator  is an instance of automatic regularity in dimension one as in Seidel \cite[Lemma 13.2]{se:bo}.
The cokernel of $ \ti{D}_{{u}}$
may be identified with the kernel of the adjoint operator.
Necessarily, any $\eta \in \coker(\ti{D}_{{u}})$ is perpendicular to variations
of the form $(J \dd {u} \alpha)^{0,1} $ produced by
variations 
\[ T_j \J(S) =\{ \alpha \in \Omega^0(S,\End(TS))  \ |  \ j \alpha + \alpha j
= 0 \} \]
of the conformal structure on $(S,\ul{z},\ul{z}')$, 
with notation from \eqref{zprime}.  By definition the
moduli space of curves is a slice for the action of the diffeomorphism
group on the space of conformal variations
\[ T_j \J(S) = T_S \M \oplus \Vect(S,\ul{z},\ul{z}') \]
with
\[ 
\Vect(S,\ul{z},\ul{z}') := \{ v \in \Vect(S) | v(z) =0 , \forall z \in
\ul{z} \cup \ul{z'} \} \]
where if $S$ unstable then $T_S \M$ is defined to be trivial.  The
image of the operator $\ti{D}_{{u}}$ is unchanged
if one extends the domain to allow all deformations: Let
$v \in \Vect(S)$ vanish at the points
$\ul{z},\ul{z}' \in {u}^{-1}(\infty)$.  The variation of the 
complex structure and map with respect to $v$ is given by 
elements
\[ \alpha(v) \in \Omega^0(S, \End(TS)), \quad L_v {u} \in
\Omega^0( S, {u}^* T\P^1) \]
with
\[ (J \dd {u} \alpha(v))^{0,1} + D_{{u}} L_v u = 0  .\]
Since ${u}$ has derivatives vanishing up to order
$m(z_j) $ at each $z_j \in \ul{z}, \ul{z}'$, the derivative
$L_\alpha {u}$ has derivatives vanishing up to order
$m(z_j)-1$.  Similarly, for the points
$z \in {u}^{-1}(\infty)$, if $u$ has derivatives up to order
$m(z)$ vanishing at $z$ then $\dd u$ has derivatives up to
order $m - 1$ vanishing at $z$.  Since $v$ vanishes at $z$, the
derivative $L_v {u}$ has derivatives up to order $m$
vanishing at $z$ as well.  Hence $L_v {u}$ defines an
element in the domain of $D_{{u}}$.  Thus, the term
$(J \dd {u} \alpha(v))^{0,1} $ lies in the image of
$ D_{{u}}$.  Since $\dd {u}$ is an isomorphism away
from the finitely many critical points of ${u}$, there are no
$0,1$-forms perpendicular to such variations
$(J \dd {u} \alpha)^{0,1}$ for all $\alpha$.  The claim follows. 
\end{proof}

\begin{lemma} \label{lem:standard2} For the standard complex structure on
  $\CC^n$ and a boundary condition given by an asymptotically cylindrical path  $\gamma: \R \to \C$, every holomorphic treed disk $u: C \to \CC^n[k]$
   with boundary on $H_\gamma$ not meeting $0 \in \CC^n$ is regular
  and the evaluation map at any point on the boundary $z \in \partial S$ has 
  surjective linearization in $T_{u(z)} H_\gamma$.
\end{lemma}

\begin{proof}  We will show that the vertical and horizontal parts of the linearized
operator are both surjective.  First let $\bGamma$ be a combinatorial type of map
with a single vertex $v$ and strip-like and cylindrical ends $S$.
Consider the moduli space ${\M}_{\bGamma}(\phi_\gamma)$ of holomorphic maps
$u: S \to X$ with boundary in $H_\gamma$.   By the long exact sequence, to show
  regularity it suffices to show that the cokernels of $D_u^v$
  and $\ti{D}_u^h$ vanish.  The
  homogeneity argument implies that the higher cohomology of the
  vertical part vanishes:  The action of $SO(n)$ on $ \CC^n$
  preserves the complex structure and Lagrangians $H_\gamma$.
  Given a holomorphic disk $u: S \to \CC^n$ with boundary on
  $H_\gamma$, one obtains an inclusion
   \[ \so(n) \to \ker(D_{ u^* T^v \CC^n}), \quad \xi 
   \mapsto \ddt |_{t=0} \exp(t \xi) u \] 
   by mapping each Lie algebra element $\xi \in \so(n)$ to the
   corresponding infinitesimal deformation of the map.  In particular, the evaluation map
\[ 
\ker(D_{ u^* T^v \CC^n}) \to T^v H_\gamma, 
\quad \xi \mapsto \xi(z) \]  
at any point $z \in \partial S^\circ$ is a submersion.   
By Lemma \ref{horizlem}, 
$\coker(\ti{D}_u^h)$ vanishes as well. 
The vertical part 
$(u^* T^v \CC^n, (\partial u)^* T^v H_\gamma)$ has spanning sections
at any point given by the action of $SO(n)$.  This fact implies that the
linearization of the evaluation map is surjective.

Similar arguments apply to the case of sublevels mapping to the neck piece,
that is, components $u_v$ mapping to $\CC^n - \{ 0 \}$
with boundary on $(\R^n \cup i \R^n) - \{ 0 \}$.  Given such a map, 
the projection $\pi \circ u_v$ is a map to $\C$ with boundary values on $\R$
with corners of $u_v$ mapping to zeros of $\pi \circ u_v$.  Such maps have surjective linearization by Lemma \ref{lem:standard1}.  

The previous paragraphs show that the moduli space 
$\M_{\bGamma(v)}(\phi_\gamma)$ is cut out transversally for each vertex 
$v \in \Ver(\Gamma)$, and the evaluation map at any point on the boundary
is a submersion.   As in the proof of Lemma \ref{lem:standard1}, an induction shows
that the evaluation maps at the nodes from $H^0(u^* T\C^n)$ are transverse to the diagonal and 
so the moduli space $\M_{\bGamma}(\phi_\gamma)$ is transversally cut out. 
\end{proof} 

The following Lemma gives the reader some idea of the relationship between maps bounding the handle and their projections under the Lefschetz fibration map.  Denote by $\M^f_{\bGamma}(\phi_\gamma) \subset 
 \M_{\bGamma}(\phi_\gamma)$ the locus of levels $u$ 
with $\pi \circ u$ having finite limit  $\ev_e(\pi \circ u) \neq \infty$
along some cylindrical end $e \in \Edge_{\rightarrow,\black}(\Gamma)$.
Let $\M^n_{\bGamma}(\phi_\gamma) \subset 
 \M_{\bGamma}(\phi_\gamma)$ denote the locus of maps $u$ with $\ev_e(u)$ lying in the nullcone $N$
at a cylindrical end $e$.

\begin{lemma} \label{lem:nullcone} For each type $\bGamma$,
the loci $ \M^f_{\bGamma}(\phi_\gamma),\M^n_{\bGamma}(\phi_\gamma)$  are transversally cut out and codimension at least two. In particular,  for rigid map types $\bGamma$ 
the locus  $\M^f_{\bGamma}(\phi_\gamma) \cup \M^n_{\bGamma}(\phi_\gamma)$ of maps having limit in the null-cone or with $\pi \circ u$ having a 
finite limit along some cylindrical end is empty.
\end{lemma}

\begin{proof}   Let $u:S \to \C^n$ be a map 
so that $\pi \circ u: S \to \C$ has finite
evaluation at a cylindrical end $e$.
The locus of such maps has formal tangent space
given by sections $\xi$ of $u^* T\C^n$ in $\ker(\ti{D}_u)$
such that $D\pi \xi$ is finite at the end.    For simplicity, 
suppose there is a single such end with multiplicity $m_\black(z_1')$. 
Evaluating the coefficients of $z^{-1},\ldots, z^{-2m_\black(z_1')}$
of any section at the end induces a long exact sequence 
\begin{equation} \label{eq:les4} \ker(\ti{D}_{\pi \circ u}) \to \ker(\ti{D}_u^h)
\to (T_{u(z_1')} \C)^{2m_\black(z_1)} \to \coker(\ti{D}_{\pi \circ u}) \to \ldots  .\end{equation}
The connecting homomorphism 
$(T_{u(z_1')} \C)^{2m_\black(z_1)} \to \coker(\ti{D}_{\pi \circ u}) $
maps the Taylor coefficients $(c_1,\ldots,c_{2m_\black(z_1')}) $
to the image of $\ti{D}_u^h \xi $ in $\coker(\ti{D}_{\pi \circ u}) $ 
where $\xi$ has the given coefficients
in the expansion at $z_1'.$  Since $\ti{D}_{\pi \circ u}$
 is surjective, the long exact sequence  \eqref{eq:les4} implies that the evaluation map $\ev_e$
at the cylindrical end from $\ker(\ti{D}_{\pi \circ u})$
to $(u^* T \C)_c$ has surjective linearization $D \ev_e$.   It follows that the locus of maps $u$ 
with $\ev_e(\pi \circ u) \neq \infty $ is transversally cut out.  The first claim follows.  
The second claim follows from the last statement 
in Lemma \ref{horizlem}.   
\end{proof}

By  Lemma \ref{lem:nullcone} rigid maps have evaluations only in the complement of the null-cone, assuming transversality holds.   To give the reader an idea of what kind of maps with strip-like ends of non-minimal length occur, we  classify such maps for the standard path.

\begin{proposition}  \label{higherclass}  For the standard path $\gamma(t) = t + i 2 \eps$,
any map $u$ from the complement $S$ of a finite set in $\HH$
whose evaluations $\ev_e(u)$ along the cylindrical ends
$e \in \mE_{\black}(S)$ do not lie in the  null-cone $N$ 
is of the form  
\begin{equation} \label{uz}
u(z) =  b(z)^{-1/2}
\prod_{\Im(\alpha_i) < 0 } (z - \alpha_i)^{1/2} 
\prod_{\Im(\alpha_i) > 0 } (z - \ol{\alpha_i})^{-1/2}  \ti{u} 
\end{equation} 
for some 
polynomial $b(z)$ and
complex numbers $\alpha_i$  and polynomial
\begin{equation}
\ti{u}(z) = 
(c_{d_-}z^{d_-} + \ldots + c_1 z + c_0) .\end{equation}
for some complex constants $c_0, ,\ldots, c_{d_-}$ satisfying 
\[  \ti{u}(z) \cdot \ti{u}(z) = f_-(z) f^*_-(z) = 
\prod_{\Im(\alpha_i) > 0 } |z - \alpha_i|^2 
 .\]
\end{proposition}

\begin{proof}  Suppose
the domain of $u$ is the punctured upper half-plane 
\[ S = \{ \on{Im}(z) \ge 0 \} - \{ z_1,\ldots, z_{e(\white)-1} \}
- \{ z_1',\ldots, z_{e(\black)}' \} . \]
The difference
$(\pi \circ u)(z) - 2 i \eps$ is a rational function
$a(z)/b(z)$ of $z$ by the reflection principle.   Since $(\pi \circ u)(z) - 2 i \eps$ 
has real boundary values, there exist real polynomials $a(z),b(z)$ so that 
\[ (\pi \circ u)(z) =  \frac{a(z)}{b(z)} + i 2 \eps .\]
By assumption $u(z)$ goes to infinity as $z \to \infty$ and the limit in $\CP^{n-1}$ is not in the null-cone. So 
\[ \deg(a) \ge \deg(b), \quad 
\deg( (\pi \circ u)b)  = \deg (a)  .\]
The composition $\pi \circ u$ admits a factorization
\[ a(z) + i 2 \eps b(z)  =  c f_+(z) f_-(z) \]
where 
\[ f_\pm(z) := \prod_{ i = 1}^{d_\pm} (z - \alpha_{i,\pm} ) \] 
has $d_\pm $ roots $\alpha_{i,\pm}$ in the lower resp.  upper half-plane, that is, with $\mp \on{Im}(\alpha_{i,\pm}) >0 $. Define 
\[ f_\pm^*(z) := \prod_{ i = 1}^{d_\pm} (z - \ol{\alpha}_{i,\pm} ) \]
and 
\[ \ti{u}: \HH \to \C^n, \quad  z \mapsto  u(z) f_+(z)^{-1/2} f_-^*(z)^{1/2} b(z)^{1/2} . \] 
Since by assumption the limits of $u$ along the punctures
do not lie in the null-cone, the poles of $\pi \circ u$
are twice the order of those of $u$.     Since these are the 
order of vanishing of $b(z)$, the map $\ti{u}$ is a polynomial.
The degree of $\ti{u}$ is 
\begin{eqnarray*} \deg(\ti{u}) &=& 
\frac{1}{2} ( \deg (\pi \circ u ) + \deg(f_-) - \deg(f_+) + \deg(b))\\
&=& \frac{1}{2} (\deg (a) + \deg(f_-)  -  \deg(f_+) ) = d_- 
\end{eqnarray*}
and so of the claimed form.  \label{liftu} Conversely, $u$ may be constructed from $\ti{u}$ as follows.
Assume that for $z$ real
\begin{equation} \label{eq:real}
\ti{u}(z) \cdot \ti{u}(z) = f_-(z) f^*_-(z) = 
\prod_{\Im(\alpha_i) > 0 } |z - \alpha_i|^2 
 .\end{equation}
The map 
\begin{eqnarray} 
u(z)  &:=&   f_+(z)^{1/2} f_-^*(z)^{-1/2} b(z)^{-1/2} \ti{u}(z) \\
\nonumber &=&  b(z)^{-1/2}
\prod_{\Im(\alpha_i) < 0 } (z - \alpha_i)^{1/2} 
\prod_{\Im(\alpha_i) > 0 } (z - \ol{\alpha_i})^{-1/2} \\
\nonumber && 
(c_{d_-}z^{d_-} + \ldots + c_1 z + c_0) \end{eqnarray}
has the required boundary values.  \end{proof}

\begin{remark}  We have shown 
that the moduli spaces of levels in $\CC^n$ are already 
regular  without using a domain-dependent almost complex structure.   By an argument
using Sard's theorem, there exists a complex hypersurface $D_{\subset}$ in $\CP^n$ so that the moduli spaces $\M_{\bGamma}(\phi_\gamma,D_\subset,\Sigma)$ 
are cut out transversally  whenever $\bGamma$ is an uncrowded type of expected dimension at most one.  Let $f(\bGamma)$ be the type of  map without interior edges obtained
by forgetting the interior edges of $\bGamma$.   There is a forgetful map  
\[ \M_{\bGamma}(\CC^n - \{ 0 \},\phi_\gamma,D_\subset,\Sigma)
\to \M_{f(\bGamma)}(\CC^n - \{ 0 \},\phi_\gamma,\Sigma) \] 
forgetting
the interior edges which produces a bijection between the moduli spaces with and without
the Donaldson hypersurface.   For this reason, in this section we use treed disks
rather than the adapted treed disks in the earlier chapters.
\end{remark}

We introduce further notation on the combinatorial type.  Assume that
$\bGamma$ is a type of map to $\C^n - \{  0 \} $ with $e(\white)$ boundary leaves representing strip-like ends asymptotic to Reeb chords, and $d(\white) \ge e(\white)$ boundary leaves in total.  We assume there are no  cylindrical ends asymptotic to Reeb orbits.   Denote by $p(\bGamma)$ the corresponding type of stable map bounding $\RP^{n-1}$ obtained by replacing the homology classes in $\C^n$ with homology classes
in $\CP^{n-1}$, and forgetting the intersection multiplicities at the ends; assuming that this construction yields a stable map.  For each strip-like end, the resulting map to 
$\CP^{n-1}$ takes values in the divisor at infinity $\CP^{n-2} \subset \CP^{n-1}$.  
We denote by $\M_{p(\bGamma)}(\CP^{n-1}, \RP^{n-1})$ the corresponding moduli space of 
treed disks in $\CP^{n-1}$ bounding $\RP^{n-1}$, where the map on the edges
corresponding to strip-like ends is a Morse trajectory on $\RP^{n-2}$.

\begin{lemma} \label{lem:dim}  With assumptions from the previous paragraph, the unconstrained moduli space  $\M_{\bGamma}(\C^n,\phi_\gamma)$ has dimension  
\begin{equation} \label{dimM}
\dim \M_{\bGamma}(\C^n, \phi_\gamma) = d_+ + n + (n-1) d_- + d(\white) - 3 
\end{equation}
where $(d_-,d_+)$ is the bidegree of $\pi \circ u$.
Each such map has projection to $\CP^{n-1}$ given by a disk of degree
$d-1$, and the dimension of the moduli space of maps of type $p(\bGamma)$ is 
\begin{equation} \label{dimM3}
\dim \M_{p(\bGamma)}(\CP^{n-1}, \RP^{n-1}) = (n-1) + n d_- + d(\white)  - e(\white) - 3 .\end{equation}
\end{lemma} 

\begin{proof} 
By the index formula in Proposition \ref{prop:index}, 
\begin{eqnarray} \label{eq:dimM2}
\dim \M_{\bGamma}(\C^n, \phi_\gamma) &=& 2d_+ + n(1 + d_-) +  d(\white) - 3
 - \sum_{i=1}^{e(\white)} 2m_\white(z_i)    \\
 &=& 2d_+ + n(1 + d_-) - 3 + d(\white) - (d_- + d_+) \\
 &=& d_+ + n + (n-1) d_- + d(\white) - 3 .
\end{eqnarray}
Here $d(\white) - 3$ is the dimension of the moduli space of disks, and 
$2d_+ + n(1 + d_-)$ is the Maslov index of the linearized boundary value problem. 
Given such a map $u$, the composed map $p \circ u $ hits the null-cone
\[ Q := [ p^{-1}(0)  - \{ 0 \} ] \subset \CP^{n-1} \] 
exactly $d_-$ times. Its doubled
 map from $\CP^1$ to $\CP^{n-1}$ meets the degree two hypersurface $Q$ exactly $2d_-$ times,  and so is a map of degree $d_-$.  Thus $p \circ u$ is a disk of degree $d_-$. The second dimension formula \eqref{dimM3} is standard and follows from Riemann-Roch, 
 taking into account the codimension one constraints at the $e(\white)$ ends mapping to 
 $\RP^{n-2}$.
\end{proof}

We now investigate the map on moduli spaces induced from the map to projective space.
Denote by $\M_{\bGamma}(\C^n - \{ 0 \} , \phi_\gamma)$ the locus of maps 
disjoint from $0$, that is, the critical locus of $\pi$.  If $d_- \ge 1$ then composition induces a map on moduli spaces 
\begin{equation} \label{eq:pbg} p_\bGamma:  \M_{\bGamma}(\C^n - \{ 0 \}, \phi_\gamma)
\to \M_{p(\bGamma)}(\CP^{n-1}, \RP^{n-1}), \quad u \mapsto p \circ u. \end{equation}

\begin{proposition}  \label{prop:dimker2} 
If $d_- > d_+$ resp. $d_- < d_+$
then the map $Dp_\bGamma$ has positive dimensional kernel
resp. cokernel on each tangent space, while if $d_- = d_+$ the map $Dp_\bGamma$ is an isomorphism on each tangent space.
\end{proposition}

\begin{proof} By the long exact sequence \eqref{les}, the map $Dp_\bGamma$ 
is an injection resp. surjection on each tangent space of 
the vertical part of the cohomology vanishes in degree zero resp. one. 
On the other hand, the kernel $\ker(D_u^v)$ resp. cokernel $\coker(D_u^v)$ vanishes if the 
Maslov index of $(u^* T^v \C^n, u^* T^v H_\gamma)$ is negative resp. positive, 
as in Oh \cite{oh:rh}.
\end{proof}

We now wish to classify which configurations in the local model may be rigid.
 Let $d_1(\bGamma)$  denote the number of cellular constraints labeled $\sigma_1$
and $e_m(\bGamma)$  the number of strip-like ends labeled with minimal length Reeb chords and with a non-trivial constraint.

\begin{lemma} \label{lem:noothers1} 
Let $\bGamma$ be a labeled type of treed holomorphic 
building $(C,u: C \to \C^n)$ bounding $\phi_\gamma$ with constraints $\Sigma$.  Assume that either $n = 2$ or $d_1(\bGamma) \leq 2$, so that there are at most two  boundary constraints labeled by the cell $\sigma_1$.   If $\bGamma$ is a rigid type for the given constraints $\Sigma$ then one of the following holds:
\begin{enumerate}
    \item  \label{fpos}
The number of strip-like ends is $e(\white) = 1$, the Reeb chord at the puncture
is minimal length and the type is one of the minimal types considered in Section \ref{sec:mintype}; or
\item
the number of strip-like ends is
$e(\white) \ge  2$ and the numbers $d_1(\bGamma)$ resp. $e_m(\bGamma)$ of constraints labeled 
$\sigma_1$ resp. strip-like ends asymptotic to minimal-length Reeb chords satisfy
the inequality 
\[ d_1(\bGamma) + e_m(\bGamma) \ge 3 . \]
\end{enumerate}
If $n = 2$ then only the first possibility \eqref{fpos} occurs.
\end{lemma} 

\begin{proof}  We first rule out the case that the domain is a treed sphere, 
that is, the domain has empty boundary.  In this case, the dimension of 
the moduli space of maps of any type $\bGamma$
with $\Gamma$ representing a top-dimensional stratum and map with degree $d$ to $\CP^n$ is 
by Riemann-Roch
\begin{eqnarray*} 
\dim(\M_{\bGamma}(\phi_\gamma)) &=& 2d(n+1) + 2n  - 6  - \sum_{i=1}^{e(\black)} 
2(m_\black(z_i) - 1)   \\\
&=& (2d + 2)n   - 6  + 2 e(\black) \\ 
&\ge & 2 e(\black)( n+1) + 2n - 6 . \end{eqnarray*}
Here $2d(n+1)$ is the contribution of the first Chern cass, $2n$
is the dimension of the ambient manifold, $6  = \dim(\Aut(\CP^1))$, and the last term is
the contribution from the intersection constraints at the intersections the 
divisor at infinity. 
Each cylindrical end has a constraint which cut down the dimension 
by at most $2(n-1)$.   It follows that for $n \ge 2$ such types cannot be rigid.

Therefore, it suffices to consider the case that the domain has non-empty boundary, 
that is, the domain is a treed disk.  Suppose first $n = 2$ and there are no cylindrical ends.  The dimension of the 
moduli space from \eqref{dimM} is 
\[ d_+ + 2 +   d_- + d(\white) - 3  \ge 2 d(\white) - 1 . \]
Since the constraints at the boundary leaves cut down the dimension 
by at most $1$, the constrained moduli space is expected dimension 
at least $d(\white) - 1$, and we must have $d(\white) = 1$ for rigidity to hold. 

In the case of arbitrary dimension, each end $e$  contributes some pair $(d_+(e), d_-(e))$ to the bidegree of the composition
 $\pi \circ u$, measured as the number of points  in a fiber of $\pi \circ u$ near infinity.  Since the map at each end is asymptotic to a Reeb chord whose projection is
 a circular arc whose angle is an integer multiple of $\pi/2$,
\begin{equation} \label{eq:diff} |d_+(e) - d_-(e)| \leq 1. \end{equation}

\vskip .1in 
\noindent {\em Case 1:  There is a single strip-like end. }  The dimension
of the unconstrained moduli space from \eqref{dimM} is 
$d_+ + n + (n-1) d_- - 2$.

\vskip .1in 
\noindent {\em Case 1a:  $d_- \leq 1$.}  The unconstrained dimension must be at most $2(n-1)$.  Thus either  
\[ (d_- , d_+) = (0,1) \quad \text{or} \quad (d_- , d_+) = (1,0) ;\] 
in the latter case by rigidity there is a constraint at the end
and a cellular constraint labeled $\sigma_1$.

\vskip .1in 
\noindent {\em Case 1b:  $d_- \geq 2$.} By \eqref{eq:diff}, $d_+ \ge 1$ and the dimension of the unconstrained moduli space
\[ d_+ + n + (n-1) d_- - 2 \geq  3n -3 .  \] 
Since each cellular constraint lowers dimension by $n-2$ and the constraint at the strip-like end lowers dimension by $n-1$, the constrained moduli space cannot be made rigid
by adding at most two cellular constraints.  

\vskip .1in  \noindent {\em Case 2:  There are at least two strip-like ends. }  
We introduce notation for the various kinds of strip-like ends.  Let $e'(\bGamma)$ denote the number of strip-like ends of type $(0,1)$, $e''(\bGamma)$ the number of strip-like ends of type $(1,0)$,
and $e'''(\bGamma)$ resp. $e''''(\bGamma)$ the number of strip-like ends with 
\[ 0 < d_-(e)  =  d_+(e) - 1  \ \text{ resp.} \ 0 < d_+(e) = d_-(e) - 1 . \]
Since $d_+\leq d_-$ by  Proposition \ref{prop:dimker2}  we must have 
\begin{equation} \label{eq:es} e'(\bGamma) + e'''(\bGamma) \leq e''(\bGamma) + e''''(\bGamma).\end{equation}
The dimension of the constrained moduli space from  \eqref{dimM} satisfies the inequality
\begin{eqnarray} \nonumber \dim(\M_{\bGamma}(\phi_\gamma,\Sigma)) &\ge& 
d_+ + n + (n-1) d_- + d(\white) - 3 - e_\white(\bGamma)(n-1)  - 
d_1(\bGamma)(n-2)
 \\\nonumber  &\ge&  (n-1) + \sum_e  ((n-1) d_-(e) + d_+(e) - (n-1))   - 
d_1(\bGamma)(n-2)  \\
\label{eq:es2} & \ge &  (n-1) -  e'(\bGamma) (n-2) \\ && +  2 e'''(\bGamma) +  e''''(\bGamma) (2n-1)
  - 
d_1(\bGamma)(n-2)\end{eqnarray}
where the sum is over edges $e$ representing strip-like ends. 
Here the first inequality holds since $d_\pm$ is the sum of contributions from 
$d_\pm(e)$ and the constraints at the ends lower the dimension by at most $n-1$.  We have ignored the ends with $d_-(e) = d_+(e)$, which are irrelevant for the computation.   
Equation \eqref{eq:es} and 
vanishing of \eqref{eq:es2} imply
\[ e''(\bGamma) +   e''''(\bGamma) \ge 
e'(\bGamma) \ge e''''(\bGamma) + 2 - d_1(\bGamma) .\] 
Hence
\begin{eqnarray} \nonumber e_m(\bGamma) &=& 
e'(\bGamma) + e''(\bGamma) \\  \label{eq:asdesired} &\ge& e'(\bGamma) + 2 - d_1(\bGamma) . \end{eqnarray}

\vskip .1in 
\noindent {\em Case 2a:  Suppose  $d_1(\bGamma) = 2$.}  By way of contradiction, suppose that  $e'(\bGamma) = e''(\bGamma)  = 0$. We must have  $e''''(\bGamma) \ge 1$ in order for $d_+ \leq d_-$.  Then rigidity requires $d_1(\bGamma) + e'(\bGamma) \ge 3$ which is a contradiction.  Hence $d_1(\bGamma) + e_m(\bGamma) \ge 3.$

\vskip .1in 
\noindent {\em Case 2a:   Suppose $d_1(\bGamma) \leq 1$.}  Rigidity requires  $e'(\bGamma) \ge 1$ and so  by \eqref{eq:asdesired} we have $d_1(\bGamma) + e_m(\bGamma) \ge 3$ as desired. 

\vskip .1in \noindent 
{\em Case 3: There are cylindrical ends.}  Each cylindrical end with multiplicity $m_\black(z_i')$ contributes $(2n+4)m_\black(z_i')$ to the degree terms
$2d_+ + n d_-$ in the dimension formula \eqref{dimM}.  On the other hand, any constraint
at the Reeb orbit cuts down the dimension by at most $\dim(\CP^{n-1}) = 2n-2$.
Since  
\[ (2n+4)m_\black(z_i') - (2n-2) > 0 , \] 
the arguments of the previous paragraphs cases go through as before with only improvements in the inequalities. 
\end{proof} 

\label{rulingoutneck} 

\subsection{Ruling out maps intersecting the critical locus}

In the previous section, we ruled out rigid levels with more than one end
or non-minimal Reeb length assuming that the maps avoid the critical locus. 
To deal with the case that the map intersects the critical locus, we
introduce a blow-up which allows us to extend the splitting of the
tangent bundle into horizontal and vertical parts over the critical
locus.  Let
\[ \Bl(\CC^n) = \{ (z,l) \in \CC^n \times \C P^{n-1}, z \in l \} \]  
denote the blow-up of $\CC^n$ at $0$, and let 
\[ p: \Bl(\CC^n) \to \C^n, \quad (z,l) \mapsto z \]  
denote the natural surjection. \label{rep:pcn} For any path $\gamma$
avoiding $0$, the boundary condition $\phi_\gamma$ naturally lifts to $\Bl(\CC^n)$ and we denote the lift with the same notation.  Removal of singularities defines a bijective
correspondence between maps to the blow-up and to its projection.  
For some type $\bGamma$ of level, let
$\widetilde{\M}_{\bGamma}(\Bl(\CC^n),\phi_\gamma)$ denote the moduli space
of tuples $(C,{u}, \ul{\zeta})$ of maps ${u}: S \to \Bl(\CC^n)$ where 
$(C,u)$ is a treed disk and 
with additional markings $\ul{\zeta} \subset S$ is a finite set
describing the intersection set with
the exceptional divisor in the sense that
\[ {u}^{-1}(\C P^{n-1}) = \ul{\zeta} . \]
Similarly, let  $\widetilde{\M}_{p(\bGamma)}(\CC^n,\phi_\gamma)$ denote the moduli space
of pairs $(u, \ul{\zeta})$ of maps $u: S \to \CC^n$
with additional markings $\ul{\zeta}$ describing the intersection set with
the exceptional divisor $u^{-1}(0) = \ul{\zeta}$; here $p(\bGamma)$
is the obvious type of level in $\CC^n$ obtained by projecting the homology classes of the components. Composition with the projection
\[   \widetilde{\M}_{\bGamma}(\Bl(\CC^n),\phi_\gamma) \to
\widetilde{\M}_{p(\bGamma)}(\CC^n,\phi_\gamma), \quad {u} \mapsto p \circ {u} \]
is a bijection, by removal of singularities.  

\begin{lemma} \label{lem:standard4}  For the standard complex structure on $\Bl(\C^n)$, every holomorphic treed disk $(C,u: C \to \Bl(\C^n))$ bounding $\phi_\gamma$ asymptotic to some collection of Reeb orbits and chords at infinity is regular.
\end{lemma} 

\begin{proof} The proof is essentially the same as that of Lemma \ref{lem:standard2}.
The splitting of $ u^* T \Bl(\C^n)$ into vertical and horizontal parts, 
corresponding to the tangent space to $\P^{n-1}$ and $\mO(-1)$, induces
a  long exact sequence involving  the kernels and cokernels of 
$D_{{u}}^v , D_{{u}}$ and $D_{p \circ u}$. Homogeneity implies vanishing of the vertical cokernel
$\coker(D_u^v)$, while existence of a section implies vanishing  of the horizontal
cokernel $D_{p \circ u}$.  The claim now follows from the long exact sequence similar to \eqref{les}.
\end{proof} 

 \begin{lemma} \label{lem:disjoint} For generic constraints $\Sigma$, or generic paths of constraints, for any type $\bGamma$
 of treed level the moduli space $\M_{\bGamma}(\phi_\gamma,\Sigma) $ is cut out
   transversally and any rigid treed level 
   $u \in \M_{\bGamma}(\phi_\gamma,\Sigma)$ has image disjoint from
   $0 \in \CC^n$.
\end{lemma} 

\begin{proof} By Lemma \ref{lem:standard4}, the moduli spaces
  ${\tM}_{\bGamma}(\phi_\gamma,\Sigma)$ are
  regular.    A standard codimension argument shows that
  ${\tM}_{\bGamma}(\phi_\gamma,\Sigma) $ has dimension $2(n-1)$ less than that
  of $\M_{\bGamma}(\phi_\gamma,\Sigma)$, with the obvious identification of map types.   Indeed, the relative Chern class of $\Bl(\CC^n)\to \CC^n $ is dual to $(1-n)[\C P^{n-1}]$.  Thus, the Maslov index of $u$ differs from that of $p \circ u$ by
$2(1-n)$ times the intersection number of ${u}$ with $\C P^{n-1}$.
For $\ul{\zeta}$ non-empty, if $\M_{\bGamma}(\phi_\gamma,\Sigma)$ is expected
dimension zero then $\widetilde{\M}_{\bGamma}(\phi_\gamma,\Sigma) $ is empty,
and the elements of $\M_{\bGamma}(\phi_\gamma,\Sigma)$ have images disjoint
from $0 \in \CC^n$.  The same argument works for paths of constraints, as a codimension $2(n-1)$-codimensional submanifold of a one-dimensional manifold is still empty. 
 \end{proof}

\subsection{Comparing disks bounding the flattened and unflattened
handles}
\label{sec:compare}

We wish to show that the unflattened handle is asymptotically cylindrical.

\begin{lemma} \label{lem:unflat}  The  handle $H_\gamma \subset \CC^n$ 
for the standard path $\gamma(t) = t + 2 i \eps$ is asymptotically cylindrical.
\end{lemma}

\begin{proof}  We first consider the case of dimension one targets.
The image of the line $\R + i 2 \eps$ under the coordinate
change $z \mapsto 1/z$ is a circle of diameter $1/2\eps$
with center at $i/4\eps$, 
described in coordinates $z = x + iy$ as the solution set to 
\[ (\R + i 2 \eps)^{-1} = \{ x^2 + (y - (4\eps)^{-1})^2 = (4\eps)^{-2} \} .\]
It suffices to consider the case $\eps = 1/4$.   
The handle $H_\gamma$ is the pre-image of $\R + i/2$
under the square map 
\[ (x,y) \mapsto (x^2 - y^2, 2xy) .\]
Thus $H_\gamma$ is given in coordinates near infinity 
by 
\begin{eqnarray*} 
\{ (x^2 - y^2)^2 + ( (2xy) - 1)^2  = 1   \} &=& 
\{ 
y^4 - 2 x^2 y^2 + x^4 + 4 x^2 y^2  - 4 xy  +1 = 1 \} \\ 
&=& \{   
y^4 + 2 x^2 y^2 + x^4  - 4 xy   = 0 \} . \end{eqnarray*}
Write $y = xz$.  
The equation
\[  xz = \frac{(x^2z^2 + x^2)^2}{4x}= \frac{x^3 (1 + z^2)^2}{4}
\]
has a smooth solution of $z$ in terms of $x$, since 
$z  \mapsto z/(1 + z^2)^2$ is a diffeomorphism near $z = 0$. By symmetry, each branch is a smooth manifold with boundary at $x = y = 0$.

For higher dimension, the handle is the flow-out of the dimension one handle under the action of the special orthogonal group. 
Let $\ol{H}_\gamma^n \subset \CP^n$ denote the handle in dimension $n$.  Consider the action of $O(n)$ on $\CP^n$
on the first $n$ coordinates.  The locus 
\[ \ol{H}_\gamma^1  = \{ [z,0,\ldots, 0,1] \}  \cap
\ol{H}_\gamma^n \] 
has stabilizer groups contained in $O(1) \times O(n-1)$.
Thus we have a homeomorphism 
\[ \ol{H}_\gamma^n =  O(n) \ol{H}_\gamma^1 
\cong ( O(n)
\times 
\ol{H}_\gamma^1 ) / (O(1) \times O(n-1)) \] 
which is a diffeomorphism away from the boundary.  
Since each branch of $\ol{H}_\gamma^1$ is a smooth submanifold with boundary, 
$\ol{H}_\gamma^n$ has a smooth structure for which the inclusion in $\ol{X}$
is an immersion.  
\end{proof}

\begin{proposition} \label{prop:deform} Let $\bGamma$ be 
a primitive type.
  There exists an oriented cobordism  from
  $\M_\bGamma(\vv{\phi}_\gamma,\Sigma )$ to $  \M_\bGamma({\phi}_\gamma,\Sigma)$,
  where $\vv{\phi}_\gamma: \vv{H}_\gamma \to \CC^n$ is the flattened
  embedding of \eqref{Hepsp}.
\end{proposition} 

\begin{proof}  We will construct a cobordism between the two moduli spaces by viewing them as moduli spaces of curves bounding cleanly-intersecting Lagrangians.  The closures of $\vv{H}_\gamma$
and $H_\gamma$ in $\CP^n$ are contained in cleanly-intersecting
Lagrangians by Proposition \ref{prop:cleanly2}.  Furthermore, this extension of $\vv{H}_\gamma$
is an exact deformation of $H_\gamma$, by the description from 
\eqref{fprime}.  Choose a family of Lagrangians $H_{\gamma}^t$ interpolating between $H_\gamma$ and $\vv{H}_\gamma$, for example, 
by using $  \rho_t = (1-t) \rho + t(\ln(r)  - |\eps|)$ 
 in the definition \eqref{fprime}.  Let ${\tM}_{\bGamma}(\phi_\gamma,\Sigma)$ denote the parametrized moduli space of $J_0$-holomorphic maps for the family consisting 
 of triples $(t,C,\ol{u})$ where $t \in [0,1]$ and $(C,\ol{u})$ is a $J_0$-holomorphic curve bounding $\ol{H}_\gamma^t$ above.

 There is no bubbling in the parametrized moduli space by the primitivity assumption.   By definition any element 
 $(t,C,\ol{u})$ in ${\tM}_{\bGamma}(\phi_\gamma,\Sigma)$,
  consists of a collection of disks or spheres meeting the interior  and disks or spheres contained in the boundary
 $\CP^{n-1}$ with homology class summing to the homology class
 of the sections described above. Since the homology class is primitive, no bubbling is possible and the domain of $\ol{u}$ consists of a single   disk, with at least one corner mapping to $\CP^{n-1}$.
In particular, for every $\ol{u} \in {\tM}_{\bGamma}(\phi_\gamma,\Sigma)$,  the set of intersections $\ol{u}^{-1}(\CP^n)$ with the divisor at infinity consists only of 
 a single corner, corresponding to a single strip-like end asymptotic to a Reeb chord $\thorn_e$ of minimal length. 
\end{proof} 

\begin{remark}  In fact in the case of type $\bGamma = \bGamma_+$
in Definition \ref{def:mintype}, 
 the cobordism provided by Proposition \ref{prop:deform} is trivial:   
   By Lemma \ref{lem:standard2}, the linearized operators $\ti{D}_u$ are surjective and the evaluation map 
   at the end is surjective on the kernel of $\ti{D}_u$.   It follows that 
   ${\tM}_{\bGamma}(\phi_\gamma,\Sigma)$ is smooth and compact.
   Furthermore, the natural map ${\tM}_{\bGamma}(\phi_\gamma,\Sigma) \to [0,1]$ is a submersion (since the map is index one and the kernel of the linearization is   exactly the kernel of $\ti{D}_u$) 
   and the moduli space ${\tM}_{\bGamma}(\phi_\gamma,\Sigma)$
   is a trivial cobordism between the 
   moduli spaces ${\M}_{\bGamma}(\phi_\gamma,\Sigma) $
   and 
   ${\M}_{\bGamma}(\vv{\phi}_\gamma, \Sigma) .$ 
\end{remark}

\section{Fukaya algebras under surgery}
\label{fsurgery}

In this section, we prove the main result Theorem \ref{thm:bij} by combining the homotopy
equivalences in the previous section with broken Fukaya algebras with
the local computation in Section \ref{handlesec}.  We work with the unflattened boundary condition on the handle.  By Theorem \ref{thm:htpy}, the resulting Fukaya algebra  is homotopy equivalent to that of the surgered Lagrangian immersion $\phi_\eps$.   Theorem \ref{thm:biject} globalizes 
the results of the previous section to a bijection between buildings.
Given this, some algebra identifies the potentials up to the 
change \label{rep:change} in projective Maurer-Cartan solution in Theorem \ref{thm:bij}. 
 After this, a stabilization argument in \eqref{h0} and \eqref{heps}, and then the proof of Theorem \ref{thm:bij},  identifies the Floer cohomologies. \label{rep:stabarg}

\subsection{Morse functions on the surgered and unsurgered Lagrangian}  

The isomorphism of Floer cohomologies is induced by a map of Floer
cochains that maps the ordered self-intersection points of the
original Lagrangian to the longitudinal and meridian cells in the
surgered Lagrangian.  Topologically, the surgered Lagrangian $L_\eps$
is obtained from the unsurgered Lagrangian $L_0$ by attaching the
handle
\[ {H}_\eps \cong (-1,1) \times S^{n-1} .\]  
The boundary $\partial {H}_\eps \cong \{ -1, +1 \} \times S^{n-1}$
is glued in along small spheres around the preimages $x_\pm \in L_0 $
of the self-intersection point $\phi(x_+) = \phi(x_-) \in X$.  Choose
a Morse function on $L_0$ with critical points of zero index at the self-intersection points, so that the corresponding cell decomposition has small
balls $\sigma_{n,\pm}$ and spheres $\sigma_{n-1,\pm}$ around the
self-intersection points $x_\pm$.  Let $\sigma_{0,\pm}$ be the
zero cells in the boundary of $\sigma_{n-1,\pm}$.  A cell structure
$L_\eps$ is derived from that on $L_0$ by removing a ball around each
$x_\pm$ and gluing in a single $1$-cell and single $n$-cell
\[ \sigma_1: B^1:= [-1,1] \to L_\eps, \quad \sigma_n: B^n \cong B^{n-1}
\times [-1,1] \to L_\eps \]
along the boundary of the $0$-cells  resp. $n-1$-cells
\[ \sigma_{0,\pm}: \{ 0 \} \to L_\eps, \quad
\sigma_{n-1,\pm}: B^{n-1} \to L_\eps \]
as shown in Figure \ref{cell2}.  

\begin{figure}[ht]
      \centering
      \scalebox{.5}{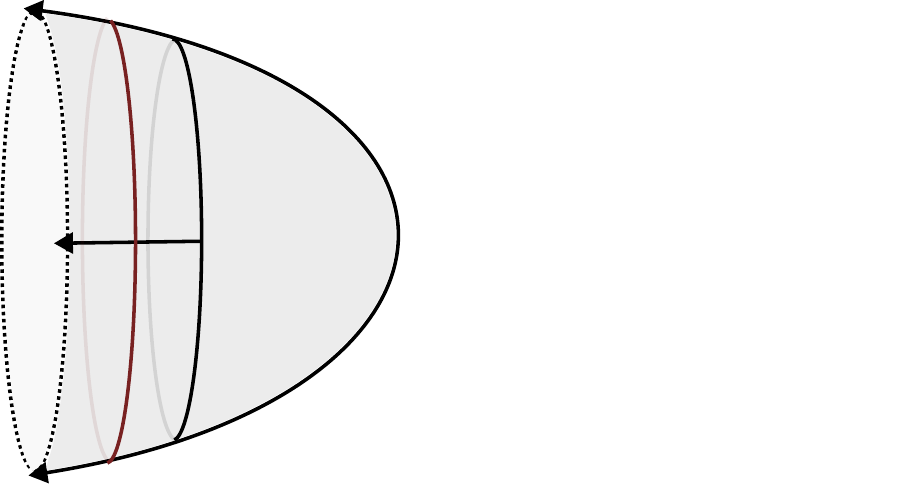}
      \caption{Cell structure on the unsurgered handle} 
\label{cell1}
\end{figure} 

\begin{figure}[ht]
      \centering
      \scalebox{.5}{
\begingroup%
  \makeatletter%
  \providecommand\color[2][]{%
    \errmessage{(Inkscape) Color is used for the text in Inkscape, but the package 'color.sty' is not loaded}%
    \renewcommand\color[2][]{}%
  }%
  \providecommand\transparent[1]{%
    \errmessage{(Inkscape) Transparency is used (non-zero) for the text in Inkscape, but the package 'transparent.sty' is not loaded}%
    \renewcommand\transparent[1]{}%
  }%
  \providecommand\rotatebox[2]{#2}%
  \newcommand*\fsize{\dimexpr\f@size pt\relax}%
  \newcommand*\lineheight[1]{\fontsize{\fsize}{#1\fsize}\selectfont}%
  \ifx\svgwidth\undefined%
    \setlength{\unitlength}{456.97225255bp}%
    \ifx\svgscale\undefined%
      \relax%
    \else%
      \setlength{\unitlength}{\unitlength * \real{\svgscale}}%
    \fi%
  \else%
    \setlength{\unitlength}{\svgwidth}%
  \fi%
  \global\let\svgwidth\undefined%
  \global\let\svgscale\undefined%
  \makeatother%
  \begin{picture}(1,0.51714315)%
    \lineheight{1}%
    \setlength\tabcolsep{0pt}%
    \put(0,0){\includegraphics[width=\unitlength,page=1]{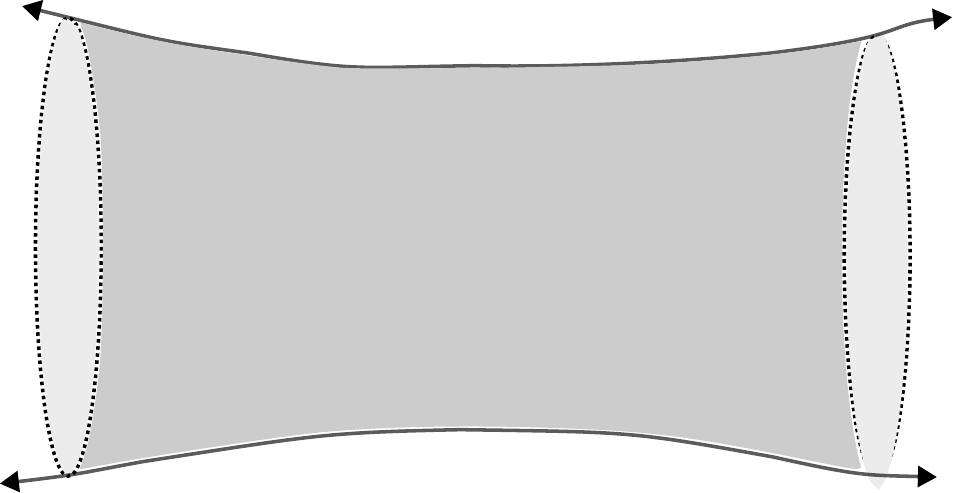}}%
    \put(0.41133869,0.22073046){\color[rgb]{0,0,0}\makebox(0,0)[lt]{\lineheight{1.25}\smash{\begin{tabular}[t]{l}$\sigma_{1}$\end{tabular}}}}%
    \put(0,0){\includegraphics[width=\unitlength,page=2]{cell2.pdf}}%
    \put(0.19783407,0.12337262){\color[rgb]{0,0,0}\makebox(0,0)[lt]{\lineheight{1.25}\smash{\begin{tabular}[t]{l}$\sigma_{n-1,-}'$\end{tabular}}}}%
    \put(0.23162023,0.16387944){\color[rgb]{0,0,0}\makebox(0,0)[lt]{\lineheight{1.25}\smash{\begin{tabular}[t]{l}$\sigma_{n-1,-}$\end{tabular}}}}%
    \put(0.78309,0.35479952){\color[rgb]{0,0,0}\makebox(0,0)[lt]{\lineheight{1.25}\smash{\begin{tabular}[t]{l}$\sigma_{n-1,+}'$\end{tabular}}}}%
    \put(0.73995366,0.3224338){\color[rgb]{0,0,0}\makebox(0,0)[lt]{\lineheight{1.25}\smash{\begin{tabular}[t]{l}$\sigma_{n-1,+}$\end{tabular}}}}%
  \end{picture}%
\endgroup%
}
      \caption{Cell structure on the surgered handle} 
\label{cell2}
\end{figure} 

\subsection{The surgered-unsurgered bijection}

With the cellular structures on the Lagrangian and its surgery
defined, we now define the chain-level map which replaces the ordered
self-intersection points to be surgered with the longitude and
meridian on the handle.  The neck-stretching argument in Theorem
\ref{thm:bthm} produces a cobordism between rigid holomorphic maps with
rigid broken holomorphic maps.  \label{rigidtwice} Let $\XX$
denote the broken manifold obtained by quotienting the spheres
$S^{2n-1}$ on either side of the $n-1$-cells $\sigma_{n-1,\pm}$ by the
$S^1$-action.  This cell structure is shown 
in Figure \ref{cell1} and \ref{cell2} as the collection of blue and red 
spheres.  The
pieces of $\XX$ are
\begin{equation} \label{threepiece}
\XX = X_\subset \cup X_0 \cup X_\supset \end{equation} 
where 
\begin{equation} \label{bpiece} \ol{X}_\subset \cong \C P^n, \quad X_0
  \cong \Bl(\C P^n) \quad \ol{X}_\supset = \Bl(X) \end{equation} 
is a projective space resp. the blow-up $\Bl(\C P^n)$ of projective
space at a point resp.  the blow-up $\Bl(X)$ of $X$ at the
self-intersection point $\phi_0(x_-) = \phi_0(x_+)$.  The handle $H_\gamma$ admits the standard cellular deformation of the diagonal 
induced by positive translation on $\sigma_1 \cong [-1,1]$ and the standard
Morse flow on $\sigma_{n-1} \cong S^{n-1}$.
The Fukaya algebra $CF(X,\phi_\gamma)$ is then homotopy equivalent
to the  broken Fukaya algebra $CF(\XX,\phi_\gamma)$ 
as in Theorem \ref{thm:htpy} whose structure maps have levels with cellular
constraints.

\begin{lemma} \label{lemma:genreg} There exists a regular perturbation datum
  $\ul{P} =(P_\Gamma)$ for holomorphic buildings in $\XX$ with
  boundary in $\phi_\gamma$ such that $J_\Gamma$ is the standard
  complex structure on $X_\subset = \CC^n$ and $X_0 = \CC^n - \{ 0 \}$
\end{lemma} 

\begin{proof} 
  Lemma \ref{lem:standard1} shows that rigid holomorphic maps to
  $X_\subset = \CC^n$ are automatically regular, while Lemma \ref{lem:standard4}
  shows that rigid holomorphic maps to $X_0 = \CC^n - \{ 0 \}$ are regular
  for the standard complex structure.
\end{proof}

We apply the classification of rigid maps bounding the handle
in the previous section to classify rigid holomorphic buildings.
Denote by $\phi_\gamma$ the Lagrangian boundary condition in $\XX$
defined using the Lagrangian $H_\gamma$ using a path $\gamma$ in the
local model corresponding to the surgered or unsurgered Lagrangian.

\begin{lemma}  Suppose the perturbations on $X_0$ vanish. There are no rigid buildings in $\XX$
so that the level in $X_0$ has non-trivial projection to $\CP^{n-1}$.
\end{lemma} 

\begin{proof}  The proof is by a symmetry argument.  The $\R - \{ 0 \}$ action on $\C^n - \{ 0 \} $ preserves
the boundary condition in $X_0$.  Furthermore, the gradient trajectories in $X_0$
are the lines $\R \times \{ p \}$ in $\R \times S^{n-1}$.  Thus the action of $\R - \{ 0 \}$
on the union of components mapping to $X_0$ produces a one-parameter
family of configurations satisfying the matching condition.   By rigidity, 
translation must be equivalent to an automorphism of the domain, so 
$u_0$ must be a cylinder $\C^\times \to \R \times S^{n-1}$.
\end{proof}

Define the map between generators as follows. Denote by 
\[ \mu = \sigma_{n-1,\pm}, \quad \lambda = \sigma_1 \]
the meridianal and longitudinal cells; the choice of which 
cell $\sigma_{n-1,\pm}$ is immaterial as both choices give
the same counts.   Define a map
\[ \cI(\phi_0) \to \cI(\phi_\eps), \quad {\sigma}_0 \mapsto {\sigma}_\eps \] 
by mapping the surgered self-intersection points  
\begin{eqnarray*}   
x = (x_-,x_+) \in \cI^{\si}(\phi_0) &\mapsto & \mu \in  
                                               \cI^{c}(\phi_\eps) \\
\ol{x} = (x_+,x_-) \in \cI^{\si}(\phi_0) & \mapsto &  \lambda \in  
                                                     \cI^c(\phi_\eps) \end{eqnarray*}
                                                   to the meridianal  
                                                   resp. longitudinal  
                                                   cells and leaving  
                                                   the remaining  
                                                   generators  
                                                   unchanged.  
  Given $\ul{\sigma}_0 \in \cI(\phi_0)^{d+1}$ define  
$\ul{\sigma}_\eps \in \cI(\phi_\eps)^{d+1}$ by applying this map to   
each generator.  We view this as a bijection up to the homology equivalences 
between $\sigma_{n-1,+}$ and $\sigma_{n-1,-}$ and similar which will be explained 
as a stabilization.

\begin{theorem} \label{thm:biject} {\rm (cf. Fukaya-Oh-Ohta-Ono
    \cite[55.11, Chapter 10]{fooo})}  For any  labeled rigid type $\bGamma_0$
    with positive energy, there exists a type $\bGamma_\eps$ of building
    obtained by replacing components in $X_\subset$ bounding
    $\phi_0$ with those bounding $\phi_\eps$, possibly after adding  edges labeled $\sigma_1$, so that there is a bijection between admissible rigid moduli spaces of 
    buildings of positive area 
  \begin{equation} \label{eq:corres}
    \M_{\bGamma_0}
    (\XX,\phi_0) \to
    \M_{\bGamma_\eps}
    (\XX,\phi_\eps) , \quad u_0
    \mapsto u_\eps \end{equation}
  preserving orientations.   If $u_\eps,u_0$ are related by this bijection then the
  symplectic areas are related as in Lemma \ref{lem:ncorners}:
  \[ A(u_\eps) = A(u_0) + (\kappa - \ol{\kappa}) A(\eps) \]
  where $\kappa$ resp. $\ol{\kappa}$ is the number of corners  with boundary in $\phi_0$ which map to $x$ resp. $\ol{x}$ and $A(\eps)$ is the area of \eqref{Aeps}.
\end{theorem}

\begin{proof}
  The bijection between maps with boundary in the unsurgered handle $\phi_0$ and those in
  the surgered handle $\phi_\eps$ in the local model in Proposition \ref{prop:bijprop} and Lemma \ref{lem:quadrant}  produces a correspondence
  between rigid types with only constraints $\lambda$ on each level in the
  neck region (ignoring the count of configurations with pieces in the local model 
  with Reeb chords of large angle, which vanishes by Proposition \ref{prop:noothers2}.)
The counts of broken maps are not changed by replacing the standard handle with the flattened handle, by the cobordism in Proposition \ref{prop:deform}.
   We compare the numerical invariants of the corresponding  buildings.   Lemma \ref{lem:ncorners} then implies that the areas
  differ by $(\kappa - \ol{\kappa}) A(\eps)$.  The bijection in
  \eqref{eq:corres} is sign-preserving if and only if the bijection
  between moduli spaces on the broken piece $X_\subset$ of
  \eqref{bpiece} containing the self-intersection point
  $x \in \phi(L)$ is orientation preserving. This can always be
  achieved by changing the orientation on the determinant line
  $\DD_x^+$. 
\end{proof} 

\begin{remark} \label{rem:constant} We discuss constant disks on the
  handle region; these will be needed later to prove the invariance of
  the potential.  On the pre-surgered side, there are two constant
  disks 
\[ u_\pm: S \to X, \quad u_\pm(S) = \phi(x_-) = \phi(x_+) \] 
in the case $\dim(L_0) > 2$. The 
  constant disks $u_\pm$ have inputs $x,\ol{x}$ resp. $\ol{x},x$ and outgoing labels 
  $\sigma_{0,+}$ resp. $\sigma_{0,-}$.  In the case
  $\dim(L_0) = 2$, there are arbitrary numbers of inputs, as in
  \eqref{rinvd}.  For the surgered Lagrangian, we have two constant
  configurations $u_\pm: S \to X$ corresponding to the classical
  boundary $\sigma_{0,+} - \sigma_{0,-}$ of $\sigma_1$.
\end{remark}

\begin{remark} \label{rem:following} The bijections between moduli spaces
  of holomorphic treed disks bounding the immersion and its surgery
  extend to repeated inputs. Suppose $\bGamma_0$ is a type of building 
  bounding $\phi_0$ with label $x$ appearing $l$ times, and
\[ \ul{r} = (r_1,\ldots, r_l) \] 
is a collection of integers.  Let 
  $\bGamma_\eps^{\ul{r}}$ denote the type obtained by repeating 
  the edge labeled $\sigma_{n-1,\pm}$ at the $i$-th place $r_i$ times. 
    Combining Theorem \ref{thm:biject} with
  Theorem \ref{thm:repthm} (or rather, it's extension to buildings, whose proof is the same) gives for permutation-invariant matching
  conditions bijections between moduli spaces 
  \[ \M_{\bGamma_0}(\XX,\phi_0) \to \M_{\bGamma_\eps^{\ul{r}}}(\XX,\phi_\eps) , \quad u_0 \mapsto
  u_\eps .\]
Each disk passing once through the handle in the positive
  direction meets each generic translate of the meridian
  $\sigma_{n-1,\pm}$ exactly once.  If $\dim(L_0) = 2$, then the
  longitudinal cell $\sigma_{1}$ is also codimension one.  In this
  case, let 
  \[ \ul{r} = (r_1^+, \ldots, r_l^+, r_1^-,\ldots, r_s^-) \] 
be a
  tuple of integers represented a pattern of repetitions.  If
  $\ul{\sigma}^{\ul{r}}_\eps$ is obtained by replacing the $i$-th
  occurrence $x$ resp. $\ol{x}$ with $r_i^+$ resp. $r_i^-$ copies of
  $\sigma_{n-1,\pm}$ resp $\sigma_1$, then there is a bijection as above for exactly
  one of the $r_i^+!$ resp. $r_i^-!$-factorial of the perturbations of
  the cycles $\sigma_{n-1,\pm}$ resp. $\sigma_1$.  Indeed, each curve hitting
  $\lambda$ hits each generic translate of $\sigma_1$ exactly once.
  This ends the Remark.
\end{remark}

\begin{lemma} \label{lem:special} The rigid moduli spaces
  $\M(\XX,\phi_\eps)_0$ are invariant under replacement of a
  constraint $\sigma_{n-1,+}$ with constraint $\sigma_{n-1,-}$ and
  vice-versa.
\end{lemma} 

\begin{proof} By Proposition \ref{prop:bijprop}, for each rigid building $(C,u)$ the boundary $\partial u: \partial S \to L$ meets each meridian $\sigma_{n-1,\pm}$
  the same number of times that $\partial u$ passes through the handle
  $H_\gamma \subset L$ (counted with sign), and the claim follows.
\end{proof}

\subsection{Equivalence of potentials} 

We may now prove the first part of \label{rep:mainresult} Theorem \ref{thm:bij}
using the bijection between curves contributing to the potentials.
First we relate the curvatures of the immersion and its surgery.  We
work with the broken Fukaya algebras
\[ CF(\phi_0) = 
CF(\XX,\phi_0),  
\quad CF(\phi_\eps) =
CF(\XX,\phi_\eps) \] 
where $\phi_\eps$ is the asymptotically-cylindrical Lagrangian 
defined in the local model by the standard path $\gamma(t) = t + 2 i \eps$.  These broken Fukaya algebras are homotopy equivalent to the unbroken Fukaya algebras by Theorem \ref{thm:htpy}.  Define $\Psi: b_0 \mapsto b_\eps$ as in
\eqref{psimap}.  We assume that $b_0$ vanishes in the neighborhood of
the attaching spheres in $L_0$ by Lemma \ref{lem:gaugekill}.  The
following is straight-forward from Definition \ref{onsurgery}:

\begin{proposition} The derivative
\[ D_{b_0} \Psi: CF(\phi_0) \to CF(\phi_\eps) \]  
is given by the identity on all generators in $\cI(\phi_0)$ except $x,\ol{x}$.  On
these generators we have
\begin{eqnarray} \label{DPsi}
x &\mapsto & (b_0(x) 
q^{A(\eps)})^{-1} \mu + b_0(\ol{x}) \lambda \\ 
\ol{x} & \mapsto & 
b_0(x) \lambda \end{eqnarray} 
for $\dim(L_0) > 2$; while 
\begin{eqnarray} \label{DPsi2}  x &\mapsto  &
 (b_0(x) q^{A(\eps)})^{-1} \mu + b_0(\ol{x}) (b_0(x) 
 b_0(\ol{x}) + 1)^{-1} \lambda \\  \ol{x} & \mapsto &
 b_0(x) (b_0(x) b_0(\ol{x}) + 1)^{-1}  \lambda  \end{eqnarray}
 for $\dim(L_0) = 2$. 
\end{proposition} 

We may write the higher composition maps in terms of
  correlators as follows.  For $\bullet = 0,\eps$ define {\em correlators}
\[ p_{d+1}^\bullet(\sigma_0,\ldots, \sigma_n) \in \Lambda,  \quad m_d^\bullet(\sigma_1,\ldots, \sigma_n) = \sum_{\sigma_0,\alpha} p_{d+1}^\bullet(
\sigma_0,\ldots, \sigma_n) \sigma_0 .\]

\begin{theorem} \label{thm:curve} Assume that if $\dim(L_0) = 2 $, then
  the condition in Definition \ref{rinvd} holds.  Then
  \begin{equation} \label{eq:samecorrs}
  \sum_{r \ge 0} p_{r+1}^\eps ( D_{b_0} \Psi(\sigma),
  b_\eps,\ldots, b_\eps ) = \sum_{r \ge 0} p_{r+1}^0 ( \sigma,
  b_0,\ldots, b_0 ) \end{equation}
for each generator $\sigma \in \cI(\phi_0)$. 
\end{theorem} 

  We first show that the inadmissible configurations, in the sense of Definition \ref{inadmissible} below,  do not contribute to the sum on the left-hand-side of \eqref{eq:samecorrs}. 

\begin{definition} \label{inadmissible}  A component $u_v: S_v \to X_\subset$ of a building
$u:S \to \XX$ is {\em inadmissible} if it has more than one strip-like end.   A building $u: S \to \XX$ is {\em admissible}
if it has no inadmissible components.
\end{definition}

Lemma \ref{lem:noothers1} classifies the possible rigid configurations with inadmissible components.   
\begin{proposition} \label{prop:noothers2}   The weighted count of configurations $(C, u: C \to \XX)$ bounding $\phi_\eps$ with inadmissible surface components $u_v:S_v \to X_\subset$ 
in \eqref{eq:samecorrs} and \eqref{eq:samecorrs2} vanishes.
\end{proposition}

\begin{proof}   We will show that the weighted count of configurations  is a multiple of
a coefficient of a self-intersection point in the curvature  of the weakly bounding cochain
for the unsurgered Lagrangian, which vanishes.   First, suppose there is a single inadmissible component $S_v$.   Lemma \ref{lem:noothers1} implies that rigid such configurations have at least two strip-like ends asymptotic to minimal length Reeb chords on the piece $S_v$,  so that one of the minimal length Reeb chords connects to a configuration not containing
the output.

Choose such a minimal
length chord $\thorn$  leading to a  component in $X_0$ which
does not contain the output.
Splitting at $\thorn$ divides the configuration into two pieces, which we denote by $u_+$  and $u_-$ as in Figure \ref{fig:replace}.   Since there is a single inadmissible component, the piece $u_+$ has no inadmissible components.  
We obtain a building bounding the unsurgered Lagrangian as follows. 
Let $\hat{u}: \hat{S} \to \XX$ denote the configuration obtained by replacing 
$u_-$ with the map to a single quadrant bounding $\phi_0$,
and all levels in $u_+$ mapping to $X_\subset$ with similar quadrants,
as in Figure \ref{fig:replace}.
\begin{figure}[ht]
      \centering
      \scalebox{.5}{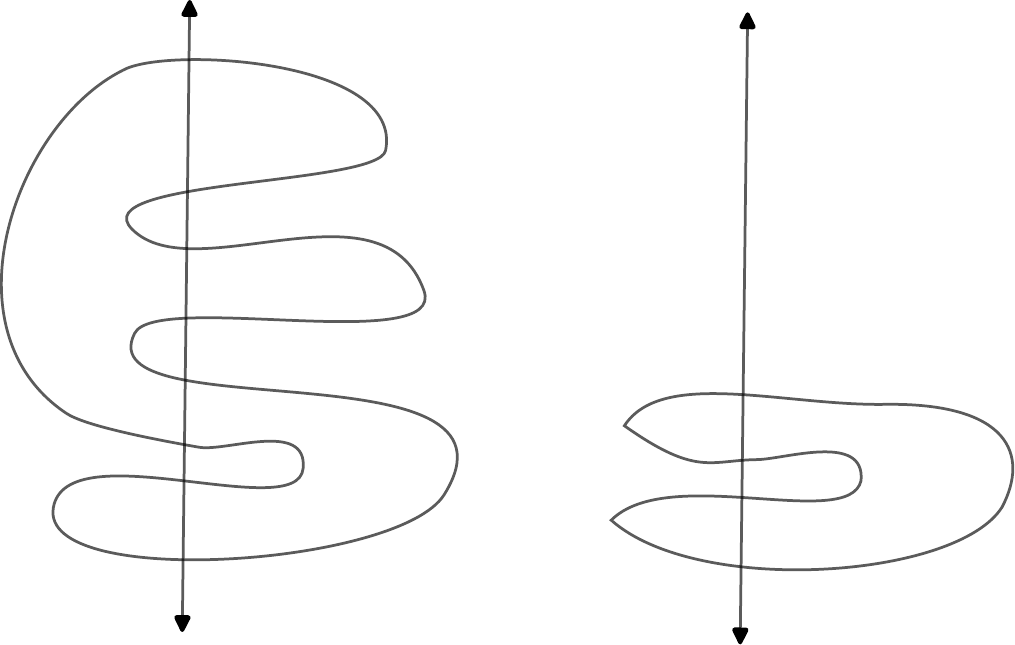}
      \caption{Eliminating levels with multiple ends} 
\label{fig:replace}
\end{figure} 
The resulting building $\hat{u} = (\hat{u}_-,\hat{u}_+)$ has inputs labeled $b_0$ and output labeled $\ol{x}$,
in the component adjacent to $\thorn$.  For each such type $\bGamma$, let $\hat{\bGamma}$ denote the type of $\hat{u}$.
We consider the union of moduli spaces of type $\bGamma$ where $\bGamma$
contains $\bGamma_-$.   
For each such type $\bGamma$, 
the corresponding moduli space  of type $\hat{\bGamma}$
is a product of the moduli space for $\bGamma_-$ and 
the moduli space for type $\hat{\bGamma}$:  We have a bijection
\[ \bigcup_{\bGamma \supset \bGamma_-} \M_{\bGamma}(\phi_\eps) \to \M_{\bGamma_-}(\phi_\eps) \times \bigcup_{\hat{\bGamma}} \M_{\hat{\bGamma}}(\phi_0),
\quad u \mapsto (u_-,\hat{u}_+) .\]
Here $\hat{\bGamma}$ ranges over types with an output labeled $x$ or $\ol{x}$, 
(depending on which minimal length chord $\thorn$ is). The weighted sum over such configurations $\hat{u}$ is the coefficient of $x$ or $\ol{x}$
 in $m_0^{b_0}(1)$ which by assumption vanishes.       

Suppose now that the configuration has an arbitrary number of inadmissible components. 
Removing the {\em inadmissible vertices} $\Ver^{in}(\bGamma) \subset \Ver(\bGamma)$ 
corresponding to inadmissible components $S_v$ of Definition \ref{inadmissible} creates a union of trees $\bGamma_1,\ldots,\bGamma_k$.
At least one of these trees $\bGamma_+ := \bGamma_i$ must be adjacent to a single inadmissible vertex joined 
at a minimal-length orbit $\thorn$, since by Lemma \ref{lem:noothers1}
any inadmissible component has 
\[ e_m(\bGamma) \ge 3  -  d_1(\bGamma) \ge 1. \] 
  That is, 
one starts from any inadmissible vertex and moves outwards along the tree away from the output, choosing an
edge in $\bGamma$ that corresponds to a minimal length Reeb chord at each inadmissible vertex.   If there are several such graphs, we may choose $\bGamma_i$ to be the 
graph containing the last interior edge in the given ordering, so that the decomposition
of $\bGamma$ into 
\[ \bGamma_- := \bGamma - \bGamma_i, \quad \bGamma_+ := \bGamma_i \]
is unique.    

We now sum over the moduli spaces corresponding to the subgraphs separately. 
Let $u_i$ denote the part of $u$ corresponding to $\bGamma_i$.   Replacing each component
of $u_i$ with the corresponding quadrant, and adding a single quadrant at $\thorn$,
produces a configuration $\ol{u}$ bounding $\phi_0$.  Let $\hat{\bGamma}$ denote the type of $\hat{u}$.   The weighted
count of configurations $\hat{u}$ over types $\hat{\bGamma}$ vanishes,   being the coefficient of $x$ or $\ol{x}$ in $m_0^{b_0}(1)$ depending on whether
$\thorn$ is a minimal length chord from $\R^n$ to $i\R^n$ or vice-versa.
\end{proof}

\begin{proof}[Proof of Theorem \ref{thm:curve}] Each correlator is a sum over contributions from disks
  that pass $k_-$ resp. $k_+$ times through the neck region in the
  negative resp.  positive direction:
\[ p_{d+1} ^\bullet (\sigma_0,\ldots, \sigma_d ) = \sum_{k_-,k_+}
p_{d+1}^{\bullet,k_-,k_+}(\sigma_0,\ldots, \sigma_d ) . \]
Each non-zero contribution to $p_{d+1}^{\eps,k_-,k_+}$ has up to $k_+$
groups of inputs labeled $\mu$ and up to $k_-$ groups of inputs
labeled $\mu$ or $\lambda$.   Let
\[ b_{\cap} = b_0 - b_0(x) x - b_0(\ol{x}) \] %
which is the collection of terms of $b_0$ and $b_\eps$ that both
share.   

We first prove the identity \eqref{eq:samecorrs} for generators away from the neck.
Choose a generator $\sigma \in \cI(\phi_0)$ not
 equal to $x,\ol{x}$.  \label{rep:sigchoose}  By definition, $D_{b_0}\Psi(\sigma) = \sigma$.  Let
 $r_j$ resp. $s_j$ be the number of repetitions of $\mu$
 resp. $\lambda$ in the $j$-th group.  Set
\[ c(\mu) = \ln( b_0(x) q^{A(\eps)}) , \quad c(\lambda) =
\ln(b_0(x)b_0(\ol{x}) + 1) .\]
Suppose first that $\dim(L_0) = 2$.  
By Remark \ref{rem:following}, the $j$-th
group of repetitions may be removed at the cost of changing the
correlator by a factorial $r_j!$, where $r_j$ is the length of the
group, so that
\begin{eqnarray} 
 \nonumber
 \sum_{r \ge 1}  p_r^\eps (\sigma,b_\eps,\ldots, b_\eps) 
&=& 
\sum_{ r \ge 1,  k_-,k_+}
\left(  \prod_{j=1}^{k_+}
\sum_{r_j \ge
0}
c(\mu)^{r_j}
(r_j!)^{-1}
\right) \\ 
 && \left( 
\prod_{j=1}^{k_-} 
\sum_{r_j \ge
0 }
(-c(\mu))^{r_j}
(r_j!)^{-1}
\left( -1 +
\sum_{s_j \ge
0 }
c(\lambda)^{s_j}
(s_j!)^{-1}
              \right)
   \right) \\
&&   
p_r^{\eps,k_-,k_+}(\sigma,b_\cap,\ldots, b_\cap) .
\label{cont} \end{eqnarray} 
Here the terms in the sum 
\[ -1 +
\sum_{s_j \ge
0 }
c(\lambda)^{s_j}
(s_j!)^{-1} \] 
come from the two ``wrong-way'' curves in the handle, corresponding to
the points in $S^{n-2} = S^0 = \{ 1, -1 \}$.  We assume that the
local system is chosen so that the boundary of the holomorphic disk in the handle bounding $\phi_\eps$ not passing through $\lambda$
has parallel transport $-1$.  The other curve crosses the
longitude once, possibly with repetitions, hence the sum over $s_j$
in the second term.   
  Denote by
\[ \ul{i}_+ \ \text{resp.} \  \ul{i}_-  \ \subset \{ 1, \ldots, r+k_-+k_+ \} \]
the positions of these groups of label $\mu$ resp. $\lambda$.  Define
\begin{equation} \label{j0}
 j_0(\ul{i}_-,\ul{i}_+) 
: \cI(\phi_0)^{r-k_- - k_+} \to \cI(\phi_0)^{r}
 \end{equation} 
 the map defined by inserting $k_\pm$ labels $x$ resp. $\ol{x}$ at the positions
 $\ul{i}_-,\ul{i}_+$. 
Continuing we have 
\begin{eqnarray}  \nonumber
 \eqref{cont} 
 \nonumber
                                                    &=&
\sum_{   k_-,k_+}  ( \exp(  \ln(
b_0(x)
q^{A(\eps)})))^{k_+
- k_-} 
(-1 + \exp (
\ln(b_0(x)b_0(\ol{x})
+ 1))^{k_-} 
  \\ \nonumber && 
p_r^{\eps,k_-,k_+}(\sigma,  b_{\cap},
                  \ldots, b_{\cap} )
\\
\nonumber
                                                    & =&  \sum_{r,k_-,k_+}    
 (b_0(x) q^{A(\eps)})^{k_+} 
 (b_0(x)^{-1}
 q^{-A(\eps)}
 (-1 + (
 b_0(x)b_0(\ol{x})
 + 1)))^{k_-}
  \\ \nonumber && 
p_r^{0,k_-,k_+}(\sigma,
                                                      b_{\cap},
                  \ldots, b_{\cap} )
\\ \nonumber
  &=& 
\nonumber
\sum_{r,k_-,k_+} 
q^{ (k_+ - k_-) A(\eps)} 
b_0(x)^{k_+} b_0(\ol{x})^{k_-} 
p_r^{\eps,k_-,k_+}(\sigma,   b_{\cap},
      \ldots, b_{\cap}) ) 
  \\ \nonumber
&=&
\sum_{r,\ul{i}_-,\ul{i}_+} b_0(x)^{k_+} b_0(\ol{x})^{k_-} 
p_r^{0,k_-,k_+}(\sigma,  j_0(\ul{i}_-,\ul{i}_+) 
                                                     (   b_{\cap},
      \ldots, b_{\cap}) )  
 \\   \label{firsteq}
&=& \sum_{ r \ge 1} p_r^0(\sigma, b_0,\ldots, b_0) .
\end{eqnarray} 
where $j_0$ is defined in \eqref{j0}.  For the contributions from non-constant disks, the first equality above is an application of Remark \ref{rem:following}, the second is by
the power series of the exponential function, the third and fourth
equalities are algebraic simplifications, the fifth is by Theorem
\ref{thm:biject} and Proposition \ref{prop:noothers2}, and the last is the expansion
$b_0 = b_\cap + b_0(x) x + b_0(\ol{x}) \ol{x}$.

The contributions of the constant disks in the above computation match
by the following argument.  The contribution of the two constant disks
with inputs $x,\ol{x}$ in Remark \ref{rem:constant} to $m_2(x,\ol{x})$  is
\[ m_2(x,\ol{x}) = b_0(x) b_0(\ol{x}) (\sigma_{0,+} - \sigma_{0,-}) + \ldots  . \]
We also have contributions from alternating inputs
$x,\ol{x}, \ldots, \ol{x}$ to $\sigma_{0,\pm}$ with coefficient
$(-1)^{d-1}/d$ by assumption, see Definition \ref{rinvd}.  The sum of
these contributions is
\[ \sum_{d \ge 1} \frac{ (-1)^{d-1}}{d} (b_0(x) b_0(\ol{x}))^d(\sigma_{0,+} -
\sigma_{0,-}) = \ln(b_0(x) b_0(\ol{x}) + 1)\  (\sigma_{0,+} -
\sigma_{0,-}) .\]
Since the classical boundary of $\sigma_1$ is
$\sigma_{0,+} - \sigma_{0,-}$, this sum matches the classical terms in
$p_\eps(\sigma_{n,\pm},c(\lambda) \lambda)$.

It remains to deal with the cases that the constraint on the output is
one of the cells on the neck.  In the case $\sigma = \mu$, the
contributions to $p_d(\mu, \ldots)$ arise from configurations $(C,u:C \to \XX)$ passing
either positively or negatively through the neck region at the
outgoing node.  Write
\[ \delta(\rho) =\ln( (b_0(x) + \rho) q^{A(\eps)} ) 
\mu 
 + \ln(b_0(x)b_0(\ol{x}) + 1) 
 \lambda .\]
Let $p_{r,r_+,r_-}$ denote contributions  to $p_r$  from disks where the first $r_+ + 1$  labels and  last $r_-$ labels  lie on the handle.  The count of disks passing through the handle (with sign depending on whether the disk passes through negatively or positively) 
is 
\begin{eqnarray*}       
  \sum_{r \ge 1}  p_r^\eps(   \mu,b_\eps,\ldots, b_\eps) &=& b_0(x) \sum_{r \ge 1}  p_r^\eps(   b_0(x)^{-1} \mu,b_\eps,\ldots, b_\eps) \\     
  &=& 
      b_0(x) \sum_{r \ge 1, r_\pm \ge 0}   p_{r,r_-,r_+}^{\eps}( b_0(x)^{-1}
      \mu, \delta(0), \ldots, \delta(0),  \\      
&& b_\eps, b_\eps,\ldots,
      b_\eps,
      \delta(0) ,\ldots, \delta(0)) \\    
  &=& 
      b_0(x) \sum_{r \ge 1, r_\pm \ge 0}  \frac{\partial}{\partial
      \rho} |_{\rho = 0}  p_{r,r_-,r_+}^{\eps}(
      \delta(\rho),  \ldots ,     \delta(\rho),  \\      
&& b_\eps, b_\eps,\ldots,
      b_\eps, 
      \delta(\rho),\ldots, \delta(\rho))  
\end{eqnarray*}
\begin{eqnarray}
&=&
b_0(x) \sum_{ r \ge 1}        \nonumber  
\frac{\partial}{\partial \rho} |_{\rho = 0} 
(b_0(x) + \rho) 
p_r^0( x , b_0,\ldots, b_0)
\\       \nonumber   && +
\frac{\partial}{\partial \rho} |_{\rho = 0} 
(b_0(x) + \rho)^{-1} 
(-1 + \exp (
\ln(b_0(x)b_0(\ol{x})
+ 1)))  \\ &&
 p_r^0 \left(  \ol{x}, b_0,\ldots, b_0 \right) \\     \nonumber  
  &=& 
      \sum_{ r \ge 1}   p_r^0( 
      b_0(x) x , b_0,\ldots, b_0) \\ &&- \sum_{r \ge 1} p_r^0 \left(  \frac{-1 + (b_0 (x)b_0(\ol{x}) + 1)}{b_0(x)}\ol{x}, b_0,\ldots, b_0 \right) \\ 
 \label{eq:mu} &=& \sum_{ r \ge 1}   p_r^0(  b_0(x) x - b_0(\ol{x}) \ol{x}, b_0,\ldots, b_0) \end{eqnarray} 
where the terms involving 
$p_r^0( x , b_0,\ldots, b_0)$
in the sum arise  \label{rep:arise} from configurations passing through the handle positively and terms involving
$p_r^0( \ol{x} , b_0,\ldots, b_0)$
arise from configurations passing through the handle negatively. The  presence of  a label $\mu$ in the first  entry forces the first node to map to the handle.  There are  contributions  from any number $r_-$ entries  $\delta(\rho)$ at  the end of the string  $\ul{\sigma}$ and $r_+$ entries $\delta(\rho)$ where those labels 
appear on the same level of the building in the configuration. These
   contribute
   by Remark
   \ref{rem:following}
   with a
   factorial
   entry
   $l = ( 1 +
   r_- +
   r_+)!^{-1}$.
   Since there
   are
   $1 + r_- +
   r_+$
   such
   entries for
   each $l$
   (depending
   on where
   the $0$-th
   entry
   appears in
   the string),
   we obtain a
   contribution
   of
   $(r_ + +
   r_+)!^{-1}$
   after
   summing
   over these
   positions.
   Similarly
   for
   $\sigma =
   \lambda$
   we have
\begin{eqnarray} \nonumber  \label{eq:lam}
  \sum_{d \ge 1}  p_d^\eps(  
  \lambda,b_\eps,\ldots, b_\eps) &=& 
                                     \sum_{ r \ge 1}   q^{A(\eps)} p_r^0(  b_0(x)^{-1} q^{-A(\eps)}  (
                                     b_0(x) b_0(\ol{x}) + 1)
                                     \ol{x}, b_0,\ldots, b_0) \\ 
                                 &=& 
                                     \sum_{ r \ge 1}   p_r^0(
                                     b_0(x)^{-1} (b_0(\ol{x}) b_0(x) +
                                     1) 
                                     \ol{x}
, b_0,\ldots, b_0).\end{eqnarray}
From this and \eqref{eq:mu} we obtain
\begin{eqnarray*} 
\nonumber \sum_{r \ge 1}  p_r^\eps(  D_{b_0} \Psi(x) ,b_\eps,\ldots, b_\eps) &=& 
\sum_{ r \ge 1} p_r^\eps \left(
b_0(x)^{-1} \mu  + \frac{b_0(\ol{x})}{ (b_0(x) b_0(\ol{x})+  1)} \lambda
, b_\eps,\ldots, b_\eps \right) \\
&=& 
\sum_{ r \ge 1} p_r^0( 
b_0(x)^{-1}
( b_0(x)  x -  b_0(\ol{x}) \ol{x}) 
\\ && +
\frac{b_0(\ol{x})  (b_0(\ol{x}) b_0(x) + 1) }{ (b_0(x) b_0(\ol{x}) + 1) b_0(x)}
                                     \ol{x}
,  b_0,\ldots, b_0) \\
&=& \sum_{r \ge 1} p_r^0(x, b_0,\ldots, b_0)  . \label{secondeq}
\end{eqnarray*} 
Finally
\begin{eqnarray} 
\nonumber \sum_{r \ge 1}  p_r^\eps(  D_{b_0} \Psi(\ol{x}) ,b_\eps,\ldots, b_\eps) &=& 
\sum_{r \ge 1} p_r^\eps( b_0(x) (b_0(x) b_0(\ol{x}) + 1)^{-1} 
 \lambda  , b_\eps, \ldots, b_\eps) 
\\
&=& 
\sum_{r \ge 1} p_r^0( \ol{x}, b_0,\ldots, b_0). \label{thirdeq}
\end{eqnarray} 
If $\sigma$ has no $x,\ol{x}$ terms then
$D_{b_0} \Psi (\sigma) = \sigma$.  Together \eqref{firsteq} and
\eqref{secondeq} imply the result for
$\dim(L_0) = 2$.

Now consider the case $\dim(L_0) > 2$.  Let $u: C \to \XX$ be a rigid broken disk bounding 
$\phi_\eps$.  Disk with levels mapping to $\C^n$ which have 
non-minimal angle, or with more than one strip-like-end, do not contribute by 
Proposition \ref{prop:noothers2}.  For levels with a single strip-like end, 
each level passing
through the handle in the positive resp. negative direction must have zero resp. one $\lambda$ label to be rigid, by the dimension formula in Lemma \ref{lem:dim}.  The computation is then the same as in the case
$\dim(L_0) = 2$, but without the repeated $\lambda$ inputs and defining
$c(\lambda) = b_0(x) b_0(\ol{x})$.
\end{proof}

\begin{proof}[Proof of the potential part of Theorem \ref{thm:bij}]  By Theorem \ref{thm:curve}, 
the coefficients of cells not on the neck in $m_0^{b_0}(1)$ and
$m_0^{b_\eps}(1)$ agree.  Furthermore, for $\beta$ in the surgery region 
\begin{eqnarray*}  
  ( m_0^{b_0} (1) , \beta) 
  &:=& \sum_{r}  p_r^0(\beta, b_0,\ldots, b_0)  \\
  &=& \sum_{r}  p_r^\eps(D_{b_0} \Psi (\beta), b_\eps,\ldots, b_\eps)  \\
  &=:&  ( m_0^{b_\eps} (1), (D_{b_0} \Psi)  \beta).
\end{eqnarray*} 
Since $D_{b_0} \Psi$ preserves the identity
$1_{\phi_0} \mapsto 1_{\phi_\eps}$, the potential
is preserved by surgery:
\[  W_0(b_0) = W_\eps(\Psi(b_0)) . \] 
\end{proof}
\begin{remark}  Equation \eqref{eq:samecorrs} also gives the equivalence of the derivatives with respect to a cochain $\tau \in \cI(\phi_0)$
  \begin{multline} \label{eq:samecorrs2}
  \sum_{r \ge 0} p_{r+1}^\eps ( D_{b_0} \Psi(\sigma),
  b_\eps,\ldots,D_{b_0} \Psi (\tau), b_\eps,\ldots, b_\eps ) = \sum_{r \ge 0} p_{r+1}^0 ( \sigma,
  b_0,\ldots, b_0,\tau,b_0, \ldots b_0 ) .\end{multline}
\end{remark}

\subsection{Equivalence of Floer cohomologies}
\label{sec:floerequiv}

To prove the isomorphisms of Floer cohomology, we introduce a quotient
$CF^{\ess}(\phi_0)$ of $CF(\phi_0)$ that captures the cohomology
$HF(\phi_0,b_0)$, and a quotient $CF^{\ess}(\phi_\eps)$ of
$CF(\phi_\eps)$ capturing the cohomology $HF(\phi_\eps,b_\eps)$.  
Recall that the generators for $\phi_\eps$ are obtained by removing
two top-dimensional cells and two ordered self-intersections and
gluing in cells of codimension $0,n-1,1,n$.  Let 
\[ CF^{\loc}(\phi_0) = \on{span} ( \{ \sigma_{n-1,\pm}, \sigma_{n,\pm}
\} ) \subset CF(\phi_0) .\]

\begin{lemma} $CF^{\loc}(\phi_0)$ is a sub-complex
of $CF(\phi_0)$.
\end{lemma} 

\begin{proof} 
By assumption, the almost complex structure $J_\Gamma$ near the
  self-intersection points $x,\ol{x}$ is the standard one.  
  For index reasons, there are
  no rigid buildings $(C,u: C \to \XX)$ with positive area $A(u)$ having
  input $\sigma_n$:  Forgetting the constraint would produce a 
  building in a moduli space of negative expected dimension.  Thus
\[ m_1^{b_0} \sigma_{n,\pm} = \partial \sigma_{n,\pm} =
\sigma_{n-1,\pm} .\] 
Since $b_0 \in MC(\phi_0)$, we have $(m_1^{b_0})^2 = 0$ and so
$m_1^{b_0} \sigma_{n-1,\pm} = 0 $ which proves the claim.
\end{proof} 

A long exact sequence argument implies that the quotient complex has isomorphic cohomology:  The quotient  
\[  CF^{\ess}(\phi_0) = CF(\phi_0)/ CF^{\loc}(\phi_0) \]
(with induced differential denoted $\ol{m}_1^{b_0}$)
fits into a short exact sequence
\[ 0 \to CF^{\loc}(\phi_0) \to CF(\phi_0) \to CF^{\ess}(\phi_0) \to
0  \] 
inducing a long exact sequence in cohomology.  Since
$CF^{\loc}(\phi_0)$ is acyclic,
\begin{equation} \label{h0} H( CF^{\ess}(\phi_0), \ol{m}_1^{b_0} )= H(
  CF(\phi_0),m_1^{b_0} ) . \end{equation}

Similar arguments apply to the cohomology of the surgered Lagrangian. Define a subspace 
$CF^{\loc}(\phi_\eps)$ generated by the 
cells $\sigma_n,\sigma_{n,\pm}, \sigma_{1,\pm}$
and their classical boundaries:
\[ 
CF^{\loc}(\phi_\eps) = \span(\{ \sigma_n , \sigma_{n-1,+} -
\sigma_{n-1,-} 
\}) \subset CF(\phi_\eps) \]

\begin{lemma} $CF^{\loc}(\phi_\eps)$ is  a sub-complex of $CF(\phi_\eps)$
\end{lemma}

\begin{proof} Since there are no rigid treed disks with positive
area and a constraint $\sigma_n$ on the neck, there are no quantum corrections in
  the formula
\[ m_1^{b_\eps} (\sigma_n) = \partial \sigma_n = \sigma_{n-1,+} -
\sigma_{n-1,-} .\]
Since $(m_1^{b_\eps})^2 = 0$, we have
$ m_1^{b_\eps}( \sigma_{n-1,+} - \sigma_{n-1,-}) = 0 $.
\end{proof} 

The quotient complex
\[  CF^{\ess}(\phi_\eps) 
= CF(\phi_\eps)/ CF^{\loc}(\phi_\eps) \]
(with induced differential denoted $\ol{m}_1^{b_\eps}$)
fits into a short exact sequence
\[ 0 \to CF^{\loc}(\phi_\eps) \to CF(\phi_\eps) \to
CF^{\ess}(\phi_\eps) \to 0 .\]
Since $CF^{\loc}(\phi_\eps)$ is acyclic, 
\begin{equation} \label{heps} H( CF^{\ess}(\phi_\eps), \ol{m}_1^{b_\eps}
  ) \cong H( CF(\phi_\eps,),m_1^{b_\eps} ) . \end{equation}

\begin{lemma}  The complexes $CF^{\ess}(\phi_\eps)$ and $CF^{\ess}(\phi_0)$ have 
the same dimension. 
\end{lemma} 

\begin{proof}
The quotient
  $CF^{\ess}(\phi_\eps)$ has two new generators
   compared to $CF^{\ess}(\phi_0)$ corresponding to the
  longitudinal cell in dimension $1$ and the meridional cell in
  dimension $n-1$ compared to $CF^{\ess}(\phi_0)$, but two fewer
  generators corresponding to ordered self-intersection points
  $(x_+,x_-), (x_-,x_+) \in \cI^{\si}(\phi_0)$.
\end{proof}

\begin{proof}[Completion of the proof of Theorem \ref{thm:bij}]
 Let
$\sigma_{n-1,\pm}$ denote the image of $\sigma_{n,-}, \sigma_{n,+}$
in $CF^{\ess}(\phi_\eps)$.  The derivative
$D_{b_0} \Psi$ induces a map on quotient complexes by Lemma
\ref{lem:special}, for which we use the same notation.  
The complex
$CF^\ess(\phi_0)$ is generated by the  images of 
\[ \cI^{\ess}(\phi_0) = \cI(\phi_0) - \{ \sigma_{n-1,\pm},
\sigma_{n,\pm} \} \]
and similarly for $\cI^{\ess}(\phi_\eps)$.  Equation \eqref{eq:samecorrs2} gives the identity
\begin{equation} \label{eq:samecorrs3} ( m_1^{b_0} (\tau) , \alpha) = ( m_1^{b_\eps} (D_{b_0} \Psi
\tau), (D_{b_0} \Psi) \alpha), \ \quad \forall \alpha, \quad 
\text{so}    \ m_1^{b_0} = (D_{b_0} \Psi)^t m_1^{b_\eps} (D_{b_0} \Psi) . \end{equation}
Since $D_{b_0} \Psi$ is invertible, the kernels and cokernels are related by 
\[  (D_{b_0} \Psi ) 
 \ker m_1^{b_0} =
\ker m_1^{b_\eps} , \quad \im
m_1^{b_0} = 
 (D_{b_0} \Psi)^t \im m_1^{b_\eps} .\]
Hence, as claimed
\[ 
HF(\phi_\eps,b_\eps)  \cong 
HF^{\ess}(\phi_\eps,b_\eps) 
\cong HF^{\ess}(\phi_0,b_0)   \cong
HF(\phi_0,b_0)  .
\] 
\end{proof}

\begin{remark} Recall that a deformation of a complex variety $X_0$ \label{rep:complexspace} over
  a pointed base $(S,s_0)$ is a pair 
  \[ ( \pi: X \to S, \phi: \pi^{-1}(s_0) \to X_0 ) \]
  consisting of germ $\pi$ of a flat map together with an
  identification of the central fiber $\phi$.  A deformation is {\em
    versal} if it is complete, that is, if every deformation is obtained
  by pullback by some map; note that this is the weakest notion of
  versality in the literature \cite{catanese}.  There
  are natural notions
  of deformation of morphisms, coherent sheaves, and so on  as in Siu \cite{siu}.
  A naive notion of deformation of an immersed Lagrangian brane
  $\phi \to L$ is given by a family of pairs 
\[ (\phi_s:L \to X, b_s \in MC(\phi_s ))  \]  
parametrized by a point $s$ in a space $S$.  Depending on the structure of
$\phi_s, b_s$, one could speak of analytic, smooth, continuous
deformations and so on.  Clearly, this notion is inadequate as
the deformation does not include the surgered branes near $L$, and one
seems to have real-codimension-one walls \label{rep:walls} at $\val_q(b) =0$.  The results of
this paper imply that those walls vanish by adjoining the
Maurer-Cartan spaces of the surgeries.  In this somewhat vague sense,
we have shown the existence of versal deformations of Lagrangian
branes including the surgered Lagrangians.  It would be interesting to
know whether there is a more precise definition of deformation of a
Lagrangian brane similar to that of a coherent \label{rep:acoherent} sheaf in algebraic
geometry.
\end{remark}

\begin{remark} \label{rem:assumerem}
In the proof of Theorem \ref{thm:bij}, we assume that the Fukaya
   algebras $CF(\phi_0)$ and $CF(\phi_\eps)$ have been defined using
   perturbation data satisfying good invariance properties in
   Definition \ref{def:pinv} and, for Lagrangian surfaces, Definition
   \ref{rinvd}, explained in Section \ref{repeated}. We were left
   feeling that we only partially understood Definition \ref{rinvd},
   and future work will hopefully clarify the situation.  Note that in
   dimension two, one can also assume \eqref{keq} and shift the local
   system rather than the Maurer-Cartan solution to prove invariance
   which avoids the assumption in \ref{rinvd}.  
\end{remark} 

\begin{remark} The almost complex structures admit an sft-style
    limit in Section \ref{sftsec}, in which the self-intersection
    point is isolated by a neck-stretching.  For arbitrary choices of
    perturbation data, the conclusion of Theorem  \ref{thm:bij} \label{rep:bithm} 
    holds without the
    explicit formula in Definition \ref{onsurgery} for the change in
    the weakly bounding cochains $b_0,b_\eps$.\end{remark}

\subsection{Variations of local system} 

\label{Mshiftproof} 
 
The formulas in Definition \ref{onsurgery} are equivalent to slightly different formulas using changes in the local system rather than the weakly bounding cochain.
Suppose that the parallel longitudinal
  transport $\cL_\eps$ from one side of the handle
  $\{ -\infty \} \times S^{n-1}$ to the other
  $\{ \infty \} \times S^{n-1}$ using $y_\eps$ is given by
 \begin{equation} \label{Lshift} \cL_\eps = b_0(x) q^{A(\eps)} \in \Lambda_0 .\end{equation}

 \begin{theorem}  Given the local system $\cL_\eps$ above, the conclusions of Theorem \ref{thm:bij} hold for the surgered  bounding cochain
  \begin{multline}  b_\eps = 
    b_0(x)x - b_0(\ol{x}) \ol{x}  +  
 \begin{cases} 
  \ln(b_0(x)b_0(\ol{x}) + 1) 
 \lambda 
 &    \dim(L_0) = 2 \\
 b_0(x) b_0(\ol{x})  \lambda 
& \dim(L_0) > 2 \end{cases} \end{multline}
\end{theorem}

The proof is essentially the same as that of the main result Theorem \ref{thm:bij}.  In the dimension two case one may sometimes replace the change
in weakly bounding cochain with a change in local system. 
Suppose that $\dim(L_0) = 2$, $L_0$ is connected, and the weakly bounding cochain $b_0$ vanishes except on a single
  one-chain $\kappa: [-1,1] \to L_0$ connecting $x_+$ with $x_-$ which has only
  classical boundary
\begin{equation} \label{keq} m_1(\kappa) = x_+ - x_- .\end{equation}
Define $b_\eps = 0$ and set the parallel transport
$\cM_\eps$ around the meridian of the local system $y_\eps$ to be
\begin{equation} \label{Mshift} \cM_\eps = b_0(x) b_0(\ol{x}) - 1
  \in \Lambda_0 .\end{equation}
Indeed, variation of a weakly bounding cochain $b$ by a degree one
element $b' \in CF^1(\phi)$ is equivalent to a variation of the local
system $y$ by the corresponding representation $\exp(b')$ by the divisor equation \eqref{diveq} in Section \ref{divisor}.

 \section{Quasi-isomorphisms}
\label{sec:qiso}

In this section we show Theorem \ref{thm:qiso}, namely that the objects
defined by the surgered and unsurgered Lagrangian are quasi-isomorphic in a simplified
version of the Fukaya category. 

\subsection{Quasi-isomorphisms induced by Hamiltonian perturbation}

Let $\phi_0'$ be a Hamiltonian perturbation of $\phi_0$ as in the statement of Theorem \ref{thm:bij} \label{rep:bithm2} and $MC(\phi_0'),MC(\phi_0)$
the corresponding Maurer-Cartan spaces.   Symplectomorphisms induce 
\ainfty isomorphisms, since one can use the pull-back almost complex structure
for which the holomorphic disk counts are the same.   A change in almost complex structure from the pull-back almost complex structure to the original almost complex structure
then produces an \ainfty homotopy equivalence by \cite[Section 5]{flips}.  For each $b_0 \in MC(\phi_0)$ there exists a $b_0' \in MC(\phi_0')$ with the same value of the potential:
\[ w(b_0) = w(b_0') \in \Lambda .\]
Lemma \ref{lem:corners} implies that this correspondence extends to the case that the 
element $b_0$ has slightly negative $q$-valuation at $x$, so that if $b_0 \in MC_\delta(\phi_0,\eps)$ then $b_0' \in MC_\delta(\phi_0',\eps)$.
Fix such $b_0, b_0'$ and let $\Fuk_{\ti{\phi}_0}(X)$ denote the \ainfty category with objects
$(\phi_0,b_0)$ and $(\phi_0',b_0')$ with higher compositions for $d \geq 1$
\begin{multline}
m_d^{b_0,\ldots, b_d}(a_1,\ldots,a_d) = \sum_{k_0,\ldots, k_d \ge 0} m_{d + k_0 + \ldots + k_d}(\underbrace{b_0,\ldots, b_0}_{k_0} , a_1,
  \ldots, \\  \underbrace{b_{d-1},\ldots, b_{d-1}}_{k_{d-1}}, a_d,\underbrace{b_d,\ldots, b_d}_{k_d})
\end{multline}
and  define $m_0(1) = 0$; here the notation is motivated by \eqref{eqn:tiphi}.
\label{rep:tphi}
The \ainfty associativity relation for the algebra
$CF(\ti{\phi})$ for the Lagrangian $\ti{\phi}$ of \eqref{eqn:tiphi} implies that 
$\Fuk_{\ti{\phi}_0}(X)$  is a flat \ainfty category.

  \begin{figure}
      \centering
\begingroup%
  \makeatletter%
  \providecommand\color[2][]{%
    \errmessage{(Inkscape) Color is used for the text in Inkscape, but the package 'color.sty' is not loaded}%
    \renewcommand\color[2][]{}%
  }%
  \providecommand\transparent[1]{%
    \errmessage{(Inkscape) Transparency is used (non-zero) for the text in Inkscape, but the package 'transparent.sty' is not loaded}%
    \renewcommand\transparent[1]{}%
  }%
  \providecommand\rotatebox[2]{#2}%
  \newcommand*\fsize{\dimexpr\f@size pt\relax}%
  \newcommand*\lineheight[1]{\fontsize{\fsize}{#1\fsize}\selectfont}%
  \ifx\svgwidth\undefined%
    \setlength{\unitlength}{372.45750127bp}%
    \ifx\svgscale\undefined%
      \relax%
    \else%
      \setlength{\unitlength}{\unitlength * \real{\svgscale}}%
    \fi%
  \else%
    \setlength{\unitlength}{\svgwidth}%
  \fi%
  \global\let\svgwidth\undefined%
  \global\let\svgscale\undefined%
  \makeatother%
  \begin{picture}(1,0.32331776)%
    \lineheight{1}%
    \setlength\tabcolsep{0pt}%
    \put(0,0){\includegraphics[width=\unitlength,page=1]{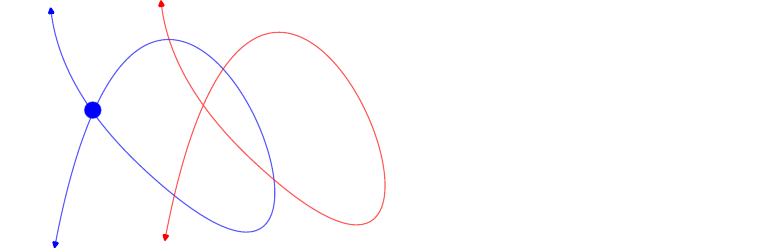}}%
    \put(-0.00119561,0.25000739){\color[rgb]{0,0,0}\makebox(0,0)[lt]{\lineheight{1.25}\smash{\begin{tabular}[t]{l}$\phi_0$\end{tabular}}}}%
    \put(0.22098042,0.3021964){\color[rgb]{0,0,0}\makebox(0,0)[lt]{\lineheight{1.25}\smash{\begin{tabular}[t]{l}$\phi_0'$\end{tabular}}}}%
    \put(0.11362019,0.14712047){\color[rgb]{0,0,0}\makebox(0,0)[lt]{\lineheight{1.25}\smash{\begin{tabular}[t]{l}x\end{tabular}}}}%
    \put(0,0){\includegraphics[width=\unitlength,page=2]{perturb.pdf}}%
    \put(0.57955753,0.29960602){\color[rgb]{0,0,0}\makebox(0,0)[lt]{\lineheight{1.25}\smash{\begin{tabular}[t]{l}$\phi_\eps$\end{tabular}}}}%
    \put(0.72402382,0.29670969){\color[rgb]{0,0,0}\makebox(0,0)[lt]{\lineheight{1.25}\smash{\begin{tabular}[t]{l}$\phi_0'$\end{tabular}}}}%
    \put(0,0){\includegraphics[width=\unitlength,page=3]{perturb.pdf}}%
  \end{picture}%
\endgroup%

      \caption{The unsurgered immersion and its perturbation; the surgered immersion and its perturbation}
      \label{drawing}
  \end{figure}

\begin{lemma}
$(\phi_0',b_0')$ is quasi-isomorphic to $(\phi_0,b_0)$ in $\Fuk_{\ti{\phi}_0}(X)$.
\end{lemma} 

\begin{proof}[Sketch of proof] The space $CF(\phi_0,\phi_0')$ is naturally 
a $(CF(\phi_0),CF(\phi_0'))$-bimodule as explained in 
Charest-Woodward \cite[Chapter 6]{flips}.  The structure maps are defined by a count of 
$(J,H)$-holomorphic strips bounding $\phi_0,\phi_0'$, where $H:[0,1] \times X \to \R$ is the Hamiltonian  whose flow defines $\phi_0'$.   Let 
\[ \ti{H} \in C^\infty(\R \times [0,1] \times X)) \] 
be a function limiting to $0$ as $s \to -\infty$ and $H$ as $s \to + \infty$.  A
count of treed $(J,\ti{H})$-holomorphic strips implies that 
the bimodule $(CF(\phi_0),CF(\phi_0'))$ is homotopy equivalent to 
$CF(\phi_0)$.  The image of the unit $1_{\phi_0}$ under the homotopy equivalence
with $CF(\phi_0,\phi_0')$, and the image of the unit $1_{\phi_0}$ under
the homotopy equivalence with $CF(\phi_0',\phi_0)$ provide \label{rep:provides} the necessary 
elements 
\begin{equation} \label{alphabeta2} \alpha_0 \in CF(\phi_0,\phi_0'),  \ \  \beta_0 \in CF(\phi_0',\phi_0),  \ \ 
\delta_0 \in CF(\phi_0,\phi_0), \ \  \delta_0' \in CF(\phi_0',\phi_0') \end{equation}
as in \eqref{alphabeta}.  
The fact that the composition of these two homotopy
equivalences is homotopic to the identity implies the necessary composition relation for
$m_2(\alpha_0,\beta_0)$ and $m_2(\beta_0,\alpha_0).$
\end{proof} 

We also have a Fukaya category with the same two objects, but where the 
structure maps are defined by pseudoholomorphic buildings.  As a special case of 
Theorem \ref{thm:htpy} (with notation from \eqref{eqn:tiphi}) 
counts of quilted disks define an \ainfty homotopy equivalence
\[ CF(X,\ti{\phi}_0) \cong 
CF(\XX,\ti{\phi}_0) 
.\]
Let $  \Fuk_{\ti{\phi}_0}
(\XX) $ be the category whose 
objects are $\phi_0$ and $\phi_0'$, and whose morphisms 
are the sub-spaces of $CF(X,\ti{\phi}_0)$ in the obvious way.
The homotopy equivalence of \ainfty algebras induces a homotopy 
equivalence of \ainfty categories
\[\Fuk_{\ti{\phi}_0}(X) \to
\Fuk_{\ti{\phi}_0}(\XX) .\] 
The existence of quasi-isomorphisms with the broken limit $CF(\XX,\ul{\phi}_0)$ 
implies that $(\phi_0,b_0)$ and $(\phi_0',b_0')$ define quasi-isomorphic objects in 
$ \Fuk_{\ti{\phi}_0}(\XX)$.

\subsection{Quasi-isomorphism with the surgery}
\label{sec:qi}
We now use the computations above to show that the surgered Lagrangian is quasi-isomorphic
to the unsurgered Lagrangian as objects in the Fukaya category. 
Since the intersections of $\phi_0$ and $\phi_0'$ are disjoint from the surgery 
regions, we have a canonical bijection between intersection points $  (\phi_0 \times \phi_0')^{-1}(\Delta) $ and  $(\phi_\eps \times \phi_0')^{-1}(\Delta)$.  These induce identifications 
\begin{equation} \label{eq:isos}
CF(\phi_0,\phi_0') \cong CF(\phi_\eps, \phi_0'), 
\quad CF(\phi_0', \phi_0) \cong CF(\phi'_0, \phi_\eps ) .\end{equation}
We denote by 
\[ \alpha_\eps \in CF(\phi_\eps,\phi_0'),  \quad \beta_\eps \in CF(\phi_0',\phi_\eps) \] 
the images of $\alpha_0$ and $\beta_0$ under the isomorphisms \eqref{eq:isos}. 
Furthermore, let 
\[ \delta_\eps \in CF(\phi_\eps,\phi_\eps) \] 
the image of $\delta_0 $ under the map $D \Psi$.   We claim
\begin{equation} \label{eq:m2} 1_{\phi_0} = m_2^{b_0,b_0',b_0}(\beta_0,\alpha_0) - m_1^{b_0,b_0}(\delta_0) = (D_{b_0} \Psi)^t (m_2^{b_\eps,b_0',b_\eps}(\beta_\eps,\alpha_\eps) - m_1^{b_\eps,b_\eps}(\delta_\eps))  .\end{equation}
The configurations contributing to the composition maps in the last expression in \eqref{eq:m2} are broken maps whose components mapping to $X_\subset$
either have a single strip-like end (in which case they are classified in Section \ref{sec:mintype}) or those with more than one strip-like end.  The identity 
\[   m_1^{b_0,b_0}(\delta_0) = (D_{b_0} \Psi)^tm_1^{b_\eps,b_\eps}(\delta_\eps) \]
is a special case of \eqref{eq:samecorrs3}.  To prove 
\[ m_2^{b_0,b_0',b_0}(\beta_0,\alpha_0)  = (D_{b_0} \Psi)^t (m_2^{b_\eps,b_0',b_\eps}(\beta_\eps,\alpha_\eps)  \] 
note that if a component  $S_{v_0}$
mapping to $X_\subset \cong \C^n$ has more than one strip-lie end, 
then one of the components of $S - S_{v_0}$ must contain the corners 
labeled $\alpha_0,\beta_0$ and the others must contain only edges labeled $b_0$, 
as in Figure \ref{fig:replace2}; compare with Figure \ref{fig:replace} which involved
a single Lagrangian.

\begin{figure}[ht]
      \centering
      \scalebox{.5}{
\begingroup%
  \makeatletter%
  \providecommand\color[2][]{%
    \errmessage{(Inkscape) Color is used for the text in Inkscape, but the package 'color.sty' is not loaded}%
    \renewcommand\color[2][]{}%
  }%
  \providecommand\transparent[1]{%
    \errmessage{(Inkscape) Transparency is used (non-zero) for the text in Inkscape, but the package 'transparent.sty' is not loaded}%
    \renewcommand\transparent[1]{}%
  }%
  \providecommand\rotatebox[2]{#2}%
  \newcommand*\fsize{\dimexpr\f@size pt\relax}%
  \newcommand*\lineheight[1]{\fontsize{\fsize}{#1\fsize}\selectfont}%
  \ifx\svgwidth\undefined%
    \setlength{\unitlength}{332.57425383bp}%
    \ifx\svgscale\undefined%
      \relax%
    \else%
      \setlength{\unitlength}{\unitlength * \real{\svgscale}}%
    \fi%
  \else%
    \setlength{\unitlength}{\svgwidth}%
  \fi%
  \global\let\svgwidth\undefined%
  \global\let\svgscale\undefined%
  \makeatother%
  \begin{picture}(1,0.91600993)%
    \lineheight{1}%
    \setlength\tabcolsep{0pt}%
    \put(0,0){\includegraphics[width=\unitlength,page=1]{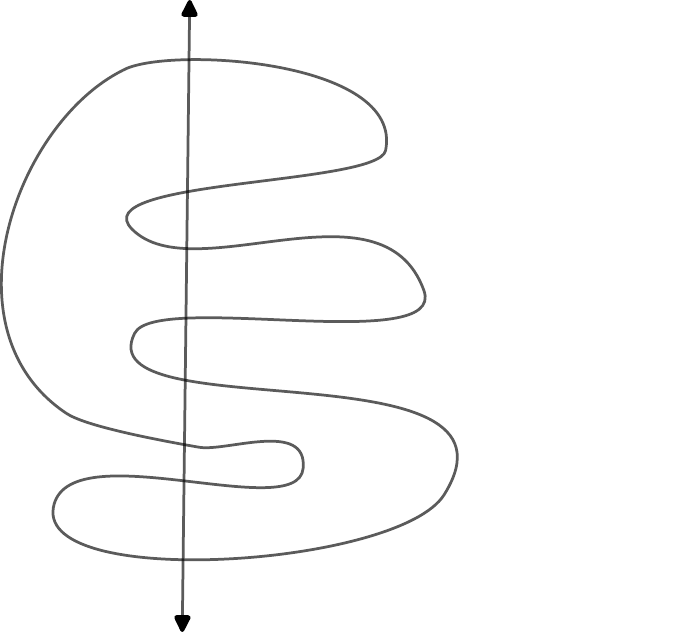}}%
    \put(0.50430822,0.24657254){\color[rgb]{1,0.33333333,0.33333333}\makebox(0,0)[lt]{\lineheight{1.25}\smash{\begin{tabular}[t]{l}$u_+$\end{tabular}}}}%
    \put(0,0){\includegraphics[width=\unitlength,page=2]{replace2.pdf}}%
    \put(0.07359551,0.44852748){\color[rgb]{0,0,0}\transparent{0.64502198}\makebox(0,0)[lt]{\lineheight{1.25}\smash{\begin{tabular}[t]{l}$u_-$\end{tabular}}}}%
    \put(0,0){\includegraphics[width=\unitlength,page=3]{replace2.pdf}}%
    \put(0.26654692,0.31118368){\color[rgb]{0,0,0}\transparent{0.64502198}\makebox(0,0)[lt]{\lineheight{1.25}\smash{\begin{tabular}[t]{l}$\gamma$\end{tabular}}}}%
    \put(0,0){\includegraphics[width=\unitlength,page=4]{replace2.pdf}}%
    \put(0.01240944,0.78260846){\color[rgb]{0,0,0}\makebox(0,0)[lt]{\lineheight{1.25}\smash{\begin{tabular}[t]{l}$\phi_\eps$\end{tabular}}}}%
    \put(0.88417251,0.68333296){\color[rgb]{0,0,0}\makebox(0,0)[lt]{\lineheight{1.25}\smash{\begin{tabular}[t]{l}$\phi_0'$\end{tabular}}}}%
  \end{picture}%
\endgroup%
}
      \caption{Eliminating levels with multiple ends, second version} 
\label{fig:replace2}
\end{figure} 

Since at least one of the other ends maps to a minimal length Reeb chord, the 
contributions from such configurations vanish, as in Proposition \ref{prop:noothers2}.
The correspondence in Theorem \ref{thm:biject} then implies the required identity.
A similar identity holds for the composition in the reverse order, showing that the composition is the identity for $\phi_\eps$.  Hence $\alpha_\eps,\beta_\eps$ are quasi-isomorphisms as claimed.  This completes the proof of Theorem \ref{thm:qiso}.

\subsection{Mapping cones}
\label{sec:mappingcone}

In the case of a single intersection point of a pair of embedded Lagrangians, the main
result of this paper reproduces the identification 
of the surgery with the mapping cone, which was the original intent of
Fukaya-Oh-Ohta-Ono \cite[Chapter 10]{fooo}, see also Abouzaid
\cite{ab:fr}, Mak-Wu \cite{mak:dehn}, Tanaka \cite{tanaka}, and
Chantraine-Dimitroglou-Rizell-Ghiggini-Golovko \cite[Chapter
8]{chantraine}.  The special case that one of the Lagrangians is a
Lagrangian sphere was treated earlier by Seidel \cite{se:lo} in his
paper on symplectic Dehn twists.  Pascaleff-Tonkonog \cite{pasc:wall}
have developed a generalization to clean intersections, related to
higher-dimensional analogs of Lagrangian mutation.

We put ourselves in the following simple version of the Fukaya category, 
generated by two branes.   Suppose
that the immersion $\phi_0:L_0 \to X$ is the disjoint union of
immersions $\phi_{\pm}:L_\pm \to X$ intersecting transversally
equipped with weakly bounding cochains $b_\pm \in MC(\phi_\pm)$.
Denote the combined immersion by
\[ \phi_0 = \phi_- \sqcup \phi_+ : L_- \sqcup L_+ \to X .\]
Recall that $CF(\phi_-,\phi_+)$ is the subspace of $CF(\phi_0)$ generated by
the intersection points of $\phi_-$ and $\phi_+$.  As vector spaces
\[ CF(\phi_0) \cong CF(\phi_-) \oplus CF(\phi_+) \oplus CF(\phi_-,\phi_+) \oplus
CF(\phi_+,\phi_-) \]
and $CF(\phi_\pm)$ are \ainfty sub-algebras.  The space $CF(\phi_-,\phi_+)$ is
naturally an \ainfty bimodule over the \ainfty algebras $CF(\phi_-)$ and
$ CF(\phi_+)$.  Let
\[  c \in CF(\phi_-,\phi_+), \quad m_1^{b_-, b_+}(c) = 0 \]
be a cocycle.  Let $\psi:K \to X$ be another immersed Lagrangian brane
in $X$ meeting $\phi_+,\phi_-$ transversally and disjoint from
$\phi(L_+ ) \cap \phi(L_-)$.  Suppose that $K$ is equipped with a bounding cochain
$k \in MC(\psi)$ with
\[ W(k) = W(b_-) = W(b_+) .\]
The complex $\Hom(\Cone(c),K)$ is by definition
\[ \Hom(\Cone(c),K) = CF(L_-,K)[1] \oplus CF(L_+,K) \]
with differential $m_1^{b_- + b_+ + c,k}$ induced by the differentials
on $CF(L_\pm,K)$ and composition with $c$, see for example Seidel
\cite[2.10]{se:ho}.

\begin{theorem} \label{thm:mappingcone} {\rm (cf. \cite[Remark 54.9,
    Chapter 10]{fooo})}    
     Suppose $L_\pm,K$ are as above and
  $\dim(L_\pm) > 2$. Suppose that $x \in \phi_-(L_-) \cap \phi_+(L_+)$
  is an odd self-intersection point and
  \[ 
  c = q^{-A(\eps)} x \in CF(L_-,L_+), \quad m_1^{b_-,b_+} c = 0 \]
  is a cocycle.  Let $\phi_\eps: L_\eps \to X$ denote the $\eps$-surgery at $x$ with
  cochain $b_\eps = b_+ + b_-$ with $b_\pm$ vanishing in an open
  neighborhood of $x$.  Then the complex $CF(\phi_\eps,K)$ with
  differential $m_1^{b_\eps,k}$ is homotopy equivalent to the mapping
  cone $\Hom(\Cone(c),K)$.
\end{theorem}

\begin{remark} The special case that one of the Lagrangians is a
  Lagrangian sphere was treated earlier by Seidel \cite{se:lo}.  In
  this case, say $L_-$ is a sphere, the surgery
  $\phi_\eps:L_\eps \to X$ is embedded and Hamiltonian isotopic to the
  {\em Dehn twist} $\tau_{L_-} L_+ $ of $L_+$ around $L_-$.  Here the
  Dehn twist $\tau_{L_-} \in \Aut(X,\omega)$ is a symplectomorphism on
  $X$ that restricts to minus the identity on $L_-$ and is supported
  on a neighborhood of $L_-$.  Surgering all self-intersections
  simultaneously gives an exact triangle in the derived Fukaya
  category
\[ \Hom(L_-,L_+) L_- \to L_+ \to \tau_{L_-}(L_+) \to \Hom(L_-,L_+) L_-
[1] ,\]
see Seidel \cite[Proposition 9.1]{se:ho}.    This ends the Remark.
\end{remark} 

\begin{proof}[Proof of Theorem \ref{thm:mappingcone}]
    We apply neck stretching to obtain homotopy-equivalent complexes defined by broken maps. Let
  $\phi_\pm:L_\pm \to X$ be embeddings as in the statement of the
  theorem and $b_\pm \in MC_\delta(L_\pm)$ projective Maurer-Cartan
  solutions.  As in Theorem \ref{thm:htpy}, the complexes $CF(\phi_\eps,K)$,
  $\Hom(\Cone(c),K)$ are homotopy equivalent to those defined by curve counts in the broken limit $\XX$ in
  which the almost complex structure is stretched along a sphere
  enclosing the given intersection point $x \in L_- \cap L_+$. 
  
  We now compare the broken limits.   Any configuration
  $(C,u_0 : C \to \XX)$ with boundary on $L_0 \cup L_1 \cup K$, and a single corner at $c$, contributing to a structure map of $\Hom(\Cone(c),K)$
  corresponds under the map in \label{rep:intheorem} Theorem \ref{thm:biject} with a curve
  $(C,u_\eps: C \to \XX)$ with boundary on $L_\eps \cup K$.  The number of
  corners of $u_0$ on $x$ is equal to the number of times that $u_\eps$
  passes through the handle $H_\eps$ positively.  
  Since  $\dim(L_0) > 2$ by assumption, any rigid curve $u_\eps$ passes in the positive
  direction on the handle by Theorem \ref{thm:biject}. 
    That is, there are
  no ``wrong way'' corners to deal with in the bijection between
  holomorphic disks. 
    On the surgered side, only configurations passing through the handle in the positive direction can be rigid, by the dimension formula in Lemma \ref{lem:dim}.
  The area of $A(u_\eps)$ is
  $A(u_0) - \kappa A(\eps)$ as in Lemma \ref{lem:ncorners}.  Counting
  rigid curves $u_\eps$ defines the differential on $CF(\phi_\eps,K)$
  using the bounding cochain $b_- + b_+$.  One obtains an
  identification of the complex $CF(\phi_\eps,K)$ with
  $\Hom(\Cone(c),K)$.  \end{proof}

\begin{remark} Fukaya-Oh-Ohta-Ono \cite[Theorem 56.14, Chapter 10]{fooo}
use this identification with the mapping cone to show that there exists
a Lagrangian in the six-dimensional symplectic torus whose Fukaya algebra
(defined using their foundational system, presumably equivalent to ours) 
has no projective Maurer-Cartan solutions.  
\end{remark}

 \def\C Prime{$'$} \def\C Prime{$'$} \def\C Prime{$'$} \def\C Prime{$'$}
\def\C Prime{$'$} \def\C Prime{$'$}
\def\polhk#1{\setbox0=\hbox{#1}{\ooalign{\hidewidtht 
      \lower1.5ex\hbox{`}\hidewidth\crcr\unhbox0}}} \def\C Prime{$'$}
\def\C Prime{$'$}


\begin{thebibliography}{10}

\bibitem{abbas:com}
C.~Abbas. 
\newblock {\em An introduction to compactness results in symplectic field 
  theory}. 
\newblock Springer, Heidelberg, 2014.

\bibitem{ab:fr} M.~Abouzaid.
\newblock
On the Fukaya categories of higher genus surfaces.
\newblock {\em Advances in Mathematics} 217: 1192--1235, 2008.

\bibitem{ab:ex} M.~Abouzaid.  \newblock Framed bordism and Lagrangian
  embeddings of exotic spheres.  \newblock {\em Ann. of Math.} (2) 175:71--185,
  2012.

\bibitem{akaho}
M.~Akaho and D.~Joyce. 
Immersed Lagrangian Floer Theory. 
{\em J. Differential Geom.}  86:381--500, 2010.


\bibitem{ar:alg2}
E.~Arbarello, M.~Cornalba, and P.~A. Griffiths.
\newblock {\em Geometry of algebraic curves. {V}olume {II}}, volume 268 of {\em
  Grundlehren der Mathematischen Wissenschaften [Fundamental Principles of
  Mathematical Sciences]}.
\newblock Springer, Heidelberg, 2011.
\newblock With a contribution by Joseph Daniel Harris.

\bibitem{ab:ym} M. F. Atiyah and R. Bott.
\newblock The Yang-Mills Equations over Riemann Surfaces.
Philosophical Transactions of the Royal Society of London. Ser. A. 308
(1505): 523--615.

\bibitem{audin:morse} Mich\'ele Audin and Mihai Damian.  \newblock
  Morse Theory and Floer Homology.  \newblock Translated from the 2010
  French original by Reinie Ern\'e. Universitext. Springer, London,
  2014.

\bibitem{auroux:asym}
D.~Auroux. 
\newblock Asymptotically holomorphic families of symplectic submanifolds. 
\newblock {\em Geom. Funct. Anal.}, 7(6):971--995, 1997. 

\bibitem{auroux1} D. Auroux. Mirror symmetry and T-duality in the
  complement of an anticanonical divisor. J. Gokova Geom. Topol.,
  1:51–91, 2007.

\bibitem{auroux2} D. Auroux. Special Lagrangian fibrations,
  wall-crossing, and mirror symmetry. In Geometry, analysis, and
  algebraic geometry, volume 13 of Surveys in Differential Geometry,
  pages 1–47. Intl. Press, 2009.
  
  
\bibitem{auroux:complement} D.~Auroux, D.~Gayet, and J.-P.~Mohsen. 
  \newblock Symplectic hypersurfaces in the complement of an isotropic 
  submanifold.  \newblock {\em Math. Ann.}, 321(4):739--754, 2001. 


\bibitem{bc:ql}
P.~Biran and O.~Cornea. 
\newblock Quantum structures for {L}agrangian submanifolds. 
\newblock 
\newblock \href{http://www.arxiv.org/abs/0708.4221}{arxiv:0708.4221}.

\bibitem{bo:com}
F.~Bourgeois, Y.~Eliashberg, H.~Hofer, K.~Wysocki, and E.~Zehnder. 
\newblock Compactness results in symplectic field theory. 
\newblock {\em Geom. Topol.}, 7:799--888 (electronic), 2003.

\bibitem{catanese} Catanese, F. A superficial working guide to
  deformations and moduli. Handbook of moduli. Vol. I, 161–215,
  Adv. Lect. Math. (ALM), 24, Int. Press, Somerville, MA,
  2013. 
\href{http://www.mathe8.uni-bayreuth.de/pdf/CataneseFabrizio/128.pdf}
{Retrieved at http://www.mathe8.uni-bayreuth.de/pdf/CataneseFabrizio/128.pdf}

\bibitem{ciel:switch} K. Cieliebak, T. Ekholm, and J. Latschev. Compactness for holomorphic curves with switching Lagrangian
  boundary conditions. J. Symplectic Geom. 8 (2010), no. 3, 267–298.

\bibitem{chanda} S. Chanda.   Floer Cohomology and Higher Mutations. 
 \href{http://www.arxiv.org/abs/2301.08311}{arxiv:2301.08311}.

\bibitem{chantraine}
B.~Chantraine, G.~Dimitroglou Rizell, P.~Ghiggini, R.~Golovko. 
Geometric generation of the wrapped Fukaya category of Weinstein 
manifolds and sectors. 
\href{http://www.arxiv.org/abs/1712.09126}{arxiv:1712.09126}



\bibitem{flips}
F.~Charest and C.~Woodward. 
\newblock Floer theory and flips. 
\newblock Mem. Amer. Math. Soc. 279, 2022. 



\bibitem{cho} C.-H.~Cho. Products of Floer cohomology of torus fibers 
  in toric Fano manifolds.  Comm. Math. Phys.  260: 613--640, 2005. 

\bibitem{chooh:toric}
C.-H.~Cho and Y.-G.~Oh. 
\newblock Floer cohomology and disc instantons of {L}agrangian torus fibers in 
  {F}ano toric manifolds. 
\newblock {\em Asian J. Math.}, 10(4):773--814, 2006. 

\bibitem{cm:com} K.~Cieliebak and K.~Mohnke.  \newblock Compactness 
  for punctured holomorphic curves.  \newblock {\em J. Symplectic 
    Geom.}, 3(4):589--654, 2005. 

\bibitem{cm:trans}
K.~Cieliebak and K.~Mohnke. 
\newblock Symplectic hypersurfaces and transversality in {G}romov-{W}itten 
  theory. 
\newblock {\em J. Symplectic Geom.}, 5(3):281--356, 2007.


\bibitem{de:ca}
J.~P.~Demailly.
\newblock 
Complex Analytic and Differential Geometry.
\newblock  
Universit\'e de Grenoble.
\newblock  
\href{http://www-fourier.ujf-grenoble.fr/~demailly/manuscripts/agbook.pdf}{http://www-fourier.ujf-grenoble.fr/$\sim$demailly/manuscripts/agbook.pdf}


\bibitem{det:ref}
G.~Dimitroglou Rizell, T.~Ekholm, and D.~Tonkonog. 
 \newblock Refined disk potentials for immersed Lagrangian surfaces.
\newblock {\em J. Differential Geom.} 121:459--539, 2022.
\newblock \href{http://www.arxiv.org/abs/1806.03722}{arxiv:1806.03722}.

\bibitem{don:symp}
S.~K. Donaldson. 
\newblock Symplectic submanifolds and almost-complex geometry. 
\newblock {\em  J. Differential Geom.} 44 (1996), no. 4, 666--705. 


\bibitem{ekholm:morse}
Tobias Ekholm.
\newblock Morse flow trees and {L}egendrian contact homology in 1-jet spaces.
\newblock {\em Geom. Topol.}, 11:1083--1224, 2007.



\bibitem{fang}  K.-Y.~Fang.  \newblock Geometric constructions of mapping
  cones in the Fukaya category.   \newblock Ph.D. Thesis, Berkeley, 2018. 

\bibitem{floer:monopoles} A. Floer. Monopoles on asymptotically flat 
  manifolds. In: Hofer H., Taubes C.H., Weinstein A., Zehnder E. (eds) 
  The Floer Memorial Volume. Progress in Mathematics, vol 133. 
  Birkhäuser Basel.


\bibitem{forster} O~Forster and K.~Knorr. \"{U}ber die Deformationen
  von Vektorraumb\"{u}ndeln auf kompakten komplexen
  R\"{a}umen. (German) Math. Ann. 209 (1974), 291–346.

\bibitem{totreal}
U.~Frauenfelder and K.~Zehmisch. 
\newblock Gromov compactness for holomorphic discs with totally real boundary conditions. 
\newblock {\em  J. Fixed Point Theory Appl. }  17  (2015),  no. 3, 521--540. 
		
\bibitem{fooo} K.~Fukaya, Y.-G.~Oh, H.~Ohta, and K.~Ono.  \newblock 
  {\em Lagrangian intersection {F}loer theory: anomaly and 
    obstruction.}, volume~46 of {\em AMS/IP Studies in Advanced 
    Mathematics}.  \newblock American Mathematical Society,
  Providence, RI, 2009.  
Orientation chapter at 
  \href{https://www.math.kyoto-u.ac.jp/~fukaya/bookchap9071113.pdf}{https://www.math.kyoto-u.ac.jp/$\sim$fukaya/bookchap9071113.pdf 
    version 2007}.  Surgery chapter at 
  \href{https://www.math.kyoto-u.ac.jp/~fukaya/Chapter10071117.pdf}{https://www.math.kyoto-u.ac.jp/~fukaya/Chapter10071117.pdf}. 

\bibitem{ganatra}
S.~Ganatra. 
\newblock 
Symplectic Cohomology and Duality for the 
Wrapped Fukaya Category. 
\newblock PhD Thesis, Massachusetts Institute of Technology, 2006. 

\bibitem{hp:remov}
R. Harvey and J. Polking.
Removable singularities of solutions of linear partial differential equations.
{\em Acta Math.} {125}:{39--56} 1970.

\bibitem{haug} L. Haug. 
\newblock
Lagrangian antisurgery. \newblock {\em Math. Res. Lett.} 27:1423--1464, 2020.


\bibitem{hicks:wall} J.~Hicks. 
Wall-crossing from Lagrangian Cobordisms.
Algebr. Geom. Topol. 24:  3069--3138, 2024.
\href{https://arxiv.org/abs/1911.09979}{arXiv:1911.09979}.


\bibitem{hicks} J.~Hicks. 
Lagrangian cobordisms and Lagrangian surgery.
Comment. Math. Helv. 98:509--595, 2023. 
\href{https://arxiv.org/abs/2102.10197}{arXiv:2102.10197}.

\bibitem{hkl}
H.~Hong, Y.~Kim, and S.-C.~Lau. 
\newblock  Immersed two-spheres and SYZ for Gr(2,$\CC^4$) and 
  OG(1,$\CC^5$). 
\href{https://arxiv.org/abs/1805.11738}{arxiv:1805.11738}.

\bibitem{jy} A. Jacob, and S.-T. Yau.  \newblock A special Lagrangian
  type equation for holomorphic line bundles.  \newblock {\em
    Math. Ann.} 369 (2017), no.1– 2, 869--898.

\bibitem{joyce:conjectures}
D.~Joyce.
\newblock
Conjectures on Bridgeland stability for Fukaya categories of
Calabi-Yau manifolds, special Lagrangians, and Lagrangian mean
curvature flow.  EMS Surv. Math. Sci. 2 (2015). \newblock \href{http://www.arxiv.org/abs/1401.4949}{arxiv:1401.4949}.


\bibitem{kap:dbranes}
A.~{Kapustin} and Y.~{Li}.
\newblock {D-branes in Landau-Ginzburg models and algebraic geometry}.
\newblock {\em Journal of High Energy Physics}, 12:5, 2003.
curvature flow.  EMS Surv. Math. Sci. 2 (2015). \newblock \href{http://www.arxiv.org/abs/hep-th/0210296}{arxiv:hep-th/0210296}.


\bibitem{km:cohft}
M. Kontsevich and Y. Manin.
\newblock Gromov-Witten classes, quantum cohomology, and enumerative geometry.
\newblock {\em Comm. Math. Phys.} 164:525-562, 1994.


\bibitem{ks} M. Kontsevich and Y. Soibelman. Affine structures and
  non-Archimedean analytic spaces. In {\em The unity of Mathematics},
  volume 244 of Progr. Math., pages 321–385. Birkh\"auser Boston,
  Boston, MA, 2006.


\bibitem{ls:sous}
F.~Lalonde and J.-C.~Sikorav.
\newblock Sous-vari\'et\'es lagrangiennes et lagrangiennes exactes des
fibr\'es cotangents. 
\newblock Commentarii mathematici Helvetici, 66(1):18–33, 1991.


\bibitem{lm} R.~B.~Lockhart and R.~C.~McOwen.  \newblock Elliptic
  differential operators on noncompact manifolds \newblock {\em Annali
    della Scuola Normale Superiore di Pisa - Classe di Scienze},
  S\'erie 4, Volume 12, no. 3, p. 409--447, 1985.

\bibitem{ms:jh}
D.~McDuff and D.~Salamon. 
\newblock {\em {$J$}-holomorphic curves and symplectic topology}, volume~52 of 
  {\em American Mathematical Society Colloquium Publications}. 
\newblock American Mathematical Society, Providence, RI, 2004.

\bibitem{mak:dehn} C.~Y.~Mak and W.~Wu.  \newblock Dehn twist exact 
  sequences through Lagrangian cobordism. 
\newblock {\em Compos. Math.}  154, no. 12, 2485--2533, 2018.
  \href{http://www.arxiv.org/abs/1509.08028}{arxiv:1509.08028}.

\bibitem{ainfty} S.~Ma'u, K.~Wehrheim, and C.T. Woodward.  \newblock
  ${A}_\infty$-functors for {L}agrangian correspondences.  \newblock
  {\em Selecta Math.} 24(3), p. 1913--2002, 2018.
  
\bibitem{oh:rh} Y.~G.~Oh. 
\newblock Riemann-Hilbert problem and application 
  to the perturbation theory of analytic discs. {\em Kyungpook Math. J.},
  35(1):39–75, 1995.


\bibitem{pw:flow}
J.~Palmer and C.~Woodward. 
\newblock Immersed Lagrangian Floer theory and Maslov flow.
\newblock {\em
Algebr. Geom. Topol.}  21(5): 2313--2410, 2021.
\newblock \href{http://www.arxiv.org/abs/1804.06799}{arXiv:1804.06799}.

\bibitem{pasc:wall}
J.~Pascaleff and D.~Tonkonog. 
\newblock  
 The wall-crossing formula and Lagrangian mutations.   {\em Adv. Math. } 
 361:106850,  2020. \newblock \href{http://www.arxiv.org/abs/1711.03209}{arxiv:1711.03209}. 


\bibitem{surger} L.~Polterovich.  \newblock The surgery of Lagrange
  submanifolds.  \newblock {\em Geom. Funct. Anal. }  1: 198--210,  1991.

\bibitem{pozniak} 
 M. Pozniak.
\newblock  Floer homology, Novikov rings and clean intersections.
Northern California Symplectic Geometry Seminar, 119-181, AMS Transl. Ser. 2, 196, AMS, Providence, RI, 1999. 		

\bibitem{rs:asym} J.~Robbin and D.~Salamon.  Asymptotic behaviour of 
  holomorphic strips.  {\em Ann. Inst. H. Poincar\'e Anal. Non 
    Lin\'eaire} 18:573--612, 2001. 

\bibitem{clean} 
F.~Schm\"aschke. 
\newblock 
Floer homology of Lagrangians in clean intersection. 
\newblock 
\href{http://www.arxiv.org/abs/1606.05327}{arxiv:1606.05327}. 
  
\bibitem{seeley}  R.~T.~Seeley.  Extension of $C^\infty$ functions defined in a half space.  {\em Proc. Amer. Math. Soc.} 15: 625--626,
1964.

\bibitem{se:bo}
P.~Seidel.
\newblock {\em Fukaya categories and {P}icard-{L}efschetz theory}.
\newblock Zurich Lectures in Advanced Mathematics. European Mathematical
  Society (EMS), Z\"urich, 2008.

\bibitem{se:lo}
P.~Seidel. 
\newblock A long exact sequence for symplectic {F}loer cohomology. 
\newblock {\em Topology}, 42(5):1003--1063, 2003.

\bibitem{se:gr}
P.~Seidel. 
\newblock Graded {L}agrangian submanifolds. 
\newblock {\em Bull. Soc. Math. France}, 128(1):103--149, 2000.

\bibitem{se:ho} P.~Seidel.  \newblock Homological mirror symmetry for 
  the quartic surface.  \newblock Memoirs of the American Mathematical 
  Society, 2015, Volume 236. 
 \newblock \href{http://www.arxiv.org/abs/math/0310414}{arXiv:0310414}

\bibitem{seidel:genustwo}
P.~Seidel. 
\newblock Homological mirror symmetry for the genus two curve. 
\newblock {\em J. Algebraic Geom.} 20:727--769, 2011.


\bibitem{ss:cob}
N.~Sheridan and I.~Smith. 
\newblock Lagrangian cobordism and tropical curves. 
\newblock J. Reine Angew. Math. 774 (2021), 219--265. 
\href{http://www.arxiv.org/abs/1805.07924}{arXiv:1805.07924}.

\bibitem{siu} 
Y.T.~Siu and G.~Trautmann.
\newblock  Deformations of coherent analytic sheaves with compact supports. 
\newblock {\em Mem. Amer. Math. Soc.} 29 (1981), no. 238.

\bibitem{tanaka}
H.~L.~Tanaka.  
\newblock Surgery induces exact sequences in Lagrangian cobordisms.
\newblock  
\href{http://www.arxiv.org/abs/1805.07424}{arXiv:1805.07424}.


\bibitem{thomasyau} R.P.~Thomas and S.~T.~Yau.  Special Lagrangians,
  stable bundles and mean curvature flow.  {\em Comm. Anal. Geom. } 10
  (2002), no. 5, 1075--1113.

\bibitem{tfuk} S.~Venugopalan and C.~Woodward. 
Tropical Fukaya algebras.  To appear in {\em Memoirs of the Eur. Math. Soc.}
\href{http://arxiv.org/abs/2004.14314}{arxiv:2004.14314}. 

\bibitem{wo:ex}
K.~Wehrheim and C.~Woodward. 
\newblock Exact triangle for fibered {D}ehn twists. 
\newblock {\em Res. Math. Sciences} 3:17, 2016. 
\newblock \href{http://arxiv.org/abs/1503.07614}{arxiv:1503.07614}.
		
\bibitem{orient}
K.~Wehrheim and C.T. Woodward. 
\newblock Orientations for holomorphic quilts. 
\newblock 
\href{http://www.arxiv.org/abs/1503.07803}{arXiv:1503.07803}. 

\bibitem{weinstein}
A.~Weinstein. 
\newblock Removing intersections of Lagrangian immersions. 
\newblock {\em Illinois J. Math.} 27: 484--500, 1983.
\end{thebibliography}

\end{document}